\documentclass[11pt]{article}
\usepackage[margin=1in,footskip=0.25in]{geometry}
\usepackage{amsmath,amsthm,thmtools,thm-restate,mathrsfs,subcaption}
\usepackage[numbers]{natbib}
\usepackage{amssymb,appendix}
\usepackage{xcolor}
\usepackage{mathtools}
\usepackage{fancybox}
\usepackage{hyperref}
\usepackage{enumitem}
\newcommand{\hell}[2]{d_{\mathsf{hel}}^2\left(#1,#2\right)}
\newcommand{\tv}[2]{d_{\mathsf{TV}}\left(#1,#2\right)}
\newcommand{\MIhell}[2]{\mathbf{I}_{\mathsf{hel}} \left( {{#1}; {#2}}  \right)}
\newcommand{\MItv}[2]{\mathbf{I}_{\mathsf{TV}} \left( {{#1}; {#2}}  \right)}
\newcommand\polylog{\operatorname{polylog}}
\declaretheorem{lemma}
\declaretheorem{theorem}
\declaretheorem{proposition}
\declaretheorem{corollary}
\declaretheorem[style=definition]{definition}
\declaretheorem{remark}
\declaretheorem{fact}
\declaretheorem{assumption}
\declaretheorem{conjecture}
\declaretheorem{informal Result}
\usepackage{bm}
\usepackage{thmtools,thm-restate,xargs}
\usepackage{extarrows}
\usepackage{xcolor}
\usepackage{mathrsfs}

\newcommand{\UT}{{\scriptscriptstyle{\mathsf{T}}}}   %

\newcommand{\E}{\mathbb{E}}
\renewcommand{\P}{\mathbb{P}}
\newcommand{\Q}{\mathbb{Q}}
\newcommand{\Tr}{\mathsf{Tr}}

\newcommand{\dmu}[1]{\mu_{#1}}
\newcommand{\paramspace}{\mathcal{V}}
\newcommand{\dataspace}{\mathcal{X}}
\newcommand{\model}{\mathcal{P}}
\newcommand{\refmu}{\mu_0}
\newcommand{\nullmu}{\overline{\mu}}
\newcommand{\barP}{\overline{\P}}
\newcommand{\barE}{\overline{\E}}

\newcommand{\refE}{\E_0}
\newcommand{\refP}{\P_0}
\newcommand{\looP}[2]{\barP_{#1}^{(#2)}}
\newcommand{\looE}[2]{\barE_{#1}^{(#2)}}

\newcommand{\factor}{F}
\newcommand{\prior}{\pi}
\newcommand{\geometric}{\Psi}
\newcommand{\consthell}{K_{\mathsf{hel}}}
\newcommand{\ssize}{N}

\newcommand{\N}{\mathbb N}
\newcommand{\W}{\mathbb{N}_0}
\newcommand{\R}{\mathbb R}

\newcommand{\sphere}[1]{\mathbb{S}^{#1}}
\newcommand*\diff{\mathop{}\!\mathrm{d}}
\newcommand{\gauss}[2]{\mathcal{N}\left( #1,#2 \right)}
\newcommand{\explain}[2]{\overset{\text{\tiny{#1}}}{#2}}
\newcommand{\Indicator}[1]{\mathbb{I}_{ #1}}
\newcommand{\ip}[2]{\left\langle {#1}, {#2} \right\rangle}
\newcommand{\budget}{b}
\newcommand{\state}{s}
\newcommand{\batch}{n}
\newcommand{\mach}{m}
\renewcommand{\dim}{d}
\newcommand{\iters}{T}

\newcommand{\unif}[1]{\mathsf{Unif} \left( #1 \right)}
\newcommand{\goodevnt}{\mathcal{Z}}

\newcommand{\tensor}[2]{\bigotimes^{#2} #1}
\newcommand{\GaussSpace}[1]{\mathcal{L}_2 \left( \gauss{\bm 0}{\bm I_{#1}} \right)}
\newcommand{\lowdegree}[2]{\left( #1 \right)_{\leq #2}}
\newcommand{\highdegree}[2]{\left( #1 \right)_{> #2}}

\newcommand{\truncate}[2]{{#1}^{\leq #2}}
\newcommand{\op}{\mathsf{op}}

\newcommand{\loss}{\ell}
\newcommand{\nongauss}{\nu}
\newcommand{\filter}{\Delta}
\newcommand{\thermite}[1]{H^{\filter}_{#1}}
\newcommand{\intH}[3]{\overline{H}_{#1}(#2;#3)}
\newcommand{\dscore}[1]{\mathscr{L}_{#1}}
\newcommand{\barscore}{\overline{\mathscr{L}}}
\newcommand{\varproxy}{\vartheta}
\newcommand{\trunc}[1]{\mathscr{T}_{#1}}
\newcommand{\thresh}{h}

\newcommand{\net}[1]{\mathcal{N}_{#1}}
\newcommand{\mview}[2]{{#1}^{(#2)}}
\newcommand{\sgn}{\mathsf{sign}}
\newcommand{\mat}{\mathsf{Mat}}
\renewcommand{\vec}{\mathsf{Vec}}
\newcommand{\estimatespace}{\widehat{\paramspace}}
\newcommand{\pdim}{D}
\newcommand{\psnr}{\Lambda}
\renewcommand{\ln}{\log}

\newenvironment{fminipage}%
  {\begin{Sbox}\begin{minipage}}%
  {\end{minipage}\end{Sbox}\fbox{\TheSbox}}

\newenvironment{algbox}[0]{\vskip 0.2in
\noindent 
\begin{fminipage}{6.3in}
}{
\end{fminipage}
\vskip 0.2in
}
\newcommand{\rd}[1]{#1}
\newcommand{\djh}[1]{#1}

\title{\textbf{Statistical-Computational Trade-offs in Tensor PCA \\ and Related Problems via Communication Complexity}}
\author{Rishabh Dudeja\thanks{Department of Statistics, University of Wisconsin--Madison} \and Daniel Hsu\thanks{Department of Computer Science, Columbia University}}

\begin{document}
\maketitle
\begin{abstract}
  Tensor PCA is a stylized statistical inference problem introduced by Montanari and Richard to study the computational difficulty of estimating an unknown parameter from higher-order moment tensors.
  Unlike its matrix counterpart, Tensor PCA exhibits a statistical-computational gap, i.e., a sample size regime where the problem is information-theoretically solvable but conjectured to be computationally hard.
  This paper derives computational lower bounds on the run-time of memory bounded algorithms for Tensor PCA using communication complexity.
  These lower bounds specify a trade-off among the number of passes through the data sample, the sample size, and the memory required by any algorithm that successfully solves Tensor PCA.
  While the lower bounds do not rule out polynomial-time algorithms, they do imply that many commonly-used algorithms, such as gradient descent and power method, must have a higher iteration count when the sample size is not large enough.
  Similar lower bounds are obtained for Non-Gaussian Component Analysis, a family of statistical estimation problems in which low-order moment tensors carry no information about the unknown parameter.
  Finally, stronger lower bounds are obtained for an asymmetric variant of Tensor PCA and related statistical estimation problems.
  These results explain why many estimators for these problems use a memory state that is significantly larger than the effective dimensionality of the parameter of interest.
\end{abstract}
\newpage
\tableofcontents
\newpage

\section{Introduction}
Many statistical inference problems exhibit a range of sample sizes or signal-to-noise ratios in which it is information-theoretically possible to infer the unknown parameter of interest, but all known (computationally) efficient estimators fail to give accurate inferences.
It is widely conjectured that, for many such problems, no efficient algorithm can produce accurate inferences in these so-called (conjectured) ``hard'' phases, even though there may be efficient algorithms that work if the sample size or signal-to-noise ratio is sufficiently high (i.e., in the ``easy'' phase of the problem).
The existence of such a hard phase is known as a \emph{statistical-to-computational gap}.
Since proofs of such
gaps
are currently out-of-reach, a popular way to give evidence for the
gaps
is to prove that certain restricted classes of estimators fail to solve the inference problems in the conjectured hard phases.
These restrictions are often chosen to capture the techniques used by the best efficient estimators known to date: e.g., sum-of-squares relaxations~\citep{raghavendra2018high}, belief propagation and message passing~\citep{zdeborova2016statistical,bandeira2018notes}, general first-order methods~\citep{celentano2020estimation}, and low-degree polynomial functions~\citep{hopkins2018statistical, kunisky2019notes, schramm2020computational}.

Another way to constrain estimators is to require additional desirable properties, such as:
\begin{enumerate}
\item robustness to deviations from model assumptions,
\item low memory footprint,
\item low communication cost in a distributed computing environment.
\end{enumerate}
If all estimators with these properties were proved to fail in the conjectured hard phases of inference problems, then we would have a satisfying practical theory of statistical-computational gaps.
That is, 
even if efficient estimators exist in the conjectured hard phase of inference problems,
their use in practice would be limited since they would use too much memory, be non-robust to slight model mismatches, etc.

\citet*{steinhardt2016memory} provide another motivation for studying inference problems under such constraints. They hypothesize that computationally easy problems remain solvable even in the face of constraints, such as those related to robustness, memory, and communication. In contrast, hard problems have brittle solutions which are unable to endure such constraints. Hence, hard problems should exhibit certain hallmarks such as the non-robustness, high memory footprint, or high communication cost of efficient estimators. Studying inference problems under such constraints enriches our understanding of the computational complexity of these problems.

% \citet*{steinhardt2016memory} provide another motivation for studying inference problems under such constraints.
% They hypothesize that computationally easy problems can still be solved in the presence of constraints (such as those related to robustness, memory, and communication), whereas
% %
% hard problems have only brittle solutions that cannot withstand these constraints.
% Hence, hard problems should exhibit certain hallmarks such as the non-robustness, high memory footprint, or high communication cost of efficient estimators.
% Studying inference problems under such constraints enriches our understanding of the computational complexity of these problems.

The hypothesis of \citeauthor*{steinhardt2016memory} is supported by the power of \citeauthor{kearns1998efficient}' Statistical Query (SQ) model for explaining known statistical-computational gaps~\citep{kearns1998efficient,feldman2017statistical,feldman2018complexity,wang2015sharp}.
In the SQ model, estimators can only access the dataset by querying summary statistics of the dataset, and they must be tolerant to adversarial perturbations in query responses of magnitude similar to the random fluctuations of these statistics.
For many inference problems believed to exhibit a hard phase, it is known that all efficient estimators that are robust in the SQ-sense will fail to solve these inference problems in that phase (e.g., \citep{feldman2017statistical, feldman2018complexity, wang2015sharp}).

In this paper, we further investigate the hypothesis of \citeauthor*{steinhardt2016memory} by studying Tensor PCA
and related problems that exhibit a similar statistical-computational gap under \emph{memory constraints}.
Our results are, in fact, obtained by studying the effect of \emph{communication constraints}, and then leveraging a reduction from communication-bounded estimation to memory-bounded estimation.

\begin{figure}[t!]
\begin{algbox}
\textbf{Memory bounded estimation algorithm with resource profile $(\ssize,\iters,\state)$.} 

\vspace{1mm}

\textit{Input}: $\bm x_{1:\ssize} = (\bm x_1,\bm x_2,\dotsc,\bm x_{\ssize})$, a dataset of $\ssize$ samples. \\
\textit{Output}: An estimator $\hat{\bm V} \in \estimatespace$. \\
\textit{Variables}: Memory state $\mathtt{state} \in \{0,1\}^\state$, initially all zeros.
\begin{itemize}
\item For iteration $t \in \{1, 2, \dotsc, \iters\}$:
\begin{itemize}
    \item For each sample $i \in  \{1, 2, \dotsc, \ssize\}$:
      \begin{itemize}
        \item[]
                 $\mathtt{state} \gets f_{t,i}(\mathtt{state}, \bm x_i)$
      \end{itemize}
\end{itemize}
\item \textit{Return} estimator $\hat{\bm V} = g(\mathtt{state})$.
\end{itemize}
\vspace{1mm}

\end{algbox}
\caption{Template for memory bounded estimation algorithms with resource profile $(\ssize,\iters,\state)$} \label{fig: memory-bounded-algorithm}
\end{figure}

\subsection{Our Contributions}
We study iterative estimation algorithms that maintain and update an internal memory state of $\state$ bits in the course of $\iters$ passes (iterations) over a dataset of $\ssize$ samples.
A general template of such an iterative algorithm is given in Figure~\ref{fig: memory-bounded-algorithm}, and a formal definition appears in Section~\ref{sec:comp-models}; note that there is no restriction on the functions used to update the state and produce the final estimator.
This class of iterative algorithms \djh{is well-suited for modeling} commonly-used estimators, e.g., spectral estimators (using the power method) and empirical risk minimizers (using gradient descent).
For several statistical inference problems exhibiting a statistical-computational gap, we prove a lower bound on the product of resources $\ssize \cdot \iters \cdot \state$ that all iterative algorithms must use to solve the problem.
In the following, we describe these inference problems and the key results we obtain for each of them.

\subsubsection{Tensor Principal Component Analysis}

In the order-$k$ Tensor Principal Components Analysis ($k$-TPCA) problem introduced by \citet{montanari2014statistical}, one observes $\ssize$ noisy independent realizations $\bm X_{1:\ssize} = (\bm X_1, \bm X_2, \dotsc, \bm X_{\ssize})$ of a rank-$1$ symmetric $k$-tensor (the signal) corrupted by Gaussian noise:
\begin{align*} 
    \bm X_i =  \lambda \bm v^{\otimes k} + \bm W_i, \quad (W_i)_{j_1,j_2, \dotsc, j_k} \explain{i.i.d.}{\sim} \gauss{0}{1}, \quad \forall \; j_1,j_2, \dotsc ,j_k \in [d].
\end{align*} %
 The goal is to estimate the \emph{unit vector} $\bm v$ that generates the signal tensor. The parameter $\lambda\geq 0$ determines the signal-to-noise ratio for this problem. This problem is known to exhibit a sizeable computational gap: in the regime when $\dim \lesssim \ssize \lambda^2 \ll \dim^{k/2}$, it is information-theoretically possible to design consistent estimators for $\bm v$ but no computationally efficient (polynomial-time) estimator is known. Hence, the regime $\dim \lesssim \ssize \lambda^2 \ll \dim^{k/2}$ is the conjectured hard phase for this problem. It is believed that no polynomial-time estimator for $k$-TPCA exists in this regime.
 Our main result for $k$-TPCA (Theorem~\ref{thm:tpca} in Section~\ref{sec:tpca}) is stated informally as follows.
 \begin{informal Result} \label{informal:tpca} Any estimator for $k$-TPCA that uses $\ssize \lambda^2 \cdot \iters \cdot \state \ll \dim^{\lceil (k+1)/2 \rceil}$ total resources fails to solve $k$-TPCA in the regime $\lambda \asymp 1$.
 \end{informal Result}
 
 The resource lower bound in Result~\ref{informal:tpca} is tight for even $k$, because the spectral estimator of \citet{hopkins2016fast} can be implemented by an iterative algorithm
  with $\ssize = \dim^{k/2} \cdot \polylog(\dim)/\lambda^2$ samples,
  $\iters = \mathrm{log}(\dim)$ iterations, and
  $\state = \dim \cdot \polylog(\dim)$ memory bits.
  The most interesting consequence of
  Result~\ref{informal:tpca}
  is that it
  provides \emph{an unconditional separation result} between the easy and conjectured hard phase of $k$-TPCA for (nearly) linear memory iterative algorithms, i.e., algorithms whose resource profiles scale as
     $N \asymp \dim^\eta/\lambda^2$, $\iters \asymp \dim^\tau$, and $\state \asymp \dim$,
 where the sample exponent $\eta = \ln(\ssize)/\ln(\dim)$ and run-time exponent $\tau = \ln(\iters)/\ln(\dim)$ are fixed positive constants. Specifically, for these algorithms:
 \begin{enumerate}
      \item In the conjectured hard phase $1<\eta<k/2$, Result~\ref{informal:tpca} shows that the run-time exponent must be strictly positive: $\tau \geq k/2 - \eta>0$. The run-time exponents ruled out by this lower bound appear as the striped triangular region in the phase diagram in Figure~\ref{fig:tpca-phase-diagram}. 
      \item On the other hand, in the easy phase $\eta > k/2$, the spectral estimator of \citeauthor{hopkins2016fast} provides a nearly-linear memory algorithm whose run-time exponent can be made arbitrarily close to $0$. This algorithm is represented as a green dot at $(\eta = k/2, \; \tau = 0)$ in the phase diagram in Figure~\ref{fig:tpca-phase-diagram}. 
  \end{enumerate}
Since the parameter of interest in this problem $\bm v$ is $\dim$-dimensional, many estimators for this problem (like tensor power method and gradient descent on a likelihood objective) use a memory state of size $\state \asymp \dim$, and thus are subject to the lower bound of Result~\ref{informal:tpca}.

We note that our unconditional separation result is significantly weaker than the conjectured separation.
Since it is believed that there is no polynomial-time estimation algorithm (and hence, certainly not one with linear memory) for $k$-TPCA in the conjectured hard phase $1<\eta<k/2$, no finite run-time exponent should be possible in the hard regime.
Result~\ref{informal:tpca}, however, only rules out run-time exponents $\tau < k/2-\eta$.

Even so, we are not aware of any other approach that yields an unconditional lower bound for linear memory iterative algorithms that is comparable to Result~\ref{informal:tpca}.
In particular, the popular low-degree likelihood ratio framework \citep{hopkins2018statistical,kunisky2019notes} only yields lower bounds of the form $T \gtrsim \log(\dim)$ for many linear-memory iterative algorithms (under additional degree restrictions).
More details regarding the comparison with lower bounds obtained using the low-degree likelihood framework and other consequences of Result~\ref{informal:tpca} appear in Section~\ref{sec:tpca-discussion}.
\begin{figure}[ht]
\centering
\includegraphics[width=0.49\textwidth]{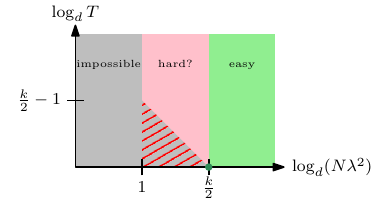}
\caption{Consequences of Result~\ref{informal:tpca} for linear memory $k$-TPCA algorithms ($k$ even).}
\label{fig:tpca-phase-diagram}
\end{figure}

\subsubsection{Non-Gaussian Component Analysis}

\citeauthor{montanari2014statistical} intended $k$-TPCA as a stylized statistical inference problem that captures computational difficulties in extracting information about a parameter of interest from the empirical order-$k$ moment tensor of a dataset.
Taking a cue from this motivation, we study the order-$k$ Non-Gaussian Component Analysis ($k$-NGCA) problem \citep{blanchard2006search}, defined as follows.
The goal is to estimate an unknown unit vector $\bm v$ from $\ssize$ i.i.d.\ realizations of a $\dim$-dimensional non-Gaussian vector $\bm x$ in which: (1) the order-$k$ moment tensor differs from the moment tensor of a standard Gaussian vector $\bm z \sim \gauss{\bm 0}{\bm I_{\dim}}$ along precisely one direction, given by $\bm v$; and (2) for any $\ell \leq k-1$, the order-$\ell$ moment tensor of $\bm x$ is identical to that of the Gaussian vector $\bm z$, and hence it reveals no information about $\bm v$.
We show (Theorem~\ref{thm:ngca} in Section~\ref{sec:ngca}) that a resource lower bound identical to Result~\ref{informal:tpca} holds for $k$-NGCA when the signal-to-noise ratio is sufficiently small as a function of $\dim$.
Since our lower bound applies to a broad family of constructions of the non-Gaussian vector $\bm x$, we obtain, as corollaries to Theorem~\ref{thm:ngca}, similar results for the estimation problems in specific statistical models, including Gaussian mixture models and certain generalized linear models.
These connections are explained in more detail in Section~\ref{sec: ngca-constructions}. 

\begin{figure}
     \centering
     \begin{subfigure}[b]{0.49\textwidth}
         \centering
         \includegraphics[width=\textwidth]{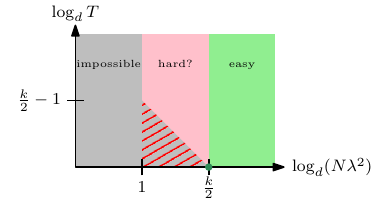}
         \caption{$s \asymp \dim^{k/2}$ bits.}
         \label{fig:atpca-overparametrized}
     \end{subfigure}
     \hfill
     \begin{subfigure}[b]{0.49\textwidth}
         \centering
         \includegraphics[width=\textwidth]{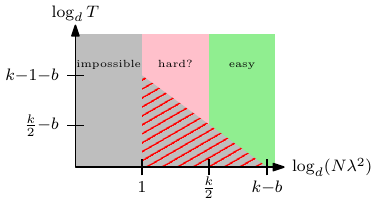}
         \caption{$s \asymp \dim^{b}$ bits, $b < k/2$.}
         \label{fig:atpca-underparameterized}
     \end{subfigure}
\caption{Consequences of Result~\ref{informal:atpca} for $k$-ATPCA algorithms with memory size of $\state \asymp \dim^{\frac{k}{2}}$ bits (left) and $s \asymp \dim^{b}$ bits for $b < k/2$ (right). The striped triangular region represents the run-time vs.\ sample size trade-offs ruled out by Result~\ref{informal:atpca}. The green dot at $(\log_{\dim}(\ssize) = k/2, \log_{\dim}(T) = 0)$ in Figure~\ref{fig:atpca-overparametrized} represents the \citeauthor{montanari2014statistical} estimator.} 
\label{fig:atpca-phase-diagram}
\end{figure}
 
\subsubsection{Asymmetric Tensor PCA}

We also study an \emph{asymmetric} version of $k$-TPCA ($k$-ATPCA), where the
signal tensor is
allowed to be an arbitrary (possibly asymmetric) rank-1 tensor.
The conjectured hard phase for this problem is the sample size regime $\dim \lesssim \ssize \lambda^2 \ll \dim^{k/2}$, which is identical to that for (symmetric) $k$-TPCA.
However, our main result for $k$-ATPCA (Theorem~\ref{thm:main-result-atpca} in Section~\ref{sec:atpca}), informally stated below, shows that this problem requires significantly more resources than $k$-TPCA. 
\begin{informal Result} \label{informal:atpca} Any estimator for $k$-ATPCA that uses $\ssize \lambda^2 \cdot \iters \cdot \state \ll \dim^{k}$ total resources fails to solve $k$-ATPCA in the regime $\lambda \asymp 1$.
 \end{informal Result}
This resource lower bound is tight when $k$ is even, because \citet{zheng2015interpolating} show that the spectral estimator of \citeauthor{montanari2014statistical} yields a consistent estimator $k$-ATPCA when $\ssize \lambda^2 \gtrsim \dim^{k/2} \polylog(\dim)$.
This estimator computes the rank-1 SVD of a $\dim^{k/2} \times \dim^{k/2}$ matrix and hence can be computed by running $\iters = O(\log(\dim))$ iterations of the power method while using a memory state of size $\state = d^{k/2}\polylog(\dim)$ bits.
The total resource requirement for this estimator is $\ssize \lambda^2 \cdot \iters \cdot \state = \dim^k \polylog(\dim)$, which matches the lower bound in Result~\ref{informal:atpca} up to polylogarithmic factors.
Note that the \citeauthor{montanari2014statistical} spectral estimator is, in a sense, overparameterized: it requires a memory state of size $\state \gtrsim \dim^{k/2}$, which is significantly larger than $k\dim$, the effective dimension of the rank-1 signal tensor, the parameter of interest.
A key consequence of Result~\ref{informal:atpca} is that this overparameterization is necessary: estimators that use a memory state of size $s \asymp \dim^b$ with $b<k/2$ have a strictly worse run-time vs.\ sample size trade-off (shown in Figure~\ref{fig:atpca-underparameterized}) compared to sufficiently overparametrized estimators which use a memory state of size $s \asymp \dim^{k/2}$ (shown in Figure~\ref{fig:atpca-overparametrized}). This explains why estimators for $k$-ATPCA ``lift'' the problem to higher dimensions.

\subsubsection{Canonical Correlation Analysis}

Since $k$-TPCA captures the computational difficulties in extracting information about a parameter of interest from the empirical $k$-moment tensor of a dataset, it is natural to expect that $k$-ATPCA should capture the computational difficulties of the same but for the empirical $k$-\emph{cross-moment tensor} of a dataset.
To develop this analogy, we study the order-$k$ Canonical Correlation Analysis problem ($k$-CCA), in which one observes $\ssize$ i.i.d.\ realizations of a $k\dim$-dimensional random vector $\bm x =  (\mview{\bm x}{1}, \mview{\bm x}{2}, \dotsc, \mview{\bm x}{k})$ consisting of $k$ separate $\dim$-dimensional ``views''. The parameter of interest is the order-$k$ cross moment tensor $\E[\mview{\bm x}{1} \otimes \mview{\bm x}{2} \otimes \dotsb \otimes \mview{\bm x}{k}]$, and hard instances of this problem have the property that no other moment tensor of order-$\ell$ with $\ell \leq k$ carries information regarding the parameter of interest. For the $k$-CCA problem, our main result (Theorem~\ref{thm:cca} in Section~\ref{sec:cca}) shows that a resource lower bound identical to Result~\ref{informal:atpca} holds for the $k$-CCA problem in the regime when signal-to-noise ratio is sufficiently small as a function of $\dim$.
Furthermore, since the problem of \emph{learning parity functions} can be reduced to the $k$-CCA instance used to prove the resource lower bound in Theorem~\ref{thm:cca}, we also obtain interesting resource lower bounds for the problem of learning parities.
This is discussed further in Section~\ref{sec:parity}.      

\subsection{Proof Techniques} \label{sec:proof-techniques}

To obtain the resource lower bounds discussed above, we rely on a reduction of \citet{alon1999space}, which was more recently used in the context of statistical inference problems in the works of  \citet{shamir2014fundamental} and \citet{dagan2018detecting}.  This reduction shows that any iterative algorithm that solves a statistical inference problem using few resources (as measured by the product $\ssize \cdot \iters \cdot \state$) can be used to solve the statistical inference problem in a distributed setting with a limited amount of communication between the machines holding the data samples. %
Consequently, the claimed resource lower bounds follow from communication lower bounds for the distributed versions of these inference problems.
For $k$-ATPCA and $k$-CCA problems, the desired communication lower bounds can be obtained from existing communication lower bounds for sparse Gaussian mean estimation \citep{braverman2016communication,acharya2020unified} and correlation detection problems \citep{dagan2018detecting}.
We refer the reader to Remarks~\ref{remark:ATPCA-SGME} and~\ref{remark:CCA-lightbulb} for a detailed discussion.
For $k$-TPCA and $k$-NGCA, we derive new communication lower bounds.
Recent works by \citet{han2018geometric, barnes2020lower,acharya2020unified} have developed general frameworks to prove communication lower bounds distributed statistical inference problems.
However, for $k$-TPCA and $k$-NGCA, we are unable to obtain the desired communication lower bounds using these frameworks (see Section~\ref{sec:prior-frameworks} for more details).
Hence, building on these works, we develop a different framework to obtain communication lower bounds for distributed inference problems.
This approach yields stronger communication lower bounds for $k$-TPCA and $k$-NGCA than those obtained using the works mentioned above.
Furthermore, this framework also yields alternative proofs for the desired communication lower bounds for $k$-ATPCA and the $k$-CCA in a unified manner.

\subsection{Related Work}

Since the work of \citeauthor{montanari2014statistical}, which introduced $k$-TPCA, a number of subsequent works have proposed and analyzed various estimators and proved different kinds of computational lower bounds for this and related problems.

\subsubsection{Hardness of Symmetric and Asymmetric Tensor PCA}

Many works have designed computationally efficient estimators for $k$-TPCA that attain the conjectured optimal sample complexity for polynomial-time estimators $(\ssize \lambda^2 \gtrsim \dim^{k/2})$.
This includes spectral estimators \citep{montanari2014statistical, hopkins2015tensor, zheng2015interpolating, hopkins2016fast, hopkins2017power,biroli2019iron}, sum-of-squares relaxations \citep{hopkins2015tensor,hopkins2016fast,hopkins2017power}, tensor power method with well-designed initializations \citep{anandkumar2017homotopy, biroli2019iron} and higher-order generalizations of belief propagation \citep{wein2019kikuchi}. The spectral estimator of \citet{hopkins2016fast} for $k$-TPCA and the spectral estimator of \citet{montanari2014statistical} for $k$-ATPCA (discussed in detail in Section~\ref{sec:tpca-discussion} and Section~\ref{sec:atpca-discussion}) are particularly relevant for our work. The resources used by these estimators (as measured by the product $\ssize \cdot \iters \cdot \state$) nearly match our resource lower bounds for $k$-TPCA and $k$-ATPCA, respectively. Several works have shown that many natural classes of estimators fail in the conjectured hard phase for $k$-TPCA ($\dim \lesssim \ssize \lambda^2  \ll \dim^{k/2}$). This includes sum-of-squares relaxations \citep{hopkins2015tensor,hopkins2016fast,hopkins2017power,bhattiprolu2016sum},  estimators that compute a low degree polynomial of the dataset \citep{hopkins2018statistical,kunisky2019notes}, and SQ algorithms \citep{brennan2020statistical,dudeja2020statistical}. A more detailed comparison with the low-degree lower bounds appears in Section~\ref{sec:tpca-discussion}. The landscape of the maximum likelihood objective for this problem has been shown to have numerous spurious critical points \citep{arous2019landscape,ros2019complex} and it is known that Langevin dynamics on the maximum likelihood objective fails to solve $k$-TPCA in the conjectured hard phase \citep{arous2020algorithmic}. Finally, using average-case reductions, it has been shown that the hardness of Hypergraph Planted Clique implies the hardness of $k$-TPCA \citep{zhang2018tensor,brennan2020reducibility}.

\subsubsection{Hardness of Non-Gaussian Component Analysis}

The $k$-NGCA problem was formally introduced by \citet{blanchard2006search}, and
various computationally efficient estimators have been proposed and analyzed \citep{vempala2011structure,tan2018polynomial,goyal2019non,mao2021optimal, davis2021clustering}.
These estimators have a sample size requirement which is significantly more than the information-theoretic sample size requirement.
In the special case when the distribution of the non-Gaussian direction is discrete, \citet{zadik2021lattice} and \citet{diakonikolas2021non} have designed computationally efficient algorithms that recover the non-Gaussian direction with the information-theoretically optimal sample complexity.
However, these algorithms are brittle and break down when the distribution of the non-Gaussian component is sufficiently nice (e.g., absolutely continuous with respect to the standard Gaussian distribution; see Remark \ref{remark:lattice} for additional details).
In this situation, \citet{diakonikolas2017statistical} have identified a sample size regime where SQ algorithms fail to identify the non-Gaussian direction with polynomially many queries.
This suggests that this problem is computationally hard in this regime.
A problem closely related to $k$-NGCA problem is the Continuous Learning With Errors problem \citep{bruna2021continuous}.
\citet{bruna2021continuous} show that this problem is computationally hard provided that a plausible conjecture from cryptography is true \citep[Conjecture 1.2]{micciancio2009lattice}.
Since $k$-NGCA is connected to many other inference problems, the SQ lower bounds of \citeauthor{diakonikolas2017statistical} are at the heart of SQ lower bounds for many other robust estimation and learning problems \citep{NGCAadversarial,diakonikolas2017statistical,NGCArobustlinear,NGCAsupervised1,NGCAsupervised2,diakonikolas2021optimality,NGCAsupervised4,NGCAsupervised5}. 

\subsubsection{Memory and Communication Lower Bounds for Statistical Inference}

Finally, we mention some closely related works which study memory or communication lower bounds for statistical inference. 
As mentioned previously, our lower bounds leverage a reduction reduction of \citet{alon1999space}, which was more recently used in the context of statistical inference problems in the works of  \citet{shamir2014fundamental} and \citet{dagan2018detecting}.
This reduction allows us to derive memory lower bounds for iterative estimation algorithms using communication lower bounds for distributed estimators.
Consequently, the works of \citet{han2018geometric,barnes2020lower,acharya2020unified}, which develop unified frameworks to derive communication lower bounds for distributed estimators, are particularly relevant to our work.
These works are discussed in greater detail in Section \ref{sec:proof-techniques} and Section \ref{sec:prior-frameworks}.
The $k$-ATPCA problem is closely related to the sparse Gaussian mean estimation problem, and the $k$-CCA problem is closely related to the correlation detection problem.
Communication lower bounds for these problems were obtained in the prior works of \citet{braverman2016communication} and \citet{dagan2018detecting}, respectively.
These connections are discussed formally in Remarks~\ref{remark:ATPCA-SGME} and~\ref{remark:CCA-lightbulb}.
A line of work \citep{steinhardt2016memory,raz2018fast,kol2017time, moshkovitz2017mixing, moshkovitz2017NN,garg2018extractor, beame2018time,moshkovitz2018entropy,garg2019time,sharan2019memory,garg2021memory} initiated by \citet{steinhardt2016memory} and \citet{raz2018fast} provides another approach to obtain memory lower bounds without relying on the connection with distributed inference problems.
We provide a comparison with lower bounds obtained using this approach in Section \ref{sec:raz-comparision}.  

\section{Statistical Inference and Computational Constraints} \label{sec:comp-models}

In this section, we introduce notations used through out this paper, the setup for general statistical inference problems, and the computational constraints on estimators that we study.

\subsection{Notation}

\paragraph{Important sets:} $\N$ and $\R$ denote the set of positive integers and the set of real numbers, respectively. $\W \explain{def}{=} \N \cup \{0\}$ is the set of non-negative integers. For each $k,\dim \in \N$, $[k]$ denotes the set $\{1, 2, 3, \dotsc, k\}$, $\R^\dim$ denotes the $\dim$-dimensional Euclidean space, $\sphere{\dim-1}$ denotes the unit sphere in $\R^\dim$, $\R^{\dim \times k}$ denotes the set of all $\dim \times k$ matrices, $\tensor{\R^\dim}{k}$ denotes the set of all $\dim \times \dim \times \dim \times \dotsb \times \dim$ ($k$ times) tensors with $\R$-valued entries, and $\tensor{\W^\dim}{k}$ denotes the set of all $\dim \times \dim \times \dim \times \dotsb \times \dim$ ($k$ times) tensors with $\W$-valued entries.

\paragraph{Linear Algebra:} We denote the $\dim$-dimensional vectors $(1, 1, \dotsc, 1)$, $(0, 0, \dotsc, 0)$ and the $\dim \times \dim$ identity matrix using $\bm 1_{\dim}$, $\bm 0_{\dim}$, and $\bm I_{\dim}$ respectively. We will omit the subscript $\dim$ when the dimension is clear from the context. The vectors $\bm e_1, \bm e_2, \dotsc, \bm e_{\dim}$ denote the standard basis vectors of $\R^{\dim}$. For a vector $\bm v \in \R^{\dim}$, $\|\bm v\|, \|\bm v\|_1, \|\bm v\|_\infty$ denote the $\ell_2, \ell_1$, $\ell_\infty$ norms of $\bm v$, and $\|\bm v\|_{0}$ denotes the sparsity (number of non-zero entries) of $\bm v$. For two vectors $\bm u, \bm v \in \R^{\dim}$, $\ip{\bm u}{\bm v}$ denotes the standard inner product on $\R^{\dim}$: $\ip{\bm u}{\bm v} \explain{def}{=} \sum_{i=1}^{\dim} u_i v_i$. For two matrices or tensors $\bm U$ and $\bm V$, we analogously define $\|\bm U\|, \|\bm U\|_1, \|\bm U\|_\infty, \|\bm U\|_0$, and $\ip{\bm U}{\bm V}$ by stacking their entries to form a vector. For a matrix $\bm A$, $\bm A^\UT$ denotes the transpose of $\bm A$ and $\|\bm A\|_{\op}$ denotes the operator (or spectral) norm of $\bm A$. For a square matrix $\bm A$, $\Tr(\bm A)$ denotes the trace of $\bm A$. Finally, for vectors $\bm v_{1:k} \in \R^{\dim}$, $\bm v_1 \otimes \bm v_2 \otimes \dotsb \otimes \bm v_k$ denotes the $k$-tensor with entries $(\bm v_1 \otimes \bm v_2 \otimes \dotsb \otimes \bm v_k)_{i_1, i_2, \dotsc, i_k} =(\bm v_1)_{i_1} \cdot (\bm v_2)_{i_2} \dotsb (\bm v_k)_{i_k}$ for $i_{1:k} \in [\dim]$. When $\bm v_1 = \bm v_2 = \dotsb = \bm v_k = \bm v$, we shorthand $\bm v \otimes \bm v \otimes \dotsb \otimes \bm v$ as $\bm v^{\otimes k}$. Analogously, given two tensors $\bm U \in \tensor{\R^{\dim}}{\ell}$ and $\bm V \in \tensor{\R^{\dim}}{m}$, $\bm U \otimes \bm V$ is the $(\ell+m)$-tensor with entries $(\bm U \otimes \bm V)_{i_1, i_2, \dotsc, i_\ell,j_1, j_2, \dotsc, j_{m}} = (\bm U)_{i_1, i_2, \dotsc, i_{\ell}} \cdot (\bm V)_{j_1, j_2, \dotsc, j_{m}}$ for $i_{1:\ell} \in [\dim], j_{1:m} \in [\dim]$. This definition is naturally extended to define the $(\ell_1 + \ell_2 + \dotsb + \ell_k)$-tensor $\bm U_1 \otimes \bm U_2 \otimes \dotsb \otimes \bm U_{k}$ for tensors $\bm U_{1:k}$ with $\bm U_i \in \tensor{\R^\dim}{\ell_i}$ for each $i \in [k]$. 

\paragraph{Asymptotic notation:} Given a two non-negative sequences $a_{\dim}$ and $b_{\dim}$ indexed by $\dim \in \N$, we use the following notations to describe their relative magnitudes for large $\dim$. We say that $a_{\dim} \lesssim b_{\dim}$ or $a_{\dim} = O(b_{\dim})$ or $b_{\dim} = \Omega(a_{\dim})$ if $\limsup_{\dim \rightarrow \infty} ( a_{\dim}/b_{\dim}) < \infty$. If $a_{\dim} \lesssim b_{\dim}$ and $b_{\dim} \lesssim a_{\dim}$, then we say that $a_{\dim} \asymp b_{\dim}$.  If there exists a constant $\epsilon > 0$ such that $a_{\dim} \cdot  \dim^{\epsilon} \lesssim b_{\dim}$ we say that $a_{\dim} \ll b_{\dim}$. We use $\polylog{(\dim)}$ to denote any sequence $a_{\dim}$ such that $a_{\dim} \asymp \log^t(\dim)$ for some fixed constant $t \geq 0$. 

\paragraph{Important distributions:} $\gauss{0}{1}$ denotes the standard Gaussian measure on $\R$, and $\gauss{\bm 0}{\bm I_{\dim}}$ denotes the standard Gaussian measure on $\R^{\dim}$. For any finite set $A$, $\unif{A}$ denotes the uniform distribution on the elements of $A$. 

\paragraph{Hermite polynomials:} We will make extensive use of the Hermite polynomials $\{ H_i : i \in \W\}$ which are the orthonormal polynomials for the Gaussian measure $\gauss{0}{1}$ and their multivariate analogs $\{H_{\bm c}: \bm c \in \W^\dim\}$, which are the orthornormal polynomials for the $\dim$-dimensional Gaussian measure $\gauss{\bm 0}{\bm I_{\dim}}$. We provide the necessary background regarding Hermite polynomials and analysis on the Gaussian Hilbert space in Appendix \ref{fourier_gauss_appendix}. 

\paragraph{Miscellaneous:} For an event $\mathcal{E}$, $\Indicator{\mathcal{E}}$ denotes the indicator random variable for $\mathcal{E}$. For $x,y \in \R$, $x \vee y$ and $x \wedge y$ denote $\max(x,y)$ and $\min(x,y)$, respectively; and $\sgn(x)$ denotes the sign function ($\sgn(x) = 1$ iff $x>0$, $\sgn(x) = -1$ iff $x<0$ and $\sgn(0) = 0$). For $x>0$, $\log(x)$ denotes the natural logarithm (base $e$) of $x$.

\subsection{Statistical Inference Problems}

A general statistical inference problem is specified by a \emph{model} $\model$, which is a collection of probability distributions on a space $\dataspace$.
Elements of $\model$ are indexed by a \emph{parameter} $\bm V \in \paramspace$, so $\model = \{ \dmu{\bm V}: \bm V \in \paramspace\}$, where $\paramspace$ is the \emph{parameter set}.
A statistical inference problem can be thought of as a game between nature and a statistician.
First, nature picks a parameter $\bm V \in \paramspace$, which is not revealed to the statistician.
Then, the $N$ samples $\bm x_{1:\ssize} = (\bm x_1, \bm x_2, \dotsc, \bm x_N)$ are drawn i.i.d.\ from $\dmu{\bm V}$ and revealed to the statistician.
The statistician constructs an \emph{estimator} $\hat{\bm V}(\bm x_{1:\ssize}) \in \estimatespace$ using the dataset $\bm x_{1:\ssize}$ and incurs a loss $\ell(\bm V, \hat{\bm V})$, where $\loss: \paramspace \times \estimatespace \rightarrow [0,\infty)$ is the \emph{loss function}. An estimator $\hat{\bm V}: \dataspace^{\ssize} \rightarrow \estimatespace$ is \emph{$(\epsilon,\delta)$-accurate} if
\begin{align} \label{eq: eps-delta-accurate-estimator}
    \sup_{\bm V \in \paramspace} \P_{\bm V} \left( \ell(\bm V, \hat{\bm V}(\bm x_{1:\ssize})) \geq \epsilon \right) & < \delta. 
\end{align}
The statistician's goal is to construct an \emph{$(\epsilon,\delta)$-accurate} estimator with the smallest sample size $\ssize$. 

\subsection{Memory Bounded and Communication Bounded Estimators}

We first present a general framework for studying memory bounded estimation algorithms.

\begin{definition}[Memory bounded estimation algorithm with resource profile $(\ssize, \iters, \state)$] \label{def:memory-bounded-algorithm}
  A \emph{memory bounded estimation algorithm} with \emph{resource profile} $(\ssize, \iters, \state)$ computes an estimator by making $\iters$ passes through a dataset of $\ssize$ samples using a memory state of $\state$ bits (initially all zeros).
  Such an algorithm is specified by the \emph{update functions} $f_{t,i}: \{0,1\}^\state \times \dataspace \rightarrow \{0,1\}^\state$ and an \emph{estimator function} $g: \{0,1\}^\state \rightarrow \estimatespace$, which are used as follows.
  In the $t$-th pass through the dataset, the algorithm considers each sample $\bm x_i$ for $i \in [\ssize]$ in sequence, and it updates the memory state by applying the update function $f_{t,i}$ to the current memory state and the sample $\bm x_i$ under consideration.
  After all $\iters$ passes are complete, the estimator is computed by applying the estimator function $g$ to the final memory state.
  A template for a general memory bounded algorithm is given in Figure \ref{fig: memory-bounded-algorithm}.
\end{definition}

We measure the computational cost of an estimation algorithm by $\iters$, the number of passes it makes through the dataset. This cost measure is not sensitive to the size of the dataset. Furthermore, the update and estimator functions are permitted to be aribtrary functions, and we do not consider their computational cost in our lower bounds.
This means that the lower bounds are conservative, in that a more detailed accounting of their costs in a concrete computational model would only improve our lower bounds.

Although our primary focus is on proving lower bounds for memory bounded estimation algorithms, these lower bounds are consequences of communication lower bounds for distributed estimation protocols in the ``blackboard'' model of communication, introduced next.

\begin{definition}[Distributed estimation protocol with parameters $(\mach, \batch, \budget)$] \label{def: distributed-algorithm}
  A \emph{distributed estimation protocol} with parameters $(\mach,\batch,\budget)$ computes an estimator based on a dataset $\{\bm x_{i,j} \in \dataspace : i \in [\mach], \; j \in [\batch] \}$ of $\ssize = \mach \batch$ samples that are distributed across $\mach$ machines, with $\batch$ samples $\bm X_i = \{\bm x_{i,j} : j \in [\batch] \}$ per machine, after each machine writes at $\budget$ bits to a (public) blackboard.
  The execution of the protocol occurs in a sequence of $\mach \budget$ rounds; a single bit is written on the blackboard per round.
  In round $t$: (1) a machine $\ell_t \in [\mach]$ is chosen as a function of the current contents of the blackboard $\bm Y_{< t} = (Y_1, Y_2, \dotsc, Y_{t-1}) \in \{0,1\}^{t-1}$; then, (2) machine $\ell_t$ computes a Boolean function of the local dataset $\bm X_{\ell_t}$ stored on machine $\ell_{t}$, as well as the current contents of the blackboard $\bm Y_{<t}$; and finally, (3) the output $Y_t \in \{0,1\}$ of the function computed by machine $\ell_t$ is then written on the blackboard.
  Each machine is chosen in $\budget$ rounds.
  At the end of the $\mach \budget$ rounds, the estimator is computed as a function of the final contents of the blackboard $\bm Y \in \{0,1\}^{\mach \budget}$. 
  A general template for a distributed estimation protocol is shown in Figure~\ref{fig: distributed-algorithm}.
\end{definition}

\begin{figure}[t!]
\begin{algbox}
\textbf{Distributed estimation protocol with parameters $(\mach,\batch,\budget)$.} 

\vspace{1mm}

\textit{Input}: $\{\bm x_{i,j} : i \in [\mach], \; j \in [\batch]\}$, a dataset of $\ssize=\mach \batch$ samples distributed across $\mach$ machines, with machine $i$ receiving $\batch$ samples $\bm X_i = \{\bm x_{i,j}: j \in [\batch]\}$. \\
\textit{Output}: An estimator $\hat{\bm V} \in \estimatespace$. \\
\textit{Variables}: Contents of the blackboard $\bm Y \in \{0,1\}^{\mach \budget}$.
\begin{itemize}
\item For round $t \in \{1, 2, \dotsc, \mach \budget\}$
\begin{itemize}
    \item Select machine $\ell_t \in [\mach]$, a function of $\bm Y_{<t}$.
    \item Machine $\ell_t$ writes bit $Y_t$, a function of $(\bm X_{\ell_t}, \bm Y_{<t})$, on the blackboard.
\end{itemize}
\item \textit{Return} estimator $\hat{\bm V}$, a function of $\bm Y$.
\end{itemize}
\vspace{1mm}

\end{algbox}
\caption{Template for distributed estimation protocols with parameters $(\mach,\batch,\budget)$.} \label{fig: distributed-algorithm}
\end{figure}

The connection between the memory bounded computational model and the distributed computational model is encapsulated in Fact~\ref{fact:reduction}, below, formalized by \citet{shamir2014fundamental} and \citet{dagan2018detecting}.
It is a consequence of a simple reduction of \citet{alon1999space} that simulates a memory bounded estimation algorithm using a distributed estimation protocol:
the machines take turns to simulate the algorithm's passes over the dataset, with one machine concluding its turn by writing the memory state on the blackboard so the next machine can continue the simulation.

\begin{fact}[\citep{alon1999space,shamir2014fundamental,dagan2018detecting}]\label{fact:reduction} A memory bounded estimation algorithm with resource profile $(\ssize,\iters,\state)$ can be simulated using a distributed estimation protocol with parameters $(\ssize/\batch, \batch, \state \iters)$ for any $\batch \in \N$ such that $\ssize/\batch \in \N$.
\end{fact}

We rely on Fact~\ref{fact:reduction} to convert lower bounds for distributed estimation protocols to lower bounds for memory bounded estimation algorithms. 
Note that in the reduction, there is some flexibility in the choice of $\batch$, the number of samples per machine.
In our use of Fact~\ref{fact:reduction}, we will set $\batch$ in a way that gives us the most interesting lower bounds for memory bounded algorithms.

\begin{remark}[Deterministic vs.\ Randomized Distributed Estimation Protocols]\label{rem:deterministic}
  In Definition~\ref{def: distributed-algorithm}, we have defined distributed estimation protocols to be deterministic, so they do not use any additional randomness apart from the dataset.
  However, all of our lower bounds for determistic protocols also apply to randomized protocols, in which all computations (of the $\ell_t$'s, $Y_t$'s, and $\hat{\bm V}$) are permitted to additionally depend on a (shared) uniformly random bit vector $\bm q \in \{0,1\}^R$.
  This is because, to rule out $(\epsilon,\delta)$-accurate distributed estimators, we study the Bayesian version of the inference problem, in which the parameter is drawn from a prior $\bm V \sim \pi$.
  In the Bayesian problem, there is no advantage of using a randomized protocol: one can always use the deterministic protocol corresponding to the bit vector $\bm q$ that achieves the lowest Bayes risk (averaged over the realization of $\bm V \sim \prior$). This deterministic protocol is guaranteed to perform as well as the original randomized protocol. 
\end{remark}

\section{Lower Bounds for Distributed Estimation Protocols} \label{sec:framework}

As a consequence of the reduction from memory bounded estimation to communication bounded estimation, we focus our attention on proving lower bounds for distributed estimation protocols.
In this section, we introduce a general lower bound technique for showing that if an estimator $\hat{\bm V}$ is computed by a distributed estimation protocol using insufficiently-many resource (as measured by the parameters $(\mach, \batch, \budget)$), then it is not $(\epsilon,\delta)$-accurate (for suitable choices of $\epsilon$ and $\delta$):
\begin{align*}
    \sup_{\bm V \in \paramspace} \P_{\bm V} \left( \ell(\bm V, \hat{\bm V}) \geq \epsilon \right) & \geq \delta.
\end{align*}
To show this, we consider the Bayesian (a.k.a.\ average-case) version of the statistical inference problem, in which nature draws the parameter $\bm V$ from a prior $\prior$ on the parameter space $\paramspace$.
Since,
\begin{align*}
    \sup_{\bm V \in \paramspace} \P_{\bm V} \left( \ell(\bm V, \hat{\bm V}) \geq \epsilon \right) & \geq \int \P_{\bm V} \left( \ell(\bm V, \hat{\bm V}) \geq \epsilon \right) \prior(\diff \bm V),
\end{align*}
it is enough to show that the RHS of the above display is at least $\delta$. In order to do so, we will rely on Fano's Inequality for Hellinger Information \citep{chen2016bayes}, which we introduce next. 

\subsection{Hellinger Information and Fano's Inequality}

Recall that in a statistical inference problem, the $\ssize$ samples $\{\bm x_{i,j}: \; i \in [\mach], \; j \in [\batch]\} \subset \dataspace$ are drawn i.i.d.\ from $\dmu{\bm V}$.
In the present distributed setting, the samples are distributed across $\mach = \ssize/\batch$ machines, with $\batch$ samples per machine.
The dataset at machine $i$ is denoted by $\bm X_i \in \dataspace^{\batch}$.
The machines then communicate via a distributed estimation protocol to write a transcript $\bm Y \in \{0,1\}^{\mach \budget}$ on the blackboard; the final estimator $\hat{\bm V}$ is only a function of $\bm Y$.
Let $\P(\bm Y = \bm y | \bm X_{1:\mach})$ denote the conditional probability that the final transcript is $\bm y \in \{0,1\}^{\mach \budget}$ given the datasets $\bm X_{1:\mach}$.
Now define
\begin{align} 
    \dmu{\bm V}(\diff \bm X_i) &\explain{def}{=}  \prod_{j=1}^\batch \dmu{\bm V}(\diff \bm x_{i,j}), \label{eq:law-per-machine-data}\\
        \P_{\bm V}(\bm Y = \bm y) & \explain{def}{=} \int \P(\bm Y = \bm y | \bm X_{1:m}) \cdot \dmu{\bm V}(\diff \bm X_{1})  \cdot \dmu{\bm V}(\diff \bm X_{2}) \dotsm \dmu{\bm V}(\diff \bm X_{\mach}).
\end{align}
In words, $\dmu{\bm V}(\diff \bm X_i)$ and $\P_{\bm V}(\bm Y = \bm y)$ are, respectively, the marginal laws of $\bm X_i$ (the dataset at machine $i$) and the blackboard transcript $\bm Y$ when the parameter picked by nature is $\bm V$.
We compare two distributions $\P_1$ and $\P_2$ on $\{0,1\}^{\mach \budget}$ using the \emph{squared Hellinger distance}, defined by
\begin{align*}
    \hell{\P_1}{\P_2} & = \frac{1}{2} \sum_{\bm y \in \{0,1\}^{\mach\budget}} \left( \sqrt{\P_{1}(\bm Y = \bm y)} - \sqrt{\P_{2}(\bm Y = \bm y)} \right)^2.
\end{align*}
With these preliminary definitions in place, we now define the \emph{Hellinger Information between the parameter $\bm V$ and the blackboard transcript $\bm Y$} by
\begin{align} \label{eq: hellinger_definition}
    \MIhell{\bm V}{\bm Y} & \explain{def}{=} \inf_{\Q}  \int \hell{\P_{\bm V}}{\Q} \prior(\diff \bm V),
\end{align}
where the infimum is taken over all probability measures on $\{0,1\}^{m\budget}$.

Fano's Inequality for Hellinger Information (due to \citet{chen2016bayes}) provides a lower bound on the error of any estimator for $\bm V$ based on the transcript $\bm Y$ in terms of the Hellinger Information $\MIhell{\bm V}{\bm Y}$ between $\bm V$ and $\bm Y$.

\begin{fact}[Fano's Inequality for Hellinger Information \citep{chen2016bayes}] \label{fact: fano} Let $\ell : \paramspace \times \estimatespace \rightarrow \{0,1\}$ be an arbitrary 0-1 loss. Let $\pi$ be an arbitrary prior on $\paramspace$. Define
\begin{align*}
    R_0(\prior) \explain{def}{=} \min_{\bm u \in \estimatespace} \left( \int_\paramspace \ell(\bm V, \bm u) \prior (\diff \bm V)  \right).
\end{align*}
Then, for any estimator $\hat{\bm V} : \{0,1\}^{m\budget} \rightarrow \estimatespace$, we have
\begin{align*}
    \int_{\paramspace} \E_{\bm V}[\ell(\bm V, \hat{\bm V}(\bm Y))] \; \prior(\diff \bm V) & \geq R_0(\pi) -  \sqrt{2\MIhell{\bm V}{\bm Y}}.
\end{align*}
In the above display,  $\MIhell{\bm V}{\bm Y}$ denotes the Hellinger information between the random variables: $\bm V \sim \prior$ and $\bm Y \sim \P_{\bm V}$.
\end{fact}
\begin{proof}
The above claim is a minor modification of a result proved by \citet[Corollary 7, item (ii)]{chen2016bayes}. We provide a derivation in Appendix~\ref{appendix:fano} for completeness. 
\end{proof}

Note that $R_0(\prior)$ is the lowest possible estimation error when the transcript $\bm Y$ is not observed.  The above fact says that $\sqrt{2\MIhell{\bm V}{\bm Y}}$ is an upper bound on the reduction in estimation error possible by leveraging information contained in the transcript $\bm Y$. Since we wish to lower bound
\begin{align*}
    \int \P_{\bm V} \left( \ell(\bm V, \hat{\bm V}) \geq \epsilon \right) \prior(\diff \bm V),
\end{align*}
we will apply Fano's inequality with the 0-1 loss $\widetilde{\ell}$ defined as follows:
\begin{align*}
    \widetilde{\ell}(v,\hat{v}) & \explain{def}{=} \begin{cases} 0 & \text{if } \ell(v,\hat{v}) < \epsilon , \\ 1 &\text{if }  \ell(v,\hat{v}) \geq \epsilon . \end{cases}
\end{align*}

\subsection{Information Bound for Distributed Estimation Protocols}

Next, we present a general upper bound on $\MIhell{\bm V}{\bm Y}$ for distributed estimation protocols. 
\begin{proposition}\label{prop: main_hellinger_bound} Let
\begin{enumerate}
\item $\prior$ be a prior distribution on the parameter space $\paramspace$;
\item $\refmu$ be a reference probability measure on $\dataspace^{\batch}$ such that $\dmu{\bm V} \ll \refmu$ for all $\bm V \in \paramspace$;
\item $\nullmu$ be a null probability measure on $\dataspace^{\batch}$ such that $\nullmu$ and $\refmu$ are mutually absolutely continuous;
\item $\goodevnt \subset \dataspace^{\batch}$ be an event such that
\begin{align*}
    \left\{ \bm X \in \dataspace^{\batch}: \left|\frac{\diff \nullmu}{\diff \refmu} (\bm X) - 1 \right| \leq \frac{1}{2}  \right\} & \subset \goodevnt ,
\end{align*}
and let $Z_i$ for $i \in [\mach]$ be the indicator random variables defined by $Z_i \explain{def}= \Indicator{\bm X_i \in \goodevnt}$.

\end{enumerate}
%Let $\looE{0}{i}[\cdot]$ and $\looE{0}{i}[\cdot |\cdot]$ denote unconditional and conditional expectations under 
\rd{Consider} a hypothetical setup in which machine $i \in [m]$ is exceptional, and the data $\bm X_{1:m}$ are sampled independently as follows:
\begin{align*}
    (\bm X_j)_{j \neq i} \explain{i.i.d.}{\sim} \nullmu, \quad \bm X_i \sim \refmu.
\end{align*}
\rd{Let} $\looP{0}{i}$ and $\looE{0}{i}$ denote the probabilities and expectations in this setup:
\begin{align*}
        \looP{0}{i}(\bm Y = \bm y) & \explain{def}{=} \int \P(\bm Y = \bm y | \bm X_{1:m}) \; \refmu(\diff \bm X_i) \cdot  \prod_{j\neq i }  \nullmu(\diff \bm X_{j}), \\
        \looE{0}{i} f(\bm X_{1:m}, \bm Y) & \explain{def}{=} \int \sum_{\bm y \in \{0,1\}^{\mach\budget}} f(\bm X_{1:m}, \bm y) \;  \P(\bm Y = \bm y | \bm X_{1:m}) \; \refmu(\diff \bm X_i) \cdot  \prod_{j\neq i }  \nullmu(\diff \bm X_{j}) .
\end{align*}
\rd{Also let} $\looE{0}{i}[\cdot |\cdot ]$ denote conditional expectations in this setup.

There is a universal constant $\consthell$ such that, if $\bm V \sim \prior$, $\bm X_{1:\mach} \explain{i.i.d.}{\sim} \dmu{\bm V}$, and $\bm Y$ is the transcript produced by a distributed estimation protocol with parameters $(\mach, \batch, \budget)$, then
\begin{align*}
     \MIhell{\bm V}{\bm Y} & \explain{}{\leq} {\consthell}\sum_{i=1}^m   \looE{0}{i}\left[ Z_i \cdot \int \left({\looE{0}{i}\left[  \frac{\diff \dmu{\bm V}}{\diff \refmu} (\bm X_i) - \frac{\diff \nullmu}{\diff \refmu} (\bm X_i)  \bigg| \bm Y, Z_i, (\bm X_j)_{j \neq i}  \right]}\right)^2 \prior(\diff \bm V) \right]  \\& \hspace{8cm} +  \frac{m\consthell }{2} \left( \int \dmu{\bm V}(\goodevnt^c) \pi(\diff \bm V) +  \nullmu(\goodevnt^c) \right).
\end{align*}
\end{proposition}
\begin{proof}
  The proof of this result is presented in Appendix~\ref{sec: proof-hellinger-bound}.  
\end{proof}

In order to apply Proposition~\ref{prop: main_hellinger_bound}, one needs to suitably choose the reference measure $\refmu$, null measure $\nullmu$, and the event $\goodevnt$.
The considerations involved in these choices are as follows:
\begin{enumerate}
    \item The reference measure $\refmu$ is chosen so that it is easy to analyze the concentration behavior of the following likelihood ratios when $\bm X \sim \refmu$:
    \begin{align*}
        \frac{\diff \dmu{\bm V}}{\diff \refmu} (\bm X ), \quad \frac{\diff \nullmu}{\diff \refmu} (\bm X) .
    \end{align*}
    Typically, $\refmu$ will be the standard Gaussian measure over $\dataspace^{\batch}$.
    \item We will often set the null measure $\nullmu = \refmu$. However, in some cases, we will be able to obtain improved lower bounds with the following choice:
    \begin{align*}
        \nullmu(\cdot) = \int \dmu{\bm V}(\cdot) \; \prior(\diff \bm V). 
    \end{align*}
    The above measure is the marginal law of the dataset in a single machine after integrating out $\bm V \sim \prior$. 
    \item Finally, we will typically set $\goodevnt$ minimally as
    \begin{align*}
        \goodevnt = \left\{ \bm X \in \dataspace^{\batch}: \left|\frac{\diff \nullmu}{\diff \refmu} (\bm X) - 1 \right| \leq \frac{1}{2}  \right\}.
    \end{align*}
    However, in some cases, we will find it helpful to enrich $\goodevnt$ with other high probability events that facilitate the analysis of (our upper bound on) Hellinger information.
\end{enumerate}

\subsection{Linearization}

To use our upper bound on Hellinger information, we develop upper bounds on
\begin{align*}
\geometric^2_i(\bm y, z_i, (\bm X_j)_{j \neq i}) \explain{def}{=} \int \left({\looE{0}{i}\left[  \frac{\diff \dmu{\bm V}}{\diff \refmu} (\bm X_i) - \frac{\diff \nullmu}{\diff \refmu} (\bm X_i)  \bigg| \bm Y = \bm y, Z_i = z_i, (\bm X_j)_{j \neq i} \right]}\right)^2 \prior(\diff \bm V).
\end{align*}
A useful technique to control $\geometric^2_i(\bm y, z_i, (\bm x_j)_{j \neq i})$ is linearization, described in the following lemma. 
\begin{lemma}[Linearization]  \label{lemma: linearization} We have
\begin{align*}
    \Psi(\bm y, z_i, (\bm X_j)_{j \neq i}) & = \sup_{\substack{S: \paramspace \rightarrow \R \\ \|S\|_{\prior} \leq 1}} {\looE{0}{i}\left[  \ip{\frac{\diff \dmu{\bm V}}{\diff \refmu} (\bm X_i) - \frac{\diff \nullmu}{\diff \refmu} (\bm X_i)}{S}_{\prior}  \bigg|\bm Y = \bm y, Z_i = z_i, (\bm X_j)_{j \neq i} \right]},
\end{align*}
where $\|\cdot\|_\prior$ and $\ip{\cdot}{\cdot}_{\prior}$ denote the $L_2$ norm and inner product with respect to the prior $\pi$:
\begin{align*}
    \|S\|_\prior^2 = \int S^2(\bm V) \prior (\diff \bm V), \\
    \ip{\frac{\diff \dmu{\bm V}}{\diff \refmu} (\bm X_i) - \frac{\diff \nullmu}{\diff \refmu} (\bm X_i)}{S}_{\prior} & = \int S(\bm V) \cdot \left( \frac{\diff \dmu{\bm V}}{\diff \refmu} (\bm X_i) - \frac{\diff \nullmu}{\diff \refmu} (\bm X_i)\right) \prior(\diff \bm V).
\end{align*}
\end{lemma}
\begin{proof}
The proof follows from the following identities:
\begin{align*}
     \Psi(\bm y, z_i, (\bm X_j)_{j \neq i}) & \explain{(a)}{=} \left\| {\looE{0}{i}\left[  \frac{\diff \dmu{\bm V}}{\diff \refmu} (\bm X_i) - \frac{\diff \nullmu}{\diff \refmu} (\bm X_i)  \bigg| \bm Y = \bm y, Z_i = z_i, (\bm X_j)_{j \neq i}  \right]} \right\|_{\prior} \\
     & \explain{(b)}{=} \sup_{\substack{S: \paramspace \rightarrow \R \\ \|S\|_{\prior} \leq 1}} \ip{S}{\; {\looE{0}{i}\left[  \frac{\diff \dmu{\bm V}}{\diff \refmu} (\bm X_i) - \frac{\diff \nullmu}{\diff \refmu} (\bm X_i)  \bigg| \bm Y = \bm y, Z_i = z_i, (\bm X_j)_{j \neq i}  \right]}}_\prior \\
     & \explain{(c)}{=} \sup_{\substack{S: \paramspace \rightarrow \R \\ \|S\|_{\prior} \leq 1}} \looE{0}{i}\left[ \ip{S}{\; \frac{\diff \dmu{\bm V}}{\diff \refmu} (\bm X_i) - \frac{\diff \nullmu}{\diff \refmu} (\bm X_i) }_\prior \bigg| \bm Y = \bm y, Z_i = z_i, (\bm X_j)_{j \neq i}  \right].
\end{align*}
In the step marked (a), we used the definition of $\geometric$ and $\|\cdot\|_\prior$.
In the step marked (b), we used Cauchy-Schwarz inequality (and its tightness condition).
In the step marked (c), we used Fubini's Theorem to move the inner product $\ip{\cdot}{\cdot}_\prior$ inside the conditional expectation. 
\end{proof}

\subsection{Geometric Inequalities}

In order to upper bound
\begin{align} \label{eq:geometric-inequality-target}
   \left(  {\looE{0}{i}\left[  \ip{S}{\; \frac{\diff \dmu{\bm V}}{\diff \refmu} (\bm X_i) - \frac{\diff \nullmu}{\diff \refmu} (\bm X_i)}_{\prior}  \bigg| \bm Y = \bm y, Z_i = z_i, (\bm X_j)_{j \neq i}    \right]} \right)^2,
\end{align}
we will use the framework of Geometric Inequalities introduced by \citet{han2018geometric} which shows that the task of upper bounding \eqref{eq:geometric-inequality-target} can be reduced to the task of understanding the concentration properties of the following function when $\bm X \sim \refmu$:
\begin{align*}
    f(\bm X_i) = \ip{S}{\; \frac{\diff \dmu{\bm V}}{\diff \refmu} (\bm X) - \frac{\diff \nullmu}{\diff \refmu} (\bm X)}_{\prior} .
\end{align*}
Similar results were known in the concentration of measure literature prior to the work of \citeauthor{han2018geometric} under the name ``Transportation Lemma'' \citep[see, e.g.,][Lemma 4.18]{boucheron2013concentration}).
This result has also been used in other works studying communication lower bounds for distributed estimation \citep{barnes2020lower,acharya2020unified}.
The following proposition summarizes this technique in our context.
\begin{proposition} [\citet{boucheron2013concentration,han2018geometric}] \label{prop: geometric inequality}
Let $f : \dataspace^\batch \rightarrow \R$ be given, and consider $\bm X \sim \refmu$.
\begin{enumerate} 
\item For any $\xi > 0$,
\begin{align*}
     &\left|  {\looE{0}{i}\left[ f(\bm X_i)  \bigg| \bm Y = \bm y, Z_i = z_i, (\bm X_j)_{j \neq i}   \right]} \right|  \\& \hspace{4cm}\leq \frac{\ln (\refE[ e^{\xi f(\bm X)}] \vee \refE[ e^{-\xi f(\bm X)}])}{\xi} + \frac{1}{\xi}  \ln  \frac{1}{\looP{0}{i}(\bm Y = \bm y, Z_i = z_i | (\bm X_j)_{j \neq i})}.
\end{align*} %
\item For any $q \geq 1$,
\begin{align*}
    \left|  {\looE{0}{i}\left[ f(\bm X_i)  \bigg| \bm Y = \bm y, Z_i = z_i, (\bm X_j)_{j \neq i} \right]} \right| & \leq \left(\frac{\refE |f(\bm X)|^{q}}{\looP{0}{i}(\bm Y = \bm y, Z_i = z_i | (\bm X_j)_{j \neq i} ) } \right)^{\frac{1}{q}}
\end{align*}
\end{enumerate}
\end{proposition}
\begin{proof}
  For completeness, the proof of this result is presented in Appendix~\ref{sec:geometric-inequality-proof}.
\end{proof}

We have now introduced all the key elements of the framework to prove lower bounds on distributed estimation protocols for solving general statistical inference problems.
We instantiate our framework in Sections~\ref{sec:tpca}--\ref{sec:cca} to obtain lower bounds for specific problems.

\subsection{Comparison to Prior Works} \label{sec:prior-frameworks}

Recent works by \citet{han2018geometric,barnes2020lower,acharya2020unified} have developed general frameworks to obtain communication lower bounds for distributed statistical inference problems.
The general information bounds developed in these works yield lower bounds for the simpler ``hide-and-seek'' variant of the inference problems \citep{shamir2014fundamental}.
In the hide-and-seek variant, the statistician knows the entire parameter vector $\bm V \in \{\pm 1\}^{\dim}$ except for a single coordinate, hidden at an unknown index $i \in [\dim]$.
The goal is to infer the sign of the hidden coordinate.

A hide-and-seek version inference problem can always be solved in the distributed setting with the information-theoretic sample complexity as long as each machine is allowed to communicate at least $\Omega(\dim \cdot \polylog(\dim))$ bits.
To see this, note that because the statistician knows entire parameter vector except for a single coordinate hidden at an unknown index $i \in [\dim]$, the possible parameter space for the inference problem is a discrete set of size $2 \dim$---there are $\dim$ possibilities for the index of the unknown coordinate, and two possibilities for the sign of the unknown coordinate.
Hence, each machine $j \in [\mach]$ can transmit the likelihoods of all $2\dim$ elements of this discrete set given its own dataset $\bm X_j$, using $O(\dim \cdot \polylog(\dim))$ bits of communication.
These likelihoods can be aggregated \djh{(by taking their product)} to obtain the likelihoods given \emph{all} of the data; this is a sufficient statistic for any inference problem.

Since the information bounds developed in the previously mentioned works \citep{han2018geometric,barnes2020lower,acharya2020unified} apply to the hide-and-seek variant of inference problems, we are unable to use them directly to obtain non-trivial lower bounds for $k$-TPCA and $k$-NGCA in the regime where the problems are information-theoretically solvable $(\ssize = \mach \cdot \batch \gtrsim \dim)$ and each machine is allowed at least $\budget \gtrsim \dim \cdot \polylog(\dim)$ bits of communication.
The information bound in Proposition~\ref{prop: main_hellinger_bound} builds on these works to address this limitation.

\section{Symmetric Tensor PCA} \label{sec:tpca}

\subsection{Problem Formulation}

In the symmetric order-$k$ Tensor PCA ($k$-TPCA) problem introduced by \citet{montanari2014statistical}, one observes $\ssize$ i.i.d.\ tensors $\bm X_{1:m} \in \tensor{\R^d}{k}$ sampled
as follows:
\begin{align} \label{eq: symmetric_tpca_setup}
    \bm X_i = \frac{\lambda \bm V^{\otimes k}}{\sqrt{d^k}} + \bm W_i, \quad (W_i)_{j_1,j_2, \dots j_k} \explain{i.i.d.}{\sim} \gauss{0}{1}, \quad \forall \; j_1,j_2, \dots ,j_k \in [d].
\end{align}
In the above display, $\lambda > 0$ is the signal-to-noise ratio, and $\bm V \in \paramspace$ is the unknown parameter one seeks to estimate.
The parameter space for this problem is $\paramspace = \{\bm V \in \R^\dim : \|\bm V\| = \sqrt{\dim} \}$.
We let the probability measure $\dmu{\bm V}$ denote the distribution of a single sample $\bm X_i$ in \eqref{eq: symmetric_tpca_setup}.
 
\subsection{Statistical-Computational Gap in $k$-TPCA}

Depending on the scaling of the effective sample size $\ssize \lambda^2$, $k$-TPCA exhibits three phases:
\begin{description}
\item [Impossible phase.] In the regime $\ssize \lambda^2 \ll \dim$, it is information-theoretically impossible to recover $\bm V$~\citep{montanari2014statistical}.

\item [Conjectured hard phase.] In the regime $\dim \lesssim \ssize \lambda^2 \ll \dim^{k/2}$, the maximum likelihood estimator succeeds in recovering $\bm V$ \citep{montanari2014statistical}. However, it is not known how to compute the maximum likelihood estimator using a polynomial-time algorithm. No known polynomial-time estimation algorithm has a non-trivial performance in this phase. Based on evidence from the low degree likelihood ratio framework \citep{kunisky2019notes}, the statistical query framework \citep{brennan2020statistical,dudeja2020statistical}, the sum-of-squares hierarchy framework \citep{hopkins2017power} and the average-case reductions framework \citep{zhang2018tensor,brennan2020reducibility}, it is believed that no polynomial-time algorithm can have non-trivial performance in this phase. 

\item [Easy phase.] In the regime $\ssize \lambda^2 \gtrsim \dim^{k/2}$, there are polynomial-time algorithms that accurately estimate $\bm V$~\citep{montanari2014statistical,hopkins2015tensor,zheng2015interpolating,hopkins2016fast,anandkumar2017homotopy,biroli2019iron,wein2019kikuchi}. 

\end{description}

\subsection{Computational Lower Bound}

The following is our computational lower bound for $k$-TPCA.

\begin{theorem} \label{thm:tpca}
  Let $\hat{\bm V}$ denote any estimator for $k$-TPCA with $k \geq 2$ and  $\lambda \asymp 1$ (as $\dim \rightarrow \infty$) that can be computed using a memory bounded estimation algorithm with resource profile $(\ssize, \iters, \state)$ scaling with $\dim$ as
  \begin{align*}
    \ssize \asymp \dim^{\eta}/\lambda^2, \quad \iters \asymp \dim^\tau, \quad \state \asymp \dim^{\mu}
  \end{align*}
  for any constants $\eta \geq 1, \; \tau \geq 0, \; \mu \geq 0$. If
  \begin{align*}
    \eta + \tau + \mu  < \left\lceil\frac{k+1}{2} \right\rceil,
  \end{align*}
  then, for any $t \in \R$,
  \begin{align*}
    \limsup_{\dim \rightarrow \infty} \inf_{\bm V \in \paramspace} \P_{\bm V}\left( \frac{|\langle{\bm V,}{\hat{\bm V}} \rangle|^2}{\|\bm V\|^2 \|\hat{\bm V}\|^2}  \geq \frac{t^2}{d} \right) & \leq 2\exp\left( - \frac{t^2}{2}\right).
  \end{align*}
\end{theorem}

The above result shows that if the total resources used by a memory-bounded estimator $\hat{\bm V}$ (as measured by the product $\ssize \cdot \iters \cdot \state$) is too small, there is a worst-case choice of $\bm V \in \paramspace$ such that, on an event of probability arbitrarily close to $1$, we have
\begin{align*}
    \frac{|\langle{\bm V,}{\hat{\bm V}} \rangle|^2}{\|\bm V\|^2 \|\hat{\bm V}\|^2} & \lesssim \frac{1}{\dim}.
\end{align*}
On the other hand, for any $\bm V \in \paramspace$, the trivial estimator $\hat{\bm V} \sim \gauss{\bm 0}{\bm I_\dim}$ achieves
\begin{align*}
    \frac{|\langle{\bm V,}{\hat{\bm V}} \rangle|^2}{\|\bm V\|^2 \|\hat{\bm V}\|^2} & \gtrsim\frac{1}{\dim},
\end{align*}
with probability arbitrarily close to 1.
Hence, memory bounded estimation algorithms using too few total resources perform no better than a random guess.

\subsection{Discussion of Theorem~\ref{thm:tpca}} \label{sec:tpca-discussion}

We now discuss some key implications of Theorem~\ref{thm:tpca}.
Recall that we consider the scaling regime where $\dim \rightarrow \infty$ and $\lambda \asymp 1, \ssize \asymp \dim^{\eta}, \; \iters \asymp \dim^{\tau}, \; \state \asymp \dim^{\mu}$; the exponents $\eta \geq 1$, $\tau \geq 0$, and $\mu \geq 0$ are fixed constants. We additionally restrict our discussion to the case where \emph{$k$ is even}, because our lower bounds appear to be deficient by a factor of $\sqrt{\dim}$ when $k$ is odd (additional details are provided in
Appendix~\ref{appendix:odd-case}
%the supplement \citep[Appendix D.1]{supplement}
regarding the odd case).

\subsubsection{Consequences for Linear Memory Algorithms}

Theorem~\ref{thm:tpca} has some interesting consequences for memory-bounded estimation algorithms with a memory state of size $s \asymp \dim$ bits. We call such algorithms linear memory algorithms. Theorem \ref{thm:tpca} shows that for such algorithms to have a non-trivial performance for Tensor PCA,  the sample size exponent $\eta = \ln(\ssize \lambda^2)/\ln(\dim)$ and the run-time exponent $\tau = \ln(\iters)/\ln(\dim)$ must satisfy
\begin{align} \label{eq:runtime-ssize-tradeoff}
    \tau + \eta \geq  \frac{k}{2}.
\end{align}
This gives a lower bound on the run-time exponent as a function of the sample-size exponent $\tau \geq k/2 - \eta$, which rules out certain run-time exponents in the conjectured hard phase for Tensor PCA ($1<\eta<k/2$). The run-time exponents ruled out by Theorem \ref{thm:tpca} is the triangular sub-region of the conjectured hard phase shaded in red and gray in Figure \ref{fig:tpca-phase-diagram}. 

\subsubsection{Consequences for Common Iterative Algorithms}

The class of linear memory algorithms is large enough to include a broad class of iterative algorithms that maintain a sequence of iterates $\bm u_t \in \R^\dim$ and run for a total of $\iters$ iterations.
At iteration $t$, the iterate $\bm u_{t}$ is generated using the dataset $\bm X_{1:\ssize}$ and the previous iterate $\bm u_{t-1}$ as follows:
\begin{align} \label{eq:common-iteration-template}
    \bm u_t = \frac{1}{\ssize} \sum_{i=1}^\ssize \bm X_i\{\psi_t(\bm u_{t-1}),\cdot \},
\end{align}
where $\psi_t : \R^\dim \rightarrow \tensor{\R^\dim}{k-1}$ maps the previous iterate $\bm u_{t-1}$ to an order-$(k-1)$ tensor; and for tensors $\bm X \in \tensor{\R^\dim}{k}$ and $\bm \Psi \in \tensor{\R^\dim}{k-1}$, the tensor contraction $\bm X\{\bm \Psi, \cdot\}$ yields a vector in $\R^\dim$ defined by
\begin{align*}
    \bm X\{\bm \Psi, \cdot\}_i = \sum_{j_1, j_2, \dotsc, j_{k-1}} X_{j_1, j_2, \dotsc, j_{k-1}, i} \Psi_{j_1, j_2, \dotsc, j_{k-1}}.
\end{align*}

One can implement $\iters$ iterations of the above scheme using a memory bounded algorithm (recall Definition \ref{def:memory-bounded-algorithm}) with a memory state of size $\state = \dim \cdot \polylog(\dim)$ bits and $\iters$ passes through the data.
In order to see this, let us first consider the situation when the memory bounded algorithm is allowed a real-valued memory state (instead of a Boolean memory state).
In this situation, the update in~\eqref{eq:common-iteration-template} can be implemented using a memory bounded algorithm that maintains two $\dim$-dimensional state variables $\mathtt{PartialSum} \in \R^\dim$ and $\mathtt{iterate} \in \R^\dim$. This implementation is shown in Figure \ref{fig: memory-bounded-algorithm-common-iteration}. By using $\polylog(\dim)$ bits to represent a real number, one can approximate the real-valued state variables $\mathtt{PartialSum}$ and $\mathtt{iterate}$ using a Boolean vector of size $\dim \cdot \polylog(\dim)$, while ensuring that the quantization error is negligible. Consequently, $\iters$ iterations of \eqref{eq:common-iteration-template} can be implemented using a memory bounded estimation algorithm with resource profile $(\ssize, \iters, \state = \dim \cdot \polylog(\dim))$. Hence, the run-time vs.\ sample size tradeoffs given in \eqref{eq:runtime-ssize-tradeoff} of the preceding paragraph also apply to iterative algorithms with update rules as in~\eqref{eq:common-iteration-template}. 

\begin{figure}[t!]
\begin{algbox}
\textbf{Memory bounded implementation of the iterative algorithm with update rule \eqref{eq:common-iteration-template}.}

\vspace{1mm}

\textit{Input}: $\bm X_{1:\ssize}$, a dataset of $\ssize$ tensors. \\
\textit{Output}: An estimator $\hat{\bm V} \in \R^\dim$. \\
\textit{Variables}: $\mathtt{iterate} \in \R^\dim$ and $\mathtt{PartialSum} \in \R^\dim$
\begin{itemize}
\item For iteration $t \in \{1, 2, \dotsc, \iters\}$
\begin{itemize}
    \item Set $\mathtt{PartialSum} = \bm 0_\dim$. 
    \item For sample $i \in  \{1, 2, \dotsc, \ssize\}$
    \begin{equation*}
                  \mathtt{PartialSum} \gets \mathtt{PartialSum} + \frac{\bm X_i \left\{ \psi_t(\mathtt{iterate}); \cdot\right\}}{\ssize}
    \end{equation*}
    \item Update $\mathtt{iterate} \gets \mathtt{PartialSum}$.
\end{itemize}
\item \textit{Return} Estimator $\hat{\bm V} = \mathtt{iterate}$.
\end{itemize}
\vspace{1mm}

\end{algbox}
\caption{Memory bounded implementation of the iterative algorithm with update rule~\eqref{eq:common-iteration-template}.} \label{fig: memory-bounded-algorithm-common-iteration}
\end{figure}

By suitably defining the function $\psi_t$ in~\eqref{eq:common-iteration-template} one can obtain many algorithms for $k$-TPCA that have been proposed in prior works. This means that the run-time vs.\ sample size trade-off obtain in~\eqref{eq:runtime-ssize-tradeoff} also applies to such algorithms. Two examples include:
\begin{description}
\item [Tensor power method.]  The tensor power method is given by the iterations:
\begin{align*}
    \bm u_t = \frac{1}{\ssize} \sum_{i=1}^\ssize \bm X_i \left\{ \frac{\bm u_{t-1}^{\otimes k-1}}{\|\bm u_{t-1}\|^{k-1}}, \cdot  \right\} 
\end{align*}
Hence, the tensor power method can be obtained from the general iteration \eqref{eq:common-iteration-template} by choosing:
\begin{align*}
    \psi_t(\bm u) & = \frac{\bm u^{\otimes k-1}}{\|\bm u\|^{k-1}}.
\end{align*}
\item [Spectral method with partial trace.] This estimator is given by the leading eigenvector of the matrix $\bm M$ whose entries are constructed as follows:
\begin{align*}
    M_{\alpha, \beta} \explain{def}{=} \frac{1}{\ssize} \sum_{i=1}^\ssize \sum_{\gamma_1, \gamma_2, \dotsc, \gamma_\ell \in [\dim]} (\bm X_i)_{\gamma_1, \gamma_1, \gamma_2, \gamma_2, \dotsc, \gamma_\ell, \gamma_{\ell}, \alpha, \beta}.
\end{align*}
In the above display, we defined $\ell \explain{def}{=} k/2-1$. This estimator is due to \citet{hopkins2016fast}; see \citet{biroli2019iron} for a simple analysis.
It can be verified that
\begin{align*}
    \bm M \explain{d}{=} \frac\lambda\dim \bm V \bm V^\UT + \sqrt{\frac{\dim^\ell}{\ssize}} \cdot \bm Z, 
\end{align*}
where $\bm Z\in \R^{\dim \times \dim}$ is a random $\dim \times \dim$ random matrix with i.i.d.\ $\gauss{0}{1}$ entries. In the regime when $\lambda \asymp 1$ and $\ssize \lambda^2 \gg \dim^{\frac{k}{2}}$, the largest eigenvector of $\bm M$ yields a consistent estimator for $\bm V$. Furthermore, in this regime $\bm M$ exhibits a spectral gap of size $\Delta \gtrsim 1$ \citep[see, e.g.,][for detailed arguments]{biroli2019iron}. Hence, the largest eigenvector of $\bm M$ can be computed by running $\iters \asymp \log(\dim)$ iterations of the power method beginning from a random initialization. The power iterations are given by the update rule
\begin{align*}
    \bm u_t & = \bm M \cdot \frac{\bm u_{t-1}}{\|\bm u_{t-1}\|}.
\end{align*}
Recalling the formula for $\bm M$, we can express the update rule for the power iteration as
\begin{align*}
    \bm u_t & = \frac{1}{\ssize} \sum_{i=1}^\ssize \bm X_i \left\{ \underbrace{\bm I_\dim \otimes \bm I_\dim \otimes \dotsb \otimes \bm I_\dim}_{\ell \text{ times}} \otimes \frac{\bm u_{t-1}}{\|\bm u_{t-1}\|} , \cdot\right\}.
\end{align*}
This is an instantiation of the general iteration in \eqref{eq:common-iteration-template} with
\begin{align*}
    \psi_t(\bm u) & = \underbrace{\bm I_\dim \otimes \bm I_\dim \otimes \dotsb \otimes \bm I_\dim}_{\ell \text{ times}} \otimes \frac{\bm u}{\|\bm u\|} .
\end{align*}
\end{description}

\subsubsection{Tightness of Theorem~\ref{thm:tpca}}

As discussed in the previous paragraph, the spectral method with partial trace provides a memory bounded estimation algorithm with resource profile: $$(\ssize \asymp \dim^{\frac{k}{2}}/\lambda^2, \iters = \polylog(\dim), \state = \dim \cdot \polylog(\dim)).$$ This leads to a (nearly) linear memory estimation algorithm for Tensor PCA whose sample size exponent $\eta$ and run-time exponent $\tau$ satisfy $\eta + \tau \leq k/2 + \epsilon$ for arbitrary $\epsilon > 0$. This shows that the run-time v.s. sample size tradeoffs implied for linear-memory algorithms by Theorem \ref{thm:tpca} are tight. Furthermore, this shows that Theorem \ref{thm:tpca} provides a weak separation between the easy and the conjectured hard phases:
\begin{enumerate}
    \item In the easy phase, when the sample size exponent $\eta > k/2$, there are (nearly) linear memory algorithms whose run-time exponent is arbitrarily close to zero ($\tau \leq \epsilon$ for any $\epsilon>0$). 
    \item In contrast, in the hard phase, when the sample size exponent $\eta<k/2$, Theorem~\ref{thm:tpca} shows that any linear memory algorithm must have a strictly positive run-time exponent $\tau \geq k/2-\eta$. 
\end{enumerate}

\subsubsection{Comparison with Low-Degree Lower Bounds}

Lastly, it is interesting to compare the lower bounds implied by Theorem~\ref{thm:tpca} with the lower bounds obtained using the low-degree likelihood framework.
\citet[Theorem 3.3]{kunisky2019notes} show that when
\begin{align*}
    \ssize \lambda^2  \lesssim  \frac{\dim^{\frac{k}{2}}}{D^{\frac{k-2}{2}}},
\end{align*}
any procedure that computes a degree-$D$ polynomial of the dataset $\bm X_{1:\ssize}$ fails to solve $k$-TPCA~\citep[Theorem 4]{kunisky2019notes}.
In general, the lower bounds obtained from Theorem~\ref{thm:tpca} are incomparable to those obtained from the low-degree framework for the following reasons:
\begin{enumerate}
    \item The low-degree polynomial makes no restrictions on the amount of memory used to compute the polynomial. 
    \item There are no degree restrictions placed on memory bounded estimation algorithms.  
\end{enumerate}
However, one can still make interesting comparisons between lower bounds obtained for estimators that can be implemented in both computational models. One such example is an estimator computed by running $\iters$ iterations of the scheme in \eqref{eq:common-iteration-template}:
\begin{align*}
    \bm u_t = \frac{1}{\ssize} \sum_{i=1}^\ssize \bm X_i\{\psi_t(\bm u_{t-1}),\cdot \},
\end{align*}
under the additional assumption that each entry of $\psi_t(\bm u)$ is a degree-$c$ polynomial in $\bm u$ for some constant $c > 1$. An example of such an iteration is the tensor power method, where each entry of $\psi_t(\bm u)$ is a degree $k-1$ homogeneous polynomial in $\bm u$. As discussed previously, this iteration can be implemented using memory bounded algorithm with resource profile $(\ssize, \iters, \state = \dim \cdot \polylog(\dim))$. Furthermore since $\bm u_{\iters}$ is a polynomial in $\bm X_{1:\ssize}$, this estimator can also be implemented in the low-degree framework. Unfortunately, the degree of the polynomial $\bm u_{\iters}(\bm X_{1:\ssize})$ can be as large as $D = c^{\Omega(\iters)}$. Hence, low-degree lower bounds only show the failure such iterative schemes in the conjectured hard phase of $k$-TPCA for $\iters \lesssim \log(\dim)$ iterations.

Low-degree lower bounds fail to give iteration lower bounds of the form $\iters \gtrsim \dim^\delta$ for any $\delta > 0$ because of the following reasons:
\begin{enumerate}
    \item The low-degree framework measures the computational cost of computing a polynomial only using its degree. Hence, in order to show an iteration lower bound of $\iters \gtrsim \dim^\delta$, one would have to show that polynomials of degree $D = \exp(O(\dim^\delta))$ fail to solve $k$-TPCA.

    \item However, it is known that for every $\epsilon \in (0,1)$, there is a computationally inefficient estimator based on a degree $D \lesssim \dim^\epsilon$ polynomial that solves $k$-TPCA in a part of the hard phase with sample-size exponent $\eta = \epsilon + k(1-\epsilon)/2<k/2$ (see discussion in \citep[Page 16]{kunisky2019notes} and references therein). 
\end{enumerate}
In contrast, stronger lower bounds are obtained via Theorem~\ref{thm:tpca} since this result leverages the fact that $\bm u_{\iters}(\bm X_{1:\ssize})$ has an additional structural property not shared by arbitrary polynomials of comparable degree: $\bm u_{\iters}(\bm X_{1:\ssize})$ can be computed by $\iters$ iterations of a linear memory algorithm.

\subsection{Proof of Theorem~\ref{thm:tpca}} 

Theorem~\ref{thm:tpca} is obtained by transferring a communication lower bound for distributed estimation protocols for $k$-TPCA to memory bounded estimators for the same problem using the reduction in Fact~\ref{fact:reduction}.

In the (Bayesian) distributed setup for $k$-TPCA, the parameter $\bm V$ is drawn from the prior $\pi \explain{def}{=} \unif{\{\pm 1\}^\dim}$, and then $\bm X_{1:\ssize}$ are sampled i.i.d.\ from $\dmu{\bm V}$; these tensors are distributed across $\mach = \ssize$ machines with $\batch = 1$ sample/machine.
The execution of a distributed estimation protocol with parameters $(\mach, \batch = 1, \budget)$ results in a transcript $\bm Y \in \{0,1\}^{\mach \budget}$ written on the blackboard.

To prove lower bounds on the performance of distributed estimation protocols, we instantiate Fano's Inequality for Hellinger Information (Fact~\ref{fact: fano}) to obtain the following corollary.
 \begin{corollary}[Fano's Inequality for $k$-TPCA] \label{coro: fano-stpca} For any estimator $\hat{\bm V}(\bm Y)$ for $k$-TPCA computed by a distributed estimation protocol, and for any $t \in \R$, we have
\begin{align*}
\inf_{\bm V \in \paramspace} \P_{\bm V}\left( \frac{|\langle{\bm V,}{\hat{\bm V}} \rangle|^2}{\|\bm V\|^2 \|\hat{\bm V}\|^2}  \geq \frac{t^2}{d} \right) & \leq 2\exp\left( - \frac{t^2}{2}\right) + \sqrt{2\MIhell{\bm V}{\bm Y}}.
\end{align*}
\end{corollary}

\begin{proof}
We apply Fano's Inequality (Fact \ref{fact: fano}) with the following loss function:
\begin{align*}
    \ell(\bm V, \bm u) \explain{def}{=} \begin{cases} 1 & \text{if $\frac{|\langle{\bm V,}{\bm u} \rangle|^2}{\|\bm V\|^2 \|\bm u\|^2} < \frac{t^2}{d}$} , \\ 0 & \text{otherwise} . \end{cases}
\end{align*}
To do so, we need to compute a lower bound on $R_0(\pi)$. By Hoeffding's inequality, for any fixed unit vector $\bm u$, we have
\begin{align*}
 \P\left( \frac{|\langle{\bm V,}{\bm u} \rangle|^2}{\|\bm V\|^2 }  \geq \frac{t^2}{d} \right) & \leq 2\exp\left( - \frac{t^2}{2}\right).
\end{align*}
Consequently $R_0(\pi) \geq 1 - 2 e^{-t^2/2}$. The claim is now immediate from Fact \ref{fact: fano}.
\end{proof}

The main technical result needed to prove Theorem~\ref{thm:tpca} is the \djh{following bound on $\MIhell{\bm V}{\bm Y}$ for $k$-TPCA in Proposition~\ref{prop: stpca-info-bound}}

\begin{proposition} [Information Bound for $k$-TPCA] \label{prop: stpca-info-bound}
  Let $\bm Y \in \{0,1\}^{\mach \budget}$ be the transcript generated by a distributed estimation protocol for $k$-TPCA with parameters $(\mach, 1, \budget)$. Then
\begin{align*}
  \MIhell{\bm V}{\bm Y}
     & \leq C_k \left( \sigma^2 \cdot \mach \cdot b + \frac{1}{\dim}  + \lambda^2  \cdot \budget \cdot  \left( \frac{\lambda^2  \vee \ln(m \cdot \dim)}{d}\right)^{\frac{k}{2}} + \inf_{\alpha \geq 2} \frac{\lambda^2 \alpha}{\dim} + \mach \cdot \left(\frac{C_k \alpha \lambda^2}{\sqrt{d^k}} + e^{-\dim} \right)^{\frac{\alpha}{2}}  \right) ,
\end{align*}
where
\begin{align*}
    \sigma^2 \explain{def}{=}  \begin{cases} C_k \cdot \lambda^2 \cdot  d^{-\frac{k+2}{2}} & \text{if $k$ is even} , \\  C_k \cdot \lambda^2 \cdot  d^{-\frac{k+1}{2}} & \text{if $k$ is odd} ; \end{cases}
\end{align*}
and $C_k>0$ is a positive constant that depends only on $k$.
In particular, in the scaling regime (as $d\to\infty$)
\begin{align*}
    \lambda \asymp 1, \quad \mach \asymp \dim^{\eta}, \quad \budget \asymp \dim^\beta
\end{align*}
for any constants $\eta \geq 1$ and $\beta \geq 0$ that satisfy
\begin{align*}
    \eta + \beta  < \left\lceil\frac{k+1}{2} \right\rceil,
\end{align*}
we have $\MIhell{\bm V}{\bm Y} \rightarrow 0$ as $\dim \rightarrow \infty$. 
\end{proposition}

Proposition~\ref{prop: stpca-info-bound} is proved in Appendix~\ref{appendix:tpca}.With this information bound in hand, we can complete the proof of Theorem~\ref{thm:tpca}.

\begin{proof}[Proof of Theorem \ref{thm:tpca}] Appealing to the reduction in Fact~\ref{fact:reduction} with the choice $n=1$, we note that any memory bounded estimator $\hat{\bm V}$ with resource profile $(\ssize, \iters, \state)$ can be implemented using a distributed estimation protocol with parameters $(\ssize, 1, \state \iters)$. Applying Corollary~\ref{coro: fano-stpca} and Proposition~\ref{prop: stpca-info-bound} to the distributed implementation of the memory bounded estimator immediately yields Theorem~\ref{thm:tpca}.
\end{proof}

\section{Asymmetric Tensor PCA} \label{sec:atpca}

\subsection{Problem Formulation}

In the asymmetric order-$k$ Tensor PCA ($k$-ATPCA) problem, one observes $\ssize$ i.i.d.\ tensors $\bm X_{1:\ssize} \in \tensor{\R^\dim}{k}$ sampled as follows:
\begin{align} \label{eq: symmetric_atpca_setup}
    \bm X_i = \frac{\lambda \bm V_1 \otimes \bm V_2 \dotsb \otimes \bm V_k}{\sqrt{\dim^k}} + \bm W_i, \quad (W_i)_{j_1,j_2, \cdots j_k} \explain{i.i.d.}{\sim} \gauss{0}{1}, \quad \forall \; j_1,j_2, \dotsc, j_k \in [\dim].
\end{align}
 In the above display, $\lambda > 0$ is the signal-to-noise ratio, and $\bm V_1, \bm V_2, \dotsc, \bm V_k$ are unknown vectors in $\R^\dim$ with $\|\bm V_i\| = \sqrt{\dim}$.
 The goal is to estimate the rank-$1$ signal tensor $\bm V \explain{def}{=} \bm V_1 \otimes \bm V_2 \dotsb \otimes \bm V_k$.
 The parameter space for this problem is
    $\paramspace = \{  \bm V_1 \otimes \bm V_2 \otimes \dotsb \otimes \bm V_k: \bm V_i \in \R^\dim, \; \|\bm V_i\| = \sqrt{\dim} \; \forall i \in [k] \}$.
We let the probability measure $\dmu{\bm V}$ denote the distribution of a single sample $\bm X_i$ in \eqref{eq: symmetric_atpca_setup}.

\subsection{Statistical-Computational Gap in $k$-ATPCA}

The delineations between the impossible phase, conjectured hard phase, and easy phase for $k$-ATPCA are the same as in (symmetric) $k$-TPCA.
In particular, the known polynomial-time algorithms that accurately estimate $\bm V$ require $\ssize \lambda^2 \gtrsim \dim^{k/2}$~\citep{montanari2014statistical,zheng2015interpolating}.

\subsection{Computational Lower Bound}

The following is our computational lower bound for $k$-ATPCA.
\begin{theorem}\label{thm:main-result-atpca}
Let $\hat{\bm V} \in \tensor{\R^\dim}{k}$ denote any estimator for $k$-ATPCA with $k \geq 2$ and signal-to-noise ratio $\lambda \asymp 1$ (as $\dim \rightarrow \infty$) that can be computed using a memory bounded estimation algorithm with resource profile $(\ssize, \iters, \state)$ scaling with $\dim$ as
 \begin{align*}
     \ssize \asymp \dim^{\eta}/\lambda^2, \quad \iters \asymp\dim^\tau, \quad \state \asymp\dim^{b}
 \end{align*}
 for any constants $\eta > 0, \; \tau \geq 0, \; b \geq 0$. If
\begin{align*}
    \eta + \tau + b  < k,
\end{align*}
then, for any $t \in \R$,
\begin{align*}
    \limsup_{\dim \rightarrow \infty} \inf_{\bm V \in \paramspace} \P_{\bm V}\left( \frac{|\langle{\bm V,}{\hat{\bm V}} \rangle|^2}{\|\bm V\|^2 \|\hat{\bm V}\|^2}  \geq \frac{t^2}{d^k} \right) & \leq \frac{1}{t^2}.
\end{align*}
\end{theorem}

Just as in the case of (symmetric) $k$-TPCA, Theorem~\ref{thm:main-result-atpca} shows that memory bounded estimation algorithms for $k$-ATPCA using too few total resources (as measured by the product $\ssize \cdot \iters \cdot \state$) perform no better than a random guess.

\subsection{Discussion of Theorem \ref{thm:main-result-atpca}}
\label{sec:atpca-discussion}
We now discuss some implications of Theorem~\ref{thm:main-result-atpca}.
We restrict attention to the situation where $\lambda \asymp 1$ and $k = 2\ell$ is even. 

\begin{description}
\item [Price of Asymmetry.] A comparison of the computational lower bound for $k$-ATPCA (Theorem~\ref{thm:main-result-atpca} and $k$-TPCA (Theorem~\ref{thm:tpca}) reveals that $k$-ATPCA is a more resource-intensive inference problem. The minimum amount of resources (as measured by the product $\ssize \cdot \iters \cdot \state$) needed to solve $k$-ATPCA is strictly more than the minimum amount of resources required to solve $k$-TPCA.

\item [Tightness of Theorem \ref{thm:main-result-atpca}.] Let $\overline{\bm X} \in \tensor{\R^\dim}{k}$ denote the empirical average of $\bm X_{1:\ssize}$.
  \citet{montanari2014statistical} proposed estimating $\bm V$ by the best rank-1 approximation of the matrix obtained by flattening $\overline{\bm X}$ into a $\dim^{\ell} \times \dim^{\ell}$ matrix.
  The estimator is based on the \emph{matricization operation} $\operatorname{Mat}: \tensor{\R^\dim}{k} \rightarrow \R^{\dim^{\ell}} \times \R^{\dim^{\ell}}$ which reshapes a tensor into a matrix.
  To define $\operatorname{Mat}(\bm T)$ for
  a tensor $\bm T \in \tensor{\R^\dim}{k}$,
  we index the rows and columns of $\operatorname{Mat}(\bm T)$ by $\ell$-tuples of indices $(i_1, i_2, \dotsc, i_\ell ) \in [\dim]^\ell$, so
  the entries of $\operatorname{Mat}(\bm T)$ are given by
\begin{equation} \label{eq:matricization}
    \operatorname{Mat}(\bm T)_{(i_1, i_2, \dotsc, i_\ell) ; (j_1, j_2, j_3, \dotsc, j_\ell)} \explain{def}{=} T_{i_1, i_2, \dotsc, i_\ell, j_1, j_2, \dotsc, j_\ell}.
\end{equation}
The estimator $\hat{\bm V}_{\operatorname{MR}}$ of \citeauthor{montanari2014statistical} is defined by
\begin{align*}
    \hat{\bm V}_{\operatorname{MR}} & = \operatorname{Mat}^{-1}(\hat{\bm M}),
\end{align*}
where $\hat{\bm M}$ is the best rank-1 approximation (or the rank-1 SVD) of $\operatorname{Mat}(\overline{\bm X})$.
  This estimator was analyzed by \citet{zheng2015interpolating} for $k$-ATPCA.
  Their analysis shows that in the regime $\lambda \asymp 1$, when $\ssize \lambda^2 \gtrsim \dim^{\frac{k}{2}}$, $ \hat{\bm V}_{\operatorname{MR}}$ is a consistent estimator for $\bm V$.
  Moreover, in this regime, the matrix $\hat{\bm M}$ has a spectral gap of size $\Delta \gtrsim 1$. Consequently, $\hat{\bm V}_{\operatorname{MR}}$ can be computed using $\polylog(\dim)$ iterations of the power method. Since $\operatorname{Mat}(\overline{\bm X}) \in \R^{\dim^\ell \times \dim^\ell}$ with $\ell = k/2$, in order to implement the power method using a memory bounded algorithm, one requires a memory state of size $\state \asymp \dim^\ell \polylog(\dim)$ bits.
  Consequently, this estimator can be computed using a memory bounded estimation algorithm with resource profile
\begin{align*}
    (\ssize \asymp \dim^{\frac{k}{2}}/\lambda^2, \iters \asymp \polylog(\dim), \state \asymp \dim^{\frac{k}{2}} \polylog(\dim)).
\end{align*}
The total resources consumed by this estimation algorithm satisfies $\ssize \cdot \iters \cdot \state \ll \dim^{k+\epsilon}$ for any $\epsilon>0$. This shows that the resource lower bound in Theorem~\ref{thm:main-result-atpca} is nearly tight. 
\item [A Seperation between Easy and Hard phases.]
  Theorem~\ref{thm:main-result-atpca} has interesting consequences for memory bounded estimation algorithms that have a memory requirement comparable to the spectral estimator of \citeauthor{montanari2014statistical}, i.e., $\state \asymp \dim^{\frac{k}{2}}$.
  For such algorithms to have a non-trivial performance for $k$-ATPCA, the sample size exponent $\eta = \ln(\ssize \lambda^2)/\ln(\dim)$ and the run-time exponent $\tau = \ln(\iters)/\ln(\dim)$ must satisfy
\begin{align} \label{eq:atpca-runtime-ssize-tradeoff}
    \tau + \eta \geq  \frac{k}{2}.
\end{align}
This gives a lower bound on the run-time exponent as a function of the sample-size exponent: $\tau \geq k/2 - \eta$.
This rules out certain run-time exponents in the conjectured hard phase for $k$-ATPCA ($1<\eta<k/2$), specifically those in the striped triangular region in Figure~\ref{fig:atpca-overparametrized}.
(The spectral estimator of \citeauthor{montanari2014statistical} is depicted by the green dot at $(\log_{\dim}(\ssize \lambda^2) = k/2, \log_{\dim}(T) = 0)$ in Figure~\ref{fig:atpca-overparametrized}.)
Hence, Theorem~\ref{thm:main-result-atpca} provides a weak separation between the easy and the conjectured hard phases, similar to that provided by Theorem~\ref{thm:tpca} for linear memory algorithms and $k$-TPCA.
\item [Necessity of Overparameterization.] The Montanari-Richard spectral estimator is \emph{overparameterized} in the sense that it uses a memory state of $\state \gtrsim \dim^{\frac{k}{2}}$ bits, which is significantly larger than the effective dimension of the parameter of interest $\bm V$, namely $k\dim$, whenever $k\geq3$.
  Theorem~\ref{thm:main-result-atpca} shows that this amount of overparameterization is necessary. To see this, we instantiate Theorem~\ref{thm:main-result-atpca} for memory state sizes of $\state \asymp \dim^b$ bits for some $b<k/2$.
  For such memory bounded estimation algorithms to have a non-trivial performance for $k$-ATPCA, the sample size exponent $\eta = \ln(\ssize \lambda^2)/\ln(\dim)$ and the run-time exponent $\tau = \ln(\iters)/\ln(\dim)$ must satisfy
\begin{align} \label{eq:atpca-runtime-ssize-tradeoff-underparameterized}
    \tau + \eta \geq  \frac{k}{2} + \left(\frac{k}{2}-b \right).
\end{align}
The trade-off in \eqref{eq:atpca-runtime-ssize-tradeoff-underparameterized} is strictly worse than the trade-off obtained from \eqref{eq:atpca-runtime-ssize-tradeoff}; compare the phase diagram in Figure~\ref{fig:atpca-underparameterized} to that in Figure~\ref{fig:atpca-overparametrized}. Hence, one cannot significantly reduce the overparameterization level (as measured by the size of the memory state) of the \citeauthor{montanari2014statistical} spectral estimator without increasing its run-time or sample-size exponents. 
\end{description}

\subsection{Proof of Theorem \ref{thm:main-result-atpca}}

Similar to Theorem~\ref{thm:tpca}, we prove
 Theorem~\ref{thm:main-result-atpca} by transferring a communication lower bound for distributed estimation protocols for $k$-ATPCA to memory bounded estimators for the same problem using the reduction in Fact~\ref{fact:reduction}.

In the (Bayesian) distributed setup for $k$-ATPCA, the parameter $\bm V$ is drawn from the prior
\begin{equation} \label{eq:prior-atpca}
  \prior \explain{def}{=} \unif{\{ \sqrt{d^k} \cdot \bm e_{i_1}\otimes \bm e_{i_2} \dotsb \otimes \bm e_{i_k}: \; i_1, i_2, \dotsc, i_k \in [\dim]\}} .
\end{equation}
Here, $\bm e_{i}$ denotes the $i$-th standard basis vector in $\R^\dim$, so $\bm V \sim \prior$ is a uniformly random $1$-sparse tensor. 
The tensors $\bm X_{1:\ssize}$ are sampled i.i.d.\ from $\dmu{\bm V}$, and are distributed across $\mach = \ssize$ machines with $\batch = 1$ sample/machine.
The execution of a distributed estimation protocol with parameters $(\mach, \batch = 1, \budget)$ results in a transcript $\bm Y \in \{0,1\}^{\mach \budget}$ written on the blackboard.

We obtain the following corollary of Fano's Inequality for Hellinger Information (Fact~\ref{fact: fano}).

\begin{corollary}[Fano's Inequality for $k$-ATPCA] \label{coro: fano-atpca} For any estimator $\hat{\bm V}(\bm Y)$ for $k$-ATPCA computed by a distributed estimation protocol, and for any $t \in \R$, we have
\begin{align*}
\inf_{\bm V \in \paramspace} \P_{\bm V}\left( \frac{|\langle{\bm V,}{\hat{\bm V}} \rangle|^2}{\|\bm V\|^2\|\hat{\bm V}\|^2}  \geq \frac{t^2}{\dim^k} \right) & \leq \frac{1}{t^2} + \sqrt{2\MIhell{\bm V}{\bm Y}}.
\end{align*}
\end{corollary}
\begin{proof}
As in the proof of Corollary \ref{coro: fano-stpca}, we apply Fact \ref{fact: fano} with the following loss function:
\begin{align*}
    \ell(\bm V, \bm U) \explain{def}{=} \begin{cases} 1 & \text{if $\frac{|\langle{\bm V,}{\bm U} \rangle|^2}{\|\bm V\|^2 \|\bm U\|^2} < t^2/d^k$} , \\ 0 & \text{otherwise}. \end{cases}
\end{align*}
We also compute a lower bound on $R_0(\pi)$: by Markov's inequality, for any fixed tensor $\bm U \in \tensor{\R^\dim}{k}$ with $\|\bm U\| = 1$, we have
\begin{align*}
 \P\left( \frac{|\langle{\bm V,}{\bm U} \rangle|^2}{\|\bm V\|^2 }  \geq \frac{t^2}{d^k} \right) & \leq \frac{d^k}{t^2} \cdot \int_\paramspace \frac{|\langle{\bm V,}{\bm U} \rangle|^2}{\|\bm V\|^2} \prior(\diff \bm V) = \frac{\|\bm U\|^2}{t^2} = \frac{1}{t^2}.
\end{align*}
Consequently $R_0(\pi) \geq 1 - 1/t^2$. The claim is now immediate from Fact \ref{fact: fano}.
\end{proof}

The main technical result is the following information bound for $k$-ATPCA.

\begin{proposition} \label{prop: info-bound-atpca}
  Let $\bm Y \in \{0,1\}^{\mach \budget}$ be the transcript generated by a distributed estimation protocol for $k$-ATPCA with parameters $(\mach, 1, \budget)$. Then
\begin{align*}
    \MIhell{\bm V}{\bm Y} & \leq C \left(  \frac{\delta^2 \mach\budget }{\dim^k}  +   \frac{1}{m} \right) ,
\end{align*}
where
\begin{align*}
    \delta \explain{def}= \exp\left( \frac{3\lambda^2}{2} + 2\lambda \sqrt{\ln(\dim^k) + \ln(\mach)}  \right) - 1.
\end{align*}
In the above display $C$ is a universal constant (independent of $\mach,\budget,\dim,\lambda)$.
In particular, in the scaling regime (as $d\to\infty$)
\begin{align*}
    \lambda \asymp 1, \quad m \asymp \dim^{\eta}, \quad b \asymp \dim^\beta
\end{align*}
for any constants $\eta \geq 1$ and $\beta \geq 0$ that satisfy
\begin{align*}
    \eta + \beta  < k,
\end{align*}
we have $\MIhell{\bm V}{\bm Y} \rightarrow 0$ as $\dim \rightarrow \infty$.
\end{proposition}
Proposition~\ref{prop: info-bound-atpca} is proved in Appendix~\ref{sec:information-bound-atpca-proof}. 
With this information bound in hand, we can prove Theorem~\ref{thm:main-result-atpca}.

\begin{proof}[Proof of Theorem~\ref{thm:main-result-atpca}] Appealing to the reduction in Fact~\ref{fact:reduction} with the choice $n=1$, we note that any memory-bounded estimator $\hat{\bm V}$ with resource profile $(\ssize, \iters, \state)$ can be implemented using a distributed algorithm with parameters $(\ssize, 1, \state \iters)$. Applying Corollary~\ref{coro: fano-atpca} and Proposition~\ref{prop: info-bound-atpca} to the distributed implementation of the memory-bounded estimator immediately yields Theorem~\ref{thm:main-result-atpca}. 
\end{proof}

We end this section with the following remark, which discusses the connection between $k$-ATPCA and the sparse Gaussian mean estimation problem studied in prior work \citep{braverman2016communication,acharya2020unified}.
\begin{remark}[Connection with Sparse Gaussian Mean Estimation]\label{remark:ATPCA-SGME} Observe that due to the choice of the prior in \eqref{eq:prior-atpca}, the instance of $k$-ATPCA used to obtain the communication lower bound is also an instance of the $1$-sparse Gaussian mean estimation problem in dimension $D = \dim^k$. Communication lower bounds for this problem in the blackboard model (cf. Definition \ref{def: distributed-algorithm}) were obtained in prior work by \citet{braverman2016communication}. This result is sufficient to obtain Theorem \ref{thm:main-result-atpca}. Recent work by \citet{acharya2020unified} also provides an alternate proof for the communication lower bounds for sparse Gaussian mean estimation. We present another proof of these results using the information bound in Proposition \ref{prop: main_hellinger_bound}, which is used to derive all communication lower bounds presented in this paper. 
\end{remark}

\section{Non-Gaussian Component Analysis} \label{sec:ngca}

\subsection{Problem Formulation}

In the Non-Gaussian Component Analysis (NGCA) problem, one seeks to estimate an unknown vector $\bm V \in \R^\dim$ with $\|\bm V\| = \sqrt{\dim}$ from an i.i.d.\ sample $\bm x_{1:\ssize}$ generated as follows:
\begin{subequations} \label{eq: ngca-model}
\begin{align}
  \bm x_i  = \eta_i \frac1{\sqrt{\dim}} \bm V + \left( \bm I_\dim - \frac1{\dim} \bm V \bm V^\UT \right) \bm z_i, \;
\end{align}
where $\eta_i \in \R$ and $\bm z_i \in \R^\dim$ are independent random variables with distributions
\begin{align}
    \bm z_i \sim \gauss{\bm 0}{\bm I_\dim}, \quad \eta_i \sim \nongauss .
\end{align}
\end{subequations}
In the above display, $\nongauss$ is a non-Gaussian distribution on $\R$.
Let $\dmu{\bm V}$ denote the distribution of $\bm x_i$ described by the above generating process \eqref{eq: ngca-model}.
The likelihood ratio (with respect to the standard Gaussian distribution $\refmu$) of a single sample $\bm x \in \R^\dim$ from the model \eqref{eq: ngca-model} is
\begin{align} \label{eq: ngca-likelihood}
  \frac{\diff \dmu{\bm V}}{\diff \refmu}(\bm x) = \frac{\diff \nongauss}{\diff \refmu}(\eta),
\end{align}
where $\eta \explain{def}{=} \ip{\bm x}{ \frac1{\sqrt{\dim}} \bm V}$.

\begin{remark}
  We overload the symbol $\refmu$ to mean
  $\refmu = \gauss{\bm 0}{\bm I_\dim}$ on the left-hand side of \eqref{eq: ngca-likelihood}, and $\refmu = \gauss{0}{1}$ on the right-hand side.
  We will use this overloaded notation throughout our analysis of NGCA, but the meaning of $\refmu$ should be clear from the context. 
\end{remark}

\subsubsection{Degree of Non-Gaussianity}

The statistical and computational difficulty of estimating $\bm V$ depends on how non-Gaussian $\nongauss$ is.
For positive integer $k\geq2$, order-$k$ NGCA ($k$-NGCA) refers to instances of NGCA in which the first $k-1$ moments of $\nongauss$ are identical to a standard Gaussian random variable,
\begin{align*}
    \int x^i  \nongauss(\diff x) = \E Z^i, \quad Z \sim \gauss{0}{1}, \quad \forall i \in [k-1] ,
\end{align*}
but the $k$-th moment differs from the corresponding standard Gaussian moment,
\begin{align*}
    \left|\int x^k  \nongauss(\diff x) - \E[Z^k] \right| = \lambda > 0, \quad Z \sim \gauss{0}{1}.
\end{align*}
The parameter $\lambda > 0$ is the signal-to-noise ratio for this problem.

\subsubsection{Assumptions on the Non-Gaussian Component}
\label{sec:ngca-assumptions}

The computational lower bounds we prove holds for a broad class of non-Gaussian distributions $\nongauss$ that have a density with respect to the standard Gaussian measure $\refmu = \gauss{0}{1}$ on $\R$, and that satisfy some additional assumptions.
Before stating these assumptions, for any probability measure $\nongauss$ on $\R$, we define the coefficients $\hat{\nongauss}_i$ for any $i \in \W$ as follows:
\begin{align*}
    \hat{\nongauss}_i \explain{def}{=} \E[H_i(\eta)] , \quad \eta \sim \nongauss .
\end{align*}
(Recall that $\{ H_i \}_{i \in \W}$ are the orthonormalized Hermite polynomials.)
Note that since $H_0(z)= 1$, we have $\hat{\nongauss}_0 = 1$. Since we always assume that $\nu$ has a density with respect to $\refmu = \gauss{0}{1}$, we can equivalently write
\begin{align*}
    \hat{\nongauss}_i = \refE\left[ \frac{\diff \nongauss}{\diff \refmu}(Z) H_i(Z) \right], \quad Z \sim \refmu = \gauss{0}{1} .
\end{align*}
Hence, $\hat{\nongauss}_i$ is the $i$-th Hermite coefficient of the likelihood ratio function $\diff \nongauss/\diff \refmu$.
By Plancheral's identity,
\begin{align*}
    \refE\left[\left( \frac{\diff \nongauss}{\diff \refmu}(Z) - 1\right)^2\right] = \sum_{i=1}^\infty \hat{\nongauss}_i^2.
\end{align*}

We now state our assumptions below.

\begin{assumption} %
  \label{ass: moment-matching}
  Distribution $\nongauss$ satisfies the \emph{Moment Matching Assumption} with parameter $k \in \N$, $k \geq 2$ if
    \begin{align*} %
               \hat{\nongauss}_i = 0 , \quad \forall \; i \in [k-1] . %
    \end{align*}
    Equivalently, %
    \begin{align*}
         \int z^i \nongauss(\diff z) = \int z^i \refmu(\diff z) , \quad \forall \; i \in [k-1] .
    \end{align*}
\end{assumption}

\begin{assumption} %
  \label{ass: bounded-snr}
  Distribution $\nongauss$ satisfies the \emph{Bounded Signal Strength Assumption} with parameters $(\lambda,K)$ for some $\lambda \geq 0$ and $K \geq 0$ if
\begin{equation*}
     \sum_{i=1}^\infty \hat{\nongauss}_i^2 \leq K^2 \lambda^2, \quad Z \sim \gauss{0}{1}. 
\end{equation*}
\end{assumption}

\begin{assumption} %
  \label{ass: locally-bounded-LLR}
  Distribution $\nongauss$ satisfies the \emph{Locally Bounded Likelihood Ratio Assumption} with parameters $(\lambda,K,\kappa)$ for some $\lambda \geq 0$, $K \geq 0$, and $\kappa \geq 0$ if
\begin{align*}  
        \left| \frac{\diff \nongauss}{\diff \refmu}(z) - 1\right| & \leq K\lambda (1 + |z|)^\kappa \quad \forall \; z \in \R \; \text{s.t.} \; K \lambda (1 + |z|)^\kappa \leq 1. %
    \end{align*}
\end{assumption}

\begin{assumption} %
  \label{ass: min-snr}
  Distribution $\nongauss$ satisfies the \emph{Minimum Signal Strength Assumption} with parameters $(\lambda, k)$ for some $\lambda > 0$ and $k \in \N$, $k \geq 2$ if
\begin{align*}
     \left| \refE Z^k - \int x^k \nongauss(\diff x) \right| = \lambda, \quad Z \sim \refmu = \gauss{0}{1} .
\end{align*}
\end{assumption}

\begin{assumption} %
  \label{ass: subgauss}
  % Random vector $\bm x  \sim \dmu{\bm V}$ is \emph{sub-Gaussian} with variance proxy $\varproxy$ for some $\varproxy \geq 1$ if $\E_{\bm V}[ \exp(\ip{\bm u}{\bm x}) ] \leq \exp(\varproxy \|\bm u\|^2/2)$ for all $\bm u \in \R^{\dim}$.
 The random vector $\bm x  \sim \dmu{\bm V}$ is \emph{sub-Gaussian} with variance proxy $\varproxy$ for some $\varproxy \geq 1$.\footnote{%
 \djh{A random vector $\bm w \in \R^{\dim}$ is \emph{sub-Gaussian with variance proxy $v$} (a.k.a.~\emph{$v$ sub-Gaussian}) if $\E[\bm w] = 0$ and $\E[ \exp(\ip{\bm u}{\bm w}) ] \leq \exp(v \|\bm u\|^2/2)$ for all  $\bm u \in \R^{\dim}$.} Note that \djh{$\bm x \sim \dmu{\bm V}$} is sub-Gaussian with variance proxy $\varproxy$ if $\eta \sim \nongauss$ is sub-Gaussian with variance proxy $\varproxy$.}
\end{assumption}

\subsection{Statistical-Computational Gap in $k$-NGCA}

Similar to $k$-TPCA, the $k$-NGCA problem exhibits three phases depending on the effective sample size $\ssize \lambda^2$:

\begin{description}
\item [Impossible phase.] When $\ssize \lambda^2 \ll \dim$, there is no consistent estimator for the non-Gaussian direction $\bm V$. This follows from standard lower bounds based on Fano's Inequality (see Proposition \ref{prop:ngca-it-lb} in Appendix \ref{sec:ngca-it-lb}). 
\item [Conjectured hard phase.] When $\dim \lesssim \ssize \lambda^2 \ll \dim^{k/2}$ and $\lambda \lesssim 1$, there is a consistent, but computationally inefficient estimator for the non-Gaussian direction $\bm V$ (\rd{provided Assumptions \ref{ass: min-snr} and \ref{ass: subgauss} hold}). This estimator is described and analyzed in Appendix~\ref{sec: ngca-bruteforce}. The lower bounds of \citet{diakonikolas2017statistical} show that SQ algorithms fail to estimate the non-Gaussian direction with polynomially many queries in this regime. This suggests that this regime is the conjectured hard phase for $k$-NGCA.  In Appendix~\ref{sec:ngca-LDLR}, we provide additional evidence for this using the low-degree likelihood ratio framework of \citet{hopkins2018statistical}. In the situation when the non-Gaussian measure $\nongauss$ is a mixture of Gaussians, similar lower bounds appear in the work of \citet{mao2021optimal}. Alternatively, low-degree lower bounds for this problem can also be derived from the SQ lower bounds of \citet{diakonikolas2017statistical} by verifying the general conditions proposed by \citet{brennan2020statistical}  which ensure equivalence between the low-degree computational model and the SQ model.%
\item [Easy phase.] When $\ssize \lambda^2 \gtrsim \dim^{k/2}$, there are polynomial-time estimators for $k$-NGCA.
  In Appendix~\ref{sec:spectral-ngca}, we study a spectral estimator for $k$-NGCA (with even $k$) that estimates the non-Gaussian direction $\bm V$ by the leading eigenvector $\hat{\bm V}$ (in the magnitude) of a data-dependent matrix $\hat{\bm M}$:
\begin{subequations} \label{eq: spectral-estimator-ngca-intro}
\begin{align}
    \hat{\bm M} &\explain{def}{=} \frac{1}{\ssize} \sum_{i=1}^\ssize (\|\bm x_i\|^2- \dim)^{\frac{k-2}{2}} \bm x_i \bm x_i^\UT - \E[(\|\bm z\|^2 -\dim)^{\frac{k-2}{2}} \bm z \bm z^\UT], \\
    \hat{\bm V} &\explain{def}{=} \max_{\|\bm u\| = 1} |\bm u^\UT  \hat{\bm M} \bm u|.
\end{align}
\end{subequations}  
In the above display $\bm z \sim \gauss{\bm 0}{\bm I_\dim}$. When $\ssize \lambda^2 \gg \dim^{\frac{k}{2}}$, we show that $\hat{\bm V}$ is a consistent estimator for the non-Gaussian direction (\rd{provided Assumptions \ref{ass: min-snr} and \ref{ass: subgauss} hold}). This estimator generalizes spectral estimators proposed in prior work of \citet{mao2021optimal} and \citet{davis2021clustering} for the special case $k=4$. 
\end{description}
\begin{remark}[Lattice Based Algorithms for Non-Gaussian Component Analysis] \label{remark:lattice}When the non-Gaussian measure is discrete or close to discrete, \citet{diakonikolas2021non} and \citet{zadik2021lattice} have designed estimators for the non-Gaussian direction which use $\ssize = \dim + 1$ samples and run in polynomial-time. In contrast, \citet{davis2021clustering} leverage the results of \citet{ghosh2020sum} show that estimators based on sum-of-squares relaxations fail to solve these instances when $\ssize \ll \dim^{3/2}$. Since we assume that the non-Gaussian distribution $\nongauss$ has a density with respect to $\gauss{0}{1}$ and the signal-to-noise ratio $\lambda \lesssim 1$ as $\dim \rightarrow \infty$, these estimators are not applicable to the instances of $k$-NGCA studied in this paper. 
\end{remark}

\subsection{Connections to Other Inference Problems} 
\label{sec: ngca-constructions}
By considering particular families of the non-Gaussian distribution $\nongauss$, we can relate $k$-NGCA to other inference problems.
In this section, we provide two constructions of non-Gaussian distributions satisfying the assumptions from Section~\ref{sec:ngca-assumptions}.
\begin{enumerate}
    \item In the first construction, the non-Gaussian distribution is a mixture of Gaussians. We use this construction to relate $k$-NGCA to the problem of estimating Gaussian mixture models. 
    \item The second construction is designed such that the likelihood ratio of $\nongauss$ with respect to the standard Gaussian measure $\gauss{0}{1}$ is uniformly bounded. We use this construction to relate $k$-NGCA to the problem of learning binary generalized linear models.
\end{enumerate}

\subsubsection{Learning Mixtures of Gaussians}

The following lemma provides a construction where $\nongauss$ is a mixture of Gaussian distributions. 

\begin{lemma} \label{lemma: ngca-mog-construction} 
For each even $k = 2\ell \in \N$, there are two positive constants $\lambda_k > 0$ and $K_k < \infty$ (depending only on $k$) such that for any $\lambda \in (0,\lambda_k/2]$, there is a probability measure $\nongauss$ has the following properties:
\begin{enumerate}
\item $\nongauss$ is a mixture of Gaussians on $\R$ with $\ell$ components with equal variances:
\begin{align*}
    \nongauss & = \sum_{i=1}^\ell p_i \cdot \gauss{w_i}{\sigma^2},
\end{align*}
for some probability weights $p_1,p_2,\dotsc,p_\ell$, mean parameters $w_1,w_2,\dotsc,w_\ell \in \R$, and variance parameter $0<\sigma^2<1$.
\item $\nongauss$ satisfies the Moment Matching Assumption (Assumption~\ref{ass: moment-matching}) with parameter $k$.
\item $\nongauss$ satisfies the Bounded Signal Strenth Assumption (Assumption~\ref{ass: bounded-snr}) with parameters $(\lambda,K_k)$. 
\item $\nongauss$ satisfies the Locally Bounded Likelihood Ratio Assumption (Assumption~\ref{ass: locally-bounded-LLR}) with parameters $(\lambda,K_k,\kappa = k)$.
\item $\nongauss$ satisfies the Minimum Signal Strength Assumption (Assumption~\ref{ass: min-snr}) with parameters $(\lambda,k)$.
\item $\nongauss$ is sub-Gaussian (Assumption~\ref{ass: subgauss}) with variance proxy $\varproxy=1$. 
\item Furthermore we have $\lambda^{1/k}/K_k \leq \min_{i \neq j} |w_i - w_j| \leq  \max_{i \neq j} |w_i - w_j| \leq K_k \cdot \lambda^{1/k}$.
\end{enumerate}
\end{lemma}
\begin{proof}
See Appendix \ref{appendix: ngca-mog-construction}. 
\end{proof}
This construction also appears in the work of \citet{diakonikolas2017statistical}, who use it to prove computational lower bounds for estimating Gaussian Mixture Models in the SQ model.
We use the above construction to relate $k$-NGCA to the problem of estimating Gaussian mixture models.

\paragraph{Mixtures of Gaussians and $k$-NGCA.}

Consider the problem of fitting a Gaussian mixture model, where the mixture components have identical but unknown covariance matrices.
Formally, one is given a dataset $\bm x_{1:\ssize} \in \R^\dim$ generated i.i.d.\ from the Gaussian mixture model:
\begin{align} \label{eq:gmm}
    \bm x_{1:\ssize} \explain{i.i.d.}{\sim} \sum_{i=1}^\ell p_i \cdot  \gauss{\bm \mu_i}{\bm \Sigma},
\end{align}
where the mean vectors $\bm \mu_{1:\ell}$ and the covariance matrix $\bm \Sigma$ are unknown.
The goal is to estimate the mean vectors $\bm \mu_{1:\ell}$.
Observe that when the dataset $\bm x_{1:\ssize}$ is generated by the $k$-NGCA model with the non-Gaussian measure $\nongauss$ from Lemma~\ref{lemma: ngca-mog-construction} and non-Gaussian direction $\bm V$, then, recalling \eqref{eq: ngca-model}, we obtain
\begin{align*}
    \bm x_{1:\ssize} & \explain{i.i.d.}{\sim} \sum_{i=1}^\ell p_i \cdot \gauss{\frac{w_i}{\sqrt{\dim}} \bm V}{ \bm I_\dim - \frac{(1-\sigma^2)}{\dim} \bm V \bm V^\UT}.
\end{align*}
This is an instance of a Gaussian mixture model~\eqref{eq:gmm}.
The parameter $\lambda$ from the Bounded Signal Strength Assumption (Lemma~\ref{lemma: ngca-mog-construction}), which determines the statistical difficulty of the $k$-NGCA problem, can be reinterpreted as the minimum separation between the component means:
\begin{align*}
    \lambda^{1/k} \asymp \min_{i \neq j} \|\bm \mu_i - \bm \mu_j\|.
\end{align*}
This is a natural notion of the signal strength for the fitting Gaussian mixture models.
The parameter $k$ from the Moment Matching Assumption (Lemma~\ref{lemma: ngca-mog-construction}), which determines the computational difficulty of the $k$-NGCA problem, relates to the number of mixing components in the Gaussian mixture model instance via the relation $k = 2\ell$.
In summary, the lower bounds we prove for the $k$-NGCA problem automatically yield lower bounds for estimating Gaussian mixture models. 

\subsubsection{Learning Binary Generalized Linear Models}

The following lemma provides a construction for $\nongauss$ designed such that the likelihood ratio of $\nongauss$ with respect to the standard Gaussian measure $\gauss{0}{1}$ is uniformly bounded. 

\begin{lemma} \label{lemma: ngca-construction-boundedLLR}
  For every $k \in \N$, there is a positive constant $\lambda_k>0$ that depends only on $k$ such that, for any $\lambda \in (0,\lambda_k]$, there is a probability measure $\nongauss$ with the following properties:
\begin{enumerate}
    \item $\nongauss$ satisfies the Moment Matching Assumption (Assumption~\ref{ass: moment-matching})  with parameter $k$.
    \item $\nongauss$ has a bounded density with respect to $\refmu = \gauss{0}{1}$:
    \begin{align*}
        \left|\frac{\diff \nongauss}{\diff \refmu}(x) - 1 \right| \leq \frac{\lambda}{\lambda_k} \leq 1.
    \end{align*}
    In particular, $\nongauss$ satisfies the Bounded Signal Strength Assumption (Assumption~\ref{ass: bounded-snr}) with parameters $(\lambda, K = 1/\lambda_k)$ and the  Locally Bounded Ratio Assumption (Assumption~\ref{ass: locally-bounded-LLR}) with parameters $(\lambda, K = 1/\lambda_k, \kappa = 0).$

    \item $\nongauss$ satisfies the Minimum Signal Strength Assumption (Assumption~\ref{ass: min-snr}) with parameters $(\lambda,k)$.

    \item $\nongauss$ is sub-Gaussian (Assumption~\ref{ass: subgauss}) with variance proxy $\varproxy \leq C$ for some universal constant $C$.

    \item When $k$ is odd, $\nongauss$ satisfies
    \begin{align*}
        \frac{1}{2} \cdot \left(\frac{\diff \nongauss}{\diff\refmu}(x) + \frac{\diff \nongauss}{\diff\refmu}(-x) \right) = 1.
    \end{align*}
\end{enumerate}
\end{lemma}
The proof of this result is provided in Appendix~\ref{appendix: ngca-boundedLLR-construction}.
We use the above construction to relate the $k$-NGCA problem to the problem of learning binary generalized linear models, which we introduce below.
The connection described below is also implicit in the work of \citet{diakonikolas2021optimality}, which studies SQ lower bounds for agnostic learning of half-spaces.

\paragraph{Generalized linear models and $k$-NGCA.}

Consider the problem of fitting a binary generalized linear model (GLM) with Gaussian covariates.
One observes a data set consisting of $\ssize$ feature-response pairs $\{(\bm f_i, r_i): \; i \in [\ssize]\} \subset \R^\dim \times \{0,1\}$ sampled i.i.d.\ as follows:
\begin{align} \label{eq:half-space-model}
    \bm f_i \sim \gauss{\bm 0}{\bm I_\dim}, \quad r_i \; | \; \bm f_i \sim \operatorname{Bernoulli}\left(\rho\bigg(\frac{\ip{\bm f_i}{\bm V}}{\sqrt{\dim}}\bigg)\right).
\end{align}
In the above display, $\bm V \in \R^\dim$ is the unknown parameter of interest with $\|\bm V\| = \sqrt{\dim}$, and $\rho: \R \rightarrow [0,1]$ is a known but arbitrary regression function.
The goal is to estimate the vector $\bm V$.

The GLM learning problem
is closely related to the $k$-NGCA problem, because for certain non-Gaussian measures $\nongauss$ (including those coming from Lemma~\ref{lemma: ngca-construction-boundedLLR} with $k$ odd), it is possible to transform a dataset $\bm x_{1:\ssize}$ sampled from the $k$-NGCA problem with non-Gaussian direction $\bm V$ into a dataset $\{(\bm f_i, r_i) : i \in [\ssize]\}$ sampled from the GLM~\eqref{eq:half-space-model}.
Consequently, estimators designed for learning GLMs can be used to solve the $k$-NGCA problem.
Hence, the lower bounds we prove for $k$-NGCA immediately yield lower bounds for the GLM learning problem.

We now describe the transformation that converts a dataset $\bm x_{1:\ssize}$ for the $k$-NGCA problem to a dataset $\{(\bm f_i, r_i) : i \in [\ssize]\}$ for the GLM learning problem:
\begin{align}\label{eq:halfspace-to-ngca-reduction}
    r_i \explain{i.i.d.}{\sim} \operatorname{Bernoulli}\left( \frac{1}{2} \right), \quad \bm f_i = (2r_i - 1) \cdot \bm x_i.
\end{align}
We verify that $(\bm f_i, r_i)$ are samples from \eqref{eq:half-space-model}.
First we observe that conditioned on $r_i$, we can compute the distribution of $\bm f_i$:
\begin{align} \label{eq:feature-given-response}
    \bm f_i \; |\; r_i = 1 \sim \dmu{\bm V}, \quad \bm f_i | r_i = 0 \sim \dmu{\bm V}^- ,
\end{align}
where $\dmu{\bm V}$ is the measure from \eqref{eq: ngca-model},
and $\dmu{\bm V}^-$ is the measure defined as follows by its likelihood ratio with respect to $\refmu = \gauss{\bm 0}{\bm I_\dim}$:
\begin{align*}
    \frac{\dmu{\bm V}^-}{\diff \refmu}(\bm x) \explain{def}{=} \frac{\dmu{\bm V}}{\diff \refmu}(-\bm x) \explain{\eqref{eq: ngca-likelihood}}{=} \frac{\diff \nongauss}{\diff \refmu}\bigg(-\frac{\ip{\bm x}{\bm V}}{\sqrt{\dim}}\bigg).
\end{align*}
If the likelihood ratio of the non-Gaussian distribution $\nongauss$ with respect to $\gauss{0}{1}$ satisfies
\begin{align} \label{eq:condition-for-reduction}
    \frac{1}{2} \cdot \left(\frac{\diff \nongauss}{\diff\refmu}(x) + \frac{\diff \nongauss}{\diff\refmu}(-x) \right) = 1,
\end{align}
then computing the marginal distribution of $\bm f_i$ from \eqref{eq:feature-given-response} yields $\bm f_i \sim \gauss{\bm 0}{\bm I_\dim}$.
The requirement in \eqref{eq:condition-for-reduction} is satisfied, for instance, when $\nongauss$ is the non-Gaussian measure constructed in Lemma \ref{lemma: ngca-construction-boundedLLR} for odd $k$. 
Under this condition, an application of Bayes' rule to \eqref{eq:feature-given-response} gives the conditional distribution of $r_i \; |  \; \bm f_i$:
\begin{align*}
    r_i \; | \; \bm f_i \sim \operatorname{Bernoulli}\left(\frac{1}{2} \cdot \frac{\diff\nongauss}{\diff \refmu}\bigg(\frac{\ip{\bm f_i}{\bm V}}{\sqrt{\dim}}\bigg)\right).
\end{align*}
This verifies the transformation in \eqref{eq:halfspace-to-ngca-reduction} produces an instance of the GLM learning problem with the regression function
\begin{align} \label{eq:nu-to-rho}
    \rho(\xi) = \frac{1}{2} \cdot \frac{\diff\nongauss}{\diff \refmu}(\xi).
\end{align} 
Finally, we estimate the parameters $\lambda$ and $k$, which respectively determine the statistical and computational difficulty of the $k$-NGCA problem in terms of the regression function $\rho$.
Recall that when the non-Gaussian measure satisfies the Minimum Signal Strength Assumption (Assumption~\ref{ass: min-snr}) and the Bounded Signal Strength Assumption (Assumption~\ref{ass: bounded-snr}), we have
\begin{align*}
    \lambda^2 \asymp \operatorname{Var}\left( \frac{\diff\nongauss}{\diff\refmu}(Z) \right), \quad k = \min \left\{ \ell \in \N: \E \left[ \frac{\diff\nongauss}{\diff\refmu}(Z) \cdot H_\ell(Z)\right] \neq 0 \right\}, \quad  Z \sim \gauss{0}{1}.
\end{align*}
Hence, \eqref{eq:nu-to-rho} shows that the statistical difficulty of the GLM learning problem is determined by
\begin{align}
    \lambda^2 &\asymp \operatorname{Var}\left( \rho(Z) \right), \label{eq:snr-halfspace}
\end{align}
where as the computational difficulty is determined by
\begin{align}
    k &= \min \left\{ \ell \in \N: \E \left[ \rho(Z) \cdot H_\ell(Z)\right] \neq 0 \right\}. \label{eq:comp-half-space}
\end{align}
We note that $\operatorname{Var}\left( \rho(Z) \right)$ appears to be a natural notion of signal strength for the GLM learning problem, since if $\operatorname{Var}\left( \rho(Z) \right) = 0$, we have $\rho(\xi) = 1/2$ almost everywhere.
This means that the feature and response are independent and carry no information about the parameter $\bm V$.
An analog of \eqref{eq:comp-half-space} (for the case when $k=2$) appears in the work of \citet[Theorem 2]{mondelli2019fundamental}, who show that if $\E[\rho(Z) H_2(Z)] = 0$, then a broad class of spectral estimators fail to have a non-trivial performance in the regime $\ssize \asymp \dim$.
Furthermore, the hard instance constructed in the work of \citet[Proposition 2.1]{diakonikolas2021optimality} to prove SQ lower bounds for agnostic learning of half-spaces has the property that the parameter $k$ (as defined in \eqref{eq:comp-half-space}) is large. 

\subsection{Computational Lower Bound}

The following is our computational lower bound for $k$-NGCA.

\begin{theorem} \label{thm:ngca}
  Consider the $k$-NGCA problem with non-Gaussian distribution $\nongauss$ satisfying
 \begin{enumerate}
    \item the Moment Matching Assumption (Assumption~\ref{ass: moment-matching}) with parameter $k\geq 2$ and $k \asymp 1$;
    \item the Bounded Signal Strength Assumption (Assumption~\ref{ass: bounded-snr}) with parameters $(\lambda,K \asymp 1)$;
    \item the Locally Bounded Likelihood Ratio Assumption (Assumption~\ref{ass: locally-bounded-LLR}) with parameters $(\lambda, K \asymp 1, \kappa \asymp 1)$.
\end{enumerate}
{Suppose that $\lambda \asymp \dim^{-\gamma}$ (as $\dim \rightarrow \infty$) for any constant $\gamma > 2 \lceil (k+1)/2 \rceil + \kappa$.} Let $\hat{\bm V} \in \R^\dim$ denote any estimator for this $k$-NGCA problem that can be computed using a memory bounded estimation algorithm with resource profile $(\ssize, \iters, \state)$  scaling with $\dim$ as %
 \begin{align*}
      \ssize \lambda^2  \asymp \dim^{\eta}, \quad \iters  \asymp \dim^\tau, \quad \state \asymp \dim^{\mu}
 \end{align*}
 for any constants $\eta \geq 1, \; \tau \geq 0, \; \mu \geq 0$. If
\begin{align*}
    \eta + \tau + \mu  < \left\lceil\frac{k+1}{2} \right\rceil,
\end{align*}
then, for any $t \in \R$,
\begin{align*}
    \limsup_{\dim \rightarrow \infty} \inf_{\bm V \in \paramspace} \P_{\bm V}\left( \frac{|\langle{\bm V,}{\hat{\bm V}} \rangle|^2}{\|\bm V\|^2 \|\hat{\bm V}\|^2}  \geq \frac{t^2}{d} \right) & \leq 2\exp\left( - \frac{t^2}{2}\right).
\end{align*}
\end{theorem}
Theorem~\ref{thm:ngca}
shows that if the signal-to-noise ratio $\lambda$ is sufficiently small,
then memory bounded estimation algorithms using too few total resources
(as measured by the product $\ssize \lambda^2 \cdot \iters \cdot \state$) perform no better than a random guess.

\subsection{Discussion of Theorem~\ref{thm:ngca}}

Theorem~\ref{thm:ngca} is quantitatively similar to the computational lower bound obtained for $k$-TPCA (modulo the condition on the signal-to-noise ratio), so most of the implications discussed in Section~\ref{sec:tpca-discussion} continue to hold.
This includes the following (again, just considering even $k$).
\begin{enumerate}
  \item Theorem~\ref{thm:ngca} gives a nearly tight lower bound on the total resources, as evidenced by the existence of the spectral estimator from~\eqref{eq: spectral-estimator-ngca-intro} that can be implemented by a memory bounded estimation algorithm with resource profile $(N \asymp d^{k/2} \cdot \polylog(d) / \lambda^2, T \asymp \polylog(d), s \asymp d \cdot \polylog(d))$.

  \item The run-time vs.\ sample-size trade-offs for (nearly) linear memory estimators, shown in Figure~\ref{fig:tpca-phase-diagram} for $k$-TPCA, also applies to $k$-NGCA.
    Nearly linear memory estimators for $k$-NGCA include the spectral estimator from \eqref{eq: spectral-estimator-ngca-intro}, the tensor power method on the empirical order-$k$ moment tensor, and gradient descent on natural non-convex objectives~\citep{wang2020efficient,davis2021clustering}.

  \item For many nearly linear memory algorithms, stronger iteration lower bounds can be obtained using Theorem~\ref{thm:ngca} in the low signal-to-noise regime as compared to the low-degree likelihood ratio framework \citep{hopkins2018statistical,kunisky2019notes}, which only yields lower bounds of the form $T \gtrsim \log(\dim)$.
\end{enumerate}

\begin{remark}The computational lower bound of Theorem~\ref{thm:ngca} applies only in the regime when the signal-to-noise ratio $\lambda^2$ is sufficiently small. This requirement is an inherent limitation of the proof technique, which derives a lower bound for memory bounded estimation algorithms from a communication lower bound for distributed estimation algorithms. In Remark~\ref{remark:small-snr-details}, we discuss a simple distributed estimation algorithm that rules out the required communication lower bound in the high signal-to-noise ratio regime.
\end{remark}

\subsection{Proof of Theorem \ref{thm:ngca}}

Similar to Theorems~\ref{thm:tpca} and~\ref{thm:main-result-atpca}, we prove
Theorem~\ref{thm:ngca}
 by transferring a communication lower bound for distributed estimation protocols for $k$-NGCA to memory bounded estimators for the same problem using the reduction in Fact~\ref{fact:reduction}.

In the (Bayesian) distributed setup for $k$-NGCA, the parameter $\bm V$ is drawn from the prior $\pi \explain{def}{=} \unif{\{\pm 1\}^\dim}$, and then $\bm x_{1:\ssize}$ are sampled i.i.d.\ from $\dmu{\bm V}$; these samples are distributed across $\mach = \ssize/\batch \in \N$ machines with $\batch$ samples/machine.
We will obtain Theorem~\ref{thm:ngca} with a suitable choice of $\batch$.
As usual, the execution of a distributed estimation protocol with parameters $(\mach, \batch, \budget)$ results in a transcript $\bm Y \in \{0,1\}^{\mach \budget}$ written on the blackboard.

We have the following corollary of
Fano's Inequality for Hellinger Information (Fact~\ref{fact: fano}), proved in exactly the same way as Corollary~\ref{coro: fano-stpca}.

\begin{corollary}[Fano's Inequality for $k$-NGCA] \label{coro: fano-ngca} For any estimator $\hat{\bm V}(\bm Y)$ for $k$-NGCA computed by a distributed estimation protocol, and for any $t \in \R$, we have
\begin{align*}
\inf_{\bm V \in \paramspace} \P_{\bm V}\left( \frac{|\langle{\bm V,}{\hat{\bm V}} \rangle|^2}{\|\bm V\|^2 \|\hat{\bm V}\|^2}  \geq \frac{t^2}{d} \right) & \leq 2\exp\left( - \frac{t^2}{2}\right) + \sqrt{2\MIhell{\bm V}{\bm Y}}.
\end{align*}
\end{corollary}

The main technical result is the following information bound for $k$-NGCA.

\begin{proposition}\label{prop:ngca-info-bound}
  Consider the $k$-NGCA problem with non-Gaussian distribution $\nongauss$ satisfying
 \begin{enumerate}
    \item the Moment Matching Assumption (Assumption~\ref{ass: moment-matching}) with parameter $k\geq 2$;
    \item the Bounded Signal Strength Assumption (Assumption~\ref{ass: bounded-snr}) with parameters $(\lambda,K)$;
    \item the Locally Bounded Likelihood Ratio Assumption (Assumption~\ref{ass: locally-bounded-LLR}) with parameters $(\lambda, K, \kappa)$.
\end{enumerate}
  Let $\bm Y \in \{0,1\}^{\mach \budget}$ be the transcript generated by a distributed estimation protocol for this $k$-NGCA problem with parameters $(\mach, \batch, \budget)$.
Let $q \geq 2$ be arbitrary but fixed constant. Then, there is a finite constant $ C_{k,K,\kappa,q}$ depending only on $(k,K,\kappa,q)$ such that if
\begin{equation} \label{eq:ngca-conditions}
    \batch \geq C_{k,K,\kappa,q} \cdot b \cdot \dim^{\kappa +  \lceil \frac{k+1}{2} \rceil} \quad \text{and} \quad \batch \lambda^2 \leq \frac{1}{C_{k,K,\kappa,q}},
\end{equation}
then
\begin{align*}
    \MIhell{\bm V}{\bm Y} & \leq C_{k,K,\kappa,q} \cdot \left( \budget \cdot \mach  \batch  \lambda^2 \cdot \dim^{-\lceil \frac{k+1}{2} \rceil}  + \mach \cdot (\batch \lambda^2)^2 + \frac{\mach}{\dim^{\frac{q}{2}}}  + \mach \cdot \batch \cdot (\mach + \batch) \cdot e^{-\dim/C_{k,K,\kappa,q}} \right) .
\end{align*}
\end{proposition}

Proposition~\ref{prop:ngca-info-bound} is proved in Appendix~\ref{appendix:ngca}.
With this information bound in hand, we can complete the proof of Theorem~\ref{thm:ngca}.

\begin{proof}[Proof of Theorem~\ref{thm:ngca}]
  Appealing to the reduction in Fact~\ref{fact:reduction}, we note that any memory-bounded estimator $\hat{\bm V}$ with resource profile $(\ssize, \iters, \state)$ can be implemented using a distributed estimation protocol with parameters $(\ssize/\batch, \batch, \state \iters)$ for any $\batch \in \N$ such that $\mach := \ssize/\batch \in \N$.
As assumed in Theorem~\ref{thm:ngca}, we consider the situation when
\begin{align} \label{eq:thm-assump}
     \eta + \tau + \mu  < \left\lceil\frac{k+1}{2} \right\rceil, \quad \gamma > 2 \left\lceil\frac{k+1}{2} \right\rceil + \kappa.
\end{align}
We set $\batch = \dim^\xi$ with
\begin{align} \label{eq:xi-choice}
    \xi & \explain{def}= \tau + \mu + \kappa + \left\lceil\frac{k+1}{2} \right\rceil + \frac{1}{2}\underbrace{\left(\left\lceil\frac{k+1}{2} \right\rceil -  (\eta + \tau + \mu)  \right)}_{>0} > \tau + \mu + \kappa + \left\lceil\frac{k+1}{2} \right\rceil.
\end{align}
With this choice, we verify that the information bound in Proposition~\ref{prop:ngca-info-bound} shows that $\MIhell{\bm V}{\bm Y} \rightarrow 0$; combining this with Corollary~\ref{coro: fano-ngca} proves the theorem.
We begin by observing
\begin{align*}
    \gamma &>\left\lceil\frac{k+1}{2} \right\rceil + \kappa + \eta + \tau + \mu + \left( \left\lceil\frac{k+1}{2} \right\rceil - (\eta + \tau + \mu)  \right) \\
    &= \eta + \xi + \frac{1}{2} \left( \left\lceil\frac{k+1}{2} \right\rceil - (\eta + \tau + \mu)  \right) \\
    &> \eta + \xi.
\end{align*}
Next, we verify the conditions required for Proposition~\ref{prop:ngca-info-bound}:
\begin{enumerate}
    \item Since $\eta > \tau + \mu + \kappa + \lceil (k+1)/2 \rceil$ we have $\batch \gg \budget \cdot \dim^{\kappa +  \lceil \frac{k+1}{2} \rceil}$ as required. 
    \item Since $\gamma > \eta + \xi > \xi$ we have $\batch \lambda^2 \ll 1$ as required. 
\end{enumerate}
Now, from the information bound of Proposition~\ref{prop:ngca-info-bound}, for any $q \geq 2$, we have:
\begin{align*}
    \MIhell{\bm V}{\bm Y} & \leq C_{k,K,\kappa,q} \cdot \left( \budget \cdot \mach  \batch  \lambda^2 \cdot \dim^{-\lceil \frac{k+1}{2} \rceil}  + \mach \cdot (\batch \lambda^2)^2 + \frac{\mach}{\dim^{\frac{q}{2}}}  + \mach \cdot \batch \cdot (\mach + \batch) \cdot e^{-\dim/C_{k,K,\kappa,q}} \right) \\
    & = C_{k,K,\kappa,q} \cdot \left( \budget \cdot \ssize  \lambda^2 \cdot \dim^{-\lceil \frac{k+1}{2} \rceil}  + (\ssize \lambda^2) \cdot (\batch \lambda^2) + \frac{\mach}{\dim^{\frac{q}{2}}}  + \mach \cdot \batch \cdot (\mach + \batch) \cdot e^{-\dim/C_{k,K,\kappa,q}} \right).
\end{align*}
We now check that this bound on $\MIhell{\bm V}{\bm Y}$ vanishes with $d\to\infty$ with a suitable choice of $q$:
\begin{enumerate}
    \item The assumption $\eta + \tau + \mu < \lceil (k+1)/2 \rceil$ guarantees $\budget \cdot \ssize  \lambda^2 \cdot \dim^{-\lceil \frac{k+1}{2} \rceil} \to 0$.
    \item Since $\gamma > \eta + \xi$, we have $(\ssize \lambda^2) \cdot (\batch \lambda^2) \to 0$. 
    \item Observe that $\mach = (\ssize \lambda^2)/(\batch \lambda^2) = \dim^{\gamma + \eta - \xi} \geq \dim^{2\eta}$. Hence, choosing $q = 2(\gamma + \eta) + \eta$ ensures $\mach / \dim^{q/2} \to 0$.
    \item Since $\batch, \mach$ scale polynomially with $\dim$, we have $\mach \cdot \batch \cdot (\mach + \batch) \cdot e^{-\dim/C_{k,K,\kappa,q}} \to 0$.
\end{enumerate}
This concludes the proof. 
\end{proof}

\begin{remark} \label{remark:small-snr-details} The computational lower bound in Theorem \ref{thm:ngca} requires that $\lambda^2$ is sufficiently small because the information bound in Proposition \ref{prop:ngca-info-bound} holds when $\batch$ is sufficiently large and $\batch \lambda^2$ is sufficiently small \eqref{eq:ngca-conditions}. In light of this, a natural question is whether an information bound of the form:
\begin{align} \label{eq:info-bound-impossible}
    \MIhell{\bm V}{\bm Y} & \explain{??}{\lesssim} \budget \cdot \mach  \batch  \lambda^2 \cdot \dim^{-\lceil \frac{k+1}{2} \rceil}, 
\end{align}
holds without assuming $\lambda^2$ is small. Unfortunately, an information of the form \eqref{eq:info-bound-impossible} is ruled out by a simple distributed estimator unless $\lambda^2 \lesssim \dim^3 \polylog{(\dim)}/\dim^{\lceil \frac{k+1}{2} \rceil}$. This distributed estimator uses the information theoretically optimal sample size of $\ssize = \mach \batch \asymp \dim/\lambda^2$ distributed across $m = \ssize/\batch$ machines with $\batch$ samples per machine. The estimator simply writes the entire dataset on the blackboard as the transcript and computes the Maximum Likelihood Estimator using the transcript. Since each sample is a $\dim$-dimensional real-valued vector, which can be quantized  to a $\dim \polylog{(\dim)}$ bit vector, the total bits written on the black board are $\mach \budget = \ssize \dim \polylog{(\dim)} = \dim^2 \polylog{(\dim)}/ \lambda^2$. Since the maximum likelihood estimator is consistent, we must have $ \MIhell{\bm V}{\bm Y} \gtrsim 1$. This leads to a contradiction to  \eqref{eq:info-bound-impossible} unless $\batch \gtrsim \dim^{\lceil \frac{k+1}{2} \rceil - 2}/\polylog{(\dim)}$. Since $\batch \leq \ssize \asymp \dim/\lambda^2$, this means that the information bound \eqref{eq:info-bound-impossible} cannot hold unless $\lambda^2 \lesssim \dim^3 \polylog{(\dim)}/\dim^{\lceil \frac{k+1}{2} \rceil}$.  For sufficiently large $k$ ($k \geq 5$), this means that \eqref{eq:info-bound-impossible} cannot hold unless $\lambda^2 \ll 1$. 
\end{remark}

\section{Canonical Correlation Analysis} \label{sec:cca}

\subsection{Problem Formulation}

In the order-$k$ Canonical Correlation Analysis ($k$-CCA) problem, one observes a dataset of $\ssize$ i.i.d.\ samples $\bm x_{1:\ssize}$, in which
each $\bm x_i \in \R^{k\dim}$ consists of $k$ ``views'' (or ``modes''):
\begin{equation*}
  \bm x_i = (\mview{\bm x_i}{1}, \mview{\bm x_i}{2}, \dotsc, \mview{\bm x_i}{k})^{\UT}.
\end{equation*}
The correlation structure between the different views is such that
\begin{align} \label{eq:k-CCA-correlation}
    \E \left[ \mview{\bm x_i}{1} \otimes \mview{\bm x_i}{2} \otimes \dotsb \otimes \mview{\bm x_i}{k} \right] = \frac{\lambda}{\sqrt{\dim^k}} \bm V,
\end{align}
where $\bm V \in \tensor{\R^\dim}{k}$ is the rank-$1$ cross-moment tensor:
\begin{align*}
    \bm V & = \sqrt{\dim^k} \cdot  \bm v_1 \otimes \bm v_2 \otimes \dotsb \otimes \bm v_k,
\end{align*}
for some unit vectors $\bm v_1, \bm v_2, \dotsc, \bm v_k \in \R^{\dim}$.
Note that $\|\bm V\| = \sqrt{\dim^k}$.
The parameter $\lambda > 0$ is the signal-to-noise ratio parameter.
The goal is to estimate the cross-moment tensor $\bm V$.

Note that we have not explicitly specified the probability measure $\dmu{\bm V}$ of $\bm x_i$, as the goal of estimating correlation structure is often considered in a non-parametric setting.
However, our lower bounds will consider a particular measure $\dmu{\bm V}$ specified by its likelihood ratio with respect to $\refmu = \gauss{\bm 0}{\bm I_{k\dim}}$:
\begin{subequations} \label{eq:kCCA-likelihood}
 \begin{align}
     \frac{\diff \dmu{\bm V}}{\diff \refmu}(\bm x) \explain{def}{=} 1 + \frac{\lambda}{\lambda_k} \cdot \sgn\left(\frac{\ip{\mview{\bm x}{1} \otimes \dotsb \otimes \mview{\bm x}{k}}{\bm V}}{\sqrt{\dim^k}}\right),
 \end{align}
 where
 \begin{align}
     \lambda_k \explain{def}{=} \left(\frac{2}{\pi} \right)^{\frac{k}{2}} = (\E |Z|)^{\frac{k}{2}}, \quad Z \sim \gauss{0}{1}.
 \end{align}
 \end{subequations}
This is a valid probability distribution that satisfies \eqref{eq:k-CCA-correlation} as long as $0 \leq \lambda \leq \lambda_k$.

Finally, our computational lower bound for $k$-CCA will hold even under further restrictions on $\bm V$, namely
\begin{align}
  \label{eq:kCCA-coordprior}
  \bm V = \sqrt{d^k} \cdot \bm e_{i_1}\otimes \bm e_{i_2} \dotsb \otimes \bm e_{i_k}
\end{align}
for some $\{ i_1, i_2, \dotsc, i_k \} \subseteq [\dim]$, where $\bm e_1, \bm e_2, \dotsc, \bm e_{\dim}$ are the standard basis vectors in $\R^\dim$.
This restriction is relevant to the connection between $k$-CCA and the parity learning problem.

\subsection{Statistical-Computational Gap in $k$-CCA}

The $k$-CCA problem exhibits the same computational gap as the other inference problems studied in this paper.
Depending upon the effective sample size $\ssize \lambda^2$, the $k$-CCA problem exhibits the following three phases:

\begin{description}
\item [Impossible phase.] When the effective sample size $\ssize \lambda^2 \ll \dim$, there is no consistent estimator for $\bm V$. This follows from standard lower bounds based on Fano's Inequality.\footnote{This follows from similar arguments as those used to prove the corresponding result for $k$-NGCA in Appendix~\ref{sec:ngca-it-lb}.}

\item [Conjectured hard phase.]
  In the regime $\dim \lesssim \ssize \lambda^2 \ll \dim^{k/2}$, there is a consistent, but computationally inefficient estimator for the cross-moment tensor $\bm V$; see Appendix~\ref{sec: cca-bruteforce} for details.
  We believe that no polynomial-time estimator can recover $\bm V$ in this phase, and in Appendix~\ref{sec:CCA-LDLR}, we give evidence for this conjecture using the low-degree likelihood ratio framework.

\item [Easy phase.] In the regime $\ssize \lambda^2 \gg \dim^{k/2}$, there are polynomial-time estimators for the $k$-CCA problem.
  The correlation structure in \eqref{eq:k-CCA-correlation} suggests that $\bm V$ can be estimated by the rank-$1$ approximation to the empirical cross-moment tensor:
\begin{align*}
    \hat{\bm T} & = \frac{1}{\ssize} \sum_{i=1}^\ssize \mview{\bm x}{1}_i \otimes \mview{\bm x}{2}_i \otimes \dotsb \otimes \mview{\bm x}{k}_i. 
\end{align*}
However, computing a rank-$1$ approximation to an order-$k$ tensor is non-trivial for $k \geq 3$.
For even $k = 2\ell$, we can reshape $\hat{\bm T}$ to a $d^{\frac{k}{2}} \times d^{\frac{k}{2}}$ matrix $\operatorname{Mat}(\bm T)$, as was done for $k$-ATPCA in~\eqref{eq:matricization}.
To estimate $\bm V$, we first estimate $\mat(\bm V)$ by computing the best rank-$1$ approximation to $\mat(\hat{\bm T})$ using SVD:
\begin{subequations} \label{eq: spectral-CCA-intro}
\begin{align}
    (\mview{\hat{\bm U}}{L}, \mview{\hat{\bm U}}{R}) \explain{def}{=} \arg \max_{\substack{\|\mview{\bm U}{L}\| = 1 \\ \|\mview{\bm U}{R}\| = 1}} \ip{\mview{\bm U}{L}}{\mat(\hat{\bm T}) \cdot \mview{\bm U}{R}}. 
\end{align}
We then construct an estimate $\hat{\bm V}$ of $\bm V$ by reshaping $\mview{\hat{\bm U}}{L} \otimes {\mview{\hat{\bm U}}{R}}$ into a tensor:
\begin{align}
    \hat{\bm V} \explain{def}{=} \mat^{-1}(\mview{\hat{\bm U}}{L} \otimes {\mview{\hat{\bm U}}{R}}).
\end{align}
\end{subequations}
Under an additional concentration assumption, we analyze this spectral estimator in Appendix~\ref{sec:cca-spectral} and show that when $\ssize \lambda^2 \gg \dim^{k/2}$, $\hat{\bm V}$ is a consistent estimator for $\bm V$. 
\end{description}

\subsection{Computational Lower Bound for $k$-CCA}
The following is our computational lower bound for $k$-CCA.
\begin{theorem} \label{thm:cca} Consider the $k$-CCA problem for $k\geq 2$ with signal-to-noise ratio $\lambda^2 \asymp \dim^{-\gamma}$ (as $\dim \rightarrow \infty$) for any constant $\gamma > 3k/2$. Let $\hat{\bm V} \in \tensor{\R^\dim}{k}$ denote any estimator for this $k$-CCA problem that can be computed using a memory bounded estimation algorithm with resource profile $(\ssize, \iters, \state)$ scaling with $\dim$ as
 \begin{align*}
 \ssize \lambda^2  \asymp \dim^{\eta}, \quad \iters  \asymp \dim^\tau, \quad \state \asymp \dim^{\mu}
 \end{align*}
 for any constants $\eta \geq 1, \; \tau \geq 0, \; \mu \geq 0$.
 If
\begin{align*}
    \eta + \tau + \mu  < k,
\end{align*}
then, for any $t \in \R$,
\begin{align*}
    \limsup_{\dim \rightarrow \infty} \inf_{\bm V \in \paramspace} \P_{\bm V}\left( \frac{|\langle{\bm V,}{\hat{\bm V}} \rangle|^2}{\|\bm V\|^2 \|\hat{\bm V}\|^2}  \geq \frac{t^2}{d^k} \right) & \leq \frac{1}{t^2}.
\end{align*}
These results hold even when $\bm V$ and $\dmu{\bm V}$ are promised to satisfy~\eqref{eq:kCCA-coordprior} and~\eqref{eq:kCCA-likelihood}.
\end{theorem}
Theorem~\ref{thm:cca}
shows that if the signal-to-noise ratio $\lambda$ is sufficiently small,
then memory bounded estimation algorithms using too few total resources
(as measured by the product $\ssize \lambda^2 \cdot \iters \cdot \state$) perform no better than a random guess.

Given the close relationship between
$k$-CCA and $k$-ATPCA (analogous to that between $k$-NGCA and $k$-TPCA),
it is not surprising that Theorem~\ref{thm:cca} and Theorem~\ref{thm:main-result-atpca} are quantitatively similiar
(modulo the condition on the signal-to-noise ratio).
So, most of the implications discussed in Section~\ref{sec:atpca-discussion} regarding $k$-ATPCA continue to hold for $k$-CCA.

\subsection{Connections to Learning Parities} \label{sec:parity}

Learning parity functions from labeled examples is a well-studied problem in computational learning theory with numerous connections to cryptography and coding theory~\citep{berlekamp1978inherent,blum1993cryptographic,blum1994weakly,blum2003noise,lyubashevsky2005parity,feldman2009agnostic,valiant2015finding,klivans2014embedding,kiltz2017efficient,raz2018fast,garg2021memory,garg2019time}.
In our generalization of this problem, one observes a data set consisting of $\ssize$ feature-response pairs $\{(\bm f_i, r_i): \; i \in [\ssize]\} \subset \R^\pdim \times \{0,1\}$ sampled i.i.d.\ as follows:
\begin{align} \label{eq:parity-model}
    \bm f_i \sim \gauss{\bm 0}{\bm I_\pdim}, \quad r_i \; | \; \bm f_i \sim \mathrm{Bernoulli}\left(\frac{1}{2}+\frac{\psnr}{2} \cdot \prod_{j=1}^k \sgn(\ip{\bm v_j}{\bm f_i}) \right).
\end{align}
In the above display, $\bm v_1, \bm v_2, \dotsc, \bm v_k \in \R^\pdim$ are unknown parameters with $\|\bm v_i\| = 1$, and $\psnr \in [0,1]$ controls the signal-to-noise ratio of the problem.
The goal is to estimate the parameter $\bm V$:
\begin{align*}
    \bm V & = \sqrt{\pdim^k} \cdot  \bm v_1 \otimes \bm v_2 \otimes \dotsb \otimes \bm v_k. 
\end{align*}
Depending on the assumptions made on $k$ and $\bm v_{1:k}$, one obtains the following different variants of the original parity learning problem:
\begin{enumerate}
    \item
      If $\bm v_1 = \bm e_{i_1}, \bm v_2 =  \bm e_{i_2}, \dotsc, \bm v_k = \bm e_{i_k}$ for some unknown subset $\{i_1, i_2, \dotsc, i_k\} \subset [\pdim]$, then this is the problem of \emph{learning $k$-sparse parities with noise} ($k$-LPN).
      Here, $\bm e_1, \bm e_2, \dotsc, \bm e_\pdim$ are the standard basis vectors in $\R^\pdim$, and we typically consider $k \asymp 1$.

    \item
      The generalization of $k$-LPN where $k \in [\pdim]$ is arbitrary (possibly growing with $\pdim$, and also possibly unknown) is called the problem of \emph{learning (non-sparse) parities with noise} (LPN).

    \item
      If $\bm v_1, \bm v_2, \dotsc, \bm v_k$ is an unknown collection of mutually orthogonal unit vectors, then this is the problem of \emph{learning $k$-sparse parities with noise in an unknown basis}.
\end{enumerate}
The computational lower bounds for $k$-CCA derived in this paper have interesting implications for each of the three variants of the parity problem introduced above.
This is because it is possible the hard instance of $k$-CCA used to prove the computational lower bounds in this paper can be transformed into an instance of $k$-LPN (for odd $k$).
Since
$k$-LPN
is the simplest of the three variants of the parity learning problem introduced above, an estimator for any of the three variants can be used to solve a $k$-LPN instance.
This means that
the lower bounds for $k$-CCA derived in this paper immediately yield computational lower bounds for each of the variants of parity learning problem mentioned above.
To make this connection precise, we give a reduction from the hard instance of $k$-CCA studied in this paper to $k$-LPN.

\paragraph{Reduction to $k$-LPN.}

In the hard instances of $k$-CCA considered in Theorem~\ref{thm:cca}, the cross-moment tensor has the form $\bm V = \sqrt{\dim^k} \bm v_1 \otimes \bm v_2 \otimes \dotsb \otimes \bm v_k$ where $\bm v_j = \bm e_{i_j}$ for some $i_1, i_2, \dotsc, i_k \in [\dim]$, as per~\eqref{eq:kCCA-coordprior}.
The dataset $\bm x_{1:\ssize} \in \R^{k\dim}$ is sampled i.i.d.\ from the probability distribution $\dmu{\bm V}$ defined via the likliehood ratio in~\eqref{eq:kCCA-likelihood}

We transform a $k$-CCA dataset $\bm x_{1:\ssize}$ into the $k$-LPN dataset $\{(\bm f_i, r_i) : i \in [\ssize]\} \subset \R^{k\dim} \times \{0,1\}$ as follows:
\begin{align*}
    r_i \explain{i.i.d.}{\sim} \mathrm{Bernoulli}\left(\frac{1}{2}\right) , \quad \bm f_i = (2 r_i - 1) \bm x_i. 
\end{align*}
Since this specifies the joint distribution of $(r_i, \bm f_i)$, one can compute the marginal distribution of $\bm f_i$ and the conditional distribution of $r_i$ given $\bm f_i$ using this information. When $k$ is odd and if $\bm x_{1:\ssize} \explain{i.i.d.}{\sim} \dmu{\bm V}$ for $\bm V =  \sqrt{d^k} \cdot \bm e_{i_1}\otimes \bm e_{i_2} \dotsb \otimes \bm e_{i_k}: \; i_{1:k} \in [\dim]$, we find that
\begin{align*}
    \bm f_i \explain{}{\sim} \gauss{\bm 0}{\bm I_{k\dim}}, \quad r_i | \bm f_i \sim \mathrm{Bernoulli}\left(\frac{1}{2}+\frac{\lambda}{2\lambda_k} \cdot \prod_{j=1}^k \sgn(\ip{\bm v_j}{\bm f_j}) \right),
\end{align*}
where
\begin{align*}
    \bm v_j & = \left(\underbrace{\bm 0, \bm 0, \dotsc, \bm 0}_{\text{$j-1$ times}}, \bm e_{i_j}, \underbrace{\bm 0, \bm 0, \dotsc, \bm 0}_{\text{$k-j$ times}} \right) \; \forall \; j \; \in \; [k].
\end{align*}
This verifies that $\{(\bm f_i, r_i) : i \in [\ssize]\} \subset \R^{k\dim} \times \{0,1\}$ is an instance of the $k$-LPN in dimension $\pdim = k\dim$ with signal-to-noise ratio $\psnr = \lambda/\lambda_k$.

\paragraph{Implications for Learning (Non-Sparse) Parities with Noise.}

To discuss the implications of the computational lower bound in Theorem \ref{thm:cca}, we focus on the problem of learning non-sparse parities.
Recall that in this problem, one is given a data set consisting of $\ssize$ feature-response pairs $\{(\bm f_i, r_i): \; i \in [\ssize]\} \subset \R^\pdim \times \{0,1\}$ sampled i.i.d.\ as follows:
\begin{align}\label{eq:non-sparse-parity-model}
    \bm f_i \sim \gauss{\bm 0}{\bm I_\pdim}, \quad r_i \sim \mathrm{Bernoulli} \left( \frac{1}{2} + \frac{\psnr}{2} \cdot \prod_{j \in S} \sgn(f_{ij})  \right),
\end{align}
where $S \subset [\pdim]$ is the unknown parameter of interest.
While this problem can be solved efficiently with $\ssize = \pdim$ samples using Gaussian elimination when $\psnr = 1$ (the noiseless setting), this problem is believed to exhibit a large computational gap when $\psnr < 1$ (the noisy setting).
The MLE for this problem consistently estimates $S$ with a sample size $\ssize \gtrsim \pdim/\psnr^2$, but requires an exhaustive search over all $2^\pdim$ possible subsets of $[\pdim]$.
No estimator with a $\mathrm{poly}(\pdim,1/\psnr)$ sample complexity and $\mathrm{poly}(\pdim,1/\psnr)$ run-time is currently known.
Some notable algorithms\footnote{%
  These works in fact study the Boolean version of the (non-sparse) parity problem where the features are drawn from $\unif{\{\pm 1\}^\pdim}$.
  However, the Gaussian and Boolean parity problems are statistically and computationally equivalent.
  Given a sample $(\bm f,y)$ from the Gaussian parity problem, $(\sgn(\bm f),y)$ is a sample from the Boolean parity problem where $\sgn(\cdot)$ acts entry-wise on $\bm f$.
  Likewise, given a sample $(\bm b, y)$ from the Boolean parity problem, $(\bm b \odot |\bm g|, y)$ is a sample from the Gaussian parity problem where $\bm g \sim \gauss{\bm 0}{\bm I_\pdim}$ and $|\bm g|$ is the entry-wise absolute value of $\bm g$ and $\bm b \odot |\bm g|$ is the entry-wise product of $\bm b$ and $|\bm g|$.%
}
that improve over the run-time of exhaustive search include the following.
\begin{enumerate}
    \item An algorithm due to \citet{blum2003noise} that solves LPN with $\ssize = 2^{O(\pdim/\log(\pdim))}$ samples and run-time in the regime $\psnr \geq 2^{-O(\pdim^\delta)}$ for any $\delta < 1$.\footnote{Though \citeauthor{blum2003noise} only state their result in the regime $\psnr \asymp 1$, their algorithm works in the regime $\psnr \geq 2^{-O(\pdim^\delta)}$ for any $\delta<1$, as stated in \citet[Propsition 4]{lyubashevsky2005parity}.}
    \item An algorithm due to \citet{lyubashevsky2005parity} that solves LPN using $\ssize \lesssim \pdim^{1+\epsilon}$ and run-time $2^{O(\pdim/\log\log(\pdim))}$ in the regime $\psnr \geq 2^{-O(\log^{\delta}(\pdim))}$ for any $\epsilon>0$ and $\delta<1$.
    \item An algorithm due to \citet{valiant2015finding} that solves $k$-LPN using $\ssize \lesssim \pdim^{(1+\epsilon)2k/3}/\psnr^{2+\epsilon}$ and run-time $O((\pdim^{(1+\epsilon)k/3}/\psnr^{2+\epsilon})^\omega)$ for any $\epsilon>0$, where $\omega < 2.372$ is the matrix multiplication exponent.
      Note that the exponent on $\pdim$ in the run-time is less than $0.8k$.
\end{enumerate}
The SQ framework has been used to provide evidence for the hardness of learning parities in the work of \citet{kearns1998efficient} and \citet{blum1994weakly}. The latter work shows that any SQ algorithm which learns noisy parities with a sample size $\ssize \leq 2^{\pdim/3}$ must make at least $2^{\pdim/3}/2$ queries.  Using the reduction between $k$-CCA and $k$-LPN outlined previously, we can obtain the following corollary for learning (non-sparse) parities.
\begin{corollary}\label{coro:parity} Consider the problem of learning non-sparse parities in dimension $\pdim$ with signal-to-noise ratio $\psnr^2 \asymp \pdim^{-\gamma}$ (as $\pdim \to \infty$).  Let $\hat{S}$ be any estimator of $S$ computed using a memory bounded estimation algorithm with resource profile $(\ssize, \iters, \state)$ scaling with $\pdim$ as
\begin{align*}
     \ssize \psnr^2  \asymp \pdim^{\eta}, \quad \iters  \asymp \pdim^\tau, \quad \state \asymp \frac{\pdim^{\mu}}{\psnr^\alpha}
\end{align*}
for any constants $\eta \geq 1, \; \tau \geq 0, \; \mu \geq 0, \; \alpha < 4/3$.
If
\begin{align*}
    \gamma > \frac{2(\eta + \tau + \mu + 2)}{4/3 - \alpha},
\end{align*}
then
\begin{align*}
    \lim_{\pdim \rightarrow \infty} \inf_{S \subset [\pdim]} \P_S\left( S = \hat{S}\right) & = 0.
\end{align*}
\end{corollary}

Informally, the above corollary shows that for any $\alpha < 4/3$, there is no memory-bounded estimation algorithm which solves the parity problem with an effective sample size $\ssize \psnr^2 = \mathrm{poly}(\pdim)$, a memory state of size $\state = \mathrm{poly}(\pdim)/\psnr^{\alpha}$ after making $T = \mathrm{poly}(\pdim)$ passes through the data set, provided the signal-to-noise ratio $\psnr$ is sufficiently small. 

\begin{proof}[Proof of Corollary \ref{coro:parity}] Let $k \in \N$ be a parameter to be determined. Consider an arbitrary memory bounded estimation algorithm for LPN with signal-to-noise ratio $\psnr$ which has a resource profile $(\ssize, \iters, \state)$ where,
\begin{align*}
     \psnr^2 \asymp \pdim^{-\gamma}, \; \ssize \psnr^2  \asymp \pdim^{\eta}, \; \iters  \asymp \pdim^\tau, \; \state \asymp \frac{\pdim^{\mu}}{\psnr^\alpha},
\end{align*}
for arbitrary constants $\eta \geq 1, \; \tau \geq 0, \; \mu \geq 0, \alpha < 4/3$.
As a consequence of the reduction from $k$-CCA to $k$-LPN, we obtain using Theorem~\ref{thm:cca} that, if we choose $k$ odd such that
\begin{align} \label{eq:interval}
    k \in \left(\eta + \tau + \mu + \frac{\alpha \gamma}{2},\frac{2\gamma}{3} \right),
\end{align}
then
\begin{align*}
    \lim_{\pdim \rightarrow \infty} \inf_{\substack{S \subset [\pdim]\\ |S| = k}} \P_S\left( S = \hat{S}\right) & = 0.
\end{align*}
Under the assumptions on $\gamma, \alpha$ stated in the corollary, the interval in \eqref{eq:interval} is non-empty and has a width $>2$. Hence, one can indeed find an odd $k \in \N$ which satisfies \eqref{eq:interval}. Hence, the claim of the corollary follows. 
\end{proof}

\subsubsection{Comparison to Prior Works}
\label{sec:raz-comparision}
 A recent line of work initiated by \citet{steinhardt2016memory} and \citet{raz2018fast} has obtained memory vs.\ sample-size lower bounds for \emph{single-pass} memory-bounded estimation algorithms for learning parities:
\begin{enumerate}
    \item \citet{raz2018fast} showed that $1$-pass ($\iters = 1$) memory-bounded estimation algorithms for learning noiseless ($\psnr=1$) parities require either a memory state of size $\state \gtrsim \pdim^2$ or an exponential sample size $\ssize \geq 2^{\Omega(\pdim)}$, proving a conjecture of \citet{steinhardt2016memory}. 
    \item \citet{garg2021memory} studied the problem of learning noisy parities (i.e., $\psnr \in (0,1)$) using the techniques of \citeauthor{raz2018fast} and showed that $1$-pass ($T=1$) memory-bounded estimation algorithms for learning noisy parities require either a memory state of size $\state \gtrsim \pdim^2/\psnr$ or an exponential sample size $\ssize \geq 2^{\Omega(\pdim)}$.
    \item \citeauthor{garg2021memory} conjectured that $1$-pass ($T=1$) memory-bounded estimation algorithms for learning noisy parities require either a memory state of size $\state \gtrsim \pdim^2/\psnr^2$ or an exponential sample size $\ssize \geq 2^{\Omega(\pdim)}$. The information-theoretic sample complexity of learning noisy parities scales as $\ssize \asymp \pdim/\psnr^2$. Hence, an interpretation of this conjecture is that any estimation algorithm which learns noisy parities with $\ssize = \mathrm{poly}(\pdim)$ sample complexity must have the capacity to memorize a dataset of size $\ssize \asymp \pdim/\psnr^2$ (the information-theoretic sample complexity). 
\end{enumerate}

In comparison to the results discussed above, a key weakness of the lower bound in Corollary \ref{coro:parity} is that it requires the signal-to-noise ratio $\psnr$ to decay as a sufficiently large power of $\pdim$. In contrast the results of \citeauthor{raz2018fast} and \citeauthor{garg2021memory} can allow any $\psnr \in (0,1]$. This is a limitation of the proof approach which relies on the connection between estimation with limited memory and estimation with limited communication in a distributed setting (recall Fact \ref{fact:reduction}). The techniques used by \citeauthor{raz2018fast} and \citeauthor{garg2021memory} are very different and do not rely on this connection. On the other hand, an advantage of the lower bounds obtained using communication complexity is that they apply to multi-pass estimation algorithms whereas it seems challenging to extend the approach of \citeauthor{raz2018fast} to the multi-pass setting. The work of \citet{garg2019time} is the current state-of-the-art result in this direction and shows that $2$-pass ($T=2$) memory bounded estimation algorithms for noiseless parity ($\psnr = 1$) require a memory state of size $s \gtrsim \pdim^{3/2}$ or a sample size of $\ssize \geq 2^{\Omega{(\sqrt{\pdim})}}$.

\subsection{Proof of Theorem~\ref{thm:cca}}

As with the other main theorems of this paper, we prove
Theorem~\ref{thm:cca}
 by transferring a communication lower bound for distributed estimation protocols for $k$-CCA to memory bounded estimators for the same problem using the reduction in Fact~\ref{fact:reduction}.

In the (Bayesian) distributed setup for $k$-CCA, the cross-moment tensor $\bm V$ is drawn from the prior
\begin{align} \label{eq:prior-cca}
     \prior \explain{def}{=} \unif{\{ \sqrt{d^k} \cdot \bm e_{i_1}\otimes \bm e_{i_2} \dotsb \otimes \bm e_{i_k}: \; i_1, i_2, \dotsc, i_k \in [\dim]\}}.
\end{align}
Here, $\bm e_{i}$ denotes the $i$-th standard basis vector in $\R^\dim$, so $\bm V \sim \prior$ is a uniformly random $1$-sparse tensor. 
Then, $\bm x_{1:\ssize}$ are sampled i.i.d.\ from the distribution $\dmu{\bm V}$ specified in~\eqref{eq:kCCA-likelihood} and then
distributed across $\mach = \ssize/\batch \in \N$ machines with $\batch$ samples/machine; $\batch$ will be suitably chosen to yield Theorem~\ref{thm:cca}.
The execution of a distributed estimation protocol with parameters $(\mach, \batch, \budget)$ results in a transcript $\bm Y \in \{0,1\}^{\mach \budget}$ written on the blackboard.

The following
corollary is proved in exactly the same way as Corollary~\ref{coro: fano-atpca}.

\begin{corollary}[Fano's Inequality for $k$-CCA] \label{coro: fano-cca} For any estimator $\hat{\bm V}(\bm Y)$ for $k$-CCA computed by a distributed estimation protocol, and for any $t \in \R$, we have
\begin{align*}
\inf_{\bm V \in \paramspace} \P_{\bm V}\left( \frac{|\langle{\bm V,}{\hat{\bm V}} \rangle|^2}{\|\bm V\|^2\|\hat{\bm V}\|^2}  \geq \frac{t^2}{\dim^k} \right) & \leq \frac{1}{t^2} + \sqrt{2\MIhell{\bm V}{\bm Y}}.
\end{align*}
\end{corollary}

The main technical result is the following information bound for $k$-CCA.

\begin{proposition}\label{prop:info-bound-cca}
  Consider the $k$-CCA problem with $\dmu{\bm V}$ as defined in \eqref{eq:kCCA-likelihood}.
  Let $\bm Y \in \{0,1\}^{\mach \budget}$ be the transcript generated by a distributed estimation protocol for this $k$-CCA problem with parameters $(\mach, \batch, \budget)$.
  There is a finite constant $C_k$ depending only on $k$, such that if
\begin{equation*}
    \batch \geq C_{k} \cdot \budget \cdot \dim^{\frac{k}{2}} \quad \text{and} \quad \batch \lambda^2 \leq \frac{1}{C_{k}},
\end{equation*}
then
\begin{align*}
    \MIhell{\bm V}{\bm Y} & \leq C_k \cdot \left( \frac{\budget \cdot \mach \cdot \batch \cdot \lambda^2}{\dim^k} + \mach \cdot \batch^2 \cdot \lambda^4 \right).
\end{align*}
\end{proposition}

Proposition~\ref{prop:info-bound-cca} is proved in Appendix~\ref{appendix:cca}.
We can now complete the proof of Theorem~\ref{thm:cca}.

\begin{proof}[Proof of Theorem~\ref{thm:cca}]
  Appealing to the reduction in Fact~\ref{fact:reduction}, we note that any memory-bounded estimator $\hat{\bm V}$ with resource profile $(\ssize, \iters, \state)$ can be implemented using a distributed estimation protocol with parameters $(\ssize/\batch, \batch, \state \iters)$ for any $\batch \in \N$ such that $\mach := \ssize/\batch \in \N$.
As assumed in Theorem \ref{thm:cca}, we consider the situation when:
\begin{align} \label{eq:assumption-cca-thm}
     \eta + \tau + \mu  < k, \; \gamma > \frac{3k}{2}.
\end{align}
We set $\batch = \dim^\xi$ with
\begin{align} \label{eq:xi-choice-cca}
    \xi & \explain{def}= \tau + \mu + \frac{k}{2} + \frac{1}{2}\underbrace{\left(k -  (\eta + \tau + \mu)  \right)}_{>0} > \tau + \mu + \frac{k}{2}.
\end{align}
With this choice, we verify that the information bound in Proposition~\ref{prop:info-bound-cca} shows that $\MIhell{\bm V}{\bm Y} \rightarrow 0$.
This will yield the claim of the theorem.
We begin by observing
\begin{align}\label{eq:gamma-observe}
    \gamma > \frac{3k}{2} = \frac{k}{2} + \eta + \tau + \mu + (k- \eta - \tau - \mu) = \eta + \xi + \frac{(k- \eta - \tau - \mu)}{2} > \eta + \xi.
\end{align}
Next, we verify the conditions required for Proposition~\ref{prop:info-bound-cca}:
\begin{enumerate}
    \item Since $\eta > \tau + \mu + k/2$ (cf.~\eqref{eq:xi-choice-cca}) we have $\batch \gg \budget \cdot \dim^{k/2}$ as required. 
    \item Since $\gamma > \eta + \xi > \xi$ (cf.~\eqref{eq:gamma-observe}) we have $\batch \lambda^2 \ll 1$ as required. 
\end{enumerate}
Now, from the information bound of Proposition~\ref{prop:info-bound-cca},
\begin{align*}
    \MIhell{\bm V}{\bm Y} & \leq C_k \cdot \left( \frac{\budget \cdot \mach \cdot \batch \cdot \lambda^2}{\dim^k} + \mach \cdot \batch^2 \cdot \lambda^4 \right). \\
    & = C_k \cdot \left( \budget \cdot \ssize  \lambda^2 \cdot \dim^{-k} + (\batch \lambda^2) \cdot (\ssize \lambda^2) \right).
\end{align*}
We now check that this bound on $\MIhell{\bm V}{\bm Y}$ vanishes as $d\to\infty$:
\begin{enumerate}
    \item The assumption $\eta + \tau + \mu < k$ (cf. \eqref{eq:assumption-cca-thm}) guarantees $\budget \cdot \ssize  \lambda^2 \cdot \dim^{-k} \to 0$.
    \item Since $\gamma > \eta + \xi$, we have $(\ssize \lambda^2) \cdot (\batch \lambda^2) \to 0$. 
\end{enumerate}
This concludes the proof. 
\end{proof}

\begin{remark}[Connection with Correlation Detection]\label{remark:CCA-lightbulb} Observe that due to the choice of the prior in \eqref{eq:prior-cca}, the instance of $k$-CCA used to obtain the communication lower bound is an instance of the correlation detection problem. In this problem, the goal is to find a $k$-tuple of coordinates $(i_1, i_2, \dotsc, i_k) \subset [\dim]^k$ in the $k$ vectors $(\mview{\bm x}{1}, \mview{\bm x}{2}, \dotsc, \mview{\bm x}{k})$ such that $\mview{x}{1}_{i_1}, \mview{x}{2}_{i_2}, \dotsc, \mview{x}{k}_{i_k}$ are $k$-wise correlated using $\ssize$ i.i.d.\ realizations of $(\mview{\bm x}{1}, \mview{\bm x}{2}, \dotsc, \mview{\bm x}{k})$. Communication lower bounds for this problem in the blackboard model (cf. Definition \ref{def: distributed-algorithm}) were obtained in prior work by \citet{dagan2018detecting}. This result is sufficient to obtain Theorem \ref{thm:cca}. In this paper, we present another proof of this result using the information bound in Proposition~\ref{prop: main_hellinger_bound}, which is used to derive all communication lower bounds presented in this paper. 
\end{remark}

\appendix
\section{Proofs of the Information Bound and Geometric Inequalities} 

This appendix presents the proofs of our general information bound (Proposition~\ref{prop: main_hellinger_bound}) and the Geometric Inequalities (Proposition~\ref{prop: geometric inequality}).

\subsection{Proof of Proposition~\ref{prop: main_hellinger_bound}} \label{sec: proof-hellinger-bound}

In this section, we present the proof of Proposition~\ref{prop: main_hellinger_bound}.
This section is organized as follows:
\begin{enumerate}
    \item In Section~\ref{sec: proof-hellinger-bound-notation}, we introduce some additional notation used in the proof.
    \item In Section~\ref{sec: proof-hellinger-facts}, we collect some well-known properties of distributed estimation algorithms.
    \item In Section~\ref{sec: proof-hellinger-bound-actualproof}, we present the actual proof of Proposition~\ref{prop: main_hellinger_bound}. 
\end{enumerate}

\subsubsection{Additional Notation} \label{sec: proof-hellinger-bound-notation}
Recall that in the distributed learning setup, the data $\bm X_{1:m} \explain{i.i.d.}{\sim} \dmu{\bm V}$ (cf.~\eqref{eq:law-per-machine-data}).
We use $\P_{\bm V}$ and $\E_{\bm V}$ to denote probabilities and expectations, respectively, when the dataset of each machine is generated i.i.d.~from $\dmu{\bm V}$.
For instance, the marginal distribution of the transcript in this setup is given by
\begin{align} \label{eq: pmf_true}
    \P_{\bm V}(\bm Y = \bm y) & \explain{def}{=} \int \P(\bm Y = \bm y | \bm X_{1:m}) \; \dmu{\bm V}(\diff \bm X_{1}) \dmu{\bm V}(\diff \bm X_{2}) \dotsb \dmu{\bm V}(\diff \bm X_{\mach}).
\end{align}
Similarly, the expectation of any function $f$ of the data $\bm X_{1:m}$ and the transcript $\bm Y$ in this setup is
\begin{align} \label{eq: E_true}
    \E_{\bm V} f(\bm X_{1:m}, \bm Y) & \explain{def}{=} \int \sum_{\bm y \in \{0,1\}^{m\budget}} f(\bm X_{1:m}, \bm y) \;  \P(\bm Y = \bm y | \bm X_{1:m}) \; \dmu{\bm V}(\diff \bm X_{1}) \dmu{\bm V}(\diff \bm X_{2}) \dotsb \dmu{\bm V}(\diff \bm X_{\mach}).
\end{align}

For our analysis, it will be helpful to consider additional hypothetical setups in which the datasets for some (or all) of the machines are generated from a distribution other than $\dmu{\bm V}$ (such as the null measure $\nullmu$ or the reference measure $\refmu$ introduced in Proposition \ref{prop: main_hellinger_bound}).
We introduce the following three hypothetical setups:
\begin{description}
\item [Setup 1:] Here, the data samples $\bm X_{1:m} \explain{i.i.d.}{\sim} \nullmu$. We use $\barP$ and $\barE$ to denote the probabilities and expectations in this setup:
\begin{subequations}
\label{eq: null_EP}
\begin{align}
        \barP(\bm Y = \bm y) & \explain{def}{=} \int \P(\bm Y = \bm y | \bm X_{1:m}) \;  \nullmu^{\otimes m} (\diff \bm X_{1:m}), \\
        \barE f(\bm X_{1:m}, \bm Y) & \explain{def}{=} \int_{\dataspace^m} \sum_{\bm y \in \{0,1\}^{m\budget}} f(\bm X_{1:m}, \bm y) \;  \P(\bm Y = \bm y | \bm X_{1:m}) \; \nullmu^{\otimes m} (\diff \bm X_{1:m}).
\end{align}
\end{subequations}
We also use $\barE[g(\bm X_{1:m}) | \bm Y = \bm y]$ to denote conditional expectations in this setup.

\item[Setup 2:] Here, data samples $\bm X_{1:m} \explain{i.i.d.}{\sim} \refmu$. We use $\refP$ and $\refE$ to denote the probabilities and expectations in this setup:
\begin{subequations}
\label{eq: ref_EP}
\begin{align}
        \refP(\bm Y = \bm y) & \explain{def}{=} \int \P(\bm Y = \bm y | \bm X_{1:m}) \;  \refmu^{\otimes m} (\diff \bm X_{1:m}), \\
        \refE f(\bm X_{1:m}, \bm Y) & \explain{def}{=} \int_{\dataspace^m} \sum_{\bm y \in \{0,1\}^{m\budget}} f(\bm X_{1:m}, \bm y) \;  \P(\bm Y = \bm y | \bm X_{1:m}) \; \refmu^{\otimes m} (\diff \bm X_{1:m}).
\end{align}
\end{subequations}
We also use $\refE[g(\bm X_{1:m}) | \bm Y = \bm y]$ to denote conditional expectations in this setup.
\item[Setup 3:] Here, a fixed machine $i \in [m]$ is exceptional, and the data $\bm X_{1:m}$ are sampled independently as follows:
\begin{align*}
    (\bm X_j)_{j \neq i} \explain{i.i.d.}{\sim} \nullmu, \quad \bm X_i \sim \dmu{\bm V}.
\end{align*}
We use $\looP{\bm V}{i}$ and $\looE{\bm V}{i}$ to denote the probabilities and expectations in this setup:
\begin{subequations}
\label{eq: alt_loo_EP}
\begin{align}
        \looP{\bm V}{i}(\bm Y = \bm y) & \explain{def}{=} \int \P(\bm Y = \bm y | \bm X_{1:m}) \; \dmu{\bm V}(\diff \bm X_i) \cdot  \prod_{j\neq i }  \nullmu(\diff \bm X_{j}), \\
        \looE{\bm V}{i} f(\bm X_{1:m}, \bm Y) & \explain{def}{=} \int_{\dataspace^m} \sum_{\bm y \in \{0,1\}^{m\budget}} f(\bm X_{1:m}, \bm y) \;  \P(\bm Y = \bm y | \bm X_{1:m}) \; \dmu{\bm V}(\diff \bm X_i) \cdot  \prod_{j\neq i }  \nullmu(\diff \bm X_{j}).
\end{align}
\end{subequations}
We also use $\looE{\bm V}{i}[g(\bm X_{1:m}) | \bm Y = \bm y]$ to denote conditional expectations in this setup.
\item [Setup 4:] Here, a fixed machine $i \in [m]$ is exceptional, and the data $\bm X_{1:m}$ are sampled independently as follows:
\begin{align*}
    (\bm X_j)_{j \neq i} \explain{i.i.d.}{\sim} \nullmu, \quad \bm X_i \sim \refmu.
\end{align*}
We use $\looP{0}{i}$ and $\looE{0}{i}$ to denote the probabilities and expectations in this setup:
\begin{subequations}
\label{eq: ref_loo_EP}
\begin{align}
        \looP{0}{i}(\bm Y = \bm y) & \explain{def}{=} \int \P(\bm Y = \bm y | \bm X_{1:m}) \; \refmu(\diff \bm X_i) \cdot  \prod_{j\neq i }  \nullmu(\diff \bm X_{j}), \\
        \looE{0}{i} f(\bm X_{1:m}, \bm Y) & \explain{def}{=} \int_{\dataspace^m} \sum_{\bm y \in \{0,1\}^{m\budget}} f(\bm X_{1:m}, \bm y) \;  \P(\bm Y = \bm y | \bm X_{1:m}) \; \refmu(\diff \bm X_i) \cdot  \prod_{j\neq i }  \nullmu(\diff \bm X_{j}).
\end{align}
\end{subequations}
We also use $\looE{0}{i}[g(\bm X_{1:m}) | \bm Y = \bm y]$ to denote conditional expectations in this setup.
\end{description}
(Note that Setup 4 is the hypothetical setup defined in Proposition~\ref{prop: main_hellinger_bound}.)

\subsubsection{Properties of Distributed Algorithms} \label{sec: proof-hellinger-facts}

We recall two well-known properties of distributed estimation protocols in the blackboard model of communication (Definition~\ref{def: distributed-algorithm}), taken from \citet{bar2004information} and \citet{jayram2009hellinger}.

\begin{fact}[\citet{bar2004information}] \label{fact: bb_conditional_independence} Suppose datasets $\bm X_{1:\mach}$ are distributed across $\mach$ machines. Let $\bm Y \in \{0,1\}^{\mach \budget}$ be the transcript produced by a distributed estimation protocol.
\begin{enumerate} \item The likelihood of the transcript given the data factorizes as follows:
\begin{align*}
    \P(\bm Y = \bm y | \bm X_{1:\mach}) & = \prod_{i=1}^\mach \factor_i( \bm y | \bm X_i),
\end{align*}
where each $\factor_i(\bm y | \bm X_i)$ takes values in $[0,1]$. 
\item Suppose that the datasets $\bm X_{1:\mach}$ are drawn from a product measure,
\begin{align*}
    \bm X_{1:\mach} \sim \bigotimes_{i=1}^\mach \nu_i,
\end{align*}
then the conditional distribution of $\bm X_{1:\mach}$ given $\bm Y = \bm y$ is also a product measure:
\begin{align*}
    \bm X_{1:\mach} | \bm Y = \bm y & \sim \bigotimes_{i=1}^\mach \nu_i^{\bm y},
\end{align*}
where, for each $i \in [\mach]$,
\begin{align*}
    \nu_i^{\bm y}(\diff \bm X) & = \frac{\factor_i(\bm y | \bm X) \nu_i(\diff \bm X)}{\int_\dataspace \factor_i(\bm y | \bm X) \nu_i(\diff \bm X)}. 
\end{align*}
\end{enumerate}
\end{fact}

We also use the following bound on the Hellinger distance, which is a consequence of the ``cut-and-paste'' property of distributed estimation protocols~\citep{bar2004information,jayram2009hellinger}.
This result has been used in several prior works that prove lower bounds for such protocols~\citep[e.g.,][]{braverman2016communication,acharya2020unified}. 

\begin{fact}[\citet{jayram2009hellinger}] \label{hellinger_BB_fact} Recall the definitions of $\P_{\bm V}$ from \eqref{eq: pmf_true}, $\barP$ from \eqref{eq: null_EP} and $\looP{\bm V}{i}$ from \eqref{eq: alt_loo_EP}. There exists a universal constant $\consthell$ such that
\begin{align*}
    \hell{\P_{\bm V}}{\barP} & \leq \consthell \cdot \sum_{i=1}^{\mach} \hell{\looP{\bm V}{i}}{\barP}.
\end{align*}
\end{fact}

We are now ready to present the proof of Proposition \ref{prop: main_hellinger_bound}.

\subsubsection{Proof of Proposition \ref{prop: main_hellinger_bound}}
\label{sec: proof-hellinger-bound-actualproof}
\begin{proof}[Proof of Proposition \ref{prop: main_hellinger_bound}]
  We use $\Q = \barP$ in \eqref{eq: hellinger_definition} to obtain the bound
\begin{align*}
    \MIhell{\bm V}{\bm Y} & \explain{}{\leq}   \int \hell{\P_{\bm V}}{\barP} \prior(\diff \bm V).
\end{align*}
By Fact~\ref{hellinger_BB_fact}, we have
\begin{align*}
    \hell{\P_{\bm V}}{\barP} & \leq \consthell \cdot \int  \sum_{i=1}^{\mach}  \hell{\barP_{\bm V}^{(i)}}{\barP} \pi(\diff \bm V).
\end{align*}
Recall that
\begin{align*}
    \hell{\barP_{\bm V}^{(i)}}{\barP} & = \frac{1}{2} \sum_{\bm y \in \{0,1\}^{\mach\budget}} \left(\sqrt{\barP_{\bm V}^{(i)}(\bm Y = \bm y)}-\sqrt{\barP(\bm Y = \bm y)} \right)^2 \\
    & = \frac{1}{2} \sum_{\bm y \in \{0,1\}^{\mach\budget}} \looP{0}{i}(\bm Y = \bm y) \cdot  \left(\sqrt{\frac{\looP{\bm V}{i}(\bm Y = \bm y)}{\looP{0}{i}(\bm Y = \bm y)}}-\sqrt{\frac{\barP(\bm Y = \bm y)}{\looP{0}{i}(\bm Y = \bm y)}}\right)^2.
\end{align*}
Next we observe that, by Fact \ref{fact: bb_conditional_independence},
\begin{align*}
    \barP(\bm Y = \bm y) & = \prod_{j=1}^{\mach} \barE_{} F_j(\bm y | \bm X_j), \\
    \looP{0}{i}(\bm Y = \bm y) & = \refE F_i(\bm y | \bm X_i) \cdot \prod_{\substack{j=1 , \\ j \neq i}}^{\mach} \barE F_j(\bm y | \bm X_j), \\
    \looP{\bm V}{i}(\bm Y = \bm y) & = \E_{\bm V} F_i(\bm y | \bm X_i) \cdot \prod_{\substack{j=1 , \\ j \neq i}}^{\mach} \barE F_j(\bm y | \bm X_j).
\end{align*}
Hence, 
\begin{align*}
    \frac{\looP{\bm V}{i}(\bm Y = \bm y)}{\looP{0}{i}(\bm Y = \bm y)}  &= \frac{\E_{\bm V} F_i(\bm y | \bm X_i)}{\refE F_i(\bm y | \bm X_i)} = \frac{1}{\refE F_i(\bm y | \bm X_i)} \refE \left[ \frac{\diff \dmu{\bm V}}{\diff \refmu} (\bm X_i)    F_i(\bm y | \bm X_i)\right] \explain{(a)}{=} \looE{0}{i}\left[  \frac{\diff \dmu{\bm V}}{\diff \refmu} (\bm X_i) \bigg| \bm Y = \bm y \right], \\
    \frac{\barP(\bm Y = \bm y)}{\looP{0}{i}(\bm Y = \bm y)}  &= \frac{\barE F_i(\bm y | \bm X_i)}{\refE F_i(\bm y | \bm X_i)} = \frac{1}{\refE F_i(\bm y | \bm X_i)} \refE \left[ \frac{\diff \nullmu}{\diff \refmu} (\bm X_i)    F_i(\bm y | \bm X_i)\right] \explain{(a)}{=} \looE{0}{i}\left[  \frac{\diff \nullmu}{\diff \refmu} (\bm X_i) \bigg| \bm Y = \bm y \right].
\end{align*}
In the step marked (a) above, we used the characterization on the conditional distribution of $\bm X_i$ given $\bm Y = \bm y$.
Hence, we have obtained
\begin{align*}
    &\MIhell{\bm V}{\bm Y}  \\&\leq \frac{\consthell}{2}\sum_{i=1}^{\mach} \sum_{\bm y \in \{0,1\}^{\mach\budget}}  \looP{0}{i}(\bm Y = \bm y) \int \left(\sqrt{\looE{0}{i}\left[  \frac{\diff \dmu{\bm V}}{\diff \refmu} (\bm X_i) \big| \bm Y = \bm y \right]}-\sqrt{\looE{0}{i}\left[  \frac{\diff \nullmu}{\diff \refmu} (\bm X_i) \big| \bm Y = \bm y \right]}\right)^2 \prior(\diff \bm V).
\end{align*} %
We can write
\begin{align*}
    \looE{0}{i}\left[  \frac{\diff \dmu{\bm V}}{\diff \refmu} (\bm X_i) \bigg| \bm Y = \bm y \right] & =  \looE{0}{i}\left[  \frac{\diff \dmu{\bm V}}{\diff \refmu} (\bm X_i) \Indicator{\bm X_i \in \goodevnt} \bigg| \bm Y = \bm y \right]  + \looE{0}{i}\left[  \frac{\diff \dmu{\bm V}}{\diff \refmu} (\bm X_i) \Indicator{\bm X_i \notin \goodevnt} \bigg| \bm Y = \bm y \right],
\end{align*}
and analogously for the term involving the likelihood ratio $\diff \nullmu/\diff\refmu$.
For any $a_1,a_2, \epsilon_1,\epsilon_2 \geq 0$, we have the scalar inequality
\begin{align*}
    (\sqrt{a_1 + \epsilon_1} - \sqrt{a_2 + \epsilon_2})^2 &= \epsilon_1 + \epsilon_2 + (\sqrt{a_1} - \sqrt{a_2})^2 + 2 \sqrt{a_1 a_2} - 2 \sqrt{(a_1+\epsilon_1)(a_2 + \epsilon_2)} \\
    & \leq \epsilon_1 + \epsilon_2 + (\sqrt{a_1} - \sqrt{a_2})^2 .
\end{align*}
This gives us
\begin{align*}
    \MIhell{\bm V}{\bm Y} & \leq \frac{\consthell}{2} \cdot \left( \mathsf{I} + \mathsf{II} \right) ,
\end{align*}
where
\begin{align*}
\mathsf{I}&\explain{def}{=}\int \sum_{i=1}^m \sum_{\bm y}  \looP{0}{i}(\bm Y = \bm y) \cdot \\
          & \qquad \left( \looE{0}{i}\left[  \frac{\diff \dmu{\bm V}}{\diff \refmu} (\bm X_i) \Indicator{\bm X_i \notin \goodevnt} \bigg| \bm Y = \bm y \right] +   \looE{0}{i}\left[  \frac{\diff \nullmu}{\diff \refmu} (\bm X_i) \Indicator{\bm X_i \notin \goodevnt} \bigg| \bm Y = \bm y \right] \right) \pi(\diff \bm V)
\end{align*}
and
\begin{align*}
    \mathsf{II}&\explain{def}{=}\sum_{i=1}^m \sum_{\bm y}  \looP{0}{i}(\bm Y = \bm y) \cdot \\
               & \qquad \int \left(\sqrt{\looE{0}{i}\left[  \frac{\diff \dmu{\bm V}}{\diff \refmu} (\bm X_i) \Indicator{\bm X_i \in \goodevnt} \bigg| \bm Y = \bm y \right]}-\sqrt{\looE{0}{i}\left[  \frac{\diff \nullmu}{\diff \refmu} (\bm X_i) \Indicator{\bm X_i \in \goodevnt} \bigg| \bm Y = \bm y \right]}\right)^2 \prior(\diff \bm V).
\end{align*} %
We simplify $\mathsf{I}$ and $\mathsf{II}$ separately below.

\paragraph{Analysis of $\mathsf{I}$.} By the tower property of conditional expectations,
\begin{align*}
    \mathsf{I}  &=  \int \sum_{i=1}^{\mach}  \left( \looE{0}{i}\left[  \frac{\diff \dmu{\bm V}}{\diff \refmu} (\bm X_i) \Indicator{\bm X_i \notin \goodevnt}  \right] + \looE{0}{i}\left[  \frac{\diff \nullmu}{\diff \refmu} (\bm X_i) \Indicator{\bm X_i \notin \goodevnt} \right] \right) \pi(\diff \bm V) \\
    & = \int \sum_{i=1}^{\mach}  \left( \refE\left[  \frac{\diff \dmu{\bm V}}{\diff \refmu} (\bm X_i) \Indicator{\bm X_i \notin \goodevnt}  \right] + \refE\left[  \frac{\diff \nullmu}{\diff \refmu} (\bm X_i) \Indicator{\bm X_i \notin \goodevnt} \right] \right) \pi(\diff \bm V) \\
    & = \mach \cdot \left( \int \dmu{\bm V}(\goodevnt^c) \pi(\diff \bm V) +  \nullmu(\goodevnt^c) \right). 
\end{align*}

\paragraph{Analysis of $\mathsf{II}$.} Note that
\begin{align*}
    \looE{0}{i}\left[  \frac{\diff \dmu{\bm V}}{\diff \refmu} (\bm X_i) \Indicator{\bm X_i \in \goodevnt} \bigg| \bm Y = \bm y \right] & = \looE{0}{i}\left[  \frac{\diff \dmu{\bm V}}{\diff \refmu} (\bm X_i) \bigg| \bm Y = \bm y, \bm X_i \in \goodevnt \right] \cdot \looP{0}{i}(\bm X_i \in \goodevnt | \bm Y = \bm y).
\end{align*}
Analogously,
\begin{align*}
    \looE{0}{i}\left[  \frac{\diff \nullmu}{\diff \refmu} (\bm X_i) \Indicator{\bm X_i \in \goodevnt} \bigg| \bm Y = \bm y \right] & = \looE{0}{i}\left[  \frac{\diff \nullmu}{\diff \refmu} (\bm X_i) \bigg| \bm Y = \bm y, \bm X_i \in \goodevnt \right] \cdot \looP{0}{i}(\bm X_i \in \goodevnt | \bm Y = \bm y).
\end{align*}
And hence,
\begin{align*}
  \lefteqn{
    \left(\sqrt{\looE{0}{i}\left[  \frac{\diff \dmu{\bm V}}{\diff \refmu} (\bm X_i) \Indicator{\bm X_i \in \goodevnt} \bigg| \bm Y = \bm y \right]}-\sqrt{\looE{0}{i}\left[  \frac{\diff \nullmu}{\diff \refmu} (\bm X_i) \Indicator{\bm X_i \in \goodevnt} \bigg| \bm Y = \bm y \right]}\right)^2
  } \\
  & = \looP{0}{i}(\bm X_i \in \goodevnt | \bm Y = \bm y)
    \left(\sqrt{\looE{0}{i}\left[  \frac{\diff \dmu{\bm V}}{\diff \refmu} (\bm X_i)  \bigg| \bm Y = \bm y, \bm X_i \in \goodevnt \right]}-\sqrt{\looE{0}{i}\left[  \frac{\diff \nullmu}{\diff \refmu} (\bm X_i) \bigg| \bm Y = \bm y, \bm X_i \in \goodevnt \right]}\right)^2.
\end{align*}
Note that by definition of $\goodevnt$,
\begin{align*}
    \looE{0}{i}\left[  \frac{\diff \nullmu}{\diff \refmu} (\bm X_i) \bigg| \bm Y = \bm y, \bm X_i \in \goodevnt \right] \geq \frac{1}{2}.
\end{align*}
Note the scalar inequality for any $a_1 \geq 0, a_2 \geq 1/2$,
\begin{align*}
    (\sqrt{a_1} - \sqrt{a_2})^2 = \frac{(a_1-a_2)^2}{(\sqrt{a_1} + \sqrt{a_2})^2} \leq \frac{(a_1-a_2)^2}{a_2} \leq 2 (a_1 - a_2)^2.  
\end{align*}
This gives us
\begin{align*}
    \tfrac{1}{2}(\mathsf{II}) & \leq \sum_{i=1}^{\mach} \sum_{\bm y \in \{0,1\}^{\mach\budget}}  \looP{0}{i}(\bm Y = \bm y,\bm X_i \in \goodevnt) \int \left({\looE{0}{i}\left[  \frac{\diff \dmu{\bm V}}{\diff \refmu} (\bm X_i) - \frac{\diff \nullmu}{\diff \refmu} (\bm X_i)  \bigg| \bm Y = \bm y, {\bm X_i \in \goodevnt} \right]}
    \right)^2 \prior(\diff \bm V).
\end{align*}
Recall that $Z_i = \Indicator{\bm X_i \in \goodevnt}$, so
\begin{align*}
\tfrac{1}{2}(\mathsf{II}) 
    & \leq \sum_{i=1}^{\mach} \sum_{\bm y \in \{0,1\}^{\mach\budget}}  \looP{0}{i}(\bm Y = \bm y,Z_i =1) \int \left({\looE{0}{i}\left[  \frac{\diff \dmu{\bm V}}{\diff \refmu} (\bm X_i) - \frac{\diff \nullmu}{\diff \refmu} (\bm X_i)  \bigg| \bm Y = \bm y, {Z_i = 1} \right]}
    \right)^2 \prior(\diff \bm V) \\ %
    & \leq \sum_{i=1}^{\mach} \sum_{(\bm y,z) \in \{0,1\}^{\mach\budget+1}} \kern-10pt Z_i \cdot \looP{0}{i}(\bm Y = \bm y,Z_i =z)  \int \left({\looE{0}{i}\left[  \frac{\diff \dmu{\bm V}}{\diff \refmu} (\bm X_i) - \frac{\diff \nullmu}{\diff \refmu} (\bm X_i)  \bigg| \bm Y = \bm y, {Z_i = z} \right]}\right)^2 \prior(\diff \bm V) \\
    & = \sum_{i=1}^{\mach}   \looE{0}{i}\left[ Z_i \cdot \int \left({\looE{0}{i}\left[  \frac{\diff \dmu{\bm V}}{\diff \refmu} (\bm X_i) - \frac{\diff \nullmu}{\diff \refmu} (\bm X_i)  \bigg| \bm Y, Z_i  \right]}\right)^2 \prior(\diff \bm V) \right].
\end{align*} %
Note that, due to the conditional independence property given in Fact~\ref{fact: bb_conditional_independence} (item 2), we have
\begin{align*}
    \bm X_i | \bm Y \explain{d}{=} \bm X_i | (\bm Y, (\bm X_j)_{j \neq i}) ,
\end{align*}
where $\explain{d}{=}$ denotes equality of distributions.
Since $Z_i$ is a function of $\bm X_i$, we have
\begin{align*}
    (\bm X_i, Z_i) | \bm Y \explain{d}{=} (\bm X_i, Z_i) | \bm Y, (\bm X_j)_{j \neq i} \implies \bm X_i | Z_i, \bm Y \explain{d}{=} \bm X_i | Z_i,\bm Y, (\bm X_j)_{j \neq i}.
\end{align*}
Hence
\begin{align*}
     \mathsf{II} & \leq {2}\sum_{i=1}^m   \looE{0}{i}\left[ Z_i \cdot \int \left({\looE{0}{i}\left[  \frac{\diff \dmu{\bm V}}{\diff \refmu} (\bm X_i) - \frac{\diff \nullmu}{\diff \refmu} (\bm X_i)  \bigg| \bm Y, Z_i, (\bm X_j)_{j \neq i}  \right]}\right)^2 \prior(\diff \bm V) \right].
     \qedhere
\end{align*}
\end{proof}

\subsection{Proof of Proposition \ref{prop: geometric inequality}}
\label{sec:geometric-inequality-proof}
In this section, we present the proof of Proposition \ref{prop: geometric inequality}.
\begin{proof}[Proof of Proposition \ref{prop: geometric inequality}]
The proof follows the argument from \citet{han2018geometric}. We prove each item separately.
Fix any $z \in \{0,1\}$, and define $\goodevnt^{(1)} = \goodevnt$ and $\goodevnt^{(0)} = \goodevnt^c$.
\begin{enumerate}
\item Consider the following sequence of inequalities:
\begin{align*}
    \left|  {\looE{0}{i}\left[ f(\bm X_i)  \bigg| \bm Y = \bm y, Z_i = z, (\bm X_j)_{j \neq i}   \right]} \right|^{q} &\explain{(a)}{\leq}   {\looE{0}{i}\left[ |f(\bm X_i)|^{q}  \bigg|\bm Y = \bm y, Z_i = z, (\bm X_j)_{j \neq i}  \right]}  \\ & \explain{}{=}   \frac{ \int_{\goodevnt^{(z)}} |f(\bm X_i)|^{q} \cdot  \P(\bm Y = \bm y | \bm X_{1:\mach})  \; \refmu( \diff \bm X_i) }{\looP{0}{i}(\bm Y = \bm y, Z_i = z | (\bm X_j)_{j \neq i})}  \\
    & \explain{(b)}{\leq}  \frac{ \int_{\dataspace} |f(\bm X_i)|^{q}   \refmu( \diff \bm X_i) }{\looP{0}{i}(\bm Y = \bm y, Z_i = z | (\bm X_j)_{j \neq i})} \\
    & = \frac{\refE[ |f(\bm X)|^{q}]}{\looP{0}{i}(\bm Y = \bm y, Z_i = z | (\bm X_j)_{j \neq i})}.
\end{align*}
In the step marked (a) above, we used Jensen's Inequality; in the step marked (b), we used $\P(\bm Y = \bm y | (\bm X_j)_{j \in [m]}) \leq 1$.
Hence,
\begin{align*}
     \left|  {\looE{0}{i}\left[ f(\bm X_i)  \bigg| \bm Y = \bm y, Z_i = z, (\bm X_j)_{j \neq i}   \right]} \right| & \leq \left(\frac{\refE |f(\bm X)|^{q}}{\looP{0}{i}(\bm Y = \bm y, Z_i = z | (\bm X_j)_{j \neq i} ) } \right)^{\frac{1}{q}},
\end{align*}
as claimed.
\item For any $\eta \in \R$,
\begin{align*}
    &\eta \cdot {\looE{0}{i}\left[ f(\bm X_i)  \bigg| \bm Y = \bm y, Z_i = z, (\bm X_j)_{j \neq i}   \right]}   \explain{(c)}{\leq} \ln \E \left[ e^{\eta f(\bm X_i)} \bigg| \bm Y = \bm y, Z_i = z, (\bm X_j)_{j \neq i}   \right] \\
    &\hspace{6cm} \explain{(d)}{\leq}  \ln \left( \frac{\refE[ e^{\eta f(\bm X_i)}]}{\looP{0}{i}(\bm Y = \bm y, Z_i = z | (\bm X_j)_{j \neq i})} \right) \\
    &\hspace{6cm}=\ln \refE[ e^{\eta f(\bm X)}]  +  \ln  \frac{1}{\looP{0}{i}(\bm Y = \bm y, Z_i = z | (\bm X_j)_{j \neq i})}.
\end{align*}
The step marked (c) above uses Jensen's inequality, and the step marked (d) relies on the fact that $\P(\bm Y = \bm y | (\bm X_j)_{j \in [m]}) \leq 1$.
Hence, for any $\eta\in \R$,
\begin{align} \label{eq: geometric-mgf}
    \eta \cdot {\looE{0}{i}\left[ f(\bm X_i)  \bigg| \bm Y = \bm y, Z_i = z, (\bm X_j)_{j \neq i})   \right]} \leq \ln \refE[ e^{\eta f(\bm X)}]  +  \ln  \frac{1}{\looP{0}{i}(\bm Y = \bm y, Z_i = z | (\bm X_j)_{j \neq i})}.
\end{align}
Now, fix $\xi > 0$, and set $\eta$ as follows:
\begin{align*}
    \eta = \xi \cdot \mathsf{sign} \left( {\looE{0}{i}\left[ f(\bm X_i)  \bigg| \bm Y = \bm y, Z_i = z, (\bm X_j)_{j \neq i}\right]} \right),
\end{align*}
so \eqref{eq: geometric-mgf} with this choice of $\eta$ yields
\begin{align*}
  \lefteqn{
   \left|{\looE{0}{i}\left[ f(\bm X_i)  \bigg| \bm Y = \bm y, Z_i = z, (\bm X_j)_{j \neq i}  \right]}\right|
  } \\
   & \leq \frac{\ln \refE[ e^{\eta f(\bm X)}]}{\xi} + \frac{1}{\xi}  \ln  \frac{1}{\looP{0}{i}(\bm Y = \bm y, Z_i = z | (\bm X_j)_{j \neq i})} \\
    & \leq \frac{\ln (\refE[ e^{\xi f(\bm X)}] \vee \refE[ e^{-\xi f(\bm X)}])}{\xi} + \frac{1}{\xi}  \ln  \frac{1}{\looP{0}{i}(\bm Y = \bm y, Z_i = z | (\bm X_j)_{j \neq i})}.
    \qedhere
\end{align*}
\end{enumerate}
\end{proof}

\section{Proofs for Tensor PCA} \label{appendix:tpca}
\subsection{Setup}
This appendix is devoted to the proof Proposition~\ref{prop: stpca-info-bound}, the information bound for the distributed $k$-TPCA problem. Recall that in the distributed $k$-TPCA problem:
\begin{enumerate}
    \item An unknown parameter $\bm V \sim \prior$ is drawn from the prior $\prior = \unif{\{\pm 1\}^\dim}$.
    \item A dataset consisting of $\mach$ tensors $\bm X_{1:\mach}$ is drawn i.i.d. from $\dmu{\bm V}$, where $\dmu{\bm V}$ is the distribution of a single tensor from the $k$-TPCA problem (recall \eqref{eq: symmetric_tpca_setup}). This dataset is divided among $\mach$ machines with 1 tensor per machine.
    \item The execution of a distributed estimation protocol with parameters $(\mach, \batch = 1, \budget)$ results in a transcript $\bm Y \in \{0,1\}^{\mach \budget}$ written on the blackboard.
\end{enumerate}
The information bound stated in Proposition~\ref{prop: stpca-info-bound} is obtained using the general information bound given in Proposition~\ref{prop: main_hellinger_bound} with the following choices:
\begin{description}
\item [Choice of $\refmu$: ] Under the reference measure, $\bm X \sim \refmu$ is a $k$-tensor with i.i.d.\ $\gauss{0}{1}$ coordinates.
\item [Choice of $\nullmu$: ] Under the measure $\nullmu$, the sample in each machine is sampled i.i.d.\ from:
\begin{align*}
        \nullmu(\cdot) \explain{def}{=} \int \dmu{\bm V}(\cdot) \;  \prior(\diff \bm V).
\end{align*}
\item [Choice of $\goodevnt$: ] We choose the event $\goodevnt$ as follows:
\begin{align*}
        \goodevnt & \explain{def}= \left\{ \bm X \in \tensor{\R^\dim}{k}: \left|\frac{\diff \nullmu}{\diff \refmu} (\bm X) - 1 \right| \leq \frac{1}{2}  \right\}.
    \end{align*}
\end{description}
This appendix is organized into subsections as follows.
\begin{enumerate}
    \item To prove Proposition~\ref{prop: stpca-info-bound}, we rely on certain analytic properties of the likelihood ratio for the $k$-TPCA problem. These properties are stated (without proofs) in Appendix~\ref{sec: tpca-harmonic}. 
    \item Using these properties, Proposition~\ref{prop: stpca-info-bound} is proved in Appendix~\ref{appendix:stpca-info-bound-proof}.
    \item Finally, the proofs of the analytic properties of the likelihood ratio are given in Appendix~\ref{appendix:stpca-harmonic-proofs}. 
\end{enumerate}

\subsection{The Likelihood Ratio for Symmetric Tensor PCA} \label{sec: tpca-harmonic}
In this section, we collect some important properties of the likelihood ratio for the Tensor PCA problem without proofs. The proofs of these properties are provided in Appendix~\ref{appendix:stpca-harmonic-proofs}. This section requires familiarity with Hermite polynomials and their some of their properties, which are reviewed in Appendix~\ref{fourier_gauss_appendix}.  

In order to prove our desired information bound (Proposition~\ref{prop: stpca-info-bound}) we will find it useful to decompose the likelihood ratio for Tensor PCA in the orthogonal basis given by the Hermite polynomials. This decomposition is given in the lemma stated below.

\begin{lemma}[Hermite Decomposition for Tensor PCA] \label{lemma: hermite-decomp-tpca}For any $\bm X \in \tensor{\R^\dim}{k}$, we have
\begin{align*}
      \frac{\diff \dmu{\bm V}}{\diff \refmu}(\bm X) = \sum_{i=0}^\infty \frac{\lambda^i}{\sqrt{i!}} \cdot H_{i}\left(\frac{\ip{\bm X}{\bm V^{\otimes k}}}{\sqrt{\dim^k}}\right).
\end{align*}
\end{lemma}
\begin{proof}
See Appendix~\ref{sec:tpca-hermite-proof}. 
\end{proof}
Next, we introduce the following family of functions derived from the Hermite polynomials.
\begin{definition}[Integrated Hermite Polynomials] \label{def: integrated-Hermite-tpca} Let $S: \{\pm 1\}^\dim \rightarrow \R$ be a function with $\|S\|_\pi = 1$. For any $i \in \W$, the \emph{integrated Hermite polynomials} are defined as
\begin{align*}
    \intH{i}{\bm X}{S} \explain{def}{=} \int H_{i}\left(\frac{\ip{\bm X}{\bm V^{\otimes k}}}{\sqrt{\dim^k}}\right) \cdot S(\bm V) \; \prior(\diff \bm V).
\end{align*}
\end{definition}
Our rationale for introducing this definition is that proving the communication lower bounds using Proposition~\ref{prop: main_hellinger_bound} requires understanding the following quantities derived from the likelihood ratio:
\begin{align*}
    \frac{\diff \nullmu}{\diff\refmu}(\bm X) &\explain{def}{=} \int \frac{\diff \dmu{\bm V}}{\diff \refmu}(\bm X)  \; \pi(\diff \bm V), \\
    \ip{\frac{\diff \dmu{\bm V}}{\diff\refmu}(\bm X)}{S}_\prior &\explain{def}{=} \int \frac{\diff \dmu{\bm V}}{\diff \refmu}(\bm X) \cdot S(\bm V)  \; \pi(\diff \bm V).
\end{align*}
Using Lemma~\ref{lemma: hermite-decomp-tpca}, these quantities are naturally expressed in terms of the integrated Hermite polynomials:
\begin{align*}
    \frac{\diff \nullmu}{\diff\refmu}(\bm x) &\explain{}{=} \sum_{i=0}^\infty \frac{\lambda^i}{\sqrt{i}}  \cdot \intH{i}{\bm X}{1}, \\
    \ip{\frac{\diff \dmu{\bm V}}{\diff\refmu}(\bm X)}{S}_\prior &\explain{}{=}\sum_{i=0}^\infty \frac{\lambda^i}{\sqrt{i}}  \cdot \intH{i}{\bm X}{S}.
\end{align*}
The following lemma shows that the integrated Hermite polynomials inherit the orthogonality property of the standard Hermite polynomials.
\begin{lemma} \label{lemma: integrate-hermite-orthogonality-tpca} For any $i,j \in \W$ such that $i \neq j$, we have
\begin{align*}
    \refE[\intH{i}{\bm X}{S} \cdot \intH{j}{\bm X}{S}] = 0,
\end{align*}
where $\bm X \sim \refmu$.
\end{lemma}
\begin{proof}
See Appendix~\ref{sec:tpca-integrated-orthogonality-proof}. 
\end{proof}
Though the integrated Hermite polynomials are orthogonal, they do not have unit norm.  In general, the norm of these polynomials depends on the choice of the function $S$ in Definition~\ref{def: integrated-Hermite-tpca}. The following lemma provides bounds on the norm of the integrated Hermite polynomials.
\begin{lemma} \label{lemma: norm-integrated-hermite-tpca} There is a universal constant $C$ (independent of $\dim$) such that, for any $i \in \W$, we have the following.
\begin{enumerate}
    \item For any $S: \{\pm 1 \}^\dim \rightarrow \R$ with $\|S\|_\pi \leq 1$, we have $\refE[\intH{i}{\bm X}{S}^2] \leq (Cki)^{\frac{ki}{2}} \cdot  d^{-\lceil\frac{ki}{2}\rceil}$. 
    \item For any $S: \{\pm 1 \}^\dim \rightarrow \R$ with $\|S\|_\pi \leq 1, \; \ip{S}{1}_\prior = 0$, we have $\refE[ \intH{i}{\bm X}{S}^2] \leq (Cki)^{\frac{ki}{2}} \cdot  d^{-\lceil\frac{ki+1}{2}\rceil}$,
\end{enumerate}
where $\bm X \sim \refmu$.
\end{lemma}
\begin{proof}
See Appendix~\ref{sec:tpca-integrated-norm-proof}. 
\end{proof}
As a consequence of the orthogonality property of integrated Hermite polynomials (Lemma~\ref{lemma: integrate-hermite-orthogonality-tpca}) and the estimates obtained in Lemma~\ref{lemma: norm-integrated-hermite-tpca}, one can easily estimate the second moment of functions constructed by linear combinations of the integrated Hermite polynomials:
\begin{align*}
    \left\| \sum_{i=0}^\infty \alpha_{i} \cdot \intH{i}{\bm X}{S} \right\|_2^2 \explain{def}{=} \refE \left( \sum_{i=0}^\infty \alpha_{i} \cdot \intH{i}{\bm X}{S}\right)^2 = \sum_{i=0}^\infty \alpha_{i}^2 \cdot   \refE[\intH{i}{\bm X}{S}^2].
\end{align*}
In our analysis, we will also find it useful to estimate the $q$-norms of linear combinations of integrated Hermite polynomials for $q \geq 2$:
\begin{align*}
    \left\| \sum_{i=0}^\infty \alpha_{i} \cdot \intH{i}{\bm X}{S} \right\|_q^q \explain{def}{=} \refE \left| \sum_{i=0}^\infty \alpha_{i} \cdot \intH{i}{\bm X}{S} \right|^q.
\end{align*}
The following lemma uses Gaussian Hypercontractivity (Fact~\ref{fact: hypercontractivity}) to provide an estimate for the above quantity.
\begin{lemma}\label{lemma: integrated-hermite-hypercontractivity-tpca} Let $\{\alpha_{i} : i \in \W\}$ be an arbitrary collection of real-valued coefficients. For any $q \geq 2$, we have
\begin{align*}
     \left\|  \sum_{i=0}^\infty \alpha_{i} \cdot \intH{i}{\bm X}{S}\right\|_q^2 & \leq \sum_{i=0}^\infty (q-1)^{i} \cdot \alpha_{\bm i}^2 \cdot \refE[\intH{i}{\bm X}{S}^2]
\end{align*}
Furthermore, the inequality holds as an equality when $q = 2$.
\end{lemma}
\begin{proof}See Appendix~\ref{sec:tpca-integrated-hypercontractitivty-proof}.
\end{proof}
\subsection{Proof of Proposition~\ref{prop: stpca-info-bound}}
\label{appendix:stpca-info-bound-proof}
In this subsection, we present a proof of the information bound for distributed tensor PCA (Proposition~\ref{prop: stpca-info-bound}). We begin by recalling the general information bound from Proposition~\ref{prop: main_hellinger_bound}:
\begin{align*}
    \frac{\MIhell{\bm V}{\bm Y}}{\consthell} & \explain{}{\leq} \sum_{i=1}^m   \looE{0}{i}\left[ \int \left({\looE{0}{i}\left[  \frac{\diff \dmu{\bm V}}{\diff \refmu} (\bm X_i) - \frac{\diff \nullmu}{\diff \refmu} (\bm X_i)  \bigg| \bm Y, Z_i, (\bm X_j)_{j \neq i}  \right]}\right)^2 \prior(\diff \bm V) \right]  +  {m \nullmu(\goodevnt^c)}. \\
\end{align*}
In order to analyze the conditional expectation of the centered likelihood ratio, we will approximate it by a low-degree polynomial. Recall that in Lemma~\ref{lemma: hermite-decomp-tpca}, we computed the following expansion of the likelihood ratio in terms of the Hermite polynomials:
\begin{align*}
     \frac{\diff \dmu{\bm V}}{\diff \refmu}(\bm X) = \sum_{i=0}^\infty \frac{\lambda^i}{\sqrt{i!}} \cdot H_{i}\left(\frac{\ip{\bm X}{\bm V^{\otimes k}}}{\sqrt{\dim^k}}\right).
\end{align*}
Recalling the definition of integrated Hermite polynomials (Definition~\ref{def: integrated-Hermite-tpca}), and also that
\begin{align*}
    \frac{\diff \nullmu}{\diff \refmu} & = \int \frac{\diff \dmu{\bm V}}{\diff \refmu} \; \prior(\diff \bm V),
\end{align*}
we can express the integrated likelihood ratio in terms of the integrated Hermite polynomials:
\begin{align*}
    \frac{\diff \nullmu}{\diff \refmu} & = \sum_{i=0}^\infty \frac{\lambda^i}{\sqrt{i!}} \cdot \intH{i}{\bm X}{1}.
\end{align*}
For any $t \in \N$, we define the degree $t$-approximation to the centered likelihood ratio:
\begin{align*}
    \lowdegree{ \frac{\diff\dmu{\bm V}}{\diff \mu_0}(\bm X) - \frac{\diff\nullmu}{\diff \mu_0}(\bm X)}{t}  \explain{def}{=} \sum_{i=0}^t \frac{\lambda^i}{\sqrt{i!}} \cdot \left(H_{i}\left(\frac{\ip{\bm X}{\bm V^{\otimes k}}}{\sqrt{\dim^k}}\right) - \intH{i}{\bm X}{1} \right)
\end{align*}
and the corresponding truncation error:
\begin{align*}
     \highdegree{ \frac{\diff\dmu{\bm V}}{\diff \mu_0}(\bm X) - \frac{\diff\nullmu}{\diff \mu_0}(\bm X)}{t}  \explain{def}{=} \sum_{i=t+1}^\infty \frac{\lambda^i}{\sqrt{i!}} \cdot \left(H_{i}\left(\frac{\ip{\bm X}{\bm V^{\otimes k}}}{\sqrt{\dim^k}}\right) - \intH{i}{\bm X}{1} \right).
\end{align*}
By choosing $t$ large enough, we hope that:
\begin{align*}
    \looE{0}{i}\left[  \frac{\diff \dmu{\bm V}}{\diff \refmu} (\bm X_i) - \frac{\diff \nullmu}{\diff \refmu} (\bm X_i)  \bigg| \bm Y, Z_i, (\bm X_j)_{j \neq i}  \right] & \approx \looE{0}{i}\left[ \lowdegree{ \frac{\diff \dmu{\bm V}}{\diff \refmu} (\bm X_i) - \frac{\diff \nullmu}{\diff \refmu} (\bm X_i)}{t}  \bigg| \bm Y, Z_i, (\bm X_j)_{j \neq i}  \right].
\end{align*}
We estimate the approximation error in the above equation using the following lemma. 
\begin{lemma} \label{lemma : high_degree_gaussian}
  Let $\bm X \sim \refmu$.
  Suppose that:
\begin{align*}
    t \geq (\lambda^2 e^2) \vee \ln \frac{4}{\epsilon} \vee 1.
\end{align*}
Then
\begin{align*}
    \refE\left[ \highdegree{\frac{\diff\dmu{\bm V}}{\diff \mu_0}(\bm X) - \frac{\diff\nullmu}{\diff \mu_0}(\bm X)}{t}^2 \right] \leq \epsilon.
\end{align*}
\end{lemma}
\begin{proof}
The proof of this result appears at the end of this subsection (Appendix~\ref{sec: stpca-truncation-error}). 
\end{proof}

Finally to analyze the conditional expectation of the low degree approximation using the Geometric Inequality (Proposition~\ref{prop: geometric inequality}), we need to understand the concentration properties of the low-degree approximation of the likelihood ratio. This is done using the moment estimates provided in the following lemma.

\begin{lemma} \label{lemma: concentration_TPCA} Let $\bm X \sim \refmu$. There exists a finite constant $C_k$ depending only on $k$ such that for any $q \geq 2$ which satisfies:
\begin{align*}
 \lambda^2 (q-1) \leq \frac{1}{C_k} \cdot \frac{\dim^{\frac{k}{2}}}{t^{\frac{k-2}{2}}},
\end{align*}
we have
\begin{align*}
    \sup_{\substack{S: \{\pm 1\}^d \rightarrow \R \\ \|S\|_\prior \leq 1}}  \left(\refE \left[ \left| \ip{\lowdegree{\frac{\diff\dmu{\bm V}}{\diff \mu_0}(\bm X) - \frac{\diff\nullmu}{\diff \mu_0}(\bm X)}{t}}{S}_\prior \right|^{q} \right] \right)^{\frac{2}{q}}\leq  (q-1) \cdot \sigma^2,
\end{align*}
where
\begin{align*}
    \sigma^2 \explain{def}{=}  \begin{cases} C_k \cdot \lambda^2 \cdot  d^{-\frac{k+2}{2}} & \text{if $k$ is even} ; \\  C_k \cdot \lambda^2 \cdot  d^{-\frac{k+1}{2}} & \text{if $k$ is odd} . \end{cases}
\end{align*}
\end{lemma}
\begin{proof}
The proof of this result appears at the end of this subsection (Appendix~\ref{section: tpca_lowdegree}). 
\end{proof}
Finally, we also need to estimate $\nullmu(\goodevnt^c)$ to upper bound the Hellinger Information using Proposition~\ref{prop: main_hellinger_bound}. This is the content of the following lemma.

\begin{lemma} \label{lemma: tpca_bad_event}
Consider the event $\goodevnt$:
\begin{align*}
    \goodevnt & \explain{def}{=} \left\{ \bm x \in \dataspace: \left|\frac{\diff \nullmu}{\diff \refmu} (\bm x) - 1 \right| \leq \frac{1}{2}  \right\},
\end{align*}
There exists a universal constant $C_k$ (depending only on $k$) such that, for any $2 \leq q \leq \dim/(C_k \lambda^2)$, we have
\begin{align*}
    \nullmu(\goodevnt^c) & \leq  \left(\frac{C_k q\lambda^2}{\sqrt{d^k}} + e^{-\dim} \right)^{\frac{q}{2}}.
\end{align*}
\end{lemma}
\begin{proof}
The proof of this lemma appears at the end of this subsection (Appendix~\ref{section: tpca_bad_event}).
\end{proof}
With these results, we are now ready to provide a proof of Proposition~\ref{prop: stpca-info-bound}. 

\begin{proof}[Proof of Proposition~\ref{prop: stpca-info-bound}]

Recall that in Proposition~\ref{prop: main_hellinger_bound} we showed:
\begin{align*}
    \frac{\MIhell{\bm V}{\bm Y}}{\consthell} & \explain{}{\leq} \sum_{i=1}^m   \looE{0}{i}\left[ \int \left({\looE{0}{i}\left[  \frac{\diff \dmu{\bm V}}{\diff \refmu} (\bm X_i) - \frac{\diff \nullmu}{\diff \refmu} (\bm X_i)  \bigg| \bm Y, Z_i, (\bm X_j)_{j \neq i}  \right]}\right)^2 \prior(\diff \bm V) \right]  +  {m \nullmu(\goodevnt^c)} \\
\end{align*}

The centered likelihood ratio can be decomposed as:
\begin{align*}
    \frac{\diff\dmu{\bm V}}{\diff \mu_0}(\bm X) - \frac{\diff\nullmu}{\diff \mu_0}(\bm X)& = \lowdegree{\frac{\diff\dmu{\bm V}}{\diff \mu_0}(\bm X) - \frac{\diff\nullmu}{\diff \mu_0}(\bm X)}{t} + \highdegree{\frac{\diff\dmu{\bm V}}{\diff \mu_0}(\bm X) - \frac{\diff\nullmu}{\diff \mu_0}(\bm X)}{t}.
\end{align*}
Using the inequality $(a+b)^2 \leq 2a^2 + 2b^2$ and Cauchy Schwarz Inequality:
\begin{align*}
    &\frac{1}{2}\left({\looE{0}{i}\left[  \frac{\diff \dmu{\bm V}}{\diff \refmu} (\bm X_i) - \frac{\diff \nullmu}{\diff \refmu} (\bm X_i)  \bigg| \bm Y, Z_i, (\bm X_j)_{j \neq i}  \right]}\right)^2  \leq \\ & \hspace{5.2cm}\left({\looE{0}{i}\left[  \lowdegree{\frac{\diff \dmu{\bm V}}{\diff \refmu} (\bm X_i) - \frac{\diff \nullmu}{\diff \refmu} (\bm X_i)}{t}  \bigg| \bm Y, Z_i, (\bm X_j)_{j \neq i}  \right]}\right)^2   + \\&\hspace{8.2cm}   {\looE{0}{i}\left[  \highdegree{\frac{\diff \dmu{\bm V}}{\diff \refmu} (\bm X_i) - \frac{\diff \nullmu}{\diff \refmu} (\bm X_i)}{t}^2  \bigg| \bm Y, Z_i, (\bm X_j)_{j \neq i}  \right]}.
\end{align*}
Hence,
\begin{align}
     &\frac{\MIhell{\bm V}{\bm Y}}{2\consthell} \explain{}{\leq} \nonumber\\&\hspace{1.5cm} \sum_{i=1}^m   \looE{0}{i}\left[\geometric^2_i(\bm Y, Z_i, (\bm X_j)_{j \neq i})\right]  +  \frac{m \nullmu(\goodevnt^c)}{2}  +  m \cdot\int \refE\left[\highdegree{\frac{\diff \dmu{\bm V}}{\diff \refmu} (\bm X) - \frac{\diff \nullmu}{\diff \refmu} (\bm X)}{t}^2\right] \prior(\diff \bm V), \label{eq: tpca_substitute_here}
\end{align}
where:
\begin{align*}
    &\geometric^2_i(\bm y, z_i, (\bm x_j)_{j \neq i}) \explain{def}{=} \\&\hspace{2cm}\int \left({\looE{0}{i}\left[  \lowdegree{\frac{\diff \dmu{\bm V}}{\diff \refmu} (\bm X_i) - \frac{\diff \nullmu}{\diff \refmu} (\bm X_i)}{t}  \bigg| \bm Y = \bm y, Z_i = z_i, (\bm X_j)_{j \neq i} = (\bm x_j)_{j \neq i})   \right]}\right)^2 \prior(\diff \bm V). 
\end{align*}
Our goal is to show that for any $\alpha \geq 2$, we have
\begin{align} \label{eq: stpca-goal}
    \frac{\MIhell{\bm V}{\bm Y}}{2\consthell} & \leq \underbrace{  \frac{C_k\lambda^2 \alpha}{\dim} + \mach \cdot \left(\frac{C_k \alpha \lambda^2}{\sqrt{d^k}} + e^{-\dim}\right)^{\frac{\alpha}{2}} }_{\text{Step 1}} + \underbrace{\frac{1}{\dim}}_{\text{Step 2}}\nonumber \\&\hspace{3cm}+ \underbrace{3C_k \cdot \lambda^2  \cdot \budget \left( \frac{(\lambda^2 e^2) \vee \ln(m \cdot \dim)}{d}\right)^{\frac{k}{2}} + 16 \sigma^2 \cdot \mach \cdot \budget }_{\text{Step 3}}
\end{align}
The information bound in the statement of the proposition follows by choosing $\alpha$ optimally.
The proof proceeds in several steps. In the above display, we have grouped the terms in the information bound according to the step they arise in.

\textbf{Step 1: Controlling $\nullmu(\goodevnt^c)$.} Note that if $\lambda^2 \alpha > \dim/C_k$, then the claimed upper bound \eqref{eq: stpca-goal} on $\MIhell{\bm Y}{\bm V}$ is trivial since $\MIhell{\bm Y}{\bm V} \leq 1$. Hence we assume $\lambda^2 \alpha \leq d/C_k$. Applying Lemma~\ref{lemma: tpca_bad_event} with $q = \alpha$, we have
\begin{align}
\mach \cdot \nullmu(\goodevnt^c)  \leq  \mach \cdot \left(\frac{C_k \alpha \lambda^2}{\sqrt{d^k}} + e^{-\dim}\right)^{\frac{\alpha}{2}}\leq  \frac{C_k \lambda^2 \alpha}{\dim} + \mach \cdot \left(\frac{C_k \alpha \lambda^2}{\sqrt{d^k}} + e^{-\dim} \right)^{\frac{\alpha}{2}}. \label{eq: tpca_substitute_p1}
\end{align}

\textbf{Step 2: Controlling High Degree Term. } We set:
\begin{align*}
    t =   (\lambda^2 e^2) \vee \ln(\mach \cdot \dim).
\end{align*}Applying Lemma~\ref{lemma : high_degree_gaussian}, we obtain,
\begin{align} \label{eq: tpca_substitute_p2}
    \refE\left[\highdegree{\frac{\diff \dmu{\bm V}}{\diff \refmu} (\bm X) - \frac{\diff \nullmu}{\diff \refmu} (\bm X)}{t}^2\right] \leq \frac{1}{\mach \cdot \dim}.
\end{align}

\textbf{Step 3: Controlling Low Degree Term. } Next we control $\geometric^2_i(\bm y, z_i, (\bm x_j)_{j \neq i})$. By linearization (Lemma~\ref{lemma: linearization}) we have:
\begin{align*}
    &\Psi(\bm y, z_i, (\bm x_j)_{j \neq i}) = \\&\hspace{1cm}\sup_{\substack{S: \paramspace \rightarrow \R \\ \|S\|_{\prior} \leq 1}} {\looE{0}{i}\left[  \ip{\lowdegree{\frac{\diff \dmu{\bm V}}{\diff \refmu} (\bm X_i) - \frac{\diff \nullmu}{\diff \refmu} (\bm X_i)}{t}}{S}_{\prior}  \bigg|\bm Y = \bm y, Z_i = z_i, (\bm X_j)_{j \neq i} = (\bm x_j)_{j \neq i}) \right]}.
\end{align*}
Using the Geometric Inequality framework (Proposition~\ref{prop: geometric inequality}) we can bound $|\Psi(\bm y, z_i, (\bm x_j)_{j \neq i})|$ if we can understand the concentration properties of:
\begin{align*}
    f_S(\bm X) \explain{def}{=} \ip{\lowdegree{\frac{\diff \dmu{\bm V}}{\diff \refmu} (\bm X_i) - \frac{\diff \nullmu}{\diff \refmu} (\bm X_i)}{t}}{S}_{\prior}, \;  \bm X \sim \refmu
\end{align*}
for any $S: \paramspace \rightarrow \R, \;  \|S\|_{\prior} \leq 1$. The concentration properties of $f_S(\bm X)$ are studied in Lemma~\ref{lemma: concentration_TPCA} which shows that for any $q$ such that:
\begin{align} \label{eq: tpca_q_condition}
    1 \leq q-1 \leq \frac{1}{C_k \cdot \lambda^2}  \left( \frac{d}{t} \right)^{\frac{k}{2}},
\end{align}
we have
\begin{align*}
    \sup_{\substack{S: \{\pm 1\}^d \rightarrow \R \\ \|S\|_\prior \leq 1}}  \left(\refE \left[ \left| f_S(\bm X) \right|^{q} \right] \right)^{\frac{2}{q}}\leq  \sigma^2 (q-1), \; 
\end{align*}
where:
\begin{align*}
    \sigma^2 \explain{def}{=}  \begin{cases} C_k \cdot \lambda^2 \cdot  d^{-\frac{k+2}{2}}: & \text{ $k$ is even} \\  C_k \cdot \lambda^2 \cdot  d^{-\frac{k+1}{2}} : & \text{ $k$ is odd} \end{cases}.
\end{align*}
In order to apply Proposition~\ref{prop: geometric inequality} we need to choose $q$ appropriately. The choice of $q$ depends on $$\looP{0}{i}(\bm Y = \bm y, Z_i = z_i | (\bm X_j)_{j \neq i} = (\bm x_j)_{j \neq i}).$$ We define the set of rare and frequent realizations of $\bm Y, Z_i$:
\begin{align*}
    \mathcal{R}_{\mathsf{freq}}^{(i)} &\explain{def}{=} \left\{ (\bm y, \bm z_i) \in \{0,1\}^{m\budget + 1} :  \looP{0}{i}(\bm Y = \bm y, Z_i = z_i | (\bm X_j)_{j \neq i} = (\bm x_j)_{j \neq i}) > \frac{1}{e} \right\}, \\
    \mathcal{R}_{\mathsf{rare}}^{(i)} &\explain{def}{=} \left\{ (\bm y, \bm z_i) \in \{0,1\}^{m\budget + 1} : 0 <   \looP{0}{i}(\bm Y = \bm y, Z_i = z_i | (\bm X_j)_{j \neq i} = (\bm x_j)_{j \neq i}) \leq 4^{-\budget} \right\}.
\end{align*}
By the tower property,
\begin{align*}
    \looE{0}{i}\left[\geometric^2_i(\bm Y, X_i, (\bm X_j)_{j \neq i})\right] & =   \looE{0}{i}\looE{0}{i}\left[\geometric^2_i(\bm Y, Z_i, (\bm X_j)_{j \neq i}) \; | \;  (\bm X_j)_{j \neq i}) \right] \\&=   \looE{0}{i} F_i((\bm X_j)_{j \neq i}) + \looE{0}{i} R_i((\bm X_j)_{j \neq i}) + \looE{0}{i} O_i((\bm X_j)_{j \neq i}).
\end{align*}
where:
\begin{align*}
    F_i((\bm x_j)_{j \neq i}) & \explain{def}{=} \sum_{(\bm y, z_i) \in \mathcal{R}_{\mathsf{freq}}^{(i)}} \Psi(\bm y, z_i, (\bm x_j)_{j \neq i})^2 \cdot  \looP{0}{i}(\bm Y = \bm y, Z_i = z_i | (\bm X_j)_{j \neq i} = (\bm x_j)_{j \neq i}), \\
     R_i((\bm x_j)_{j \neq i}) &\explain{def}{=}  \sum_{(\bm y, z_i) \in \mathcal{R}_{\mathsf{rare}}^{(i)}} \Psi(\bm y, z_i, (\bm x_j)_{j \neq i})^2 \cdot  \looP{0}{i}(\bm Y = \bm y, Z_i = z_i | (\bm X_j)_{j \neq i} = (\bm x_j)_{j \neq i}), \\
      O_i((\bm x_j)_{j \neq i}) & \explain{def}{=} \sum_{(\bm y, z_i) \notin \mathcal{R}_{\mathsf{freq}}^{(i)} \cup \mathcal{R}_{\mathsf{rare}}^{(i)}} \Psi(\bm y, z_i, (\bm x_j)_{j \neq i})^2 \cdot  \looP{0}{i}(\bm Y = \bm y, Z_i = z_i | (\bm X_j)_{j \neq i} = (\bm x_j)_{j \neq i}).
\end{align*}
We bound each of the terms separately.
\begin{description}
\item [Case 1: Frequent realizations. ] Consider the case when $(\bm y, z_i) \in \mathcal{R}_{\mathsf{freq}}^{(i)}$. In this case we set $q=2$. We need to check that this choice obeys \eqref{eq: tpca_q_condition}. Indeed if \eqref{eq: tpca_q_condition} is violated for $q=2$, then the upper bound on $\MIhell{\bm V}{\bm Y}$ in \eqref{eq: stpca-goal} is trivial since the term:
\begin{align*}
   3C_k \cdot \lambda^2  \cdot \budget \left( \frac{(\lambda^2 e^2) \vee \ln(m \cdot \dim)}{d}\right)^{\frac{k}{2}} > 1.
\end{align*}
Hence we may assume that $q = 2$ obeys \eqref{eq: tpca_q_condition} without loss of generality and we obtain by Proposition~\ref{prop: geometric inequality},
\begin{align*}
    |\Psi(\bm y, z_i, (\bm x_j)_{j \neq i})|  & \leq  \sigma \cdot \looP{0}{i}(\bm Y = \bm y, Z_i = z_i | (\bm X_j)_{j \neq i} = (\bm x_j)_{j \neq i})^{-\frac{1}{2}}, \; \forall \; (\bm y, z_i) \in \mathcal{R}_{\mathsf{freq}}^{(i)}.
\end{align*}
Note that $|\mathcal{R}_{\mathsf{freq}}^{(i)}| \leq e$, and hence,
\begin{align*}
    F_i((\bm x_j)_{j \neq i}) & \leq 2 \sigma^2.
\end{align*}
\item [Case 2: Rare realizations. ] Consider the case when $(\bm y, z_i) \in \mathcal{R}_{\mathsf{rare}}^{(i)}$. In this case we set $q=4$. It is straightforward to check that if $q$ doesn't satisfy \eqref{eq: tpca_q_condition},then the claimed bound on $\MIhell{\bm V}{\bm Y}$ in \eqref{eq: stpca-goal} is vacuous and hence we assume $q=4$ satisfies \eqref{eq: tpca_q_condition}. Applying Proposition~\ref{prop: geometric inequality} gives us:
\begin{align*}
    |\Psi(\bm y, z_i, (\bm x_j)_{j \neq i})|  & \leq  \sqrt{3} \cdot  \sigma \cdot \looP{0}{i}(\bm Y = \bm y, Z_i = z_i | (\bm X_j)_{j \neq i} = (\bm x_j)_{j \neq i})^{-\frac{1}{4}}, \; \forall \; (\bm y, z_i) \in \mathcal{R}_{\mathsf{rare}}^{(i)}.
\end{align*}
Hence we can upper bound $R_i$:
\begin{align*}
     R_i((\bm x_j)_{j \neq i}) &\explain{def}{=}  \sum_{(\bm y, z_i) \in \mathcal{R}_{\mathsf{rare}}^{(i)}} \Psi(\bm y, z_i, (\bm x_j)_{j \neq i})^2 \cdot  \looP{0}{i}(\bm Y = \bm y, Z_i = z_i | (\bm X_j)_{j \neq i} = (\bm x_j)_{j \neq i}) \\
     & \leq 3 \sigma^2  \sum_{(\bm y, z_i) \in \mathcal{R}_{\mathsf{rare}}^{(i)}}  \looP{0}{i}(\bm Y = \bm y, Z_i = z_i | (\bm X_j)_{j \neq i} = (\bm x_j)_{j \neq i})^{\frac{1}{2}} \\
     & \leq 3 \sigma^2 2^{-b} |\mathcal{R}_{\mathsf{rare}}^{(i)}|.
\end{align*}
Recall that we assume that the communication protocol is deterministic, i.e. the bit written by a machine is a deterministic function its local dataset and the bits written on the blackboard so far. Hence, conditional on $(\bm X_j)_{j \neq i}$ there are only $2^{\budget+1}$ possible realizations of $(\bm Y, Z_i)$ with non-zero probability.  And hence, $|\mathcal{R}_{\mathsf{rare}}^{(i)}| \leq 2^{b+1}$. Hence,
\begin{align*}
    R_i((\bm x_j)_{j \neq i}) \leq 6 \sigma^2.
\end{align*}
\item [Case 3: All other realizations. ] Now consider any realization $(\bm y, z_i) \notin \mathcal{R}_{\mathsf{rare}}^{(i)} \cup \mathcal{R}_{\mathsf{freq}}^{(i)}$. In this case, we set $q$ as:
\begin{align*}
q & =  - 2 \ln \looP{0}{i}(\bm Y = \bm y, Z_i = z_i | (\bm X_j)_{j \neq i} = (\bm x_j)_{j \neq i}).
\end{align*}
Since $(\bm y, z_i) \notin \mathcal{R}_{\mathsf{rare}}^{(i)} \cup \mathcal{R}_{\mathsf{freq}}^{(i)}$, we have
\begin{align*}
     2 \leq q \leq \budget \ln(4) \leq 2 \budget.
\end{align*}
In particular, if
\begin{align*}
    \budget \lambda^2 \leq \frac{1}{2C_k }  \left( \frac{d}{t} \right)^{\frac{k}{2}} ,
\end{align*}
then \eqref{eq: tpca_q_condition} holds for this choice of $q$.
On the other hand, if this is not the case, then the claimed upper bound \eqref{eq: stpca-goal} on $\MIhell{\bm V}{\bm Y}$ is trivial since
\begin{align*}
   3C_k \cdot \lambda^2  \cdot \budget \left( \frac{(\lambda^2 e^2) \vee \ln(m \cdot \dim)}{d}\right)^{\frac{k}{2}} > 1.
\end{align*}
Hence, we have for any $(\bm y, z_i) \notin \mathcal{R}_{\mathsf{rare}}^{(i)} \cup \mathcal{R}_{\mathsf{freq}}^{(i)}$:
\begin{align*}
     |\Psi(\bm y, z_i, (\bm x_j)_{j \neq i})|  & \leq \sqrt{2e} \cdot  \sigma \cdot (- \ln \looP{0}{i}(\bm Y = \bm y, Z_i = z_i | (\bm X_j)_{j \neq i} = (\bm x_j)_{j \neq i}))^{-\frac{1}{2}}.
\end{align*}
Hence,
\begin{align*}
     &O_i((\bm x_j)_{j \neq i})  \explain{def}{=} \sum_{(\bm y, z_i) \notin \mathcal{R}_{\mathsf{freq}}^{(i)} \cup \mathcal{R}_{\mathsf{rare}}^{(i)}} \Psi(\bm y, z_i, (\bm x_j)_{j \neq i})^2 \cdot  \looP{0}{i}(\bm Y = \bm y, Z_i = z_i | (\bm X_j)_{j \neq i} = (\bm x_j)_{j \neq i}) \\
     & \leq 2 e \sigma^2 \sum_{(\bm y, z_i) \in \{0,1\}^{m \budget}} h(\looP{0}{i}(\bm Y = \bm y, Z_i = z_i | (\bm X_j)_{j \neq i} = (\bm x_j)_{j \neq i})),
\end{align*}
where $h(\cdot)$ is the entropy function $ h(x) \explain{def}{=} -x \ln(x)$. We note that the expression appearing in the above equation is the entropy of $(\bm Y, Z_i)$ conditional on $(\bm X_j)_{j \neq i} = (\bm x_j)_{j \neq i}$. Since the protocol is deterministic (cf.\ Remark~\ref{rem:deterministic}) there are at most $2^{b+1}$ realizations $(\bm y, z_i)$ such that $\looP{0}{i}(\bm Y = \bm y, Z_i = z_i | (\bm X_j)_{j \neq i} = (\bm x_j)_{j \neq i})>0$.  Since the entropy is maximized by the uniform distribution:
\begin{align*}
    O_i((\bm x_j)_{j \neq i}) & \leq 2 \cdot e \cdot \ln(2)\cdot \sigma^2 \cdot  (b+1). 
\end{align*}
\end{description}
Using the bounds from the above 3 cases, we have:
\begin{align} \label{eq: tpca_substitute_p3}
     \looE{0}{i}\left[\geometric^2_i(\bm Y, X_i, (\bm X_j)_{j \neq i})\right] & =   \looE{0}{i}\looE{0}{i}\left[\geometric^2_i(\bm Y, Z_i, (\bm X_j)_{j \neq i}) \; | \;  (\bm X_j)_{j \neq i}) \right] \\&\leq 8 \sigma^2 +  2  e  \ln(2)\cdot \sigma^2 \cdot  (b+1) \leq 16 \sigma^2 b.
\end{align}
Substituting the estimates \eqref{eq: tpca_substitute_p1}, \eqref{eq: tpca_substitute_p2} and \eqref{eq: tpca_substitute_p3} in \eqref{eq: tpca_substitute_here} we obtain \eqref{eq: stpca-goal}. This proves the first claim made in the statement of the proposition. Lastly, we consider scaling regime:
\begin{align*}
    \lambda = \Theta(1), \; m = \Theta(\dim^{\eta}), \; b = \Theta(\dim^\beta)
\end{align*}
for some constants $\eta \geq 1, \beta \geq 0$, which satisfy:
\begin{align*}
    \eta + \beta  < \left\lceil\frac{k+1}{2} \right\rceil.
\end{align*}
We set $\alpha$ to be any constant strictly more than $\max(2, 4\eta/k)$ and observe that $\beta < k/2 + 1 -\eta \leq k/2$. Consequently, each term in the upper bound in \eqref{eq: stpca-goal} is $o_d(1)$. This finishes the proof of the proposition. 
\end{proof}

The remainder of this subsection is devoted to the proofs of Lemma~\ref{lemma: concentration_TPCA} (in Appendix~\ref{section: tpca_lowdegree}), Lemma~\ref{lemma : high_degree_gaussian} (in Appendix~\ref{sec: stpca-truncation-error}) and Lemma~\ref{lemma: tpca_bad_event} (in Appendix~\ref{section: tpca_bad_event}).
\subsubsection{Analysis of Low Degree Part}\label{section: tpca_lowdegree}
In this section, we provide a proof of Lemma~\ref{lemma: concentration_TPCA}.
\begin{proof}[Proof of Lemma~\ref{lemma: concentration_TPCA}]
Recall that,
\begin{align*}
    \lowdegree{ \frac{\diff\dmu{\bm V}}{\diff \mu_0}(\bm X) - \frac{\diff\nullmu}{\diff \mu_0}(\bm X)}{t}  \explain{def}{=} \sum_{i=0}^t \frac{\lambda^i}{\sqrt{i!}} \cdot \left(H_{i}\left(\frac{\ip{\bm X}{\bm V^{\otimes k}}}{\sqrt{\dim^k}}\right) - \intH{i}{\bm X}{1} \right),
\end{align*}
and in particular,
\begin{align*}
    \int \lowdegree{ \frac{\diff\dmu{\bm V}}{\diff \mu_0}(\bm X) - \frac{\diff\nullmu}{\diff \mu_0}(\bm X)}{t}\; \prior(\diff \bm V) & = 0.
\end{align*}
Hence,
\begin{align*}
    \ip{\lowdegree{\frac{\diff\dmu{\bm V}}{\diff \mu_0}(\bm X) - \frac{\diff\nullmu}{\diff \mu_0}(\bm X)}{t}}{S}_\prior & = \ip{\lowdegree{\frac{\diff\dmu{\bm V}}{\diff \mu_0}(\bm X) - \frac{\diff\nullmu}{\diff \mu_0}(\bm X)}{t}}{S - \ip{S}{1}_\prior }_\prior \\
    & = \ip{\lowdegree{\frac{\diff\dmu{\bm V}}{\diff \mu_0}(\bm X) }{t}}{S - \ip{S}{1}_\prior }_\prior,
\end{align*}
where,
\begin{align*}
    \lowdegree{\frac{\diff\dmu{\bm V}}{\diff \mu_0}(\bm X) }{t} \explain{def}{=} \sum_{i=0}^t \frac{\lambda^i}{\sqrt{i!}} \cdot H_{i}\left(\frac{\ip{\bm X}{\bm V^{\otimes k}}}{\sqrt{\dim^k}}\right) .
\end{align*}
Consequently,
\begin{align*}
     \sup_{\substack{S: \{\pm 1\}^d \rightarrow \R \\ \|S\|_\prior \leq 1}}  \refE \left[ \left| \ip{\lowdegree{\frac{\diff\dmu{\bm V}}{\diff \mu_0}(\bm X) - \frac{\diff\nullmu}{\diff \mu_0}(\bm X)}{t}}{S}_\prior \right|^{q} \right]  & =  \sup_{\substack{S: \{\pm 1\}^d \rightarrow \R \\ \|S\|_\prior \leq 1 \\ \ip{S}{1}_\prior = 0}}  \refE \left[ \left| \ip{\lowdegree{\frac{\diff\dmu{\bm V}}{\diff \mu_0}(\bm X) }{t}}{S}_\prior \right|^{q} \right].
\end{align*}
We can compute:
\begin{align*}
    \ip{\lowdegree{\frac{\diff\dmu{\bm V}}{\diff \mu_0}(\bm X) }{t}}{S}_\prior & = \sum_{i=1}^t \frac{\lambda^i}{\sqrt{i!}} \cdot \intH{i}{\bm X}{S}.
\end{align*}
Note that when $\bm X \sim \refmu$, $\bm X$ is a Gaussian tensor with i.i.d. $\gauss{0}{1}$ entries. Hence by Gaussian Hypercontractivity (Lemma~\ref{lemma: integrated-hermite-hypercontractivity-tpca}):
\begin{align*}
    \left(\refE \left[ \left| \ip{\lowdegree{\frac{\diff\dmu{\bm V}}{\diff \mu_0}(\bm X) }{t}}{S}_\prior \right|^{q} \right] \right)^{\frac{2}{q}} & = \sum_{i=1}^t \frac{\lambda^{2i}}{i!} \cdot  (q-1)^i \cdot  \refE[\intH{i}{\bm X}{S}^2] \\
    & \explain{(a)}{\leq} \sum_{i=1}^t \frac{\lambda^{2i}}{i!} \cdot (q-1)^i \cdot  (Cki)^{\frac{ki}{2}} \cdot  d^{-\lceil\frac{ki+1}{2}\rceil},
\end{align*}
where in the step marked (a) we used the estimate on $\refE[\intH{i}{\bm X}{S}^2]$ from Lemma~\ref{lemma: norm-integrated-hermite-tpca}. 
We split the above sum into two parts:
\begin{align*}
    \sum_{i=1}^t \frac{\lambda^{2i}}{i!} \cdot (q-1)^i \cdot  (Cki)^{\frac{ki}{2}} \cdot  d^{-\lceil\frac{ki+1}{2}\rceil} & = \sum_{\substack{i=1\\\text{$i$ is odd}}}^t \frac{\lambda^{2i}}{i!} \cdot (q-1)^i \cdot  (Cki)^{\frac{ki}{2}} \cdot  d^{-\lceil\frac{ki+1}{2}\rceil} \\& \hspace{2cm}+ \sum_{\substack{i=1\\\text{$i$ is even}}}^t \frac{\lambda^{2i}}{i!} \cdot (q-1)^i \cdot  (Cki)^{\frac{ki}{2}} \cdot  d^{-\lceil\frac{ki+1}{2}\rceil}.
\end{align*}
We first analyze the sum corresponding to the odd terms. By estimating the ratio of two consecutive terms in the sum, one can obtain a constant $C_k$ such that if,
\begin{align}\label{eq:tpca-lowdegree-moment-estimate}
    \frac{\lambda^4 \cdot (q-1)^2 \cdot  t^{k-2}}{\dim^k} & \leq \frac{1}{C_k},
\end{align}
then the sum decays geometrically by a factor of $1/2$.  Hence,
\begin{align*}
    \sum_{\substack{i=1\\\text{$i$ is odd}}}^t \frac{\lambda^{2i}}{i!} \cdot (q-1)^i \cdot  (Cki)^{\frac{ki}{2}} \cdot  d^{-\lceil\frac{ki+1}{2}\rceil} & \leq \frac{C_k \cdot \lambda^2 \cdot  (q-1)}{\dim^{\lceil\frac{k+1}{2}\rceil}} \left( 1 + \frac{1}{2} + \frac{1}{4} + \dotsc \right) \leq \frac{C_k \cdot \lambda^2 \cdot  (q-1)}{\dim^{\lceil\frac{k+1}{2}\rceil}}.
\end{align*}
The same argument can be used to estimate the sum corresponding to the even terms. Under the same assumption \eqref{eq:tpca-lowdegree-moment-estimate}, we have
\begin{align*}
    \sum_{\substack{i=1\\\text{$i$ is even}}}^t \frac{\lambda^{2i}}{i!} \cdot (q-1)^i \cdot  (Cki)^{\frac{ki}{2}} \cdot  d^{-\lceil\frac{ki+1}{2}\rceil} & \leq \frac{C_k \cdot \lambda^4 \cdot  (q-1)^2}{\dim^{k+1}} \left( 1 + \frac{1}{2}  + \dotsc \right) \leq \frac{C_k \cdot \lambda^4 \cdot  (q-1)^2}{\dim^{k+1}}.
\end{align*}
Hence,
\begin{align*}
     \sup_{\substack{S: \{\pm 1\}^d \rightarrow \R \\ \|S\|_\prior \leq 1}}  \left(\refE \left[ \left| \ip{\lowdegree{\frac{\diff\dmu{\bm V}}{\diff \mu_0}(\bm X) - \frac{\diff\nullmu}{\diff \mu_0}(\bm X)}{t}}{S}_\prior \right|^{q} \right] \right)^{\frac{2}{q}} & \leq \frac{C_k \cdot \lambda^2 \cdot  (q-1)}{\dim^{\lceil\frac{k+1}{2}\rceil}} + \frac{C_k \cdot \lambda^4 \cdot  (q-1)^2}{\dim^{k+1}} \\
     & \leq \frac{C_k \cdot \lambda^2 \cdot  (q-1)}{\dim^{\lceil\frac{k+1}{2}\rceil}}.
\end{align*}
In the above display, in order to obtain the final inequality, we again used assumption \eqref{eq:tpca-lowdegree-moment-estimate}. This concludes the proof. 
\end{proof}
\subsubsection{Analysis of High Degree Part} \label{sec: stpca-truncation-error}
In this section, we provide a proof of Lemma~\ref{lemma : high_degree_gaussian}.
\begin{proof}[Proof of Lemma~\ref{lemma : high_degree_gaussian}] Recall that,
\begin{align*}
     \highdegree{ \frac{\diff\dmu{\bm V}}{\diff \mu_0}(\bm X) - \frac{\diff\nullmu}{\diff \mu_0}(\bm X)}{t}  \explain{def}{=} \sum_{i=t+1}^\infty \frac{\lambda^i}{\sqrt{i!}} \cdot \left(H_{i}\left(\frac{\ip{\bm X}{\bm V^{\otimes k}}}{\sqrt{\dim^k}}\right) - \intH{i}{\bm X}{1} \right).
\end{align*}
Hence,
\begin{align*}
    &\refE\left[ \highdegree{\frac{\diff\dmu{\bm V}}{\diff \mu_0}(\bm X) - \frac{\diff\nullmu}{\diff \mu_0}(\bm X)}{t}^2 \right]  \leq \\& \hspace{4.7cm}2 \refE\left[ \left( \sum_{i=t+1}^\infty \frac{\lambda^i}{\sqrt{i!}} \cdot H_{i}\left(\frac{\ip{\bm X}{\bm V^{\otimes k}}}{\sqrt{\dim^k}}\right) \right)^2 \right] + 2 \refE\left[ \left(\sum_{i=t+1}^\infty \intH{i}{\bm X}{1}\right)^2 \right].
\end{align*}
By the orthogonality of Hermite and integrated Hermite polynomials (see Lemma~\ref{lemma: integrate-hermite-orthogonality-tpca}), we obtain,
\begin{align*}
     \refE\left[ \highdegree{\frac{\diff\dmu{\bm V}}{\diff \mu_0}(\bm X) - \frac{\diff\nullmu}{\diff \mu_0}(\bm X)}{t}^2 \right] & \leq 2 \left( \sum_{i=t+1}^\infty \refE\left[H_{i}^2\left(\frac{\ip{\bm X}{\bm V^{\otimes k}}}{\sqrt{\dim^k}}\right) \right] + \refE[\intH{i}{\bm X}{1}^2] \right).
\end{align*}
By Jensen's Inequality,
\begin{align*}
    \refE[\intH{i}{\bm X}{1}^2] \leq \int  \refE\left[H_{i}^2\left(\frac{\ip{\bm X}{\bm V^{\otimes k}}}{\sqrt{\dim^k}}\right) \right] \; \prior(\diff \bm V) = 1.
\end{align*}
Hence,
\begin{align*}
    \refE\left[ \highdegree{\frac{\diff\dmu{\bm V}}{\diff \mu_0}(\bm X) - \frac{\diff\nullmu}{\diff \mu_0}(\bm X)}{t}^2 \right]  \leq 4\sum_{i=t+1}^{\infty} \frac{\lambda^{2i}}{i!}  \explain{(a)}{\leq} \epsilon.
\end{align*}
In the last step, we used the hypothesis on $t$ and Fact~\ref{fact: partial_exp_series} %
given in Appendix~\ref{appendix: misc}.
\end{proof}
\subsubsection{Analysis of Bad Event}
\label{section: tpca_bad_event}
In this section, we provide a proof of Lemma~\ref{lemma: tpca_bad_event}.
\begin{proof}[Proof of Lemma~\ref{lemma: tpca_bad_event}]
For any $q \geq 2$, we have, by Chebychev's Inequality:
\begin{align*}
    \barP(\goodevnt^c) & \leq 2^q \cdot \barE \left|\frac{\diff \nullmu}{\diff \refmu} (\bm x) - 1 \right|^{q} \\
    & = 2^q \cdot \refE \frac{\diff \nullmu}{\diff \refmu} \left|\frac{\diff \nullmu}{\diff \refmu} (\bm x) - 1 \right|^{q} \\
    & \leq 2^q \left( \refE  (\bm x)\left|\frac{\diff \nullmu}{\diff \refmu} (\bm x) - 1 \right|^{q+1} + \refE \left|\frac{\diff \nullmu}{\diff \refmu} (\bm x) - 1 \right|^{q}\right)
\end{align*}
Recalling Lemma~\ref{lemma: hermite-decomp-tpca} and Definition~\ref{def: integrated-Hermite-tpca}, we have
\begin{align*}
    \frac{\diff\nullmu}{\diff \mu_0}(\bm X) & = 1 + \sum_{t = 1}^\infty \frac{\lambda^i}{\sqrt{i!}} \intH{i}{\bm X}{1}. 
\end{align*}
By Gaussian Hypercontractivity (Lemma~\ref{lemma: integrated-hermite-hypercontractivity-tpca}) we have, for any $q \geq 2$:
\begin{align*}
    \left( \refE \left|\frac{\diff \nullmu}{\diff \refmu} (\bm x) - 1 \right|^{q} \right)^{\frac{2}{q}} & \leq \sum_{i=1}^\infty \frac{\lambda^{2i}\cdot (q-1)^i}{i!} \cdot \refE[\intH{i}{\bm X}{1}^2].
\end{align*}
We split the above sum into a high-degree part and a low degree part. Let $t \in \N$ be arbitrary. 
Then
\begin{align*}
    \left( \refE \left|\frac{\diff \nullmu}{\diff \refmu} (\bm x) - 1 \right|^{q} \right)^{\frac{2}{q}} & \leq \sum_{i=1}^t \frac{\lambda^{2i}\cdot (q-1)^i}{i!} \cdot \refE[\intH{i}{\bm X}{1}^2] + \sum_{i=t+1}^\infty \frac{\lambda^{2i}\cdot (q-1)^i}{i!} \cdot \refE[\intH{i}{\bm X}{1}^2]. 
\end{align*}
\paragraph{Analysis of the low degree part: } Using the bound on $\refE[\intH{i}{\bm X}{1}^2]$ obtained in Lemma~\ref{lemma: norm-integrated-hermite-tpca}, we have
\begin{align*}
    \sum_{i=1}^t \frac{\lambda^{2i}\cdot (q-1)^i}{i!} \cdot \refE[\intH{i}{\bm X}{1}^2] & \leq \sum_{i=1}^t \frac{\lambda^{2i}\cdot (q-1)^i}{i!} \cdot \left( \frac{Cki}{\dim} \right)^{\frac{ki}{2}}.
\end{align*}
By analyzing the ratio of consecutive terms in the sum, one can find a constant $C_k$ depending only on $k$ such that, if,
\begin{align} \label{eq:t-condition-bad-evt}
    \lambda^2 \cdot (q-1) \leq \frac{1}{C_k} \cdot \frac{\dim^{\frac{k}{2}}}{t^{\frac{k-2}{2}}},
\end{align}
then the sum decays geometrically with a factor of atleast $1/2$ and hence,
\begin{align*}
     \sum_{i=1}^t \frac{\lambda^{2i}\cdot (q-1)^i}{i!} \cdot \refE[\intH{i}{\bm X}{1}^2] & \leq \frac{C_k \cdot \lambda^2 \cdot (q-1)}{\dim^{\frac{k}{2}}} \cdot \left(1 + \frac{1}{2} + \frac{1}{4} + \dotsb \right) \leq \frac{C_k \cdot \lambda^2 \cdot (q-1)}{\dim^{\frac{k}{2}}}.
\end{align*}

\paragraph{Analysis of high degree part: } Recall the definition of integrated Hermite polynomials (Definition~\ref{def: integrated-Hermite-tpca}):
\begin{align*}
    \intH{i}{\bm X}{1} & = \int H_i \left( \frac{\ip{\bm X}{\bm V^{\otimes k}}}{\dim^{\frac{k}{2}}} \right) \; \prior (\diff \bm V).
\end{align*}
Hence,  by Jensen's Inequality,
\begin{align*}
    \refE[\intH{i}{\bm X}{1}^2] & \leq \int \refE\left[H_i \left( \frac{\ip{\bm X}{\bm V^{\otimes k}}}{\dim^{\frac{k}{2}}} \right)^2 \right] \; \prior (\diff \bm V) = 1.
\end{align*}
Hence,
\begin{align*}
    \sum_{i=t+1}^\infty \frac{\lambda^{2i}\cdot (q-1)^i}{i!} \cdot \refE[\intH{i}{\bm X}{1}^2] & \leq \sum_{i=t+1}^\infty \frac{\lambda^{2i}\cdot (q-1)^i}{i!}.
\end{align*}
Appealing to Fact~\ref{fact: partial_exp_series}, we set,
\begin{align}\label{eq:t-choice-bad-evnt}
    t = (e^2 \lambda^2 (q-1)) \vee \dim \vee 1, 
\end{align}
and obtain,
\begin{align*}
    \sum_{i=t+1}^\infty \frac{\lambda^{2i}\cdot (q-1)^i}{i!} \cdot \refE[\intH{i}{\bm X}{1}^2] & \leq e^{-\dim}. 
\end{align*}
Note that the hypothesis assumed in the statememt of the lemma $\lambda^2 \cdot q \leq \dim / C_k$ guarantees that the choice of $t$ in \eqref{eq:t-choice-bad-evnt} satisfies \eqref{eq:t-condition-bad-evt}. Hence, we have shown that for any $q \geq 2$
\begin{align*}
    \left( \refE \left|\frac{\diff \nullmu}{\diff \refmu} (\bm x) - 1 \right|^{q} \right)^{\frac{2}{q}} & \leq \frac{C_k (q-1)\lambda^2}{\sqrt{d^k}} + e^{-\dim},
\end{align*}
for a universal constant $C_k$ depending only on $k$ provided $\lambda^2(q-1) \leq d/C_k$.
Applying this with $q, q+1$ we obtain:
\begin{align*}
    \nullmu(\goodevnt^c) & \leq \left(\frac{C_k q\lambda^2}{\sqrt{d^k}}  + \frac{1}{\dim} \right)^{\frac{q}{2}} + \left(\frac{C_k q\lambda^2}{\sqrt{d^k}} + e^{-\dim} \right)^{\frac{q+1}{2}}
\end{align*}
provided $\lambda^2 q \leq \dim /C_k$. Note that under this assumption, since $k \geq 2$, we have $C_k\lambda^2 q/\sqrt{\dim^k} \leq C_k\lambda^2 q/\dim \leq 1$. Hence the above bound can be simplified to:
\begin{align*}
     \nullmu(\goodevnt^c) & \leq \left(\frac{C_k q\lambda^2}{\sqrt{d^k}} + e^{-\dim} \right)^{\frac{q}{2}},
\end{align*}
for a suitably large constant $C_k$.
\end{proof}

\subsection{Omitted Proofs from Appendix~\ref{sec: tpca-harmonic}}
\label{appendix:stpca-harmonic-proofs}
This section contains the proofs of the various analytic properties of the likelihood ratio for Tensor PCA, which were stated in Appendix~\ref{sec: tpca-harmonic}. 
\subsubsection{Proof of Lemma~\ref{lemma: hermite-decomp-tpca}} \label{sec:tpca-hermite-proof}

\begin{proof}[Proof of Lemma~\ref{lemma: hermite-decomp-tpca}]
We observe that,
\begin{align*}
   \frac{\diff \dmu{\bm V}}{\diff \refmu}(\bm X) & = \exp\left( \lambda \cdot \frac{\ip{\bm X}{\bm V^{\otimes k}}}{\sqrt{\dim^k}} - \frac{\lambda^2}{2} \right).
\end{align*}
In particular, the likelihood ratio depends on $\bm X$ only through the projection $\ip{\bm X}{\bm V^{\otimes k}}$. Observe that when $\bm X \sim \refmu$,
\begin{align*}
    \frac{\ip{\bm X}{\bm V^{\otimes k}}}{\sqrt{\dim^k}} \sim \gauss{0}{1}.
\end{align*}
Since Hermite polynomials form a complete orthonormal basis for $L^2(\gauss{0}{1})$, the likelihood ratio admits an expansion of the form:
\begin{align*}
     \frac{\diff \dmu{\bm V}}{\diff \refmu}(\bm X) = \sum_{i=0}^\infty c_i \cdot H_{i}\left(\frac{\ip{\bm X}{\bm V^{\otimes k}}}{\sqrt{\dim^k}}\right).
\end{align*}
The coefficients $c_i$ are given by:
\begin{align*}
    c_i \explain{def}{=} \refE\left[ \frac{\diff \dmu{\bm V}}{\diff \refmu}(\bm X) \cdot H_i\left(  \frac{\ip{\bm X}{\bm V^{\otimes k}}}{\sqrt{\dim^k}}\right) \right],
\end{align*}
where $\bm X \sim \refmu$. We can simplify $c_i$ as follows:
\begin{align*}
    c_i &\explain{def}{=} \refE\left[ \frac{\diff \dmu{\bm V}}{\diff \refmu}(\bm X) \cdot H_i\left(  \frac{\ip{\bm X}{\bm V^{\otimes k}}}{\sqrt{\dim^k}}\right) \right] \\
    & \explain{(a)}{=} \E_{\bm V}\left[ H_i\left(  \frac{\ip{\bm X}{\bm V^{\otimes k}}}{\sqrt{\dim^k}}\right)\right] \\
    & \explain{(b)}{=} \refE \left[ H_i(\lambda + Z) \right] \\
    & \explain{(c)}{=} \frac{\lambda^i}{\sqrt{i!}}.
\end{align*}
In the above display, in the step marked (a), applied a change of measure to change the distribution of $\bm X$ to $\bm X \sim \dmu{\bm V}$. In the step marked (b), we used the fact that when $\bm X \sim \dmu{\bm V}$,
\begin{align*}
    \frac{\ip{\bm X}{\bm V^{\otimes k}}}{\sqrt{\dim^k}} \sim \lambda + Z, \quad Z \sim \gauss{0}{1}.
\end{align*}
In the step marked (c), we appealed to Fact~\ref{hermite_special_property}.
\end{proof}
\subsubsection{Proof of Lemma~\ref{lemma: integrate-hermite-orthogonality-tpca}} \label{sec:tpca-integrated-orthogonality-proof}
\begin{proof}[Proof of Lemma~\ref{lemma: integrate-hermite-orthogonality-tpca}] Using Definition~\ref{def: integrated-Hermite-tpca} and Fubini's theorem, we obtain,
\begin{align*}
      \refE[\intH{i}{\bm X}{S} \cdot \intH{j}{\bm X}{S}] = \int   \refE\left[ H_{i}\left(\frac{\ip{\bm X}{\bm V^{\otimes k}}}{\sqrt{\dim^k}}\right) H_{j}\left(\frac{\ip{\bm X}{ \widetilde{\bm V}^{\otimes k}}}{\sqrt{\dim^k}}\right) \right] \cdot S(\bm V) \cdot S(\widetilde{\bm V}) \; \prior(\diff \bm V) \; \prior(\diff \widetilde{\bm V}).
\end{align*}
Since $i \neq j$, Fact~\ref{fact: correlated-hermite} gives us, 
\begin{align*}
   \refE\left[ H_{i}\left(\frac{\ip{\bm X}{\bm V^{\otimes k}}}{\sqrt{\dim^k}}\right) H_{j}\left(\frac{\ip{\bm X}{ \widetilde{\bm V}^{\otimes k}}}{\sqrt{\dim^k}}\right) \right]  = 0.
\end{align*}
Hence, we obtain the claim of the lemma. 
\end{proof}

\subsubsection{Proof of Lemma~\ref{lemma: norm-integrated-hermite-tpca}} \label{sec:tpca-integrated-norm-proof}
\begin{proof}[Proof of Lemma~\ref{lemma: norm-integrated-hermite-tpca}]  Using Definition~\ref{def: integrated-Hermite-tpca} and Fubini's theorem, we obtain,
\begin{align*}
     \refE[\intH{i}{\bm X}{S}^2] = \int  \refE\left[ H_{i}\left( \frac{\ip{\bm X}{\bm V_1^{\otimes k}}}{\sqrt{\dim^k}} \right) H_{i}\left( \frac{\ip{\bm X}{\bm V_2^{\otimes k}}}{\sqrt{\dim^k}} \right) \right] \cdot S(\bm V_1) \cdot S(\bm V_2) \; \prior(\diff \bm V_1) \; \prior(\diff \bm V_2).
\end{align*}
Fact~\ref{fact: correlated-hermite} gives us, 
\begin{align*}
    \refE\left[ H_{i}\left( \frac{\ip{\bm X}{\bm V_1^{\otimes k}}}{\sqrt{\dim^k}} \right) H_{i}\left( \frac{\ip{\bm X}{\bm V_2^{\otimes k}}}{\sqrt{\dim^k}} \right) \right] = \left( \frac{\ip{\bm V_1}{\bm V_2}}{d} \right)^{ki}.
\end{align*}
We define $\bm V = \bm V_1 \odot \bm V_2$, where $\odot$ denotes entry-wise product of vectors and, 
\begin{align*}
    \overline{V} = \frac{1}{\dim} \sum_{i=1}^\dim V_i.
\end{align*}
Hence,
\begin{align*}
     \refE[\intH{i}{\bm X}{S}^2] &=  \int \overline{V}^{ki} \cdot S(\bm V_1) \cdot S(\bm V_2) \; \prior(\diff \bm V_1) \; \prior(\diff \bm V_2).
\end{align*}
Since $\bm V_1, \bm V_2$ are independently sampled from the prior $\pi$ and $\bm V = \bm V_1 \odot \bm V_2$, it is straight-forward to check that $\bm V_1, \bm V$ are independent, uniformly random $\{\pm 1\}^\dim$ vectors and $\bm V_2 = \bm V_1 \odot \bm V$. Hence,
\begin{align} \label{eq: intermediate-eq-integrated-hermite-norm-tpca}
     \refE[\intH{i}{\bm X}{S}^2] &=  \int \overline{V}^{ki}  \cdot S(\bm V_1) \cdot S(\bm V \odot \bm V_1) \; \prior(\diff \bm V_1) \; \prior(\diff \bm V).
\end{align}
Recall that, the collection of polynomials:
\begin{align*}
    \left\{\bm V^{\bm r} \explain{def}{=} \prod_{i=1}^\dim V_i^{r_i}: \bm r \in \{0,1\}^\dim\right\},
\end{align*}
form an orthonormal basis for functions on the Boolean hypercube $\{\pm 1\}^\dim$ with respect to the uniform distribution $\prior = \unif{\{\pm 1\}^\dim}$. Hence, we can expand $\bm S$ in this basis:
\begin{align*}
    \bm S(\bm V) & = \sum_{\bm r \in \{0,1\}^\dim} \hat{S}_{\bm r} \cdot \bm V^{\bm r}, \; \hat{S}_{\bm r} \explain{def}{=} \int S(\bm V) \cdot \bm V^{\bm r} \prior(\diff \bm V).
\end{align*}
Substituting this in \eqref{eq: intermediate-eq-integrated-hermite-norm-tpca} gives us:
\begin{align*}
    \refE[\intH{i}{\bm X}{S}^2] &=  \sum_{\bm r, \bm s \in \{0,1\}^\dim} \hat{S}_{\bm r} \hat{S}_{\bm s} \int \overline{V}^{ki}  \cdot  \bm{V}_1^{\bm r + \bm s} \cdot \bm V^{\bm s} \; \prior(\diff \bm V_1) \; \prior(\diff \bm V).
\end{align*}
Noting that, if $\bm r \neq \bm s$,
\begin{align*}
    \int   \bm{V}_1^{\bm r + \bm s} \; \prior(\diff \bm V_1) = 0,
\end{align*}
we obtain,
\begin{align*}
\refE[\intH{i}{\bm X}{S}^2] &=  \sum_{\bm r \in \{0,1\}^\dim} \hat{S}_{\bm r}^2 \int \overline{V}^{ki}   \cdot \bm V^{\bm r} \; \prior(\diff \bm V).
\end{align*}
Since $\|S\|_\prior  \leq 1$, we know that $\sum_{\bm r} \hat{S}_{\bm r}^2 \leq 1$. When $\ip{S}{1}_\prior = 0$, one additionally has $\hat{S}_{\bm 0} = 0$. Hence,
\begin{align*}
     \sup_{S: \|S\|_\pi \leq 1} \refE[\intH{i}{\bm X}{S}^2] & = \max_{\bm r \in \{0,1\}^\dim}  \int \overline{V}^{ki}   \cdot \bm V^{\bm r} \; \prior(\diff \bm V), \\
     \sup_{\substack{S: \|S\|_\pi \leq 1\\ \ip{S}{1}_\prior = 0}} \refE[\intH{i}{\bm X}{S}^2] & = \max_{\substack{\bm r \in \{0,1\}^\dim\\ \|\bm r\|_1 \geq 1}}  \int \overline{V}^{ki}   \cdot \bm V^{\bm r} \; \prior(\diff \bm V).
\end{align*}
The right hand sides of the above equations have been analyzed in Lemma~\ref{lemma: rademacher-moments}. Appealing to this result immediately yields the claims of this lemma. 
\end{proof}
\subsubsection{Proof of Lemma~\ref{lemma: integrated-hermite-hypercontractivity-tpca}} \label{sec:tpca-integrated-hypercontractitivty-proof}

\begin{proof}[Proof of Lemma~\ref{lemma: integrated-hermite-hypercontractivity-tpca}] Note that the result for $q=2$ follows from the discussion preceding this lemma. Hence we focus on proving the inequality when $q \geq 2$. Recalling Definition~\ref{def: integrated-Hermite-tpca}, we have
\begin{align*}
   \intH{i}{\bm X}{S} \explain{def}{=} \int H_{i}\left(\frac{\ip{\bm X}{\bm V^{\otimes k}}}{\sqrt{\dim^k}}\right) \cdot S(\bm V) \; \prior(\diff \bm V).
\end{align*}
Observe that for any fixed $\bm V$, the quantity
\begin{align*}
   H_{i}\left(\frac{\ip{\bm X}{\bm V^{\otimes k}}}{\sqrt{\dim^k}}\right)
\end{align*}
can be expressed as a polynomial in $\bm X$ of degree $i$ (see Fact~\ref{fact: hermite-projection-property}). Since,
\begin{align*}
    \intH{i}{\bm X}{S} \explain{def}{=} \int H_{i}\left(\frac{\ip{\bm X}{\bm V^{\otimes k}}}{\sqrt{\dim^k}}\right) \cdot S(\bm V) \; \prior(\diff \bm V),
\end{align*}
is a weighted linear combination of such polynomials, $\intH{i}{\bm X}{S}$ is also a homogeneous polynomial in $\bm X$ of degree $i$.
Hence, by the completeness of the Hermite polynomial basis, $\intH{i}{\bm X}{S}$ must have a representation of the form:
\begin{align*}
     \intH{i}{\bm X}{S}  = \sum_{\substack{\bm c \in \tensor{\W^\dim}{k} \\ \|\bm c\|_1 = i}}\beta(\bm c; S) \cdot {H_{\bm c}(\bm X)},
\end{align*}
for some coefficients $\beta(\bm c; S)$. While these coefficients can be computed, we will not need their exact formula for our discussion. Hence,
\begin{align*}
      \sum_{i=0}^\infty \alpha_{i} \cdot \intH{i}{\bm X}{S}  =  \sum_{i=0}^\infty  \sum_{\substack{\bm c \in \tensor{\W^\dim}{k} \\ \|\bm c\|_1 = i}} \alpha_{i} \cdot \beta(\bm c; S) \cdot {H_{\bm c}(\bm X)}
\end{align*}
By Gaussian Hypercontractivity (Fact~\ref{fact: hypercontractivity}),
\begin{align*}
     \left\|  \sum_{i=0}^\infty \alpha_{i} \cdot \intH{i}{\bm X}{S}\right\|_q^2 & \leq  \sum_{i=0}^\infty  \sum_{\substack{\bm c \in \tensor{\W^\dim}{k} \\ \|\bm c\|_1 = i}} \alpha_{i}^2 \cdot \beta(\bm c; S)^2 \cdot (q-1)^i \\
     & = \sum_{i=0}^\infty   \alpha_{i}^2  \cdot (q-1)^i \cdot \sum_{\substack{\bm c \in \tensor{\W^\dim}{k} \\ \|\bm c\|_1 = i}} \beta(\bm c; S)^2
\end{align*}
Observing that,
\begin{align*}
    \refE[ \intH{i}{\bm X}{S}^2] = \sum_{\substack{\bm c \in \tensor{\W^\dim}{k} \\ \|\bm c\|_1 = i}} \beta(\bm c; S)^2,
\end{align*}
yields the claim of the lemma.
\end{proof}

\section{Proofs for Asymmetric Tensor PCA} \label{sec:information-bound-atpca-proof}
\subsection{Setup}
This appendix is devoted to the proof of the information bound for distributed $k$-ATPCA (Proposition \ref{prop: info-bound-atpca}). Recall that in the distributed $k$-ATPCA problem:
\begin{enumerate}
    \item An unknown rank-$1$ signal tensor $\bm V \in \tensor{\R^\dim}{k}$ is drawn from the prior:
    \begin{align*}
         \prior \explain{def}{=} \unif{\{ \sqrt{d^k} \cdot \bm e_{i_1}\otimes \bm e_{i_2} \dotsb \otimes \bm e_{i_k}: \; i_1, i_2, \dotsc, i_k \in [\dim]\}}.
    \end{align*}
    \item A dataset consisting of $\mach$ tensors $\bm X_{1:\mach}$ is drawn i.i.d. from $\dmu{\bm V}$, where $\dmu{\bm V}$ is the distribution of a single tensor from the $k$-ATPCA problem (recall \eqref{eq: symmetric_atpca_setup}). This dataset is divided among $\mach$ machines with 1 tensor per machine.
    \item The execution of a distributed estimation protocol with parameters $(\mach, \batch = 1, \budget)$ results in a transcript $\bm Y \in \{0,1\}^{\mach \budget}$ written on the blackboard.
\end{enumerate}
We will obtain Proposition \ref{prop: info-bound-atpca} by instantiating the general information bound provided in Proposition \ref{prop: main_hellinger_bound} with the following choices:
\begin{description}
\item [Choice of $\refmu$: ] Under the reference measure, $\bm X \sim \refmu$ is a $k$-tensor with i.i.d.\ $\gauss{0}{1}$ coordinates.
\item [Choice of $\nullmu$: ] The null measure is set to $\nullmu \explain{def}=  \refmu$.
\item [Choice of $\goodevnt$: ] We choose the event $\goodevnt$ as follows:
\begin{align*}
          \goodevnt \explain{def}=  \left\{\bm X \in (\R^\dim)^{\otimes k}: \|\bm X\|_\infty \leq \lambda + \sqrt{2 (k \log(\dim) + \epsilon)}\right\}.
     \end{align*}
\end{description}
With these choices, we can write down the formula for the relevant likelihood ratio that appears Proposition~\ref{prop: main_hellinger_bound}.
 Note that if $\bm V = \sqrt{d^k} \cdot \bm e_{j_1}\otimes \bm e_{j_2} \dotsb \otimes \bm e_{j_k}$, then the likelihood ratio has the following formula:
 
 \begin{align}
     \frac{\diff \dmu{\bm V}}{\diff \refmu} (\bm X) = \exp\left( \lambda (X)_{j_{1},j_{2} \dotsc ,j_{k}} - \frac{\lambda^2}{2} \right).
 \end{align}
 Consequently, we define:
 \begin{align*}
     \frac{\diff\dmu{\lambda}}{\diff\dmu{0}}(x) \explain{def}{=} \exp\left( \lambda x - \frac{\lambda^2}{2} \right).
 \end{align*}
 Note that this is exactly the likelihood ratio between $\dmu{\lambda} = \gauss{\lambda}{1}$ and $\dmu{0} = \gauss{0}{1}$. Hence, when $\bm V = \sqrt{d^k} \cdot \bm e_{j_1}\otimes \bm e_{j_2} \dotsb \otimes \bm e_{j_k}$, we have:
 \begin{align*}
     \frac{\diff \dmu{\bm V}}{\diff \refmu} (\bm X) = \frac{\diff\dmu{\lambda}}{\diff\dmu{0}}(X_{j_{1},j_{2} \dotsc ,j_{k}}).
 \end{align*}
The proof of Proposition \ref{prop: info-bound-atpca} is presented in the following subsection.
\subsection{Proof of Proposition~\ref{prop: info-bound-atpca}}
 The proof of Proposition \ref{prop: info-bound-atpca} relies on  two intermediate results. First,  Lemma \ref{lemma: concentration_ATPCA} analyzes the concentration properties of a suitably truncated version of the Likelihood ratio. 
 
 \begin{lemma} \label{lemma: concentration_ATPCA}Let $\bm S \in ({\R^\dim})^{\otimes k}$ be an arbitrary tensor with $\|\bm S\| \leq 1$. Let $\epsilon > \lambda$ be arbitrary scalar. Let $\bm X \sim \refmu$ be Gaussian tensor with $i.i.d.$ $\gauss{0}{1}$ entries. Define the tensor $\truncate{\bm T}{\epsilon} \in ({\R^\dim})^{\otimes k}$ with entries:
 \begin{align*}
     (\truncate{T}{\epsilon})_{j_1,j_2, \dotsc ,j_k} \explain{def}{=} \left(\frac{\diff\dmu{\lambda}}{\diff\dmu{0}}(X_{j_{1},j_{2} \dotsc ,j_{k}}) - 1 \right) \cdot \Indicator{|X_{j_{1},j_{2} \dotsc ,j_{k}}| \leq \epsilon}
 \end{align*}
 Then, 
 \begin{enumerate} \item There exists a universal constant $C$ such that for any $q \in \N$,
 \begin{align*}
     (\refE |\ip{\bm S}{\truncate{\bm T}{\epsilon} - \refE \truncate{\bm T}{\epsilon}} |^q)^{\frac{2}{q}} & \leq C \cdot q \cdot \delta^2.
 \end{align*}
 In the above display $\ip{\cdot}{\cdot}$ denotes the standard inner product:
 \begin{align*}
     \ip{\bm S}{\truncate{\bm T}{\epsilon} - \refE\truncate{\bm T}{\epsilon} } = \sum_{j_1,j_2, \dotsc ,j_k \in [\dim]} (\truncate{T}{\epsilon} - \refE \truncate{T}{\epsilon})_{j_1,j_2, \dotsc ,j_k} (S)_{j_1,j_2, \dotsc ,j_k}.
 \end{align*}
 and,
 \begin{align*}
     \delta \explain{def}{=}  \max\left( e^{\lambda \epsilon - \frac{\lambda^2}{2}} -1, \lambda \epsilon + \frac{\lambda^2}{2} \right).
 \end{align*}
 \item Furthermore, 
 \begin{align*}
     \|\E \truncate{\bm T}{\epsilon}\|^2 \leq 4 \cdot  d^k \cdot  e^{-(\epsilon-\lambda)^2}.
 \end{align*}
 \end{enumerate}
 \end{lemma}
\begin{proof} The proof of this result appears at the end of this appendix (Appendix \ref{sec: proof-concentration-atpca}). 
\end{proof}
The second result required to complete the proof of Proposition \ref{prop: info-bound-atpca} is Lemma \ref{lemma: atpca_rare_event}. This result provides an estimate on $\dmu{\bm V}(\goodevnt^c)$ which appears in the information bound in Proposition \ref{prop: main_hellinger_bound}. 
 \begin{lemma} \label{lemma: atpca_rare_event} Let $\epsilon \geq 0$ be arbitrary. Let $\goodevnt$ denote the set:
 \begin{align*}
     \goodevnt = \left\{\bm X \in (\R^\dim)^{\otimes k}: \|\bm X\|_\infty \leq \lambda + \sqrt{2 (k \log(\dim) + \epsilon)}\right\}.
 \end{align*}
 Then, $\dmu{\bm V}(\goodevnt^c) \leq e^{-\epsilon}$. In the above display, $\|\bm X\|_\infty$ denotes the entry-wise $\infty$-norm:
 \begin{align*}
     \|\bm X\|_{\infty} \explain{def}{=} \max_{j_{1:k} \in [\dim]} |X_{j_1,j_2, \dotsc , j_k}|.
 \end{align*}
 \end{lemma}
 \begin{proof}
 Recall that under $\dmu{\bm V}$, we have:
 \begin{align*}
     \bm X = \frac{\lambda \bm V_1 \otimes \bm V_2 \dotsb \otimes \bm V_k}{\sqrt{\dim^k}} + \bm W, \; (W)_{j_1,j_2, \cdots j_k} \explain{i.i.d.}{\sim} \gauss{0}{1}, \; \forall \; j_1,j_2 \cdots ,j_k \in [\dim].
 \end{align*}
 The claim follows by observing $\|\bm V_i\|_{\infty} \leq \|\bm V_i\| = \sqrt{\dim}$ along with the standard tail bound for maximum of Gaussian random variables:
 \begin{align*}
     \P\left( \|\bm W\|_\infty \geq \sqrt{2 (k \log(d) + \epsilon)}\right) \leq e^{-\epsilon}.
 \end{align*}
 \end{proof}
 With these results in hand, we now present the proof of Proposition \ref{prop: info-bound-atpca}. 
 
 \begin{proof}[Proof of Proposition \ref{prop: info-bound-atpca}]
Recall that in Proposition \ref{prop: main_hellinger_bound} we showed:
\begin{align*}
    &\frac{\MIhell{\bm V}{\bm Y}}{\consthell}  \explain{}{\leq} \\&\hspace{1.5cm}\sum_{i=1}^\mach   \refE\left[ \Indicator{Z_i =1} \cdot \int \left({\refE\left[  \frac{\diff \dmu{\bm V}}{\diff \refmu} (\bm X_i) - 1 \bigg| \bm Y, Z_i, (\bm X_j)_{j \neq i}  \right]}\right)^2 \prior(\diff \bm V) \right]  +  {m \int \dmu{\bm V}(\goodevnt^c) \prior(\diff \bm V)}. \\
\end{align*}
We will choose:
\begin{align*}
    \mathcal{Z} \explain{def}{=} \left\{\bm X \in (\R^\dim)^{\otimes k}: \|\bm X\|_\infty \leq \epsilon\right\}, \; \epsilon \explain{def}{=}  \lambda + 2\sqrt{(k \ln(\dim) + \ln(\mach))}. 
\end{align*}
By Lemma \ref{lemma: atpca_rare_event}, we have, $\dmu{\bm V}(\goodevnt^c) \leq 1/m^2$. Hence,
\begin{align*}
    \frac{\MIhell{\bm V}{\bm Y}}{\consthell} & \explain{}{\leq} \sum_{i=1}^\mach   \refE\left[ \Indicator{Z_i =1} \cdot \int \left({\refE\left[  \frac{\diff \dmu{\bm V}}{\diff \refmu} (\bm X_i) - 1 \bigg| \bm Y, Z_i, (\bm X_j)_{j \neq i}  \right]}\right)^2 \prior(\diff \bm V) \right]+ \frac{1}{m}.
\end{align*}
Next we recall that we chose 
\begin{align*}
    \prior \explain{def}{=} \unif{\{ \sqrt{d^k} \cdot \bm e_{i_1}\otimes \bm e_{i_2} \dotsb \otimes \bm e_{i_k}: \; i_{1:k} \in [\dim]\}}.
\end{align*}
And when $\bm V = \sqrt{d^k} \cdot \bm e_{j_1}\otimes \bm e_{j_2} \dotsb \otimes \bm e_{j_k}$, we simplified the likelihood ratio:
\begin{align*}
    \frac{\diff \dmu{\bm V}}{\diff \refmu} (\bm X_i) - 1 & = \frac{\diff\dmu{\lambda}}{\diff \dmu{0}} ((X_i)_{j_1,j_2, \dotsc, j_k}) - 1, \; \frac{\diff\dmu{\lambda}}{\diff\dmu{0}}(x) \explain{def}{=} \exp\left( \lambda x - \frac{\lambda^2}{2} \right).
\end{align*}
We define the tensors $\bm T_i \in (\R^\dim)^{\otimes k}$ we entries:
\begin{align*}
     T_{{j_1,j_2, \dotsc, j_k}} = \frac{\diff\dmu{\lambda}}{\diff \dmu{0}} ((X_i)_{j_1,j_2, \dotsc, j_k}) - 1.
\end{align*}
With this notation, the upper bound on Hellinger Information can be written as:
\begin{align*}
    \frac{\MIhell{\bm V}{\bm Y}}{\consthell} & \explain{}{\leq} \frac{1}{\dim^k}\sum_{i=1}^\mach   \refE\left[ \Indicator{Z_i =1} \cdot \left\|{\refE\left[  \bm T_i \big| \bm Y, Z_i, (\bm X_j)_{j \neq i}  \right]}\right\|^2  \right]+ \frac{1}{m}.
\end{align*}

\paragraph{Truncation of Likelihood Ratio.}

Recall that $Z_i = \Indicator{\bm X_i \in \goodevnt}$. Consequently, we only need to analyze the likelihood ratio on $\goodevnt$. Note that on $\goodevnt$, we have, $\bm T_i = \truncate{\bm T_i}{\epsilon}$, where:
\begin{align*}
     (\truncate{T}{\epsilon})_{j_1,j_2, \dotsc ,j_k} &\explain{def}{=} \left(\frac{\diff\dmu{\lambda}}{\diff\dmu{0}}(X_{j_{1},j_{2} \dotsc ,j_{k}}) - 1 \right) \cdot \Indicator{|X_{j_{1},j_{2} \dotsc ,j_{k}}| \leq \epsilon}, \\
     \epsilon &\explain{def}{=}  \lambda + 2\sqrt{(k \ln(\dim) + \ln(\mach))}.
\end{align*}
Note that:
\begin{align*}
    \left\|{\refE\left[  \truncate{\bm T_i}{\epsilon} \big| \bm Y, Z_i, (\bm X_j)_{j \neq i}  \right]}\right\|^2 & \leq 2\left\|{\refE\left[  \truncate{\bm T_i}{\epsilon}  - \refE \truncate{\bm T_i}{\epsilon} \big| \bm Y, Z_i, (\bm X_j)_{j \neq i}  \right]}\right\|^2  + 2\|\refE \truncate{\bm T_i}{\epsilon} \|^2.
\end{align*}
Using the bound on $\|\refE \truncate{\bm T}{\epsilon} \|^2$ obtained in Lemma \ref{lemma: concentration_ATPCA}, we obtain the following bound on Hellinger information:
\begin{align*}
    \frac{\MIhell{\bm V}{\bm Y}}{\consthell} & \explain{}{\leq} \frac{2}{\dim^k}\sum_{i=1}^\mach   \refE\left[\left\|{\refE\left[  \truncate{\bm T_i}{\epsilon}  - \refE \truncate{\bm T_i}{\epsilon} \big| \bm Y, Z_i, (\bm X_j)_{j \neq i}  \right]}\right\|^2  \right] + \frac{4}{d^{4k} m^3}  +   \frac{1}{m}.
\end{align*}
\paragraph{Linearization and Geometric Inequality.} We observe that:
\begin{align*}
    \left\|{\refE\left[  \truncate{\bm T_i}{\epsilon}  - \refE \truncate{\bm T_i}{\epsilon} \big| \bm Y, Z_i, (\bm X_j)_{j \neq i}  \right]}\right\| & \leq \sup_{\bm S \in (\R^\dim)^{\otimes k}: \|\bm S\| \leq 1} {\refE\left[ \ip{\bm S}{ \truncate{\bm T_i}{\epsilon}  - \refE \truncate{\bm T_i}{\epsilon}} \big| \bm Y, Z_i, (\bm X_j)_{j \neq i}  \right]}.
\end{align*}
Using the Proposition \ref{prop: geometric inequality} along with the moment bounds in Lemma \ref{lemma: concentration_ATPCA}, we obtain:
\begin{align*}
    \left\|{\refE\left[  \truncate{\bm T_i}{\epsilon}  - \refE \truncate{\bm T_i}{\epsilon} \big| \bm Y = \bm y, Z_i = z_i, (\bm X_j)_{j \neq i}  \right]}\right\|^2  & \leq \inf_{q \geq 1} \frac{C \cdot \delta^2 \cdot q}{\refP(\bm Y = \bm y, Z_i = z_i | (\bm X_j)_{j \neq i} = (\bm x_j)_{j \neq i})^{\frac{2}{q}}} ,
\end{align*}
where
\begin{align*}
    \delta &= \max\left(\exp\left( 2\lambda\sqrt{(k \ln(\dim) + \ln(\mach))}  + \frac{\lambda^2}{2} \right) - 1 , \frac{3\lambda^2}{2} + 2 \lambda\sqrt{(k \ln(\dim) + \ln(\mach))} \right) \\
    & \leq \exp\left( 2\lambda\sqrt{(k \ln(\dim) + \ln(\mach))}  + \frac{3\lambda^2}{2} \right) - 1.
\end{align*}
We define:
\begin{align*}
    \mathcal{R}_{\mathsf{freq}}^{(i)} &\explain{def}{=} \left\{ (\bm y, \bm z_i) \in \{0,1\}^{m\budget + 1} :  \refP(\bm Y = \bm y, Z_i = z_i | (\bm X_j)_{j \neq i} = (\bm x_j)_{j \neq i}) > \frac{1}{e} \right\}.
\end{align*}
Note that $|\mathcal{R}_{\mathsf{freq}}^{(i)}| \leq e$. We set $q$ as:
\begin{align*}
    q = \begin{cases} 1 & \text{if $(\bm y, \bm z_i) \in \mathcal{R}_{\mathsf{freq}}^{(i)}$} ; \\ - 2 \ln \refP(\bm Y = \bm y, Z_i = z_i | (\bm X_j)_{j \neq i} = (\bm x_j)_{j \neq i}) & \text{if $(\bm y, \bm z_i) \notin \mathcal{R}_{\mathsf{freq}}^{(i)}$} . \end{cases}
\end{align*}
This ensures $q \geq 1$. Setting $q$ as above yields:
\begin{multline*}
     \left\|{\refE\left[  \truncate{\bm T_i}{\epsilon}  - \refE \truncate{\bm T_i}{\epsilon} \big| \bm Y = \bm y, Z_i = z_i, (\bm X_j)_{j \neq i}  \right]}\right\|^2   \\
     \leq \begin{cases} C \delta^2 & \text{if $(\bm y, \bm z_i) \in \mathcal{R}_{\mathsf{freq}}^{(i)}$} ; \\ - 2 C \delta^2 \ln \refP(\bm Y = \bm y, Z_i = z_i | (\bm X_j)_{j \neq i} = (\bm x_j)_{j \neq i}) & \text{if $(\bm y, \bm z_i) \notin \mathcal{R}_{\mathsf{freq}}^{(i)}$} . \end{cases}
\end{multline*}
Hence we obtain:
\begin{align*}
    \E\left[\left\|{\refE\left[  \truncate{\bm T_i}{\epsilon}  - \refE \truncate{\bm T_i}{\epsilon} \big| \bm Y = \bm y, Z_i = z_i, (\bm X_j)_{j \neq i}  \right]}\right\|^2  \right] & \leq  C |\mathcal{R}_{\mathsf{freq}}^{(i)}| \delta^2 + C\delta^2\mathsf{H}(\bm Y, Z_i | (\bm X_j)_{j \neq i}). 
\end{align*}
In the above display $\mathsf{H}(\bm Y, Z_i | (\bm X_j)_{j \neq i})$ denotes the conditional entropy of the $(\bm Y, Z_i)$ given $(\bm X_j)_{j \neq i}$. Since we assumed that the communication protocol used to generate $\bm Y$ is deterministic, conditioning on $(\bm X_j)_{j \neq i}$ determines all but $\budget +1$ bits of the vector $(\bm Y, Z_i)$. Hence $\mathsf{H}(\bm Y, Z_i | (\bm X_j)_{j \neq i}) \leq C(\budget + 1)$. This gives us,
\begin{align*}
     \E\left[\left\|{\refE\left[  \truncate{\bm T_i}{\epsilon}  - \refE \truncate{\bm T_i}{\epsilon} \big| \bm Y = \bm y, Z_i = z_i, (\bm X_j)_{j \neq i}  \right]}\right\|^2  \right]  & \leq C \delta^2 \budget,
\end{align*}
which in turn yields:
\begin{align*}
    \MIhell{\bm V}{\bm Y} & \explain{}{\leq} {\consthell} \left( \frac{2}{\dim^k}\sum_{i=1}^\mach   \refE\left[\left\|{\refE\left[  \truncate{\bm T_i}{\epsilon}  - \refE \truncate{\bm T_i}{\epsilon} \big| \bm Y, Z_i, (\bm X_j)_{j \neq i}  \right]}\right\|^2  \right] + \frac{4}{d^{4k} m^3}  +   \frac{1}{m} \right) \\
    & \leq C \left(  \frac{\delta^2 \mach\budget }{\dim^k} + \frac{1}{d^{4k} m^3}  +   \frac{1}{m} \right) \\
    & \leq C \left(  \frac{\delta^2 \mach\budget }{\dim^k}  +   \frac{1}{m} \right). 
\end{align*}
Finally we observe that in the scaling regime:  $\lambda = \Theta(1), \; m = \Theta(\dim^{\eta}), \; b = \Theta(\dim^\beta)$, for some constants $\eta \geq 1, \beta \geq 0$, which satisfy $\eta+ \beta <k$, the above upper bound on $\MIhell{\bm V}{\bm Y} \rightarrow 0$ as $\dim \rightarrow \infty$.
\end{proof}

\subsection{Concentration of Likelihood Ratio} \label{sec: proof-concentration-atpca}
 In this section, we provide a proof for Lemma \ref{lemma: concentration_ATPCA}. 
 \begin{proof}[Proof of Lemma \ref{lemma: concentration_ATPCA}]
 We prove the two parts separately.
 \begin{enumerate}
 \item Recall that,
 \begin{align*}
     \frac{\diff\dmu{\lambda}}{\diff\dmu{0}}(x) \explain{def}{=} \exp\left( \lambda x - \frac{\lambda^2}{2} \right).
 \end{align*}
 Observe that,
 \begin{align*}
      (\truncate{T}{\epsilon})_{j_1,j_2, \dotsc ,j_k} \explain{def}{=} \left(\frac{\diff\dmu{\lambda}}{\diff\dmu{0}}(X_{j_{1},j_{2} \dotsc ,j_{k}}) - 1 \right) \cdot \Indicator{|X_{j_{1},j_{2} \dotsc ,j_{k}}| \leq \epsilon} \leq e^{\lambda \epsilon - \frac{\lambda^2}{2}} -1. 
 \end{align*}
 Furthermore,
 \begin{align*}
     (\truncate{T}{\epsilon})_{j_1,j_2, \dotsc ,j_k} \geq \left(  e^{-\lambda \epsilon - \frac{\lambda^2}{2}} -1\right) \cdot \Indicator{|X_{j_{1},j_{2} \dotsc ,j_{k}}| \leq \epsilon} \geq - \lambda \epsilon - \frac{\lambda^2}{2}. 
 \end{align*}
 Hence,
 \begin{align*}
     \left|(\truncate{T}{\epsilon})_{j_1,j_2, \dotsc ,j_k} \right| & \leq \max\left( e^{\lambda \epsilon - \frac{\lambda^2}{2}} -1, \lambda \epsilon + \frac{\lambda^2}{2} \right) \explain{def}{=} \delta. 
 \end{align*}
 Hence, $(\truncate{T}{\epsilon})_{j_1,j_2, \dotsc ,j_k}$ are independent random variables in $[-\delta,\delta]$. Consequently, by Hoeffding's Inequality $\ip{\bm S}{\truncate{\bm T}{\epsilon} - \refE\truncate{\bm T}{\epsilon} }$ is sub-Gaussian with variance proxy $\delta^2$. The claim follows by standard estimates on the moments of sub-Gaussian random variables.
 \item Since the entries of $\truncate{\bm T}{\epsilon}$ are identically distributed, we have:
 \begin{align*}
      \|\E \truncate{\bm T}{\epsilon}\|^2  = d^k (\E (\truncate{T}{\epsilon})_{1,1,1, \dotsc ,1})^2.
 \end{align*}
 Recall that:
 \begin{align*}
     (\truncate{T}{\epsilon})_{1,1,1, \dotsc ,1} \explain{d}{=} \left( \frac{\diff \dmu{\lambda}}{\diff \dmu{0}}(Z) - 1\right) \cdot \Indicator{|Z|\leq \epsilon}, \; Z \sim \gauss{0}{1}.
 \end{align*}
 Hence,
 \begin{align*}
      |\E (\truncate{T}{\epsilon})_{1,1,1, \dotsc ,1}| & = \left|\E \left[ \left( \frac{\diff \dmu{\lambda}}{\diff \dmu{0}}(Z) - 1\right) \cdot \Indicator{|Z|\leq \epsilon} \right] \right| \\
      & = \left|\P(|Z| \leq \epsilon) - \P(|Z+\lambda| \leq \epsilon) \right| \\
      & =  \left| \P(|Z + \lambda| \geq \epsilon) - \P(|Z| \geq \epsilon) \right| \\
      & \leq \max(\P(|Z + \lambda| \geq \epsilon), \P(|Z| \geq \epsilon)) \\
      & \leq \P(|Z| \geq \epsilon - \lambda) \\
      & \leq 2 e^{-\frac{(\epsilon-\lambda)^2}{2}}. 
 \end{align*}
 Hence, 
 \begin{align*}
        \|\E \truncate{\bm T}{\epsilon}\|^2  \leq 4 \cdot \dim^k \cdot e^{-(\epsilon-\lambda)^2}.
 \end{align*}
 \end{enumerate}
 \end{proof}

\section{Proofs for Non-Gaussian Component Analysis}
\label{appendix:ngca}
\subsection{Setup}
This appendix is devoted to the proof Proposition~\ref{prop:ngca-info-bound}, the information bound for the distributed NGCA problem. Recall that in the distributed $k$-NGCA problem:
\begin{enumerate}
    \item An unknown $\dim$ dimensional parameter $\bm V \sim \prior$ is drawn from the prior $\prior = \unif{\{\pm 1\}^\dim}$.
    \item A dataset consisting of $\ssize = \mach \batch$ samples is drawn i.i.d. from $\dmu{\bm V}$, where $\dmu{\bm V}$ is the distribution of a single sample from the Non-Gaussian Component Analysis problem (recall \eqref{eq: ngca-model}). This dataset is divided among $\mach$ machines with $\batch$ samples per machine. We denote the dataset in one machine by $\bm X_i \in \R^{\dim \times \batch}$, where,
\begin{align*}
    \bm X_i = \begin{bmatrix} \bm x_{i1} & \bm x_{i2} &\dots &\bm x_{i\batch} \end{bmatrix}.
\end{align*}
Since $\bm x_{ij} \explain{i.i.d.}{\sim} \dmu{\bm V}$, $\bm X_i \explain{i.i.d.}{\sim} \dmu{\bm V}^{\otimes \batch}$. 
    \item The execution of a distributed estimation protocol with parameters $(\mach, \batch, \budget)$ results in a transcript $\bm Y \in \{0,1\}^{\mach \budget}$ written on the blackboard.
\end{enumerate}
The information bound stated in Proposition~\ref{prop:ngca-info-bound} is obtained using the general information bound given in Proposition~\ref{prop: main_hellinger_bound} with the following choices: 
\begin{description}
\item [Choice of $\refmu$: ] Under the measure $\refmu$, $\bm x_{ij} \explain{i.i.d.}{\sim} \gauss{\bm 0}{\bm I_\dim}$ for any $i \in [\mach], \; j \in [\batch]$.
\item [Choice of $\nullmu$: ] Under the measure $\nullmu$, the dataset of each machine is sampled i.i.d.\ from:
\begin{align*}
    \refmu(\cdot) & = \int \dmu{\bm V}^{\otimes \batch} (\cdot) \; \prior(\diff \bm V).
\end{align*}
Note that the data across machines is independent, but the $\batch$ samples within a machine are dependent since the were sampled from the same $\dmu{\bm V}$.
\item [Choice of $\goodevnt$: ] We choose the event $\goodevnt$ as follows:
\begin{subequations}  \label{eq: good-evnt-ngca}
\begin{align}
    \goodevnt &\explain{def}{=} \goodevnt_1 \cap \goodevnt_2, \\
    \goodevnt_1 &\explain{def}{=} \{ \bm x_{1:\batch} \in \R^\dim :  \|\bm x_i\| \leq \sqrt{2\dim} \; \forall \; i \; \in \; [\batch]\}, \\
    \goodevnt_2 & \explain{def}{=} \left\{ \bm x_{1:\batch} \in \R^\dim:  \left|\frac{\diff \nullmu}{\diff \refmu}(\bm x_{1: \batch}) - 1 \right|  \leq \frac{1}{2} \right\}.
\end{align}
\end{subequations}
\end{description}

This appendix is organized into subsections as follows.
\begin{enumerate}
    \item To prove Proposition~\ref{prop:ngca-info-bound}, we rely on certain analytic properties of the likelihood ratio for this problem (similar to $k$-TPCA). These properties are stated (without proofs) in Appendix~\ref{sec: ngca-harmonic}. 
    \item Using these properties, Proposition~\ref{prop:ngca-info-bound} is proved in Appendix~\ref{appendix:ngca-info-bound-proof}.
    \item Finally, the proofs of the analytic properties of the likelihood ratio appear in Appendix~\ref{appendix:ngca-harmonic-proofs}. 
\end{enumerate}

\subsection{The Likelihood Ratio for Non-Gaussian Component Analysis} \label{sec: ngca-harmonic}

In this section, we collect some important properties of the likelihood ratio for the Non-Gaussian Component Analysis problem without proofs.
The proofs of these properties are provided in Appendix~\ref{appendix:ngca-harmonic-proofs}.

We recall from \eqref{eq: ngca-likelihood} that the likelihood ratio is given by:
\begin{align*}
    \frac{\diff \dmu{\bm V}}{\diff \refmu}(\bm x_{1:\ssize}) = \prod_{i=1}^\ssize \frac{\diff \nongauss}{\diff \refmu}(\eta_i) , \;\eta_i = \frac{\ip{\bm x_i}{\bm V}}{\sqrt{\dim}}.
\end{align*}

Next we compute the Hermite decomposition of the $\ssize$-sample likelihood ratio. Due to the product structure of the $\ssize$-sample likelihood ratio, it is sufficient to compute the Hermite decomposition of the one-sample likelihood ratio. We have

\begin{align*}
    \frac{\diff \nongauss}{\diff \refmu}(\eta) = \sum_{t = 0}^\infty\refE \left[ \frac{\diff\nongauss}{\diff \refmu}(Z) H_t(Z) \right] \cdot H_t(\eta) = \sum_{t=0}^\infty \hat{\nongauss}_t H_t(\eta).
\end{align*}
In the last step, we defined,
\begin{align*}
    \hat{\nongauss}_t \explain{def}{=} \refE \left[ \frac{\diff\nongauss}{\diff \refmu}(Z) H_t(Z) \right] = \E_{\eta \sim \nongauss}[H_t(\eta)].
\end{align*}

Hence, we obtain the expression for the Hermite decomposition of the likelihood ratio summarized in the following lemma. 

\begin{lemma}[Hermite Decomposition for Non-Gaussian Component Analysis] \label{lemma: hermite-decomp-ngca}
  We have
\begin{align*}
      \frac{\diff \dmu{\bm V}}{\diff \refmu}(\bm x_{1:\ssize}) = \sum_{\bm t \in \W^\ssize} \hat{\nongauss}_{\bm t} \cdot H_{\bm t}\left(\frac{\ip{\bm x_1}{\bm V}}{\sqrt{\dim}},\frac{\ip{\bm x_2}{\bm V}}{\sqrt{\dim}},\dotsc, \frac{\ip{\bm x_\ssize}{\bm V}}{\sqrt{\dim}} \right),
\end{align*}
where, for any $\bm t \in \W^\ssize$,
\begin{align*}
    H_{\bm t}\left(\frac{\ip{\bm x_1}{\bm V}}{\sqrt{\dim}},\frac{\ip{\bm x_2}{\bm V}}{\sqrt{\dim}},\dotsc, \frac{\ip{\bm x_\ssize}{\bm V}}{\sqrt{\dim}} \right) &\explain{def}{=} \prod_{i=1}^\ssize H_{t_i} \left( \frac{\ip{\bm x_i}{\bm V}}{\sqrt{\dim}} \right),\\
    \hat{\nongauss}_{\bm t} &\explain{def}{=} \prod_{i=1}^\ssize \hat{\nongauss}_{t_i}.
\end{align*}
In the above display, $\hat{\nongauss}_t, \; t \; \in \; \N$ are the Hermite coefficients of the one-sample likelihood ratio: 
\begin{align*}
    \hat{\nongauss}_t = \E_{Z \sim \refmu} \left[ \frac{\diff \nongauss}{\diff \refmu}(Z) \cdot H_t(Z) \right] = \E_{\eta \sim \nongauss}\left[ H_t(\eta)\right].
\end{align*}
\end{lemma}

We introduce the following definition. 

\begin{definition}[Integrated Hermite Polynomials] \label{def: integrated-Hermite} Let $S: \{\pm 1\}^\dim \rightarrow \R$ be a function with $\|S\|_\pi = 1$. For any $\bm t \in \W^\ssize$, we define the $\ssize$-sample integrated Hermite polynomials as:
\begin{align*}
    \intH{\bm t}{\bm x_{1:\ssize}}{S} \explain{def}{=} \int \left( \prod_{i=1}^\ssize H_{t_i} \left( \frac{\ip{\bm x_i}{\bm V}}{\sqrt{\dim}} \right) \right) \cdot S(\bm V) \; \prior(\diff \bm V).
\end{align*}
\end{definition}

Our rationale for introducing this definition is that proving the low-degree and communication lower bounds require understanding the following quantities derived from the likelihood ratio:

\begin{align*}
    \frac{\diff \nullmu}{\diff\refmu}(\bm x_{1:\ssize}) &\explain{def}{=} \int \frac{\diff \dmu{\bm V}}{\diff \refmu}(\bm x_{1:\ssize})  \; \pi(\diff \bm V), \\
    \ip{\frac{\diff \dmu{\bm V}}{\diff\refmu}(\bm x_{1:\ssize})}{S}_\prior &\explain{def}{=} \int \frac{\diff \dmu{\bm V}}{\diff \refmu}(\bm x_{1:\ssize}) \cdot S(\bm V)  \; \pi(\diff \bm V).
\end{align*}

Using Lemma~\ref{lemma: hermite-decomp-ngca}, these quantities are naturally expressed in terms of the integrated Hermite polynomials:

\begin{align*}
    \frac{\diff \nullmu}{\diff\refmu}(\bm x_{1:\ssize}) &\explain{}{=} \sum_{\bm t \in \W^\ssize} \hat{\nongauss}_{\bm t}  \cdot \intH{\bm t}{\bm x_{1:\ssize}}{1}, \\
     \ip{\frac{\diff \dmu{\bm V}}{\diff\refmu}(\bm x_{1:\ssize})}{S}_\prior &\explain{}{=} \sum_{\bm t \in \W^\ssize} \hat{\nongauss}_{\bm t}  \cdot \intH{\bm t}{\bm x_{1:\ssize}}{S}.
\end{align*}

The following lemma shows that the integrated Hermite polynomials inherit the orthogonality property of the standard Hermite polynomials.

\begin{lemma} \label{lemma: integrate-hermite-orthogonality} For any $\bm s, \bm t \in \W^\ssize$ such that $\bm s \neq \bm t$, we have
\begin{align*}
    \refE[\intH{\bm s}{\bm x_{1:\ssize}}{S} \cdot \intH{\bm t}{\bm x_{1:\ssize}}{S}] = 0.
\end{align*}
\end{lemma}
\begin{proof} Using Definition~\ref{def: integrated-Hermite} and Fubini's theorem, we obtain,
\begin{align*}
     &\refE[\intH{\bm s}{\bm x_{1:\ssize}}{S} \cdot \intH{\bm t}{\bm x_{1:\ssize}}{S}] = \\&\hspace{3cm}\int\int  \prod_{i=1}^\ssize \refE\left[ H_{s_i}\left( \frac{\ip{\bm x_i}{\bm V}}{\sqrt{\dim}} \right) H_{t_i}\left( \frac{\ip{\bm x_i}{\bm V^\prime}}{\sqrt{\dim}} \right) \right] \cdot S(\bm V) \cdot S(\bm V^\prime) \; \prior(\diff \bm V) \; \prior(\diff \bm V^\prime).
\end{align*}
Since $\bm s \neq \bm t$, there must be an $i \in \N$ such that $s_i \neq t_i$, and for this $i$, Fact~\ref{fact: correlated-hermite} gives us, 
\begin{align*}
    \refE\left[ H_{s_i}\left( \frac{\ip{\bm x_i}{\bm V}}{\sqrt{\dim}} \right) H_{t_i}\left( \frac{\ip{\bm x_i}{\bm V^\prime}}{\sqrt{\dim}} \right) \right] = 0.
\end{align*}
Hence, we obtain the claim of the lemma. 
\end{proof}

Though the integrated Hermite polynomials are orthogonal, they do not have unit norm.  In general, the norm of these polynomials depends on the choice of the function $S$ in Definition~\ref{def: integrated-Hermite}. The following lemma provides worst-case bounds on the norm of the integrated Hermite polynomials. Note that these worst-case bounds can be much smaller than $1$.

\begin{lemma} \label{lemma: norm-integrated-hermite} There is a universal constant $C$ (independent of $\dim$) such that, for any $\bm t \in \W^\ssize$ with $\|\bm t\|_1 = t$, we have
\begin{enumerate}
    \item When $t$ is odd, $ \intH{\bm t}{\bm x_{1:\ssize}}{1} = 0$.
    \item When $t$ is even, $\refE[ \intH{\bm t}{\bm x_{1:\ssize}}{1}^2] \leq (Ct)^{\frac{t}{2}} \cdot  d^{-\frac{t}{2}}$.
    \item For even $t \leq \dim$, $\refE[ \intH{\bm t}{\bm x_{1:\ssize}}{1}^2] \geq (t/C)^{\frac{t}{2}} \cdot d^{-\frac{t}{2}}$. 
    \item For any $S: \{\pm 1 \}^\dim \rightarrow \R$ with $\|S\|_\pi \leq 1$, we have $\refE[ \intH{\bm t}{\bm x_{1:\ssize}}{S}^2] \leq (Ct)^{\frac{t}{2}} \cdot  d^{-\lceil\frac{t}{2}\rceil}$. 
    \item For any $S: \{\pm 1 \}^\dim \rightarrow \R$ with $\|S\|_\pi \leq 1, \; \ip{S}{1}_\prior = 0$, we have $\refE[ \intH{\bm t}{\bm x_{1:\ssize}}{S}^2] \leq (Ct)^{\frac{t}{2}} \cdot  d^{-\lceil\frac{t+1}{2}\rceil}$.
\end{enumerate}
\end{lemma}
\begin{proof}
See Appendix~\ref{sec: ngca-harmonic-integrated-lowdegree-proof}.
\end{proof}
A limitation of the bound obtained in Lemma~\ref{lemma: norm-integrated-hermite} is that it is vacuous when $\|\bm t\|_1 \gg \dim$. The following lemma provides a bound on the norm of integrated Hermite polynomials with degree $\|\bm t\|_1 \gg \dim$.

\begin{lemma} \label{lemma: norm-integrated-hermite-highdegree} For any $\bm t\in \W^\ssize$ with $\|\bm t\|_1 = t$, we have
\begin{align*}
    \sup_{S: \|S\|_\pi \leq 1} \refE[ \intH{\bm t}{\bm x_{1:\ssize}}{S}^2] & \leq 2\exp\left( -  \frac{(1-e^{-2t/\dim})}{2} \cdot \dim  \right).
\end{align*}
\end{lemma}
\begin{proof}
See Appendix~\ref{sec: ngca-harmonic-integrated-hypercontractivity-proof}.
\end{proof}
As a consequence of the orthogonality property of integrated Hermite polynomials (Lemma~\ref{lemma: integrate-hermite-orthogonality}) and the estimates obtained in Lemma~\ref{lemma: norm-integrated-hermite} and Lemma~\ref{lemma: norm-integrated-hermite-highdegree}, one can easily estimate the second moment of functions constructed by linear combinations of the integrated Hermite polynomials:
\begin{align*}
    \left\| \sum_{\bm t \in \W^\ssize} \alpha_{\bm t} \intH{\bm t}{\bm x_{1:\ssize}}{S} \right\|_2^2 \explain{def}{=} \refE \left( \sum_{\bm t \in \W^\ssize} \alpha_{\bm t} \cdot  \intH{\bm t}{\bm x_{1:\ssize}}{S} \right)^2 = \sum_{\bm t \in \W^\ssize} \alpha_{\bm t}^2 \cdot \refE[\intH{\bm t}{\bm x_{1:\ssize}}{S}^2].
\end{align*}

In our analysis, we will also find it useful to estimate the $q$-norms of linear combinations of integrated Hermite polynomials for $q \geq 2$:

\begin{align*}
    \left\| \sum_{\bm t \in \W^\ssize} \alpha_{\bm t} \intH{\bm t}{\bm x_{1:\ssize}}{S} \right\|_q^q \explain{def}{=} \refE \left| \sum_{\bm t \in \W^\ssize} \alpha_{\bm t} \cdot  \intH{\bm t}{\bm x_{1:\ssize}}{S} \right|^q.
\end{align*}

The following lemma uses Gaussian Hypercontractivity (Fact~\ref{fact: hypercontractivity}) to provide an estimate for the above quantity.

\begin{lemma}\label{lemma: integrated-hermite-hypercontractivity} Let $\{\alpha_{\bm t} : \bm t \in \W^\ssize\}$ be an arbitrary collection of real-valued coefficients. For any $q \geq 2$, we have
\begin{align*}
     \left\| \sum_{\bm t \in \W^\ssize} \alpha_{\bm t} \intH{\bm t}{\bm x_{1:\ssize}}{S} \right\|_q^2 & \leq \sum_{\bm t \in \W^\ssize} (q-1)^{\|\bm t\|_1} \cdot \alpha_{\bm t}^2 \cdot \refE[\intH{\bm t}{\bm x_{1:\ssize}}{S}^2]
\end{align*}
Furthermore, the inequality holds as an equality when $q = 2$.
\end{lemma}
\begin{proof}
 See Appendix~\ref{sec: ngca-harmonic-integrated-hypercontractivity-proof}.
\end{proof}

In the following section, we present a proof of the information bound for distributed Non-Gaussian Component Analysis (Proposition~\ref{prop:ngca-info-bound}) using the results of this section.

\subsection{Proof of Proposition~\ref{prop:ngca-info-bound}}
\label{appendix:ngca-info-bound-proof}
This section provides a proof for Proposition~\ref{prop:ngca-info-bound}, the main information bound for the Non-Gaussian Component Analysis problem. Recall that the information bound of Proposition~\ref{prop: main_hellinger_bound} is:
\begin{align} \label{eq: hellinger-recall-ngca}
     &\frac{\MIhell{\bm V}{\bm Y}}{\consthell}  \explain{}{\leq} \nonumber \\&\hspace{1.2cm} \sum_{i=1}^m   \looE{0}{i}\left[  \int \left({\looE{0}{i}\left[ \left( \frac{\diff \dmu{\bm V}}{\diff \refmu} (\bm X_i) - \frac{\diff \nullmu}{\diff \refmu} (\bm X_i) \right) \cdot \Indicator{Z_i =1}  \bigg| \bm Y, Z_i, (\bm X_j)_{j \neq i}  \right]}\right)^2 \prior(\diff \bm V) \right] + m \cdot \nullmu(\goodevnt^c),
\end{align}
where $Z_i = \Indicator{\bm X_i \in \goodevnt}$. The following lemma analyzes the failure probability $\nullmu(\goodevnt^c)$.

\begin{lemma} \label{lemma: ngca-bad-event} Suppose that $\nongauss$ satisfies the Moment Matching Assumption (Assumption~\ref{ass: moment-matching}) with constant $k\geq 2$ and the Bounded Signal Strength Assumption (Assumption~\ref{ass: bounded-snr}) with constants $(\lambda,K)$. Then, for any $q \geq 2$, there is exists a finite constant $C_{q,k,K}$ depending only on $(q,k,K)$ such that, if,
\begin{align*}
    \batch \lambda^2 \leq \frac{\dim}{C_{q,k,K}},
\end{align*}
then,
\begin{align*}
    \nullmu(\goodevnt^c) & \leq C_{q,k,K} \cdot \left( (1+\lambda^2) \cdot \batch \cdot  e^{-\frac{\dim}{C_{q,k,K}}} +  \left(\frac{\batch \lambda^2}{\dim} \right)^{\frac{q}{2}} \right).
\end{align*}
\end{lemma}
\begin{proof}The proof of this result appears at the end of this section (Appendix~\ref{sec: proof-ngca-bad-event}).
\end{proof}
We also need to analyze:
\begin{align*}
     \looE{0}{i}\left[\int \left({\looE{0}{i}\left[ \left( \frac{\diff \dmu{\bm V}}{\diff \refmu} (\bm X_i) - \frac{\diff \nullmu}{\diff \refmu} (\bm X_i) \right) \cdot \Indicator{Z_i =1}  \bigg| \bm Y, Z_i, (\bm X_j)_{j \neq i}  \right]}\right)^2 \prior(\diff \bm V) \right],
\end{align*}
For any $\bm X \in \R^{\dim \times \batch}$, $\bm X = [\bm x_1 \; \bm x_2 \; \cdots \; \bm x_{\batch}]$, $S \subset [\batch]$, we introduce the notation,
\begin{align*}
    \dscore{\bm V}(\bm X_S) &\explain{def}{=} \prod_{i \in S}  \left( \frac{\diff \dmu{\bm V}}{\diff \refmu} (\bm x_i) - 1\right), \\ \; 
    \barscore(\bm X_S) & \explain{def}{=} \int \dscore{\bm V}(\bm X_S) \; \prior (\diff \bm V).
\end{align*}
In the special case when $S = \{i\}$, we will use the simplified notation $\dscore{\bm V}(\bm x_i), \; \barscore(\bm x_i)$.
We consider the following decomposition: For any $\bm X \in \R^{\dim \times \batch}$, $\bm X = [\bm x_1 \; \bm x_2 \; \cdots \; \bm x_{\batch}]$,
\begin{align*}
    \frac{\diff \dmu{\bm V}}{\diff \refmu} (\bm X) - \frac{\diff \nullmu}{\diff \refmu} (\bm X) & = \prod_{\ell=1}^\batch \left(  1 + \frac{\diff \dmu{\bm V}}{\diff \refmu} (\bm x_\ell) - 1\right)  - \int \;\prod_{\ell=1}^\batch \left(  1 + \frac{\diff \dmu{\bm V}}{\diff \refmu} (\bm x_\ell) - 1\right) \; \prior(\diff \bm V) \\
    & = \underbrace{\sum_{\ell=1}^\batch (\dscore{\bm V}(\bm x_\ell) - \barscore(\bm x_\ell))}_{\text{Additive Term}} +  \underbrace{\sum_{\substack{S \subset [\batch], \; |S| \geq 2}}  (\dscore{\bm V}(\bm X_S)  - \barscore(\bm X_S))}_{\text{Non Additive Term}}.
\end{align*}
With this decomposition, using the elementary inequality $(a+b)^2 \leq 2 a^2 + 2 b^2$, we obtain,
\begin{align} \label{eq: conditional-expectation-decomposition}
     \looE{0}{i}\left[\int \left({\looE{0}{i}\left[ \left( \frac{\diff \dmu{\bm V}}{\diff \refmu} (\bm X_i) - \frac{\diff \nullmu}{\diff \refmu} (\bm X_i) \right) \cdot \Indicator{Z_i =1}  \bigg| \bm Y, Z_i, (\bm X_j)_{j \neq i}  \right]}\right)^2 \prior(\diff \bm V) \right] & \leq 2 \cdot (\mathsf{I}) + 2 \cdot (\mathsf{II}), 
\end{align}
where,
\begin{align*}
    \mathsf{I} &\explain{def}{=}  \looE{0}{i}\left[\int \left({\looE{0}{i}\left[ \left( \sum_{\ell=1}^\batch (\dscore{\bm V}(\bm x_{i\ell}) - \barscore(\bm x_{i\ell}))  \right) \cdot \Indicator{Z_i =1}   \bigg| \bm Y, Z_i, (\bm X_j)_{j \neq i}  \right]}\right)^2 \prior(\diff \bm V) \right], \\
    \mathsf{II} &\explain{def}{=} \looE{0}{i}\left[\int \left({\looE{0}{i}\left[ \bigg| \sum_{\substack{S \subset [\batch], \; |S| \geq 2}}  (\dscore{\bm V}( (\bm X_i)_S)  - \barscore( (\bm X_i)_S)) \bigg|  \bigg| \bm Y, Z_i, (\bm X_j)_{j \neq i}  \right]}\right)^2 \prior(\diff \bm V) \right].
\end{align*}

In order to control the term $(\mathsf{II})$, we apply Jensen's Inequality:
\begin{align*}
     \mathsf{II} &\explain{}{\leq} \int \refE\left[ \bigg| \sum_{\substack{S \subset [\batch], \; |S| \geq 2}}  (\dscore{\bm V}( (\bm X_i)_S)  - \barscore( (\bm X_i)_S)) \bigg|^2   \right] \;  \prior(\diff \bm V) \\
     & =  \int \refE\left[ \bigg| \sum_{\substack{S \subset [\batch], \; |S| \geq 2}}  \dscore{\bm V}( (\bm X_i)_S)  \bigg|^2   \right]  \;  \prior(\diff \bm V) -  \refE\left[ \bigg| \sum_{\substack{S \subset [\batch], \; |S| \geq 2}}  \barscore( (\bm X_i)_S)  \bigg|^2   \right] \\
     & \leq  \int \refE\left[ \bigg| \sum_{\substack{S \subset [\batch], \; |S| \geq 2}}  \dscore{\bm V}( (\bm X_i)_S)  \bigg|^2   \right]  \;  \prior(\diff \bm V).
\end{align*}
The following lemma analyzes the above upper bound on $(\mathsf{II})$.
\begin{lemma} \label{lemma: ngca-nonadditive} Suppose that $\nongauss$ satisfies the Bounded Signal Strength Assumption (Assumption~\ref{ass: bounded-snr}) with constants $(\lambda,K)$. Suppose that $K^2\batch \lambda^2 \leq 1/2$. Let $\bm X = [\bm x_1 \; \bm x_2 \; \dots \; \bm x_\batch]$ where $\bm x_i \explain{i.i.d.}{\sim} \gauss{\bm 0}{\bm I_\dim}$.  Then,
\begin{align*}
    \refE\left[ \bigg| \sum_{\substack{S \subset [\batch], \; |S| \geq 2}}  \dscore{\bm V}( (\bm X)_S)  \bigg|^2   \right] & \leq 2 \cdot (K^2  \batch\lambda^2 )^2. 
\end{align*}
\end{lemma}
\begin{proof}
The proof of this result appears at the end of this section (Appendix~\ref{sec:ngca-nonadditive}). 
\end{proof}

In order to control the term $(\mathsf{I})$, we will rewrite it as follows:
\begin{align*}
    \mathsf{I} & \explain{def}{=} \looE{0}{i}\left[\int \left({\looE{0}{i}\left[ \left( \sum_{\ell=1}^\batch (\dscore{\bm V}(\bm x_{i\ell}) - \barscore(\bm x_{i\ell})) \right) \cdot \Indicator{Z_i =1}   \bigg| \bm Y, Z_i, (\bm X_j)_{j \neq i}  \right]}\right)^2 \prior(\diff \bm V) \right] \\
    & \explain{(a)}{=} \looE{0}{i}\left[\int \left({\looE{0}{i}\left[ \left( \sum_{\ell=1}^\batch (\dscore{\bm V}(\bm x_{i\ell}) - \barscore(\bm x_{i\ell})) \cdot \Indicator{\|\bm x_{i\ell}\|\leq \sqrt{2\dim}} \right) \cdot \Indicator{Z_i =1}   \bigg| \bm Y, Z_i, (\bm X_j)_{j \neq i}  \right]}\right)^2 \prior(\diff \bm V) \right] \\
    & \explain{(b)}{=} \looE{0}{i}\left[ \Indicator{Z_i =1} \cdot \int  \left({\looE{0}{i}\left[ \sum_{\ell=1}^\batch (\dscore{\bm V}(\bm x_{i\ell}) - \barscore(\bm x_{i\ell})) \cdot \Indicator{\|\bm x_{i\ell}\|\leq \sqrt{2\dim}}    \bigg| \bm Y, Z_i, (\bm X_j)_{j \neq i}  \right] }\right)^2 \prior(\diff \bm V) \right].
\end{align*}
In the step marked (a), we used the identity $\Indicator{Z_i =1} = \Indicator{\bm X_i \in \goodevnt} = \Indicator{\bm X_i \in \goodevnt} \cdot \Indicator{\|\bm x_{i\ell}\|\leq \sqrt{2\dim}}$ (cf. \eqref{eq: good-evnt-ngca}). In the step marked (b), we observed that $\Indicator{Z_i =1}$ is measurable with respect to the conditioning $\sigma$-algebra. Next, we linearize the integral with respect to the prior $\prior$ (Lemma~\ref{lemma: linearization}):
\begin{align*}
    &\int  \left({\looE{0}{i}\left[  \sum_{\ell=1}^\batch (\dscore{\bm V}(\bm x_{i\ell}) - \barscore(\bm x_{i\ell})) \cdot \Indicator{\|\bm x_{i\ell}\|\leq \sqrt{2\dim}}   \bigg| \bm Y, Z_i, (\bm X_j)_{j \neq i}  \right] }\right)^2 \prior(\diff \bm V) \\& = \sup_{S: \|S\|_\prior \leq 1} \left( \looE{0}{i}\left[  \sum_{\ell=1}^\batch \ip{(\dscore{\bm V}(\bm x_{i\ell}) - \barscore(\bm x_{i\ell})) \cdot \Indicator{\|\bm x_{i\ell}\|\leq \sqrt{2\dim}}}{S}_\prior   \bigg| \bm Y, Z_i, (\bm X_j)_{j \neq i}  \right] \right)^2.
\end{align*}
We will apply the Geometric Inequality framework (Proposition~\ref{prop: geometric inequality}) to control the above conditional expectation. In order to do so, we need to understand the concentration behavior of the random variable:
\begin{align*}
    \sum_{\ell=1}^\batch \ip{(\dscore{\bm V}(\bm x_{i\ell}) - \barscore(\bm x_{i\ell})) \cdot \Indicator{\|\bm x_{i\ell}\|\leq \sqrt{2\dim}}}{S}_\prior .
\end{align*}
This is the subject of the following lemma.

\begin{lemma} \label{lemma: additive-term-ngca} Suppose that $\nongauss$ satisfies: 
\begin{enumerate}
    \item the Moment Matching Assumption (Assumption~\ref{ass: moment-matching}) with parameter $k$,
    \item the Bounded Signal Strength Assumption (Assumption~\ref{ass: bounded-snr}) with parameters $(\lambda,K)$,
    \item the Locally Bounded Likelihood Ratio Assumption (Assumption~\ref{ass: locally-bounded-LLR}) with parameters $(\lambda, K, \kappa)$.
\end{enumerate}
Then, there is a constant $C_{k,\kappa}$ that depends only on $(k,\kappa)$ such that if the parameters $(\lambda, K, \kappa)$ satisfy $ K\lambda \leq  \dim^{-\frac{\kappa}{2}}/{C_{k,\kappa}}$, we have, for any  $S: \{\pm 1\}^\dim \rightarrow \R$ with $\|S\|_\prior \leq 1$ and any $\zeta \in \R$ with $|\zeta| \leq 1/2L$,
\begin{align*}
    \ln \E \exp\left( \zeta  \sum_{\ell=1}^\batch \ip{(\dscore{\bm V}(\bm x_{\ell}) - \barscore(\bm x_{\ell})) \cdot \Indicator{\|\bm x_{\ell}\|\leq \sqrt{2\dim}}}{S}_\prior\right) \leq  \zeta n \sigma e^{-\frac{\dim}{16}} + {n \zeta^2 \sigma^2}.
\end{align*}
In the above display the parameters $L,\sigma^2$ are defined as follows:
\begin{align*}
    L & \explain{def}{=} C_{\kappa,k} \cdot K \lambda \cdot \dim^{\frac{\kappa}{2}}, \\
    \sigma^2 &\explain{def}{=} C_{k,\kappa} \cdot K^2 \lambda^2 \cdot d^{-\lceil\frac{k+1}{2}\rceil}.
\end{align*}
Furthermore,
\begin{align*}
    \left\|\sum_{\ell=1}^\batch \ip{(\dscore{\bm V}(\bm x_{\ell}) - \barscore(\bm x_{\ell})) \cdot \Indicator{\|\bm x_{\ell}\|\leq \sqrt{2\dim}}}{S}_\prior \right\|_4 & \leq \batch \cdot \sigma \cdot  e^{-\frac{\dim}{16}} + \sqrt{L \sigma} \cdot \batch^{\frac{1}{4}} + \sqrt{n \sigma},
\end{align*}
where,
\begin{align*}
    \left\|\sum_{\ell=1}^\batch \ip{(\dscore{\bm V}(\bm x_{\ell}) - \barscore(\bm x_{\ell})) \cdot \Indicator{\|\bm x_{\ell}\|\leq \sqrt{2\dim}}}{S}_\prior \right\|_4^4 \explain{def}{=} \refE \left[ \left( \sum_{\ell=1}^\batch \ip{(\dscore{\bm V}(\bm x_{\ell}) - \barscore(\bm x_{\ell})) \cdot \Indicator{\|\bm x_{\ell}\|\leq \sqrt{2\dim}}}{S}_\prior\right)^4 \right]
\end{align*}
\end{lemma}
\begin{proof}The proof of this result appears at the end of this section (Appendix~\ref{sec: proof-of-additive-lemma-ngca}).
\end{proof}
We can now use Geometric Inequalities (Proposition~\ref{prop: geometric inequality}) to control:
\begin{align*}
    \left| \looE{0}{i}\left[  \sum_{\ell=1}^\batch \ip{(\dscore{\bm V}(\bm x_{i\ell}) - \barscore(\bm x_{i\ell})) \cdot \Indicator{\|\bm x_{i\ell}\|\leq \sqrt{2\dim}}}{S}_\prior   \bigg| \bm Y = \bm y, Z_i = 1, (\bm X_j)_{j \neq i}  \right] \right|^2.
\end{align*}
We consider two cases depending upon whether $\bm y \in \mathcal{R}_{\mathsf{rare}}^{(i)}$ or $\bm y  \in \mathcal{R}_{\mathsf{freq}}^{(i)}$, where,
\begin{align*}
     \mathcal{R}_{\mathsf{rare}}^{(i)} &\explain{def}{=} \left\{ \bm y \in \{0,1\}^{m\budget} : 0 <   \looP{0}{i}(\bm Y = \bm y, Z_i = 1 | (\bm X_j)_{j \neq i}) \leq 4^{-\budget} \right\}, \\
     \mathcal{R}_{\mathsf{freq}}^{(i)} &\explain{def}{=} \left\{ \bm y \in \{0,1\}^{m\budget} :  \looP{0}{i}(\bm Y = \bm y, Z_i = 1 | (\bm X_j)_{j \neq i}) > 4^{-\budget} \right\}.
\end{align*}

\begin{description}
\item [Case 1: $\bm y \in \mathcal{R}_{\mathsf{rare}}^{(i)}$.] In this situation we apply the moment version of the Geometric Inequality (Proposition~\ref{prop: geometric inequality}, item (1)) with $q = 4$. Using the moment estimate in Lemma~\ref{lemma: additive-term-ngca}, we obtain,
\begin{align} 
     &\left| \looE{0}{i}\left[  \sum_{\ell=1}^\batch \ip{(\dscore{\bm V}(\bm x_{i\ell}) - \barscore(\bm x_{i\ell})) \cdot \Indicator{\|\bm x_{i\ell}\|\leq \sqrt{2\dim}}}{S}_\prior   \bigg| \bm Y = \bm y, Z_i = 1, (\bm X_j)_{j \neq i}  \right] \right| \nonumber   \\ &\hspace{8cm} \leq \frac{\batch \cdot \sigma \cdot  e^{-\frac{\dim}{16}} + \sqrt{L \sigma} \cdot \batch^{\frac{1}{4}} +  \sigma\sqrt{n}}{\looP{0}{i}(\bm Y = \bm y, Z_i = 1 | (\bm X_j)_{j \neq i} )^{\frac{1}{4}}}, \label{eq: ngca-geometric-rare-event}
\end{align}
where $L,\sigma$ are as defined in Lemma~\ref{lemma: additive-term-ngca}.
\item [Case 2: $\bm y \in \mathcal{R}_{\mathsf{freq}}^{(i)}$.] In this situation we apply the m.g.f. version of the Geometric Inequality (Proposition~\ref{prop: geometric inequality}, item (2)). Using the m.g.f. estimate in Lemma~\ref{lemma: additive-term-ngca}, we obtain, for any $0 < \zeta \leq 1/2L$,
\begin{align*}
     &\left| \looE{0}{i}\left[  \sum_{\ell=1}^\batch \ip{(\dscore{\bm V}(\bm x_{i\ell}) - \barscore(\bm x_{i\ell})) \cdot \Indicator{\|\bm x_{i\ell}\|\leq \sqrt{2\dim}}}{S}_\prior   \bigg| \bm Y = \bm y, Z_i = 1, (\bm X_j)_{j \neq i}  \right] \right|   \\ &\hspace{5cm} \leq   n \sigma e^{-\frac{\dim}{16}} + {n \zeta \sigma^2} + \frac{1}{\zeta} \ln \frac{1}{\looP{0}{i}(\bm Y = \bm y, Z_i = 1 | (\bm X_j)_{j \neq i} )},
\end{align*}
where $L,\sigma$ are as defined in Lemma~\ref{lemma: additive-term-ngca}. We set:
\begin{align*}
    \zeta^2 & = \frac{1}{n \sigma^2} \cdot \ln \frac{1}{\looP{0}{i}(\bm Y = \bm y, Z_i = 1 | (\bm X_j)_{j \neq i} )} \leq \frac{\budget \cdot \ln(4)}{n \sigma^2}.
\end{align*}
If,
\begin{align} \label{eq: ngca-batch-assumption}
    n & \geq \frac{4 \ln(4) \cdot b \cdot L^2}{\sigma^2},
\end{align}
then this choice is valid, i.e. $\zeta \leq 1/2L$. With this choice, we obtain,
\begin{align}
     &\left| \looE{0}{i}\left[  \sum_{\ell=1}^\batch \ip{(\dscore{\bm V}(\bm x_{i\ell}) - \barscore(\bm x_{i\ell})) \cdot \Indicator{\|\bm x_{i\ell}\|\leq \sqrt{2\dim}}}{S}_\prior   \bigg| \bm Y = \bm y, Z_i = 1, (\bm X_j)_{j \neq i}  \right] \right|   \nonumber \\ &\hspace{5cm} \leq   n \sigma e^{-\frac{\dim}{16}} + 2 \cdot \sigma \cdot \sqrt{n} \cdot \ln^{\frac{1}{2}} \left( \frac{1}{\looP{0}{i}(\bm Y = \bm y, Z_i = 1 | (\bm X_j)_{j \neq i} )} \right). \label{eq: ngca-geometric-typical-event}
\end{align}
\end{description}
With these estimates, we can control the term $\mathsf{I}$, which we decompose as follows:
\begin{align*}
    \mathsf{I} & = \looE{0}{i}\left[ \Indicator{Z_i =1} \cdot \int  \left({\looE{0}{i}\left[ \sum_{\ell=1}^\batch (\dscore{\bm V}(\bm x_{i\ell}) - \barscore(\bm x_{i\ell})) \cdot \Indicator{\|\bm x_{i\ell}\|\leq \sqrt{2\dim}}    \bigg| \bm Y, Z_i, (\bm X_j)_{j \neq i}  \right] }\right)^2 \prior(\diff \bm V) \right] \\
    & = \mathsf{(Ia)} + \mathsf{(Ib)}, \\
    \mathsf{(Ia)} &\explain{def}{=}  \looE{0}{i}\left[ \sum_{\bm y \in \mathcal{R}_{\mathsf{rare}}^{(i)} } \looP{0}{i}(\bm Y = \bm y, Z_i = 1 | (\bm X_j)_{j \neq i} ) \cdot \geometric^2_i(\bm y, (\bm X_j)_{j \neq i}) \right], \\
    \mathsf{(Ib)} &\explain{def}{=}  \looE{0}{i}\left[ \sum_{\bm y \in \mathcal{R}_{\mathsf{freq}}^{(i)} } \looP{0}{i}(\bm Y = \bm y, Z_i = 1 | (\bm X_j)_{j \neq i} ) \cdot \geometric^2_i(\bm y, (\bm X_j)_{j \neq i}) \right].
\end{align*}
In the above display, we defined,
\begin{align*}
    &\geometric_i^2(\bm y, (\bm X_j)_{j \neq i}) \explain{def}{=} \\&\hspace{2cm} \int  \left({\looE{0}{i}\left[ \sum_{\ell=1}^\batch (\dscore{\bm V}(\bm x_{i\ell}) - \barscore(\bm x_{i\ell})) \cdot \Indicator{\|\bm x_{i\ell}\|\leq \sqrt{2\dim}}    \bigg| \bm Y = \bm  y, Z_i = 1, (\bm X_j)_{j \neq i}  \right] }\right)^2 \prior(\diff \bm V).
\end{align*}
In order to control $(\mathsf{Ia})$, we rely on the estimate \eqref{eq: ngca-geometric-rare-event}:
\begin{align*}
     \mathsf{(Ia)} &\leq  (\batch \cdot \sigma \cdot  e^{-\frac{\dim}{16}} + \sqrt{L \sigma} \cdot \batch^{\frac{1}{4}} +  \sigma\sqrt{n})^2 \cdot \looE{0}{i}\left[ \sum_{\bm y \in \mathcal{R}_{\mathsf{rare}}^{(i)} } \looP{0}{i}(\bm Y = \bm y, Z_i = 1 | (\bm X_j)_{j \neq i} )^{\frac{1}{2}}\right] \\
     & \leq  (\batch \cdot \sigma \cdot  e^{-\frac{\dim}{16}} + \sqrt{L \sigma} \cdot \batch^{\frac{1}{4}} +  \sigma\sqrt{n})^2 \cdot 2^{-b} \cdot \looE{0}{i}[| \mathcal{R}_{\mathsf{rare}}^{(i)}|].
\end{align*}
Since we assume the communication protocol to be deterministic conditioned on $(\bm X_j)_{j \neq i}$, all but $\budget$ bits of $\bm Y$ are fixed. Consequently, $| \mathcal{R}_{\mathsf{rare}}| \leq 2^\budget$. Hence,
\begin{align*}
    \mathsf{(Ia)} &\leq 3 \cdot \left(\batch^2 \cdot \sigma^2 \cdot e^{-\frac{\dim}{8}} + L \sigma \sqrt{n} + n \sigma^2 \right) \explain{(c)}{\leq} 3 \cdot \left(\batch^2 \cdot \sigma^2 \cdot e^{-\frac{\dim}{8}} +  2n \sigma^2 \right).
\end{align*}
In the above display, in the step marked (c) we observed that the assumption \eqref{eq: ngca-batch-assumption} guarantees $L \sigma \sqrt{n} \leq n \sigma^2$.
In order to control $(\mathsf{Ib})$, we rely on the estimate \eqref{eq: ngca-geometric-typical-event}:
\begin{align*}
     &\mathsf{(Ib)} \leq  2 \batch^2 \sigma^2 e^{-\frac{\dim}{8}} + 8 \sigma^2 \batch \cdot \looE{0}{i}\left[ \sum_{\bm y \in \mathcal{R}_{\mathsf{freq}}^{(i)} } h(\looP{0}{i}(\bm Y = \bm y, Z_i = 1 | (\bm X_j)_{j \neq i} )) \right] \\
     & \leq 2 \batch^2 \sigma^2 e^{-\frac{\dim}{8}} + 8 \sigma^2 \batch \cdot \looE{0}{i}\left[ \sum_{(\bm y,z) \in \{0,1\}^{\budget+1}} h(\looP{0}{i}(\bm Y = \bm y, Z_i = z | (\bm X_j)_{j \neq i} )) \right], 
\end{align*}
where $h(x) \explain{def}{=} - x \ln(x)$ is the entropy function. Since we assume the communication protocol to be deterministic (cf. Remark~\ref{rem:deterministic}), conditioned on $(\bm X_j)_{j \neq i}$, all but $\budget+1$ bits of $(\bm Y, Z_i)$ are fixed. Hence conditioned on $(\bm X_j)_{j \neq i}$, the random vector $(\bm Y, Z_i)$ has a support size of at most $2^{\budget + 1}$. The maximum entropy distribution on a given set $S$ is the uniform distribution, which attains an entropy of $\ln |S|$. Hence,
\begin{align*}
    \sum_{(\bm y,z) \in \{0,1\}^{\budget+1}} \looP{0}{i}(\bm Y = \bm y, Z_i = z | (\bm X_j)_{j \neq i} ) \cdot \ln \frac{1}{\looP{0}{i}(\bm Y = \bm y, Z_i = z | (\bm X_j)_{j \neq i} )} & \leq (\budget + 1) \cdot \ln(2)
\end{align*}
This yields the estimate,
\begin{align*}
    \mathsf{(Ib)} &\leq 2 \batch^2 \sigma^2 e^{-\frac{\dim}{8}} + 8 \ln(2) \cdot \sigma^2 \batch \cdot (\budget+1).
\end{align*}
Combining the estimates on the terms $\mathsf{Ia}, \mathsf{Ib}$ we obtain, $(\mathsf{I}) \leq  5\batch^2 \sigma^2 e^{-\frac{\dim}{8}} + 18 \sigma^2 \batch \cdot \budget$. Substituting this estimate on $\mathsf{I}$ and the estimate on $\mathsf{II}$ obtained in Lemma~\ref{lemma: ngca-nonadditive} in \eqref{eq: conditional-expectation-decomposition}, we obtain,
\begin{align*} 
     &\looE{0}{i}\left[\int \left({\looE{0}{i}\left[ \left( \frac{\diff \dmu{\bm V}}{\diff \refmu} (\bm X_i) - \frac{\diff \nullmu}{\diff \refmu} (\bm X_i) \right) \cdot \Indicator{Z_i =1}  \bigg| \bm Y, Z_i, (\bm X_j)_{j \neq i}  \right]}\right)^2 \prior(\diff \bm V) \right] \\& \hspace{8cm} \leq 10\batch^2 \sigma^2 e^{-\frac{\dim}{8}} + 36 \sigma^2 \batch \cdot \budget + 4 \cdot (K^2  \batch\lambda^2 )^2.
\end{align*}
Plugging the above bound in \eqref{eq: hellinger-recall-ngca} we obtain,
\begin{align*} 
     \frac{\MIhell{\bm V}{\bm Y}}{\consthell} & \explain{}{\leq} 10 \mach \batch^2 \sigma^2 e^{-\frac{\dim}{8}} + 36 \sigma^2 \cdot \mach \cdot \batch \cdot \budget + 4 \cdot \mach \cdot  (K^2  \batch\lambda^2 )^2 +  m \cdot \nullmu(\goodevnt^c).
\end{align*}
Finally, by Lemma~\ref{lemma: ngca-bad-event}, for any $q \geq 2$, we have
\begin{align*}
    &\frac{\MIhell{\bm V}{\bm Y}}{\consthell}  \explain{}{\leq} 10 \mach \batch^2 \sigma^2 e^{-\frac{\dim}{8}} + 36 \sigma^2 \cdot \mach \cdot \batch \cdot \budget + 4 \cdot \mach \cdot  (K^2  \batch\lambda^2 )^2 \\ & \hspace{6cm} + C_{q,k,K} \cdot \mach \cdot  \left( (1+\lambda^2) \cdot \mach \cdot \batch \cdot  e^{-\frac{\dim}{C_{q,k,K}}} +  \left(\frac{\batch \lambda^2}{\dim} \right)^{\frac{q}{2}} \right).
\end{align*}
This is precisely the information bound claimed in Proposition~\ref{prop:ngca-info-bound}. This concludes the proof of Proposition~\ref{prop:ngca-info-bound}. The remainder of this section is devoted to the proof of the various intermediate results used in the above proof.

\subsubsection{Proof of Lemma~\ref{lemma: ngca-bad-event}} \label{sec: proof-ngca-bad-event}

\begin{proof}[Proof of Lemma~\ref{lemma: ngca-bad-event}]
We begin by observing that by a union bound,
\begin{align*}
    \nullmu(\goodevnt^c_1) & \leq  \sum_{i=1}^\batch \nullmu(\{\|\bm x_i\| > \sqrt{2 \dim}\}) \\
    & = \sum_{i=1}^\batch  \int \dmu{\bm V}(\{\|\bm x_i\| > \sqrt{2 \dim}\}) \; \prior(\diff \bm V).
\end{align*}
When $\bm x\explain{}{\sim} \refmu = \gauss{\bm 0}{\bm I_\dim}$, standard $\chi^2$-concentration (see for e.g. \citep[Example 2.11]{wainwright_2019}) gives us:
\begin{align*}
  \refmu(\{\|\bm x\| > \sqrt{2 \dim}\}) & \leq e^{-\dim/8}.
\end{align*}
\begin{align*}
    \dmu{\bm V}(\{\|\bm x\| > \sqrt{2 \dim}\}) & = \refE \left[ \frac{\diff \dmu{\bm V} }{\diff \refmu}(\bm x) \Indicator{\|\bm x\| > \sqrt{2d}} \right] \\ &
    \explain{(a)}{=} \refE \left[ \frac{\diff \nongauss}{\diff \refmu}\left(Z\right) \Indicator{\|\bm x\| > \sqrt{2d}} \right], \; Z \explain{def}{=} \ip{\bm x}{\frac{\bm V}{\|\bm V\|}} \\
    & \explain{(b)}{\leq} \left( \refmu\left(\{\|\bm x\| > \sqrt{2 \dim}\}\right) \cdot \refE\left[ \left( \frac{\diff \nongauss}{\diff \refmu}\left(Z\right) \right)^2 \right]  \right)^\frac{1}{2}   \\
    & \explain{(c)}{\leq} (1+ K^2\lambda^2) \cdot e^{-\frac{\dim}{16}}.
\end{align*}
In the above display, in the step marked (a), we used the formula for the likelihood ratio in \eqref{eq: ngca-likelihood}, in step (b) we used Cauchy-Schwarz inequality and in step (c) we appealed to Bounded Signal Strength Assumption (Assumption~\ref{ass: bounded-snr}).  Hence, we conclude that,
\begin{align*}
    \nullmu(\goodevnt^c_1) & \leq (1+K^2\lambda^2) \cdot\batch \cdot e^{-\frac{\dim}{16}}.  
\end{align*}
In order to analyze $\nullmu(\goodevnt_2^c)$, we recall that,
\begin{align*}
     \frac{\diff \nullmu}{\diff\refmu}(\bm x_{1:\batch}) - 1 &\explain{}{=} \sum_{\substack{\bm t \in \W^\batch \\ \|\bm t\|_1 \geq 1}} \hat{\nongauss}_{\bm t}  \cdot \intH{\bm t}{\bm x_{1:\batch}}{1}.
\end{align*}
We decompose the centered likelihood ratio into the low degree part and the high degree part:
\begin{align*}
    \frac{\diff \nullmu}{\diff\refmu}(\bm x_{1:\batch}) - 1 &\explain{}{=} \lowdegree{ \frac{\diff \nullmu}{\diff\refmu}(\bm x_{1:\batch}) - 1}{t} + \highdegree{ \frac{\diff \nullmu}{\diff\refmu}(\bm x_{1:\batch}) - 1}{t},
\end{align*}
where,
\begin{align*}
    \lowdegree{ \frac{\diff \nullmu}{\diff\refmu}(\bm x_{1:\batch}) - 1}{t} \explain{def}{=} \sum_{\substack{\bm t \in \W^\batch \\ 1 \leq \|\bm t\|_1 \leq t}} \hat{\nongauss}_{\bm t}  \cdot \intH{\bm t}{\bm x_{1:\batch}}{1}, \\
    \highdegree{ \frac{\diff \nullmu}{\diff\refmu}(\bm x_{1:\batch}) - 1}{t} \explain{def}{=} \sum_{\substack{\bm t \in \W^\batch \\  \|\bm t\|_1 > t}} \hat{\nongauss}_{\bm t}  \cdot \intH{\bm t}{\bm x_{1:\batch}}{1}.
\end{align*}
With this decomposition, for any $q \geq 1$, we have, by Markov's Inequality,
\begin{align*}
    \nullmu(\goodevnt_2^c) & \leq \nullmu\left( \left|
    \lowdegree{ \frac{\diff \nullmu}{\diff\refmu}(\bm x_{1:\batch}) - 1}{t} \right| > \frac{1}{4} \right) + \nullmu\left(
    \highdegree{ \frac{\diff \nullmu}{\diff\refmu}(\bm x_{1:\batch}) - 1}{t} > \frac{1}{4}\right) \\
    & \leq 4^q \cdot \barE \left[  \left|
    \lowdegree{ \frac{\diff \nullmu}{\diff\refmu}(\bm x_{1:\batch}) - 1}{t} \right|^q\right] + 4 \barE \left[  \left|
    \highdegree{ \frac{\diff \nullmu}{\diff\refmu}(\bm x_{1:\batch}) - 1}{t} \right|\right] \\
    & =  4^q \cdot \refE \left[    \frac{\diff \nullmu}{\diff\refmu}(\bm x_{1:\batch}) \cdot \left|
    \lowdegree{ \frac{\diff \nullmu}{\diff\refmu}(\bm x_{1:\batch}) - 1}{t} \right|^q\right] + 4 \refE \left[   \frac{\diff \nullmu}{\diff\refmu}(\bm x_{1:\batch}) \cdot \left|
    \highdegree{ \frac{\diff \nullmu}{\diff\refmu}(\bm x_{1:\batch}) - 1}{t} \right|\right] \\
    & \leq 4^q \cdot \left\| \frac{\diff \nullmu}{\diff\refmu}(\bm x_{1:\batch})\right\|_2 \cdot \left\| \lowdegree{ \frac{\diff \nullmu}{\diff\refmu}(\bm x_{1:\batch}) - 1}{t}\right\|_{2q}^{q} + 4 \cdot \left\| \frac{\diff \nullmu}{\diff\refmu}(\bm x_{1:\batch})\right\|_2 \cdot \left\| \highdegree{ \frac{\diff \nullmu}{\diff\refmu}(\bm x_{1:\batch}) - 1}{t}\right\|_2.
\end{align*}
In order to obtain the last inequality, we applied Cauchy-Schwarz inequality. We also note that all the norms $\|\cdot \|_q$ are defined with respect to $\refmu$.  We now estimate each of the norms in the above display. 
The quantity:
\begin{align*}
    \left\|\lowdegree{\frac{\diff \nullmu}{\diff \refmu} (\bm x_{1:\ssize}) - 1}{t} \right\|_q^2, 
\end{align*}
with the choice $q=2$ is a central object in the low-degree likelihood ratio framework. In Appendix~\ref{sec:ngca-LDLR} (which analyzes the Non-Gaussian Component Analysis problem in the low-degree likelihood ratio framework), Proposition~\ref{prop: low-degree-norm-ngca} shows that there is a constant $C_{k,K}>0$ depending only on $k,K$ such that, for any $q \geq 2$, if, 
\begin{align}\label{eq:ngca-ldlr-requirement}
    t & \leq \frac{1}{C_{k,K}}\cdot \frac{\dim}{(q-1)^2}, \; \ssize \lambda^2 \leq  \frac{1}{C_{k,K}}\cdot \frac{1}{(q-1)^k} \cdot  \frac{\dim^{\frac{k}{2}}}{t^{\frac{k-2}{2}}},
\end{align}
then,
\begin{align*}
   \left\|\lowdegree{\frac{\diff \nullmu}{\diff \refmu} (\bm x_{1:\ssize}) - 1}{t} \right\|_q^2  & \leq  \frac{ C_{k,K} \cdot (q-1)^k \cdot\ssize \lambda^2 \cdot t^{\frac{k-2}{2}}}{\dim^{\frac{k}{2}}} \leq 1.
\end{align*}
We set,
\begin{align*}
    t = \frac{1}{C_{k,K}} \cdot \frac{\dim}{(2q-1)^2}.
\end{align*}
The hypothesis on the effective sample size $\batch \lambda^2 \leq \dim/C_{q,k,K}$ ensures the requirement \eqref{eq:ngca-ldlr-requirement} is met, and we obtain,
\begin{align*}
    \left\| \lowdegree{ \frac{\diff \nullmu}{\diff\refmu}(\bm x_{1:\batch}) - 1}{t}\right\|_{2q}^2 & \leq \frac{ C_{k,K} \cdot (2q-1)^2 \cdot\batch \lambda^2}{\dim}, \\
     \left\| \lowdegree{ \frac{\diff \nullmu}{\diff\refmu}(\bm x_{1:\batch}) - 1}{t}\right\|_{2}^2 & \leq \frac{ C_{k,K} \cdot \batch \lambda^2}{\dim} \leq 1.
\end{align*}
On the other hand, by Lemma~\ref{lemma: integrated-hermite-hypercontractivity}, we have
\begin{align*}
    \left\| \highdegree{ \frac{\diff \nullmu}{\diff\refmu}(\bm x_{1:\batch}) - 1}{t}\right\|_2^2 & = \sum_{\substack{\bm t \in \W^\batch \\  \|\bm t\|_1 > t}} \hat{\nongauss}_{\bm t}^2  \cdot \refE[\intH{\bm t}{\bm x_{1:\batch}}{1}^2]
\end{align*}
In Lemma~\ref{lemma: norm-integrated-hermite-highdegree}, we showed that, for any $\|\bm t\|_1 > t$,
\begin{align*}
    \refE[\intH{\bm t}{\bm x_{1:\batch}}{1}^2] & \leq  2\exp\left( - \frac{\dim}{C_{q,k,K}} \right), \; C_{q,k,K} = {2}\left( 1-e^{-\frac{1}{C_{k,K} (q-1)^2}} \right)^{-1}.
\end{align*}
Hence,
\begin{align*}
     \left\| \highdegree{ \frac{\diff \nullmu}{\diff\refmu}(\bm x_{1:\batch}) - 1}{t}\right\|_2^2 & \leq 2 \cdot e^{-\frac{\dim}{C_{q,k,K}}} \cdot \sum_{\substack{\bm t \in \W^\batch \\  \|\bm t\|_1 > t}} \hat{\nongauss}_{\bm t}^2 \\
     & \leq 2 \cdot e^{-\frac{\dim}{C_{q,k,K}}} \cdot \sum_{\substack{\bm t \in \W^\batch }} \hat{\nongauss}_{\bm t}^2 \\
     & \explain{(e)}{\leq} 2 \cdot e^{-\frac{\dim}{C_{q,k,K}}} \cdot (1+K^2 \lambda^2)^\batch \\
     & \leq 2 \cdot e^{-\frac{\dim}{C_{q,k,K}}+ K^2 \lambda^2 \batch} \\
     & \explain{(f)}{\leq } 2 \cdot e^{-\frac{\dim}{2C_{q,k,K}}}.
\end{align*}
In the above display, the step (e) relies on the Bounded Signal Strength Assumption and in the step marked (f) we used the effective sample size assumption $\batch \lambda^2 \leq \dim/(2 C_{q,k,K})$.
Due to the orthogonality of integrated Hermite polynomials (Lemma~\ref{lemma: integrate-hermite-orthogonality}), one can compute:
\begin{align*}
    \left\| \frac{\diff \nullmu}{\diff\refmu}(\bm x_{1:\batch}) - 1\right\|_2^2 & = 1 +  \left\| \lowdegree{ \frac{\diff \nullmu}{\diff\refmu}(\bm x_{1:\batch}) - 1}{t}\right\|_{2}^2 + \left\| \highdegree{ \frac{\diff \nullmu}{\diff\refmu}(\bm x_{1:\batch}) - 1}{t}\right\|_2^2 \leq 1 + 1 + 2 = 4.
\end{align*}
Hence,
\begin{align*}
     \nullmu(\goodevnt_2^c) & \leq 4^q \cdot \left\| \frac{\diff \nullmu}{\diff\refmu}(\bm x_{1:\batch})\right\|_2 \cdot \left\| \lowdegree{ \frac{\diff \nullmu}{\diff\refmu}(\bm x_{1:\batch}) - 1}{t}\right\|_{2q}^{q} + 4 \cdot \left\| \frac{\diff \nullmu}{\diff\refmu}(\bm x_{1:\batch})\right\|_2 \cdot \left\| \highdegree{ \frac{\diff \nullmu}{\diff\refmu}(\bm x_{1:\batch}) - 1}{t}\right\|_2 \\
     & = \left( \frac{C_{q,k,K} \cdot  \batch \lambda^2}{\dim} \right)^{\frac{q}{2}} + 16 e^{-\dim/C_{q,k,K}}.
\end{align*}
Finally, by suitably redefining constants, we obtain, by a union bound,
\begin{align*}
    \nullmu(\goodevnt^c) & \leq C_{q,k,K} \cdot \left( (1+\lambda^2) \cdot \batch \cdot  e^{-\frac{\dim}{C_{q,k,K}}} +  \left(\frac{\batch \lambda^2}{\dim} \right)^{\frac{q}{2}} \right),
\end{align*}
as claimed. 
\end{proof}
\subsubsection{Proof of Lemma~\ref{lemma: additive-term-ngca}} \label{sec: proof-of-additive-lemma-ngca}
Let $\bm x_{1:\ssize}$ be generated as: $\bm x_i \explain{i.i.d.}{\sim} \gauss{\bm 0}{\bm I_\dim}$. Let $S: \{\pm 1\}^\dim \rightarrow \R$ be a real-valued function defined on the Boolean hypercube with $\|S\|_\prior \leq 1$. In this section, we wish to understand the concentration behavior of the random variable:
\begin{align} \label{eq: additive-term-ngca}
    \sum_{\ell=1}^\batch \ip{(\dscore{\bm V}(\bm x_{\ell}) - \barscore(\bm x_{\ell})) \cdot \Indicator{\|\bm x_{\ell}\|\leq \sqrt{2\dim}}}{S}_\prior.
\end{align}
Since this is a sum of i.i.d.\ random variables, we will find the Bernstein Inequality useful in our analysis, and we reproduce the statement of this inequality below for convenience. This result is attributed to Bernstein. The statement below has been reproduced from \citet[Proposition 2.10]{wainwright_2019}
\begin{fact}[Bernstein's Inequality] \label{fact: latala} Let $U_1, U_2 \cdots , U_n$ be i.i.d.\ random variables which satisfy:
\begin{enumerate}
    \item $\E U_i \leq u$
    \item $\mathsf{Var}(U_i) \leq \sigma^2$
    \item $|U_i| \leq L$ with probability $1$. 
\end{enumerate}
Then, for any $|\zeta| \leq 1/L$,
\begin{align*}
    \ln \E \exp\left( \zeta \sum_{i=1}^n U_i \right) \leq  \zeta n u + \frac{n \zeta^2 \sigma^2}{2(1-L|\zeta|)}.
\end{align*}
\end{fact}
 We can now provide the proof of Lemma~\ref{lemma: additive-term-ngca}. 
\begin{proof}[Proof of Lemma~\ref{lemma: additive-term-ngca}] 
The proof involves an application of Bernstein Inequality (Fact~\ref{fact: latala}) after the computation of relevant quantities, which we compute in the following paragraphs. Let $\bm x \sim \refmu = \gauss{\bm 0}{\bm I_\dim}$. We recall that,
\begin{align*}
    \dscore{\bm V}(\bm x) & \explain{def}{=} \frac{\diff \dmu{\bm V}}{\diff \refmu}(\bm x) - 1, \; \barscore(\bm x) \explain{def}{=} \int   \dscore{\bm V}(\bm x) \; \prior(\diff \bm V).
\end{align*}
\begin{description}
\item [Worst-case upper bound: ] We begin by computing:
\begin{align*}
    \sup_{\bm x \in \R^\dim} \left|\ip{(\dscore{\bm V}(\bm x) - \barscore(\bm x)) \cdot \Indicator{\|\bm x\|\leq \sqrt{2\dim}}}{S}_\prior\right| & \leq  \sup_{\bm x \in \R^\dim} \|S\|_\prior \cdot \|(\dscore{\bm V}(\bm x) - \barscore(\bm x))\|_\prior  \cdot \Indicator{\|\bm x\|\leq \sqrt{2\dim}} \\
    & \leq 2 \cdot  \sup_{\bm V \in \{\pm 1\}^\dim} \; \sup_{\bm x : \|\bm x\|^2  \leq 2 \dim} \;  \left| \frac{\diff \dmu{\bm V}}{\diff \refmu}(\bm x) - 1 \right|.
\end{align*}
Since,
\begin{align*}
    \frac{\diff \dmu{\bm V}}{\diff \refmu}(\bm x) - 1  = \frac{\diff \nongauss}{\diff \refmu}\left( \frac{\ip{\bm x}{\bm V}}{\sqrt{\dim}}\right) - 1,
\end{align*}
we obtain,
\begin{align*}
    \sup_{\bm x \in \R^\dim} \left|\ip{(\dscore{\bm V}(\bm x) - \barscore(\bm x)) \cdot \Indicator{\|\bm x\|\leq \sqrt{2\dim}}}{S}_\prior\right| & \leq 2 \cdot  \; \sup_{z: |z| \leq \sqrt{2\dim}} \;  \left| \frac{\diff \nongauss}{\diff \refmu}(z) - 1 \right|.
\end{align*}
Recall that $\nongauss$ satisfies the Locally Bounded Likelihood Ratio Assumption (Assumption~\ref{ass: locally-bounded-LLR}) with parameters $(\lambda,K,\kappa)$. Furthermore if,
\begin{align} \label{eq: additive-ngca-ass-snr}
   3^\kappa \cdot  K \cdot  \lambda \cdot \dim^{\frac{\kappa}{2}} & \leq 1,
\end{align}
Assumption~\ref{ass: locally-bounded-LLR} yields,
\begin{align} \label{eq: additive-ngca-worst-case-bound}
    \sup_{\bm x \in \R^\dim} \left|\ip{(\dscore{\bm V}(\bm x) - \barscore(\bm x)) \cdot \Indicator{\|\bm x\|\leq \sqrt{2\dim}}}{S}_\prior\right| & \leq 2 \cdot  3^\kappa \cdot  K \cdot  \lambda \cdot \dim^{\frac{\kappa}{2}} \\
    & \explain{def}{=} L/2. \nonumber
\end{align}
\item [Upper Bound on Variance:] Observe that:
\begin{align*}
    \mathsf{Var}\left( \ip{(\dscore{\bm V}(\bm x) - \barscore(\bm x)) \cdot \Indicator{\|\bm x\|\leq \sqrt{2\dim}}}{S}_\prior\right) & \leq \refE \ip{(\dscore{\bm V}(\bm x) - \barscore(\bm x)) \cdot \Indicator{\|\bm x\|\leq \sqrt{2\dim}}}{S}_\prior^2 \\
    & \leq \refE \ip{(\dscore{\bm V}(\bm x) - \barscore(\bm x))}{S}_\prior^2
\end{align*}
Observe that,
\begin{align*}
    \int \dscore{\bm V}(\bm x) - \barscore(\bm x) \; \prior(\diff \bm V) = 0.
\end{align*}
Hence,
\begin{align*}
    &\sup_{\substack{S: \{\pm 1\} \rightarrow \R\\ \|S\|_\prior \leq 1}} \mathsf{Var}\left( \ip{(\dscore{\bm V}(\bm x) - \barscore(\bm x)) \cdot \Indicator{\|\bm x\|\leq \sqrt{2\dim}}}{S}_\prior\right) \\&\hspace{7cm} \leq   \sup_{\substack{S: \{\pm 1\} \rightarrow \R\\ \|S\|_\prior \leq 1}}\refE \ip{(\dscore{\bm V}(\bm x) - \barscore(\bm x))}{S}_\prior^2 \\
    &\hspace{7cm} = \sup_{\substack{S: \{\pm 1\} \rightarrow \R \\ \|S\|_\prior \leq 1, \; \ip{S}{1}_\prior = 0}} \refE \ip{(\dscore{\bm V}(\bm x) - \barscore(\bm x))}{S}_\prior^2 \\
    &\hspace{7cm} = \sup_{\substack{S: \{\pm 1\} \rightarrow \R \\ \|S\|_\prior \leq 1, \; \ip{S}{1}_\prior = 0}} \refE \ip{\dscore{\bm V}(\bm x)}{S}_\prior^2.
\end{align*}
Next we recall the Hermite decomposition of the $\dscore{\bm V}(\bm x)$ computed in Lemma~\ref{lemma: hermite-decomp-ngca}:
\begin{align*}
    \dscore{\bm V}(\bm x) = \sum_{t = 1}^\infty \hat{\nongauss}_{t} \cdot H_{t}\left(\frac{\ip{\bm x}{\bm V}}{\sqrt{\dim}} \right).
\end{align*}
Recalling the definition of Integrated Hermite Polynomials (Definition~\ref{def: integrated-Hermite}), we can write:
\begin{align*}
    \ip{\dscore{\bm V}(\bm x)}{S}_\prior & = \sum_{t = 1}^\infty \hat{\nongauss}_t \cdot \intH{t}{\bm x}{S}.
\end{align*}
Since the integrated Hermite polynomials are orthogonal,
\begin{align*}
    \refE \ip{\dscore{\bm V}(\bm x)}{S}_\prior^2 = \sum_{t = 1}^\infty \hat{\nongauss}_t^2 \cdot \refE[\intH{t}{\bm x}{S}^2].
\end{align*}
Since $\nongauss$ satisfies the Moment Matching Assumption (Assumption~\ref{ass: moment-matching}) with parameter $k$, we have $\hat{\nongauss}_t = 0$ for any $t \leq k-1$. Hence,
\begin{align*}
    \sup_{\substack{S: \{\pm 1\} \rightarrow \R\\ \|S\|_\prior \leq 1}} \mathsf{Var}\left( \ip{(\dscore{\bm V}(\bm x) - \barscore(\bm x)) \cdot \Indicator{\|\bm x\|\leq \sqrt{2\dim}}}{S}_\prior\right) & \leq \sum_{t = k}^\infty \hat{\nongauss}_t^2 \cdot \refE[\intH{t}{\bm x}{S}^2] \\
    & \leq \left( \sup_{t \geq k}  \refE[\intH{t}{\bm x}{S}^2] \right) \cdot \sum_{t = k}^\infty \hat{\nongauss}_t^2 \\
    & \explain{(a)}{\leq} K^2 \lambda^2 \cdot \left( \sup_{t \geq k}  \refE[\intH{t}{\bm x}{S}^2] \right).
\end{align*}
In the step marked (a), we appealed to the Bounded Signal Strength Assumption (Assumption~\ref{ass: bounded-snr}). In Lemma~\ref{lemma: norm-integrated-hermite}, we showed that, $\refE[ \intH{\bm t}{\bm x_{1:\ssize}}{S}^2] \leq (Ct)^{\frac{t}{2}} \cdot  d^{-\lceil\frac{t+1}{2}\rceil}$.
Hence,
\begin{align*}
    \max_{k \leq t \leq \dim/C} \refE[\intH{t}{\bm x}{S}^2]  & \leq  (Ck)^{\frac{k}{2}} \cdot  d^{-\lceil\frac{k+1}{2}\rceil}.
\end{align*}
On the other hand, when $t \geq \dim/C$, Lemma~\ref{lemma: norm-integrated-hermite-highdegree} gives us,
\begin{align*}
     \max_{t\geq \dim/C} \refE[\intH{t}{\bm x}{S}^2]  & \leq 2 e^{-\dim/C^\prime}, \; C^\prime = \frac{2}{(1-e^{-2/C})}.
\end{align*}
By suitably defining the constant $C_k$ (depending on $k$), we arrive at the following estimate of the variance:
\begin{subequations} \label{eq: additive-ngca-variance-estimate}
\begin{align}
     \sup_{\substack{S: \{\pm 1\} \rightarrow \R\\ \|S\|_\prior \leq 1}} \mathsf{Var}\left( \ip{(\dscore{\bm V}(\bm x) - \barscore(\bm x)) \cdot \Indicator{\|\bm x\|\leq \sqrt{2\dim}}}{S}_\prior\right) &\leq  \sup_{\substack{S: \{\pm 1\} \rightarrow \R \\ \|S\|_\prior \leq 1, \; \ip{S}{1}_\prior = 0}} \refE \ip{\dscore{\bm V}(\bm x)}{S}_\prior^2 \\ &\leq  C_k \cdot K^2 \lambda^2 \cdot d^{-\lceil\frac{k+1}{2}\rceil} \\
     & \explain{def}{=} \sigma^2.
\end{align}
\end{subequations}
\item [Upper Bound on Expectation: ] As before (by centering $S$) we can argue,
\begin{align*}
    \sup_{\substack{S: \{\pm 1\} \rightarrow \R\\ \|S\|_\prior \leq 1}} \refE\left[  \ip{(\dscore{\bm V}(\bm x) - \barscore(\bm x)) \cdot \Indicator{\|\bm x\|\leq \sqrt{2\dim}}}{S}_\prior\right] & = \sup_{\substack{S: \{\pm 1\} \rightarrow \R\\ \|S\|_\prior \leq 1 \\\ip{S}{1}_\prior = 0}} \refE\left[  \ip{\dscore{\bm V}(\bm x)}{S}_\prior  \cdot \Indicator{\|\bm x\|\leq \sqrt{2\dim}} \right].
\end{align*}
Next we observe that since $\refE[\dscore{\bm V}(\bm x)] = \refE[\barscore(\bm x)] = 0$, we have $\refE[\ip{(\dscore{\bm V}(\bm x) - \barscore(\bm x))}{S}_\prior] = 0$. Hence, we can write:
\begin{align*}
    \refE\left[  \ip{(\dscore{\bm V}(\bm x) - \barscore(\bm x))}{S}_\prior  \cdot \Indicator{\|\bm x\|\leq \sqrt{2\dim}} \right] = - \refE\left[  \ip{(\dscore{\bm V}(\bm x) - \barscore(\bm x))}{S}_\prior  \cdot \Indicator{\|\bm x\|> \sqrt{2\dim}} \right].
\end{align*}
Consequently, by Cauchy-Schwarz Inequality,
\begin{align*}
     \left( \refE\left[  \ip{(\dscore{\bm V}(\bm x) - \barscore(\bm x))}{S}_\prior  \cdot \Indicator{\|\bm x\|\leq \sqrt{2\dim}} \right] \right)^2&  \leq\refmu(\|\bm x \|^2 \geq 2\dim) \cdot   \refE[\ip{\dscore{\bm V}(\bm x)}{S}_\prior^2].
\end{align*}
Standard $\chi^2$-concentration (see for e.g. \citep[Example 2.11]{wainwright_2019}) gives us $\refmu(\{\|\bm x\| > \sqrt{2 \dim}\})  \leq e^{-\dim/8}$. Combining this with the estimate in \eqref{eq: additive-ngca-variance-estimate}, we obtain,
\begin{align} \label{eq: additive-ngca-expectation-estimate}.
     \sup_{\substack{S: \{\pm 1\} \rightarrow \R\\ \|S\|_\prior \leq 1}} \refE\left[  \ip{(\dscore{\bm V}(\bm x) - \barscore(\bm x)) \cdot \Indicator{\|\bm x\|\leq \sqrt{2\dim}}}{S}_\prior\right] &  \leq \sigma e^{-\frac{\dim}{16}}.
\end{align}
\end{description}
Combining the estimates obtained in \eqref{eq: additive-ngca-worst-case-bound}, \eqref{eq: additive-ngca-variance-estimate} and \eqref{eq: additive-ngca-expectation-estimate} with the Bernstein Inequality immediately gives the claim of the first claim of the lemma. In order to obtain the second claim, we first define:
\begin{align*}
    \alpha(S) \explain{def}{=} \refE\left[  \ip{(\dscore{\bm V}(\bm x) - \barscore(\bm x)) \cdot \Indicator{\|\bm x\|\leq \sqrt{2\dim}}}{S}_\prior\right].
\end{align*}
We have
\begin{align*}
     &\left\|\sum_{\ell=1}^\batch \ip{(\dscore{\bm V}(\bm x_{\ell}) - \barscore(\bm x_{\ell})) \cdot \Indicator{\|\bm x_{\ell}\|\leq \sqrt{2\dim}}}{S}_\prior \right\|_4  \leq \\&\hspace{5.5cm} \batch \alpha(S) +  \left\|\sum_{\ell=1}^\batch \left(\ip{(\dscore{\bm V}(\bm x_{\ell}) - \barscore(\bm x_{\ell})) \cdot \Indicator{\|\bm x_{\ell}\|\leq \sqrt{2\dim}}}{S}_\prior - \alpha(S)\right) \right\|_4.
\end{align*}
We can compute:
\begin{align*}
     &\left\|\sum_{\ell=1}^\batch \left(\ip{(\dscore{\bm V}(\bm x_{\ell}) - \barscore(\bm x_{\ell})) \cdot \Indicator{\|\bm x_{\ell}\|\leq \sqrt{2\dim}}}{S}_\prior - \alpha(S)\right) \right\|_4^4 \\& \hspace{5.5cm}= \batch \cdot \E \left(\ip{(\dscore{\bm V}(\bm x_{\ell}) - \barscore(\bm x_{\ell})) \cdot \Indicator{\|\bm x_{\ell}\|\leq \sqrt{2\dim}}}{S}_\prior - \alpha(S)\right)^4  \\&\hspace{6.5cm}+ 3 \batch(\batch-1) \cdot \mathsf{Var}^2\left( \ip{(\dscore{\bm V}(\bm x) - \barscore(\bm x)) \cdot \Indicator{\|\bm x\|\leq \sqrt{2\dim}}}{S}_\prior\right).
\end{align*}
We recall that,
\begin{align*}
    \alpha(S) & \leq \sigma e^{-\frac{\dim}{16}}, \\
     \mathsf{Var}\left( \ip{(\dscore{\bm V}(\bm x) - \barscore(\bm x)) \cdot \Indicator{\|\bm x\|\leq \sqrt{2\dim}}}{S}_\prior\right) & \leq \sigma^2, \\
     \E \left(\ip{(\dscore{\bm V}(\bm x_{\ell}) - \barscore(\bm x_{\ell})) \cdot \Indicator{\|\bm x_{\ell}\|\leq \sqrt{2\dim}}}{S}_\prior - \alpha(S)\right)^4  & \leq L^2 \cdot  \mathsf{Var}\left( \ip{(\dscore{\bm V}(\bm x) - \barscore(\bm x)) \cdot \Indicator{\|\bm x\|\leq \sqrt{2\dim}}}{S}_\prior\right) \\
     & \leq L^2 \sigma^2.
\end{align*}
Hence,
\begin{align*}
    \left\|\sum_{\ell=1}^\batch \ip{(\dscore{\bm V}(\bm x_{\ell}) - \barscore(\bm x_{\ell})) \cdot \Indicator{\|\bm x_{\ell}\|\leq \sqrt{2\dim}}}{S}_\prior \right\|_4 & \leq \batch \cdot \sigma \cdot  e^{-\frac{\dim}{16}} + \sqrt{L \sigma} \cdot \batch^{\frac{1}{4}} + \sigma\sqrt{n}. 
\end{align*}
This concludes the proof of this lemma.
\end{proof}
\subsubsection{Proof of Lemma~\ref{lemma: ngca-nonadditive}} \label{sec:ngca-nonadditive}
\begin{proof}[Proof of Lemma~\ref{lemma: ngca-nonadditive}]
We have
\begin{align*}
     \refE\left[ \bigg| \sum_{\substack{S \subset [\batch], \; |S| \geq 2}}  \dscore{\bm V}( (\bm X)_S)  \bigg|^2   \right] & =  \sum_{\substack{S_1,S_2 \subset [\batch] \\ |S_1| \geq 2, |S_2| \geq 2}}  \refE[\dscore{\bm V}( (\bm X)_{S_1}) \cdot \dscore{\bm V}( (\bm X)_{S_2})]
\end{align*}
Recall that,
\begin{align*}
    \dscore{\bm V}(\bm X_S) &\explain{def}{=} \prod_{i \in S}  \left( \frac{\diff \dmu{\bm V}}{\diff \refmu} (\bm x_i) - 1\right).
\end{align*}
We observe that, $\bm x_{1:\batch}$ are independent and,
\begin{align*}
    \refE\left[ \frac{\diff \dmu{\bm V}}{\diff \refmu} (\bm x_i) - 1\right] = 0.
\end{align*}
Hence if $S_1 \neq S_2$, $\refE[\dscore{\bm V}( (\bm X)_{S_1}) \cdot \dscore{\bm V}( (\bm X)_{S_2})] = 0$. This gives us:
\begin{align*}
     \refE\left[ \bigg| \sum_{\substack{S \subset [\batch], \; |S| \geq 2}}  \dscore{\bm V}( (\bm X)_S)  \bigg|^2   \right] & =  \sum_{\substack{S \subset [\batch] \; |S| \geq 2}}  \refE[\dscore{\bm V}( (\bm X)_{S})^2].
\end{align*}
Recall the formula for the likelihood ratio for the Non-Gaussian Component Analysis problem \eqref{eq: ngca-likelihood}, we obtain,
\begin{align*}
     \refE[\dscore{\bm V}( (\bm X)_{S})^2] & = \left( \refE\left[ \left( \frac{\diff \nongauss}{\diff \refmu} (Z) - 1 \right)^2 \right] \right)^{|S|}, \; Z \sim \gauss{0}{1}.
\end{align*}
Since $\nongauss$ satisfies the Bounded Signal Strength Assumption, we have
\begin{align*}
    \refE\left[ \bigg| \sum_{\substack{S \subset [\batch], \; |S| \geq 2}}  \dscore{\bm V}( (\bm X)_S)  \bigg|^2   \right] & \leq  \sum_{\substack{S \subset [\batch] \; |S| \geq 2}}  (K^2 \lambda^2)^{|S|} \\
    & =  \sum_{s=2}^\batch \binom{\batch}{s}  (K^2 \lambda^2)^{s} \\
    & \leq \sum_{s=2}^\batch (K^2 \batch \lambda^2)^{s}.
\end{align*}
The assumption $K^2 \batch \lambda^2 \leq 1/2$ guarantees that the above sum is dominated by a Geometric series, which immediately yields the claim of the lemma. 
\end{proof}

\subsection{Omitted Proofs from Section~\ref{sec: ngca-harmonic}}
\label{appendix:ngca-harmonic-proofs}
This section contains the proofs of the various analytic properties (Lemma~\ref{lemma: norm-integrated-hermite}, Lemma~\ref{lemma: norm-integrated-hermite-highdegree} and Lemma~\ref{lemma: integrated-hermite-hypercontractivity}) of the likelihood ratio for the Non-Gaussian Component Analysis problem, which were stated in Appendix~\ref{sec: ngca-harmonic}. 
\subsubsection{Proof of Lemma~\ref{lemma: norm-integrated-hermite}} \label{sec: ngca-harmonic-integrated-lowdegree-proof}
\begin{proof}[Proof of Lemma~\ref{lemma: norm-integrated-hermite}]  Using Definition~\ref{def: integrated-Hermite} and Fubini's theorem, we obtain,
\begin{align*}
     \refE[\intH{\bm t}{\bm x_{1:\ssize}}{S}^2] = \int\int \prod_{i=1}^\ssize \refE\left[ H_{t_i}\left( \frac{\ip{\bm x_i}{\bm V_1}}{\sqrt{\dim}} \right) H_{t_i}\left( \frac{\ip{\bm x_i}{\bm V_2}}{\sqrt{\dim}} \right) \right] \cdot S(\bm V_1) \cdot S(\bm V_2) \; \prior(\diff \bm V_1) \; \prior(\diff \bm V_2).
\end{align*}
Fact~\ref{fact: correlated-hermite} gives us, 
\begin{align*}
    \refE\left[ H_{t_i}\left( \frac{\ip{\bm x_i}{\bm V}}{\sqrt{\dim}} \right) H_{t_i}\left( \frac{\ip{\bm x_i}{\bm V^\prime}}{\sqrt{\dim}} \right) \right] = \left( \frac{\ip{\bm V_1}{\bm V_2}}{d} \right)^t.
\end{align*}
We define $\bm V = \bm V_1 \odot \bm V_2$, where $\odot$ denotes entry-wise product of vectors and, 
\begin{align*}
    \overline{V} = \frac{1}{\dim} \sum_{i=1}^\dim V_i.
\end{align*}
Hence,
\begin{align*}
     \refE[\intH{\bm t}{\bm x_{1:\ssize}}{S}^2] &=  \int\int \overline{V}^t  \cdot S(\bm V_1) \cdot S(\bm V_2) \; \prior(\diff \bm V_1) \; \prior(\diff \bm V_2).
\end{align*}
Since $\bm V_1, \bm V_2$ are independently sampled from the prior $\pi$ and $\bm V = \bm V_1 \odot \bm V_2$, it is straight-forward to check that $\bm V_1, \bm V$ are independent, uniformly random $\{\pm 1\}^\dim$ vectors and $\bm V_2 = \bm V_1 \odot \bm V$. Hence,
\begin{align} \label{eq: intermediate-eq-integrated-hermite-norm}
     \refE[\intH{\bm t}{\bm x_{1:\ssize}}{S}^2] &=  \int\int \overline{V}^t  \cdot S(\bm V_1) \cdot S(\bm V \odot \bm V_1) \; \prior(\diff \bm V_1) \; \prior(\diff \bm V).
\end{align}
Recall that, the collection of polynomials:
\begin{align*}
    \left\{\bm V^{\bm r} \explain{def}{=} \prod_{i=1}^\dim V_i^{r_i}: \bm r \in \{0,1\}^\dim\right\},
\end{align*}
form an orthonormal basis for functions on the Boolean hypercube $\{\pm 1\}^\dim$ with respect to the uniform distribution $\prior = \unif{\{\pm 1\}^\dim}$. Hence, we can expand $\bm S$ in this basis:
\begin{align*}
    \bm S(\bm V) & = \sum_{\bm r \in \{0,1\}^\dim} \hat{S}_{\bm r} \cdot \bm V^{\bm r}, \; \hat{S}_{\bm r} \explain{def}{=} \int S(\bm V) \cdot \bm V^{\bm r} \prior(\diff \bm V).
\end{align*}
Substituting this in \eqref{eq: intermediate-eq-integrated-hermite-norm} gives us:
\begin{align*}
    \refE[\intH{\bm t}{\bm x_{1:\ssize}}{S}^2] &=  \sum_{\bm r, \bm s \in \{0,1\}^\dim} \hat{S}_{\bm r} \hat{S}_{\bm s} \int\int \overline{V}^t  \cdot  \bm{V}_1^{\bm r + \bm s} \cdot \bm V^{\bm s} \; \prior(\diff \bm V_1) \; \prior(\diff \bm V).
\end{align*}
Noting that, if $\bm r \neq \bm s$,
\begin{align*}
    \int   \bm{V}_1^{\bm r + \bm s} \; \prior(\diff \bm V_1) = 0,
\end{align*}
we obtain,
\begin{align*}
     \refE[\intH{\bm t}{\bm x_{1:\ssize}}{S}^2] &=  \sum_{\bm r \in \{0,1\}^\dim} \hat{S}_{\bm r}^2 \int \overline{V}^t   \cdot \bm V^{\bm r} \; \prior(\diff \bm V).
\end{align*}
With this formula, we can prove each claim of the lemma.
\begin{description}
\item [When $S = 1$:] When $S = 1$, $\hat{S}_{\bm 0} = 1$ and $\hat{\bm S}_{\bm r} = 0$ for any $\bm r \neq 0$. Hence,
\begin{align*}
    \refE[\intH{\bm t}{\bm x_{1:\ssize}}{S}^2] &=   \int \overline{V}^t    \; \prior(\diff \bm V).
\end{align*}
When $t$ is odd, this is indeed zero due to symmetry. This proves item (1). For even $t$, we recall that $\sqrt{d}\overline{V}$ is sub-Gaussian with variance proxy $1$, and standard moment bounds on sub-Gaussian random variables (see e.g. \citep[Lemma 1.4]{rigollet2015high})
\begin{align*}
    \refE[\intH{\bm t}{\bm x_{1:\ssize}}{S}^2] &=   \int \overline{V}^t    \; \prior(\diff \bm V)  \leq (4t)^{\frac{t}{2}} \cdot  \dim^{-\frac{i}{2}}.  
\end{align*}
This proves item (2). The lower bound in item (3) is obtained by appealing to Fact~\ref{fact: rademacher-moments-lb} (due to \citet{kunisky2019notes}):
\begin{align*}
 \refE[\intH{\bm t}{\bm x_{1:\ssize}}{S}^2] &=   \int \overline{V}^t    \; \prior(\diff \bm V)  \geq (t/e^2)^{\frac{t}{2}} \cdot  \dim^{-\frac{t}{2}}.
 \end{align*}
\item [General $S$ with $\|S\|_\prior \leq 1$:]
\end{description}
Since $\|S\|_\prior  \leq 1$, we know that $\sum_{\bm r} \hat{S}_{\bm r}^2 \leq 1$. When $\ip{S}{1}_\prior = 0$, one additionally has $\hat{S}_{\bm 0} = 0$. Hence,
\begin{align*}
     \sup_{S: \|S\|_\pi \leq 1} \refE[ \intH{\bm t}{\bm x_{1:\ssize}}{S}^2] & \leq \max_{\bm r \in \{0,1\}^\dim}  \int \overline{V}^t   \cdot \bm V^{\bm r} \; \prior(\diff \bm V), \\
     \sup_{\substack{S: \|S\|_\pi \leq 1\\ \ip{S}{1}_\prior = 0}} \refE[ \intH{\bm t}{\bm x_{1:\ssize}}{S}^2] & \leq \max_{\substack{\bm r \in \{0,1\}^\dim\\ \|\bm r\|_1 \geq 1}}  \int \overline{V}^t   \cdot \bm V^{\bm r} \; \prior(\diff \bm V).
\end{align*}
The right hand sides of the above equations have been analyzed in Lemma~\ref{lemma: rademacher-moments}. Appealing to this result immediately yield claim (5) and (6). 
\end{proof}

\subsubsection{Proof of Lemma~\ref{lemma: norm-integrated-hermite-highdegree}} \label{sec: ngca-harmonic-integrated-highdegree-proof}
\begin{proof}[Proof of Lemma~\ref{lemma: norm-integrated-hermite-highdegree}]
Let $\|\bm t \|_1 = t$. Recall that in the proof of Lemma~\ref{lemma: norm-integrated-hermite}, we showed:
\begin{align*}
    \sup_{S: \|S\|_\pi \leq 1} \refE[ \intH{\bm t}{\bm x_{1:\ssize}}{S}^2] & \leq \max_{\bm r \in \{0,1\}^\dim}  \int \overline{V}^t   \cdot \bm V^{\bm r} \; \prior(\diff \bm V).
\end{align*}
Hence, using the triangle inequality and the fact that $|\bm V^{\bm r}| \leq 1$ we obtain,
\begin{align*}
     \sup_{S: \|S\|_\pi \leq 1} \refE[ \intH{\bm t}{\bm x_{1:\ssize}}{S}^2] & \leq \int |\overline{V}|^t   \; \prior(\diff \bm V).
\end{align*}
Let $D_{+}(\bm V)$ denote the number of positive coordinates of $\bm V$. Let $D_{-}(\bm V)$ denote the number of negative coordinates of $\bm V$. We observe that,
\begin{align*}
    |\overline{V}| = 1 - \frac{2}{\dim} \cdot  D_+(\bm V) \wedge D_-(\bm V).
\end{align*}
Hence,
\begin{align*}
    |\overline{V}|^t & = \left( 1 - \frac{2}{\dim} \cdot  D_+(\bm V) \wedge D_-(\bm V) \right)^t \\
    & \leq \exp\left( - \frac{2t}{\dim} \cdot  D_+(\bm V) \wedge D_-(\bm V) \right) \\
    & \leq \exp\left( - \frac{2t}{\dim} \cdot  D_+(\bm V) \right) +   \exp\left( - \frac{2t}{\dim} \cdot   D_-(\bm V) \right).
\end{align*}
Observing that,
\begin{align*}
   \int \exp\left( - \frac{2t}{\dim} \cdot  D_-(\bm V) \right) \prior(\diff \bm V)  = \int \exp\left( - \frac{2t}{\dim} \cdot  D_+(\bm V) \right) \prior(\diff \bm V) &= \left( \frac{1 + e^{-2t/\dim}}{2} \right)^\dim \\
   & = \left( 1  - \frac{(1-e^{-2t/\dim})}{2} \right)^\dim \\
   & \leq \exp\left( -  \frac{(1-e^{-2t/\dim})}{2} \cdot \dim  \right).
\end{align*}
Hence,
\begin{align*}
    \sup_{S: \|S\|_\pi \leq 1} \refE[ \intH{\bm t}{\bm x_{1:\ssize}}{S}^2] & \leq 2\exp\left( -  \frac{(1-e^{-2t/\dim})}{2} \cdot \dim  \right),
\end{align*}
as claimed.
\end{proof}

\subsubsection{Proof of Lemma~\ref{lemma: integrated-hermite-hypercontractivity}} \label{sec: ngca-harmonic-integrated-hypercontractivity-proof}

\begin{proof}[Proof of Lemma~\ref{lemma: integrated-hermite-hypercontractivity}] Note that the result for $q=2$ follows from the discussion preceding this lemma. Hence we focus on proving the inequality when $q \geq 2$. Recalling Definition~\ref{def: integrated-Hermite}, we have
\begin{align*}
    \intH{\bm t}{\bm x_{1:\ssize}}{S} \explain{def}{=} \int \left( \prod_{i=1}^\ssize H_{t_i} \left( \frac{\ip{\bm x_i}{\bm V}}{\sqrt{\dim}} \right) \right) \cdot S(\bm V) \; \prior(\diff \bm V)
\end{align*}
Observe that for any fixed $\bm V$ and any $i \in [\ssize]$,
\begin{align*}
   H_{t_i} \left( \frac{\ip{\bm x_i}{\bm V}}{\sqrt{\dim}} \right),
\end{align*}
can be written as a homogeneous polynomial in $\bm x_{i}$ (see Fact~\ref{fact: hermite-projection-property}) with degree $t_i$.  Since,
\begin{align*}
    \intH{\bm t}{\bm x_{1:\ssize}}{S} = \int \left( \prod_{i=1}^\ssize H_{t_i} \left( \frac{\ip{\bm x_i}{\bm V}}{\sqrt{\dim}} \right) \right) \cdot S(\bm V)
\end{align*}
is a weighted linear combination of such polynomials, it must have a representation of the form:
\begin{align*}
     \intH{\bm t}{\bm x_{1:\ssize}}{S}  = \sum_{\substack{\bm c_{1:\ssize} \in \W^\dim\\\|\bm c_i\|_1 = t_i}}\beta(\bm c_{1:\ssize}; S) \cdot {H_{\bm c_1}(\bm x_1)\cdot H_{\bm c_2}(\bm x_2) \dotsb \cdot  H_{\bm c_\ssize}(\bm x_\ssize) },
\end{align*}
for some coefficients $\beta(\bm c_{1:\ssize}; S)$. While these coefficients can be computed, we will not need their exact formula for our discussion. Hence,
\begin{align*}
     \sum_{\bm t \in \W^\ssize} \alpha_{\bm t} \intH{\bm t}{\bm x_{1:\ssize}}{S}  =  \sum_{\bm t \in \W^\ssize} \sum_{\substack{\bm c_{1:\ssize} \in \W^\dim\\\|\bm c_i\|_1 = t_i}} \alpha_{\bm t} \cdot \beta(\bm c_{1:\ssize}; S) \cdot {H_{\bm c_1}(\bm x_1)\cdot H_{\bm c_2}(\bm x_2) \dotsb \cdot  H_{\bm c_\ssize}(\bm x_\ssize) }
\end{align*}
By Gaussian Hypercontractivity (Fact~\ref{fact: hypercontractivity}),
\begin{align*}
     \left\| \sum_{\bm t \in \W^\ssize} \alpha_{\bm t} \intH{\bm t}{\bm x_{1:\ssize}}{S} \right\|_q^2 & \leq  \sum_{\bm t \in \W^\ssize}  \alpha_{\bm t}^2\cdot \sum_{\substack{\bm c_{1:\ssize} \in \W^\dim\\\|\bm c_i\|_1 = t_i}} (q-1)^{\|\bm c_1\|_1 + \|\bm c_2\|_1 + \dotsb + \|\bm c_\ssize\|_1} \cdot \beta(\bm c_{1:\ssize}; S)^2 \\
     & = \sum_{\bm t \in \W^\ssize} (q-1)^{\|\bm t\|_1} \cdot  \alpha_{\bm t}^2\cdot \sum_{\substack{\bm c_{1:\ssize} \in \W^\dim\\\|\bm c_i\|_1 = t_i}}   \beta(\bm c_{1:\ssize}; S)^2
\end{align*}
Observing that,
\begin{align*}
    \refE[\intH{\bm t}{\bm x_{1:\ssize}}{S}^2] = \sum_{\substack{\bm c_{1:\ssize} \in \W^\dim\\\|\bm c_i\|_1 = t_i}} \beta(\bm c_{1:\ssize})^2,
\end{align*}
yields the claim of the lemma.
\end{proof}

\section{Proofs for Canonical Correlation Analysis} \label{appendix:cca}
\subsection{Setup}
This appendix is devoted to the proof of Proposition \ref{prop:info-bound-cca}, the information bound for the distributed $k$-CCA problem.  Recall that in the distributed $k$-CCA problem:
\begin{enumerate}
    \item The unknown rank-$1$ cross moment tensor $\bm V = \sqrt{\dim^k} \cdot \bm v_1 \otimes \bm v_2 \otimes \dotsb \otimes \bm v_k$ (the parameter of interest) is drawn from the prior $\pi$:
    \begin{align*}
     \bm V \sim \prior \explain{def}{=} \unif{\{ \sqrt{d^k} \cdot \bm e_{i_1}\otimes \bm e_{i_2} \dotsb \otimes \bm e_{i_k}: \; i_{1:k} \in [\dim]\}}.
 \end{align*}
    \item A dataset consisting of $\ssize = \mach \batch$ samples is drawn i.i.d. from $\dmu{\bm V}$, where $\dmu{\bm V}$ is the distribution of a single sample from the $k$-CCA problem.  Recall that for $\bm x = (\mview{\bm x}{1}, \mview{\bm x}{2}, \dotsc, \mview{\bm x}{k}) \in \R^{k\dim}$ and $\bm V = \sqrt{\dim^k} \cdot \bm v_1 \otimes \bm v_2 \otimes \dotsb \otimes \bm v_k$, $\dmu{\bm V}$ was defined using its likelihood ratio with respect to the Gaussian measure $\refmu = \gauss{\bm 0}{\bm I_{k\dim}}$:
\begin{subequations} \label{eq:cca-llr-appendix}
 \begin{align}
     \frac{\diff \dmu{\bm V}}{\diff \refmu}(\bm x) \explain{def}{=} 1 + \frac{\lambda}{\lambda_k} \cdot \sgn\left(\ip{\mview{\bm x}{1}}{\bm v_1}\right)\cdot \sgn\left(\ip{\mview{\bm x}{2}}{\bm v_2}\right) \dotsm \sgn\left(\ip{\mview{\bm x}{k}}{\bm v_k}\right),
 \end{align}
 where, 
 \begin{align}
     \lambda_k \explain{def}{=} \left(\frac{2}{\pi} \right)^{\frac{k}{2}} = (\E |Z|)^{\frac{k}{2}}, \; Z \sim \gauss{0}{1}.
 \end{align}
 \end{subequations}
\item This dataset is divided among $\mach$ machines with $\batch$ samples per machine. We denote the dataset in one machine by $\bm X_i \in \R^{\dim \times \batch}$, where,
\begin{align*}
    \bm X_i = \begin{bmatrix} \bm x_{i1} & \bm x_{i2} &\dots &\bm x_{i\batch} \end{bmatrix},
\end{align*}
with $\bm x_{ij} \explain{i.i.d.}{\sim} \dmu{\bm V}$.
\item The execution of a distributed estimation protocol with parameters $(\mach, \batch, \budget)$ results in a transcript $\bm Y \in \{0,1\}^{\mach \budget}$ written on the blackboard.
\end{enumerate}
The information bound stated in Proposition~\ref{prop:info-bound-cca} is obtained using the general information bound given in Proposition~\ref{prop: main_hellinger_bound} with the following choices: 
\begin{description}
\item [Choice of $\refmu$ and $\nullmu$: ] We set $\refmu = \nullmu = \gauss{\bm 0}{\bm I_{k\dim}}$. That is, under $\refmu = \nullmu$, $\bm x_{ij} \explain{i.i.d.}{\sim} \gauss{\bm 0}{\bm I_\dim}$ for any $i \in [\mach], \; j \in [\batch]$.
\item [Choice of $\goodevnt$: ] We choose the event $\goodevnt$ as the unrestricted sample space $\goodevnt = \R^{k \dim \times \batch}$. Since $\refmu = \nullmu$, this choice of $\goodevnt$ satisfies the requirements of Proposition \ref{prop: main_hellinger_bound}. 
\end{description}
The proof of Proposition \ref{prop:info-bound-cca} is presented in the following section.
\subsection{Proof of Proposition \ref{prop:info-bound-cca}}
We can control $\MIhell{\bm V}{\bm Y}$ using the information bound provided in Proposition \ref{prop: main_hellinger_bound}:
\begin{align} \label{eq: hellinger-recall-cca}
     \frac{\MIhell{\bm V}{\bm Y}}{\consthell} & \explain{}{\leq} \sum_{i=1}^m   \refE\left[  \int \left({\refE\left[ \left( \frac{\diff \dmu{\bm V}}{\diff \refmu} (\bm X_i) - 1 \right)   \bigg| \bm Y, (\bm X_j)_{j \neq i}  \right]}\right)^2 \prior(\diff \bm V) \right]. 
\end{align}
Hence, we need to analyze:
\begin{align*}
     \refE\left[\int \left({\refE\left[ \left( \frac{\diff \dmu{\bm V}}{\diff \refmu} (\bm X_i) -1 \right) \bigg| \bm Y,  (\bm X_j)_{j \neq i}  \right]}\right)^2 \prior(\diff \bm V) \right],
\end{align*}
For any $\bm X \in \R^{k\dim \times \batch}$, $\bm X = [\bm x_1 \; \bm x_2 \; \cdots \; \bm x_{\batch}]$, $S \subset [\batch]$, we introduce the notation,
\begin{align*}
    \dscore{\bm V}(\bm X_S) &\explain{def}{=} \prod_{i \in S}  \left( \frac{\diff \dmu{\bm V}}{\diff \refmu} (\bm x_i) - 1\right).
\end{align*}
In the special case when $S = \{i\}$, we will use the simplified notation $\dscore{\bm V}(\bm x_i)$. We consider the following decomposition: For any $\bm X \in \R^{k\dim \times \batch}$, $\bm X = [\bm x_1 \; \bm x_2 \; \cdots \; \bm x_{\batch}]$,
\begin{align*}
    \frac{\diff \dmu{\bm V}}{\diff \refmu} (\bm X) - 1 & = \prod_{\ell=1}^\batch \left(  1 + \frac{\diff \dmu{\bm V}}{\diff \refmu} (\bm x_\ell) - 1\right)  - 1 \\
    & = \underbrace{\sum_{\ell=1}^\batch \dscore{\bm V}(\bm x_\ell)}_{\text{Additive Term}} +  \underbrace{\sum_{\substack{S \subset [\batch], \; |S| \geq 2}}  \dscore{\bm V}(\bm X_S) }_{\text{Non Additive Term}}.
\end{align*}
With this decomposition, using the elementary inequality $(a+b)^2 \leq 2 a^2 + 2 b^2$, we obtain,
\begin{align} \label{eq: cca-conditional-expectation-decomposition}
     \refE\left[\int \left({\refE\left[ \left( \frac{\diff \dmu{\bm V}}{\diff \refmu} (\bm X_i) - 1 \right) \bigg| \bm Y,  (\bm X_j)_{j \neq i}  \right]}\right)^2 \prior(\diff \bm V) \right] & \leq 2 \cdot (\mathsf{I}) + 2 \cdot (\mathsf{II}), 
\end{align}
where,
\begin{align*}
    \mathsf{I} &\explain{def}{=}  \refE\left[\int \left({\refE\left[ \sum_{\ell=1}^\batch \dscore{\bm V}(\bm x_{i\ell})    \bigg| \bm Y,  (\bm X_j)_{j \neq i}  \right]}\right)^2 \prior(\diff \bm V) \right], \\
    \mathsf{II} &\explain{def}{=} \refE\left[\int \left({\refE\left[ \sum_{\substack{S \subset [\batch], \; |S| \geq 2}}  \dscore{\bm V}( (\bm X_i)_S)   \bigg| \bm Y,  (\bm X_j)_{j \neq i}  \right]}\right)^2 \prior(\diff \bm V) \right].
\end{align*}

In order to control the term $(\mathsf{II})$, we apply Jensen's Inequality:
\begin{align*}
     \mathsf{II} &\explain{}{\leq} \int \refE\left[ \bigg| \sum_{\substack{S \subset [\batch], \; |S| \geq 2}}  \dscore{\bm V}( (\bm X_i)_S) \bigg|^2   \right] \;  \prior(\diff \bm V).
\end{align*}
The following lemma analyzes the above upper bound on $(\mathsf{II})$.
\begin{lemma} \label{lemma: cca-nonadditive} Let $\bm X = [\bm x_1 \; \bm x_2 \; \dots \; \bm x_\batch]$ where $\bm x_i \explain{i.i.d.}{\sim} \gauss{\bm 0}{\bm I_{k\dim}}$.  Suppose that $n\lambda^2/\lambda_k^2 \leq 1/2$. Then,
\begin{align*}
    \refE\left[ \bigg| \sum_{\substack{S \subset [\batch], \; |S| \geq 2}}  \dscore{\bm V}( (\bm X)_S)  \bigg|^2   \right] & \leq 2 \left(\frac{\batch\lambda^2}{\lambda_k^2}\right)^2, 
\end{align*}
where $\lambda_k$ is as defined in \eqref{eq:cca-llr-appendix}. 
\end{lemma}
\begin{proof}[Proof of Lemma \ref{lemma: cca-nonadditive}]
We have
\begin{align*}
     \refE\left[ \bigg| \sum_{\substack{S \subset [\batch], \; |S| \geq 2}}  \dscore{\bm V}( (\bm X)_S)  \bigg|^2   \right] & =  \sum_{\substack{S_1,S_2 \subset [\batch] \\ |S_1| \geq 2, |S_2| \geq 2}}  \refE[\dscore{\bm V}( (\bm X)_{S_1}) \cdot \dscore{\bm V}( (\bm X)_{S_2})]
\end{align*}
Recall that,
\begin{align*}
    \dscore{\bm V}(\bm X_S) &\explain{def}{=} \prod_{i \in S}  \left( \frac{\diff \dmu{\bm V}}{\diff \refmu} (\bm x_i) - 1\right).
\end{align*}
We observe that, $\bm x_{1:\batch}$ are independent and,
\begin{align*}
    \refE\left[ \frac{\diff \dmu{\bm V}}{\diff \refmu} (\bm x_i) - 1\right] = 0.
\end{align*}
Hence if $S_1 \neq S_2$, $\refE[\dscore{\bm V}( (\bm X)_{S_1}) \cdot \dscore{\bm V}( (\bm X)_{S_2})] = 0$. This gives us:
\begin{align*}
     \refE\left[ \bigg| \sum_{\substack{S \subset [\batch], \; |S| \geq 2}}  \dscore{\bm V}( (\bm X)_S)  \bigg|^2   \right] & =  \sum_{\substack{S \subset [\batch] \; |S| \geq 2}}  \refE[\dscore{\bm V}( (\bm X)_{S})^2].
\end{align*}
We can compute:
\begin{align*}
     \refE[\dscore{\bm V}( (\bm X)_{S})^2] & = \left( \refE\left[ \left( \frac{\diff \dmu{\bm V}}{\diff \refmu} (\bm x) - 1 \right)^2 \right] \right)^{|S|}, \; \bm x \sim \gauss{0}{1}, \\
     & \explain{(a)}{=} \left(\frac{\lambda}{\lambda_k}\right)^{2|S|}.
\end{align*}
In step (a), we recalled the formula for the likelihood ratio from \eqref{eq:cca-llr-appendix}. Hence, we have
\begin{align*}
    \refE\left[ \bigg| \sum_{\substack{S \subset [\batch], \; |S| \geq 2}}  \dscore{\bm V}( (\bm X)_S)  \bigg|^2   \right] & \leq  \sum_{\substack{S \subset [\batch] \; |S| \geq 2}}   \left(\frac{\lambda^2}{\lambda_k^2}\right)^{|S|} \\
    & =  \sum_{s=2}^\batch \binom{\batch}{s}  \left(\frac{\lambda^2}{\lambda_k^2}\right)^{s} \\
    & \leq \sum_{s=2}^\batch \left(\frac{\batch\lambda^2}{\lambda_k^2}\right)^{s}.
\end{align*}
The assumption $n\lambda^2/\lambda_k^2 \leq 1/2$ guarantees that the above sum is dominated by a Geometric series, which immediately yields the claim of the lemma. 
\end{proof}

In order to control the term $(\mathsf{I})$, we recall that when $\bm V \sim \prior$, we have
\begin{align*}
    \bm V = \sqrt{\dim^k} \cdot \bm e_{j_1} \otimes \bm e_{j_2} \dotsb \otimes \bm e_{j_k}, \; j_{1:k} \explain{i.i.d.}{\sim} \unif{[d]}.  
\end{align*}
Consequently, for any $\bm x = (\mview{\bm x}{1}, \mview{\bm x}{2}, \dotsc , \mview{\bm x}{k}) \in \R^{k \dim}$,
\begin{align*}
    \dscore{\bm V}(\bm x) & = \frac{\lambda}{\lambda_k} \cdot \sgn(\mview{x}{1}_{j_1})\cdot \sgn(\mview{x}{2}_{j_2}) \dotsb \cdot \sgn(\mview{x}{k}_{j_k}).
\end{align*}
For each machine $i \in [\mach]$ we can define $\batch$ i.i.d.\ tensors $\bm T_{i1}, \bm T_{i2}, \dotsc, \bm T_{i\batch}$ as:
\begin{align*}
    \bm T_{i\ell} &\explain{def}{=} \sgn(\mview{\bm x}{1}_{i\ell})\otimes \sgn(\mview{\bm x}{2}_{i\ell}) \dotsb \otimes \sgn(\mview{\bm x}{k}_{i\ell}),
\end{align*}
where the $\sgn(\cdot)$ operation is understood to act entry-wise on a vector $\bm v \in \R^{\dim}$ to produce another vector $\sgn(\bm v) \in \{\pm 1\}^\dim$. With this notation in place, we observe that we can rewrite $\mathsf{(I)}$ as:
\begin{align*}
    \mathsf{(I)} &\explain{}{=}  \frac{\lambda^2}{\lambda_k^2} \cdot \frac{1}{\dim^k} \cdot \refE\left[\left\|{\refE\left[ \sum_{\ell=1}^\batch \bm T_{i\ell}   \bigg| \bm Y,  (\bm X_j)_{j \neq i}  \right]}\right\|^2 \right],
\end{align*}
Linearizing $\|\cdot \|$ we obtain (c.f. Lemma \ref{lemma: linearization}):
\begin{align*}
    \left\|{\refE\left[ \sum_{\ell=1}^\batch \bm T_{i\ell}   \bigg| \bm Y,  (\bm X_j)_{j \neq i}  \right]}\right\| & = \sup_{\substack{\bm S \in \tensor{\R^\dim}{k}\\ \|\bm S\| \leq 1}} \left( \refE\left[  \sum_{\ell=1}^\batch \ip{\bm T_{i\ell}}{\bm S}  \bigg| \bm Y,  (\bm X_j)_{j \neq i}  \right] \right).
\end{align*}
We will apply the Geometric Inequality framework (Proposition \ref{prop: geometric inequality}) to control the above conditional expectation. In order to do so, we need to understand the concentration behavior of the random variable:
\begin{align*}
   \sum_{\ell=1}^\batch \ip{\bm T_{i\ell}}{\bm S}.
\end{align*}
This is the subject of the following lemma.

\begin{lemma} \label{lemma: additive-term-cca} Let $\bm T, \bm T_{1}, \dotsc, \bm T_{\batch}$ be i.i.d.\ random tensors distributed as:
\begin{align*}
    \bm T = \sgn(\mview{\bm x}{1})\otimes \sgn(\mview{\bm x}{2}) \dotsb \otimes \sgn(\mview{\bm x}{k}),
\end{align*}
where $\bm x = (\mview{\bm x}{1}, \mview{\bm x}{2} \dotsc, \mview{\bm x}{k}) \sim \gauss{\bm 0}{\bm I_{k\dim}}$. Then, we have, for any  $S\in \tensor{\bm \R^\dim}{k}$ with $\|\bm S\|\leq 1$ and any $\zeta \in \R$ with $|\zeta| \leq  \dim^{-\frac{k}{2}}/2$,
\begin{align*}
    \ln \refE \exp\left( \zeta  \sum_{\ell=1}^\batch \ip{\bm T_{\ell}}{\bm S}\right) \leq  \batch \zeta^2. 
\end{align*}
Furthermore,
\begin{align*}
    \left\|\sum_{\ell=1}^\batch \ip{\bm T_{\ell}}{\bm S} \right\|_4 & \leq \sqrt{3^k \batch}.
\end{align*}
where,
\begin{align*}
    \left\|\sum_{\ell=1}^\batch \ip{\bm T_{\ell}}{\bm S} \right\|_4^4 \explain{def}{=} \refE \left[ \left( \sum_{\ell=1}^\batch \ip{\bm T_{\ell}}{\bm S}\right)^4 \right]
\end{align*}
\end{lemma}
\begin{proof}
The first claim follows from Bernstein's Inequality (Fact \ref{fact: latala}) by observing that $\ip{\bm T_i}{\bm S} \leq \|\bm T_i\| \|\bm S\|  = \sqrt{\dim^k}$ and that $\refE\ip{\bm T_i}{\bm S} = 0, \; \refE\ip{\bm T_i}{\bm S}^2 = 1$. In order to obtain the moment bound, we observe that:
\begin{align*}
    \sum_{\ell=1}^\batch \ip{\bm T_{\ell}}{\bm S},
\end{align*}
is a polynomial of degree $k$ in the $k \dim \batch$ i.i.d.\ $\unif{\{\pm 1\}}$ random variables $(\mview{{x}_\ell}{j})_{i}$ where $j \in [k], \; \ell \in [\batch], \; i \in [\dim]$. Hence by Boolean Hypercontractivity (see for e.g. \citet[Theorem 9.21]{o2014analysis}) we have
\begin{align*}
      \left\|\sum_{\ell=1}^\batch \ip{\bm T_{\ell}}{\bm S} \right\|_4^2 & \leq 3^{k} \cdot \refE \left[ \left( \sum_{\ell=1}^\batch \ip{\bm T_{\ell}}{\bm S} \right)^2\right] = 3^k \batch.
\end{align*}
\end{proof}
We can now use Geometric Inequalities (Proposition \ref{prop: geometric inequality}) to control:
\begin{align*}
    \left\|{\refE\left[ \sum_{\ell=1}^\batch \bm T_{i\ell}   \bigg| \bm Y =\bm y,  (\bm X_j)_{j \neq i}  \right]}\right\| & = \sup_{\substack{\bm S \in \tensor{\R^\dim}{k}\\ \|\bm S\| \leq 1}} \left( \refE\left[  \sum_{\ell=1}^\batch \ip{\bm T_{i\ell}}{\bm S}  \bigg| \bm Y = \bm y,  (\bm X_j)_{j \neq i}  \right] \right).
\end{align*}
We consider two cases depending upon whether $\bm y \in \mathcal{R}_{\mathsf{rare}}^{(i)}$ or $\bm y  \in \mathcal{R}_{\mathsf{freq}}^{(i)}$, where,
\begin{align*}
     \mathcal{R}_{\mathsf{rare}}^{(i)} &\explain{def}{=} \left\{ \bm y \in \{0,1\}^{m\budget} : 0 <   \refP(\bm Y = \bm y | (\bm X_j)_{j \neq i}) \leq 4^{-\budget} \right\}, \\
     \mathcal{R}_{\mathsf{freq}}^{(i)} &\explain{def}{=} \left\{ \bm y \in \{0,1\}^{m\budget} :  \refP(\bm Y = \bm y | (\bm X_j)_{j \neq i}) > 4^{-\budget} \right\}.
\end{align*}

\begin{description}
\item [Case 1: $\bm y \in \mathcal{R}_{\mathsf{rare}}^{(i)}$.] In this situation we apply the moment version of the Geometric Inequality (Proposition~\ref{prop: geometric inequality}, item (1)) with $q = 4$. Using the moment estimate in Lemma \ref{lemma: additive-term-cca}, we obtain,
\begin{align} 
    \left\|{\refE\left[ \sum_{\ell=1}^\batch \bm T_{i\ell}   \bigg| \bm Y =\bm y,  (\bm X_j)_{j \neq i}  \right]}\right\|& \leq \frac{\sqrt{3^k \cdot \batch}}{\refP(\bm Y = \bm y | (\bm X_j)_{j \neq i} )^{\frac{1}{4}}}. \label{eq: cca-geometric-rare-event}
\end{align}
\item [Case 2: $\bm y \in \mathcal{R}_{\mathsf{freq}}^{(i)}$.] I
n this situation we apply the m.g.f. version of the Geometric Inequality (Proposition~\ref{prop: geometric inequality}, item (2)). Using the m.g.f. estimate in Lemma \ref{lemma: additive-term-cca}, we obtain, for any $0 < \zeta \leq \dim^{-\frac{k}{2}}/2$,
\begin{align*} 
     \left\|{\refE\left[ \sum_{\ell=1}^\batch \bm T_{i\ell}   \bigg| \bm Y =\bm y,  (\bm X_j)_{j \neq i}  \right]}\right\|  &\leq    {n \zeta } + \frac{1}{\zeta} \ln \frac{1}{\refP(\bm Y = \bm y | (\bm X_j)_{j \neq i} )},
\end{align*}
 We set:
\begin{align*}
    \zeta^2 & = \frac{1}{\batch} \cdot \ln \frac{1}{\refP(\bm Y = \bm y | (\bm X_j)_{j \neq i} )} \leq \frac{\budget \cdot \ln(4)}{\batch}.
\end{align*}
If,
\begin{align} \label{eq: cca-batch-assumption}
    n & \geq 2 \ln(4) \cdot \budget \cdot \dim^{\frac{k}{2}},
\end{align}
then this choice is valid, i.e. $\zeta \leq  \dim^{-\frac{k}{2}}/2$. Hence,
\begin{align}\label{eq: cca-geometric-typical-event}
      \left\|{\refE\left[ \sum_{\ell=1}^\batch \bm T_{i\ell}   \bigg| \bm Y =\bm y,  (\bm X_j)_{j \neq i}  \right]}\right\|^2  &\leq 4 \cdot \batch \cdot \ln \frac{1}{\refP(\bm Y = \bm y | (\bm X_j)_{j \neq i} )}.
\end{align}
\end{description}

With these estimates, we can control the term $(\mathsf{I})$, which we decompose as follows:
\begin{align*}
    (\mathsf{I}) & =  \frac{\lambda^2}{\lambda_k^2} \cdot \frac{1}{\dim^k} \cdot \refE\left[\left\|{\refE\left[ \sum_{\ell=1}^\batch \bm T_{i\ell}   \bigg| \bm Y,  (\bm X_j)_{j \neq i}  \right]}\right\|^2 \right] \\
    & = \frac{\lambda^2}{\lambda_k^2} \cdot \frac{1}{\dim^k} \cdot \left(\mathsf{(Ia)} + \mathsf{(Ib)} \right), \\
    \mathsf{(Ia)} &\explain{def}{=}  \refE\left[ \sum_{\bm y \in \mathcal{R}_{\mathsf{rare}}^{(i)} } \refP(\bm Y = \bm y | (\bm X_j)_{j \neq i} ) \cdot \left\|{\refE\left[ \sum_{\ell=1}^\batch \bm T_{i\ell}   \bigg| \bm Y,  (\bm X_j)_{j \neq i}  \right]}\right\|^2 \right], \\
    \mathsf{(Ib)} &\explain{def}{=}  \refE\left[ \sum_{\bm y \in \mathcal{R}_{\mathsf{freq}}^{(i)} } \refP(\bm Y = \bm y | (\bm X_j)_{j \neq i} ) \cdot \left\|{\refE\left[ \sum_{\ell=1}^\batch \bm T_{i\ell}   \bigg| \bm Y,  (\bm X_j)_{j \neq i}  \right]}\right\|^2 \right].
\end{align*}
In order to control $(\mathsf{Ia})$, we rely on the estimate \eqref{eq: cca-geometric-rare-event}:
\begin{align*}
     \mathsf{(Ia)} &\leq 3^k \cdot \batch \cdot \refE\left[ \sum_{\bm y \in \mathcal{R}_{\mathsf{rare}}^{(i)} } \refP(\bm Y = \bm y | (\bm X_j)_{j \neq i} )^{\frac{1}{2}}\right] \\
     & \leq   3^k \cdot \batch \cdot 2^{-b} \cdot \refE[| \mathcal{R}_{\mathsf{rare}}^{(i)}|].
\end{align*}
Since we assume the communication protocol to be deterministic conditioned on $(\bm X_j)_{j \neq i}$, all but $\budget$ bits of $\bm Y$ are fixed. Consequently, $| \mathcal{R}_{\mathsf{rare}}| \leq 2^\budget$. Hence,
\begin{align*}
    \mathsf{(Ia)} &\leq 3^k \cdot \batch.
\end{align*}
In order to control $(\mathsf{Ib})$, we rely on the estimate \eqref{eq: cca-geometric-typical-event}:
\begin{align*}
     \mathsf{(Ib)} &\leq 4 \batch \cdot \refE\left[ \sum_{\bm y \in \mathcal{R}_{\mathsf{freq}}^{(i)} } \refP(\bm Y = \bm y | (\bm X_j)_{j \neq i} ) \cdot \ln \frac{1}{\refP(\bm Y = \bm y | (\bm X_j)_{j \neq i} )} \right].
\end{align*}
Since we assume the communication protocol to be deterministic conditioned on $(\bm X_j)_{j \neq i}$, all but $\budget$ bits of $\bm Y$ are fixed. Hence conditioned on $(\bm X_j)_{j \neq i}$, the random vector $(\bm Y)$ has a support size of at most $2^{\budget}$. The maximum entropy distribution on a given set $S$ is the uniform distribution, which attains an entropy of $\ln |S|$. Hence,
\begin{align*}
    \sum_{(\bm y,z) \in \{0,1\}^{\budget+1}} \refP(\bm Y = \bm y | (\bm X_j)_{j \neq i} ) \cdot \ln \frac{1}{\refP(\bm Y = \bm y  | (\bm X_j)_{j \neq i} )} & \leq \budget \cdot \ln(2)
\end{align*}
This yields the estimate,
\begin{align*}
    \mathsf{(Ib)} &\leq 4 \ln(2) \cdot \budget \cdot \batch.
\end{align*}
Combining the estimates on the terms $\mathsf{Ia}, \mathsf{Ib}$ we obtain, $(\mathsf{I}) \leq  C_k \cdot \budget \cdot \batch \cdot \lambda^2 /\dim^k$, where $C_k$ is a constant depending only on $k$. Substituting this estimate on $(\mathsf{I})$ and the estimate on $(\mathsf{II})$ obtained in Lemma \ref{lemma: cca-nonadditive} in \eqref{eq: cca-conditional-expectation-decomposition}, we obtain,
\begin{align*} 
     \refE\left[\int \left({\refE\left[ \left( \frac{\diff \dmu{\bm V}}{\diff \refmu} (\bm X_i) - 1\right) \bigg| \bm Y, (\bm X_j)_{j \neq i}  \right]}\right)^2 \prior(\diff \bm V) \right] & \leq C_k \cdot \left( \frac{\budget \cdot \batch \cdot \lambda^2}{\dim^k} + \batch^2 \cdot \lambda^4 \right).
\end{align*}
Plugging the above bound in \eqref{eq: hellinger-recall-cca} we obtain,
\begin{align*} 
     \frac{\MIhell{\bm V}{\bm Y}}{\consthell} & \explain{}{\leq}  C_k \cdot \left( \frac{\budget \cdot \mach \cdot \batch \cdot \lambda^2}{\dim^k} + \mach \cdot \batch^2 \cdot \lambda^4 \right).
\end{align*}
This is exactly the information bound claimed in Proposition \ref{prop:info-bound-cca}.

\section{Discussion for $k$-Tensor PCA with odd $k$} \label{appendix:odd-case}
When $k$ is odd, our computational lower bound for $k$-TPCA (Theorem~\ref{thm:tpca})  shows that any iterative algorithm which uses $\ssize$ samples, makes $T$ passes through the dataset, and has a memory state of size $\state$ bits fails to solve $k$-TPCA if:
\begin{align} \label{eq:lb}
(\ssize \lambda^2) \cdot \iters \cdot \state & \ll \sqrt{\dim^{k+1}}.
\end{align}
On the other hand, there are iterative algorithms \citep{anandkumar2017homotopy,biroli2019iron} for $k$-TPCA with odd $k$ with a resource profile:
\begin{align} \label{eq:ub}
\ssize \lambda^2 \asymp \sqrt{\dim^k} \cdot \polylog(\dim), \;  \iters  =  \polylog(\dim), \; s \asymp \dim  \cdot \polylog(\dim),
\end{align}
which succeed in estimating the unknown signal vector $\bm V$ \emph{consistently} in the sense that the estimator $\hat{\bm V}$ computed by these algorithms satisfies:
\begin{align} \label{eq:consistency}
\frac{\langle{{\bm V}},{\hat{\bm V}}\rangle^2}{\|\bm V\|^2 \|\hat{\bm V}\|^2} \rightarrow 1 \text{ as $\dim \rightarrow \infty$}.
\end{align}
We believe that the $\sqrt{\dim}$ gap between the resource lower bound in \eqref{eq:lb} and the upper bound in \eqref{eq:ub} arises due to our use of the Hellinger information. Specifically, our approach relies on showing that the Hellinger Information $\MIhell{\bm V}{\bm Y}$ between the signal vector $\bm V \sim \unif{\{\pm 1\}^\dim}$ and the transcript $\bm Y$ generated by an iterative algorithm which uses too few resources (when run in a distributed setting via the reduction in Fact 1) satisfies:
\begin{align} \label{eq:hell-criterion}
\MIhell{\bm V}{\bm Y} \rightarrow 0  \text{ as $\dim \rightarrow \infty$}.
\end{align}
Due to Fano's Inequality for Hellinger Information (Fact~\ref{fact: fano}, see also Corollary~\ref{coro: fano-stpca} for its instantiation for $k$-TPCA), showing \eqref{eq:hell-criterion} not only rules out consistent estimation (cf. \eqref{eq:consistency}) but yields a stronger-than-desired result that any estimator $\hat{\bm V}$ computed via an iterative algorithm which uses too few resources fails to achieve \emph{better-than-random estimation}:
\begin{align} \label{eq:not-better-than-random}
   \forall \; \epsilon  > 0, \;  \P \left( \frac{\langle{{\bm V}},{\hat{\bm V}}\rangle^2}{\|\bm V\|^2 \|\hat{\bm V}\|^2} \geq \frac{\dim^\epsilon}{\dim} \right) \rightarrow 0 \text{ as $\dim \rightarrow \infty$}.
\end{align}
For better-than-random estimation, the resource lower bound in \eqref{eq:lb} is, in fact, optimal.  Specifically, for any arbitrarily small constant $\epsilon \in (0,1)$, there is an iterative algorithm with the following properties:
\begin{enumerate}
\item The algorithm has a resource profile of:
\begin{align} \label{eq:subset-algo-resources}
\ssize \lambda^2 \asymp \sqrt{\dim^{k+\epsilon}}, \;  \iters  = 1, \; s \asymp \sqrt{\dim^{1+\epsilon}}.
\end{align}
In particular, it uses $\ssize \cdot \iters \cdot \state \asymp \sqrt{\dim^{k+1+2\epsilon}}$ resources, which matches the lower bound in \eqref{eq:lb} upto an arbitrarily small polynomial factor of ${\dim^\epsilon}$. \item The estimator $\hat{\bm V}_{\epsilon}$ computed by the algorithm satisfies:
\begin{align} \label{eq:better-than-random}
\frac{\langle{{\bm V}},{\hat{\bm V}}\rangle^2}{\|\bm V\|^2 \|\hat{\bm V}\|^2} \gtrsim \frac{\dim^\epsilon}{\dim} \text{ with high probability (say, $0.9$)}.
\end{align}
In particular, this estimator works better-than-random and consequently, the Hellinger information between the signal $\bm V \sim \unif{\{\pm 1\}^\dim}$ the transcript $\bm Y$ generated by this algorithm in a distributed setting must satisfy $\MIhell{\bm V}{\bm Y} \gtrsim 1$ to avoid contradicting Fano's Inequality (Corollary~\ref{coro: fano-stpca}).
\end{enumerate}
The existence of an algorithm with the above properties (described below) shows that the resource bound in \eqref{eq:lb} is the best possible lower bound that can be obtained using the approach based on Hellinger information used in our work. In order to improve the lower bound, one would need to use other information measures. A natural approach would be to show that the mutual information $\mathbf{I}_{\mathrm{KL}}(\bm V; \bm Y)$ (based on KL divergence) satisfies $\mathbf{I}_{\mathrm{KL}}(\bm V; \bm Y) \leq c\dim$ for a suitably constant $c$, which would rule out consistent estimation. However, the key challenge in bounding the mutual information $\mathbf{I}_{\mathrm{KL}}(\bm V; \bm Y)$ is that the analog of the ``cut-and-paste property'' \citep{jayram2009hellinger} (see Fact~\ref{hellinger_BB_fact}), which plays a crucial role in proving the general information bound that underlies  all our results (Proposition~\ref{prop: main_hellinger_bound}), is not known for KL divergence. This is why we chose to use Hellinger information in our analysis. 

The algorithm that satisfies the properties \eqref{eq:subset-algo-resources} and \eqref{eq:better-than-random} is as follows. The idea is that since one is only allowed a memory state of size $\state \asymp \sqrt{\dim^{1+\epsilon}} \ll \dim$, one tries to estimate only the first $s$ coordinates of the unknown signal $\bm V$. To do so, one uses $\ssize \asymp \sqrt{\dim^{k+\epsilon}}/\lambda^2$ samples $\bm T_{1:\ssize}$ (where each $\bm T_i = \lambda \cdot \dim^{-\frac{k}{2}} \cdot \bm V^{\otimes k} + \bm W_i$ where $\bm V \sim \unif{\{\pm 1\}^\dim}$ is the unknown signal vector and $\bm W_i$ is the i.i.d. Gaussian noise tensor) to compute the $s \asymp \sqrt{\dim^{1+\epsilon}}$-dimensional statistic $\bm t$ whose entries are given by:
\begin{align*}
\forall \; \ell \in [s], \; t_{\ell} \explain{}{=} \frac{1}{\ssize} \sum_{i=1}^\ssize \bm T_i( \underbrace{\bm I_{\dim}, \bm I_{\dim}, \dotsc, \bm I_{\dim}}_{q \explain{def}{=} (k-1)/2\text{ times}}, \bm e_{\ell}) \explain{def}{=} \frac{1}{\ssize} \sum_{i=1}^\ssize \sum_{j_1, j_2, \dotsc, j_q = 1}^\dim (\bm T_i)_{j_1, j_1, j_2, j_2, \dotsc, j_{q}, j_{q}, \ell}
\end{align*}
In the above display $q = (k-1)/2$ and $\bm e_{1:s}$ are the standard basis vectors of $\R^s$. The above statistic can be computed in a single pass over the data set, so the algorithm satisfies the resource requirement in \eqref{eq:subset-algo-resources}. The final estimator for $\bm V$ is obtained by appending $\dim - \state$ zeros to $\bm t$ to obtain a $\dim$-dimensional vector:
\begin{align*}
\hat{\bm V}_{\epsilon} = (\bm t^\top, \underbrace{0, \dotsc, 0}_{d-s \text{ times}})^\top
\end{align*}
To see why this algorithm yields a better-than-random estimate, we observe that the distribution of the statistic $\bm t$ is given by:
\begin{align*}
\bm t \explain{d}{=} \frac{\lambda}{\sqrt{\dim}} \cdot \bm V_{[s]} + \sqrt{\frac{\dim^{\frac{k-1}{2}}}{\ssize}} \cdot \bm g, \quad \bm g \sim \gauss{\bm 0}{\bm I_{s}}.
\end{align*}
In the above display $\bm V_{[s]} \in \{\pm 1\}^s$ represents the vector obtained by the first $s$ coordinates of $\bm V$. As a result, we have the following lower bound on the dot product:
\begin{align*}
\langle{\hat{\bm V}_\epsilon}, \bm V \rangle & \geq \frac{\lambda \state}{\sqrt{\dim}} - \sqrt{\frac{\dim^{\frac{k-1}{2}}}{\ssize}} \cdot |\langle{\bm g}, {\bm V_{[s]}} \rangle|  \explain{(a)}{\gtrsim} \frac{\lambda \state}{\sqrt{\dim}} - \sqrt{\frac{\dim^{\frac{k-1}{2}}}{\ssize}} \cdot \sqrt{\state} \explain{(b)}{\gtrsim} \lambda\sqrt{\dim^{\epsilon}}.
\end{align*}
In the above display, step (a) follows from the fact that $\langle{\bm g}, \bm V_{[s]} \rangle \sim \gauss{0}{\|\bm V_{[s]}\|^2}$ and hence satisfies $|\langle{\bm g}, \bm V_{[s]} \rangle| \lesssim \|\bm V_{[s]}\| = \sqrt{s}$ with high probability. Step (b) uses the fact that $\ssize \lambda^2 \asymp \sqrt{\dim^{k+\epsilon}}, \; s \asymp \sqrt{\dim^{1+\epsilon}}$. We can also upper bound the norm:
\begin{align*}
    \|\hat{\bm V}_\epsilon\| = \|\bm t\| \leq  \frac{\lambda\|\bm V_{[s]}\|}{\sqrt{\dim}} + \sqrt{\frac{\dim^{\frac{k-1}{2}}}{\ssize}} \cdot \|\bm g\| \explain{(a)}{\lesssim} \frac{\lambda\sqrt{s}}{\sqrt{\dim}} +  \sqrt{\frac{\dim^{\frac{k-1}{2}}}{\ssize}} \cdot \sqrt{s} \explain{(b)}{\lesssim} \frac{\lambda \cdot \dim^{\epsilon/4}}{\dim^{1/4}} + \lambda \lesssim \lambda.
\end{align*}
In the above display, step (a) follows by observing that $\|\bm V_{[s]}\| = \sqrt{s}$ and $\|\bm g\| \lesssim \sqrt{s}$ (with high probability). Step (b) uses the fact that $\ssize \lambda^2 \asymp \sqrt{\dim^{k+\epsilon}}, \; s \asymp \sqrt{\dim^{1+\epsilon}}$. Recalling that $\|\bm V\|^2 = \dim$, the above estimates yield the claim in \eqref{eq:better-than-random}. While the above discussion focused on $k$-TPCA, analogous considerations also apply to $k$-NGCA.

\section{Additional Results for Non-Gaussian Component Analysis}
\subsection{Constructions of Non-Gaussian Distributions} 
In this section we provide the proofs for Lemma~\ref{lemma: ngca-mog-construction} and Lemma~\ref{lemma: ngca-construction-boundedLLR}. 

\subsubsection{Proof of Lemma~\ref{lemma: ngca-mog-construction}} \label{appendix: ngca-mog-construction}
The proof of Lemma~\ref{lemma: ngca-mog-construction} relies on the following fact.
\begin{fact} [{\citealp[Section 2.7]{wu2021polynomial}}] \label{fact: gaussian-quadrature} Let $k = 2\ell$ be even. There is a discrete random variable $W$ with support size $\ell$ such that
\begin{align*}
    \E H_i(W) = 0 \quad \forall \;  1\leq i \leq k-1,
\end{align*}
and, $\E H_k(W) = {- \ell !}/\sqrt{k!}.$ Furthermore, $W$ is a bounded random variable $|W| \leq \sqrt{2k+2}$ and is sub-Gaussian with variance proxy $1$. 
\end{fact}
With this fact, we can now provide a proof for Lemma~\ref{lemma: ngca-mog-construction}. 
\begin{proof}[Proof of Lemma~\ref{lemma: ngca-mog-construction}]
Let $W$ be any bounded random variable from Fact~\ref{fact: gaussian-quadrature} with the property that $\E H_i(W) = 0$ for any $1 \leq i \leq k-1$ and $\E H_k(W) \neq 0$. Let $Z \sim \gauss{0}{1}$ be independent of $W$. Define:
\begin{align*}
    \lambda_k \explain{def}{=}|\E Z^k H_k(Z)| \cdot |\E H_k(W)| > 0.
\end{align*}
For any $\lambda \leq \lambda_k$, we claim that the law $\nongauss$ of the random variable $W_\lambda$ defined by
\begin{align*}
    W_\lambda \explain{def}{=} \gamma(\lambda) \cdot W + \sqrt{1- \gamma^2(\lambda)} \cdot Z, \quad  \gamma(\lambda) \explain{def}{=} \left(\frac{\lambda}{\lambda_k}\right)^{\frac{1}{k}}
\end{align*}
is a non-Gaussian measure with the desired properties.

We begin by computing $\E H_t(W_\lambda)$. Recall the generating function for the Hermite Polynomials: for any $x, w \in \R$, we have
\begin{align*}
    e^{x w - \frac{x^2}{2}} & = \sum_{t=0}^\infty \frac{x^t}{\sqrt{t!}} H_t(w).
\end{align*}
In particular:
\begin{align}\label{eq:hermite-differential-formula}
    H_t(w) & = \frac{1}{\sqrt{t!}} \frac{\diff^t}{\diff x^t} e^{x w - \frac{x^2}{2}} \bigg|_{x = 0}.
\end{align}
Hence,
\begin{align*}
    \E H_t(W_\lambda) & =  \frac{1}{\sqrt{t!}} \frac{\diff^t}{\diff x^t} \E e^{x W_\lambda - \frac{x^2}{2}} \bigg|_{x = 0} \\
    & \explain{}{=} \frac{1}{\sqrt{t!}} \frac{\diff^t}{\diff x^t} \E e^{x  \gamma W + x\sqrt{1-\gamma^2} Z- \frac{x^2}{2}} \bigg|_{x = 0} \\
    & =  \frac{1}{\sqrt{t!}} \frac{\diff^t}{\diff x^t} \E e^{x  \gamma W - \frac{x^2 \gamma^2}{2}} \bigg|_{x = 0} \\
    & = \frac{1}{\sqrt{t!}} \E \frac{\diff^t}{\diff x^t}  e^{x  \gamma W - \frac{x^2 \gamma^2}{2}} \bigg|_{x = 0}.
\end{align*}
Applying the differential identity in \eqref{eq:hermite-differential-formula} after making the change of variables $z = \gamma x$, we obtain,
\begin{align*}
    \E H_t(W_\lambda) & = \frac{1}{\sqrt{t!}} \E \frac{\diff^t}{\diff x^t}  e^{x  \gamma W - \frac{x^2 \gamma^2}{2}} \bigg|_{x = 0} = \gamma^t(\lambda) \cdot \E H_t(W).
\end{align*}
Recalling the properties of $W$ stated in Fact~\ref{fact: gaussian-quadrature}, we obtain,
\begin{align} \label{eq: fourier-coeff-nongauss}
    \E  H_t(W_\lambda) & = \begin{cases} 0  & \text{if $t \leq k-1$} ; \\ \frac{\lambda}{\lambda_k} \cdot \E H_k(W) & \text{if $t = k$} . \end{cases}
\end{align}
Furthermore, \citet[Theorem 1]{bonan1990estimates} have shown:
\begin{align*}
    C \explain{def}{=} \sup_{t \in \W} \sup_{w \in \R} \left\{ |H_t(w)| \cdot  e^{-\frac{w^2}{2}} \right\} < \infty. 
\end{align*}
Consequently, we have
\begin{align} \label{eq: fourier-decay}
    |\E H_t(W_\lambda)| & \leq C e^{\frac{\|W\|_\infty^2}{2}} \cdot \left( \frac{\lambda}{\lambda_k} \right)^{\frac{t}{k}}
\end{align}
Using \eqref{eq: fourier-coeff-nongauss} and \eqref{eq: fourier-decay},
we can now establish the desired properties of $\nongauss$:
\begin{enumerate}
    \item By \eqref{eq: fourier-coeff-nongauss}, we see that $\E H_t(Z) = \E H_t(W_\lambda)$ for any $t \leq k-1$.
      This immediately yields $\E Z^t = \E W_\lambda^t$ for any $t \leq k-1$.
      Hence, $\nongauss$ satisfies the Moment Matching Assumption with parameter $k$.

    \item We expand the likelihood ratio in the Hermite basis:
    \begin{align*}
        \frac{\diff \nongauss}{\diff \refmu}(z) &= \sum_{t=0}^\infty \left(\E H_t(Z) \cdot  \frac{\diff \nongauss}{\diff \refmu}(Z) \right) H_t(z) \\
        & = \sum_{t=0}^\infty \E H_t(W_\lambda) \cdot H_t(z) \\
        & = 1+\sum_{t=k}^\infty \E H_t(W_\lambda) \cdot H_t(z).
    \end{align*}
    In the above display, in the last step, we used the fact that $\E H_t(W_\lambda) = 0$ for any $1\leq t \leq k-1$. To verify that the second moment of the likelihood ratio is bounded, we note that,
    \begin{align*} 
        \E\left( \frac{\diff \nongauss}{\diff \refmu}(Z) - 1\right)^2 & = \sum_{t=k}^\infty |\E H_t(W_\lambda)|^2.
    \end{align*}
    Using the estimates in \eqref{eq: fourier-coeff-nongauss} and \eqref{eq: fourier-decay} and the assumption $\lambda/\lambda_k \leq 1/2$, we obtain,
    \begin{align}\label{eq: check-2mom-assump}
        \E\left( \frac{\diff \nongauss}{\diff \refmu}(Z) - 1\right)^2 & \leq \frac{C^2 \cdot e^{\|W\|_\infty^2}}{\lambda_k^2} \cdot \frac{2^{2/k}}{2^{2/k} -1} \cdot \lambda^2.
    \end{align}
    This verifies the Bounded Signal Strength Assumption.

    \item  In order to verify that the likelihood ratio is locally bounded, we begin with the estimate:
    \begin{align*}
        \left|\frac{\diff \nongauss}{\diff \refmu}(z) - 1\right| & \leq \sum_{t=k}^\infty |\E H_t(W_\lambda)| \cdot |H_t(z)|  \leq C e^{\frac{\|W\|_\infty^2}{2}} \sum_{t=k}^\infty \left( \frac{\lambda}{\lambda_k} \right)^{\frac{t}{k}} \cdot |H_t(z)|.
    \end{align*}
    Fact~\ref{fact: hermite-simple-bound} shows that $|H_t(z)| \leq (1+|z|)^t$.
    Hence,
    \begin{align*}
        \left|\frac{\diff \nongauss}{\diff \refmu}(z) - 1\right| & \leq C e^{\frac{\|W\|_\infty^2}{2}} \sum_{t=k}^\infty \left( \frac{\lambda}{\lambda_k} \right)^{\frac{t}{k}} \cdot (1+|z|)^t. %
    \end{align*}
    Under the assumption:
    \begin{align} \label{eq: check-domain}
        \frac{\lambda}{\lambda_k} \cdot (1+ |z|)^k \leq \frac{1}{2},
    \end{align}
    we obtain,
    \begin{align} \label{eq: check-local-bound}
        \left|\frac{\diff \nongauss}{\diff \refmu}(z) - 1\right| & \leq \frac{C e^{\frac{\|W\|_\infty^2}{2}}}{\lambda_k}  \cdot \frac{2^{1/k}}{2^{1/k}-1} \cdot \lambda \cdot (1+|z|)^k.
    \end{align}
    Inspecting \eqref{eq: check-domain} and \eqref{eq: check-local-bound}, we obtain that the Locally Bounded Likelihood Ratio Assumption holds with
    \begin{align*}
        \kappa = k, \quad K =  \frac{C e^{\frac{\|W\|_\infty^2}{2}}}{\lambda_k}  \cdot \frac{2^{1/k}}{2^{1/k}-1}.
    \end{align*}
    \item
      Recall that the monomial $w^k$ can be written as a linear combination of $\{ H_t(w) \}_{t \leq k}$:
    \begin{align*}
        w^k & = \sum_{t=0}^k a_t H_t(w), \quad  a_t = \E Z^t H_t(Z). 
    \end{align*}
    Hence, 
    \begin{align*}
    |\E Z^k-\E W_\lambda^k|  = |a_k \cdot \E H_k(W_\lambda)| = \lambda \cdot \frac{|\E Z^k H_k(Z)| \cdot |\E H_k(W)|}{\lambda_k} = \lambda.
    \end{align*}
    This verifies the Minimum Signal Strength Assumption (Assumption~\ref{ass: min-snr}). 
    \item Observe that $\E e^{t W_\lambda} = \E[e^{ t \gamma(\lambda) W}] \cdot e^{t^2(1-\gamma^2(\lambda)/2}$. Since $W$ is 1 sub-Gaussian, $\E \exp(t W_\lambda)  \leq e^{t^2/2}$, which verifies that $\nongauss$ is 1 sub-Gaussian.
\end{enumerate}
This concludes the proof of this lemma. 
\end{proof}

\subsubsection{Proof of Lemma~\ref{lemma: ngca-construction-boundedLLR}} \label{appendix: ngca-boundedLLR-construction}
\begin{proof}[Proof of Lemma~\ref{lemma: ngca-construction-boundedLLR}]
Consider the vector space of polynomials on $\R$. On this vector space, define the inner product:
\begin{align*}
    \ip{f}{g}_\filter \explain{def}{=}  \int_\R f(x) g(x) \filter(x) \refmu(\diff x),
\end{align*}
where the weight function $\filter(x)$ is defined as:
\begin{align*}
  \filter(x) & = \begin{cases} 1 & \text{if $|x| \leq 1$} ; \\ 0 & \text{if $|x| > 1$}. \end{cases}. 
\end{align*}
Let $(\thermite{i})_{i\in\W}$ denote the orthonormal polynomials obtained by the Gram-Schmidt orthogonalization of the ordered linearly independent collection $(x^i)_{i \in \W}$.
In particular, for all $i, j \in \W$,
\begin{itemize}
  \item $\ip{\thermite{i}}{\thermite{j}}_{\filter} = \delta_{ij}$,
  \item $\operatorname{Span}(\{1,x,\dotsc,x^i\})  = \operatorname{Span}(\{\thermite{0},\thermite{1}, \dotsc , \thermite{i}\})$,
  \item The degree of $\thermite{i}$ is exactly $i$.
\end{itemize}
Define
\begin{align*}
    \|\thermite{k} \cdot \filter\|_\infty \explain{def}{=} \sup_{x \in \R} |\thermite{k}(x) \filter(x)|, \quad \lambda_k \explain{def}{=} \frac{|\ip{x^k}{\thermite{k}}_\filter|}{\|\thermite{k} \cdot \filter\|_\infty}. 
\end{align*}
Since polynomials are uniformly bounded on compact sets, we have $\|\thermite{k} \cdot \filter\|_\infty = \sup_{|x| \leq 1} |\thermite{k}(x)| <\infty$.
Furthemore, we observe that $\lambda_k \neq 0$ (otherwise $x^k$ lies in the span of $\thermite{0:k-1}$, which is not possible since $x^k$ has degree $k$).  With these definitions, we are ready to construct the measure $\nongauss$ as follows:
\begin{align} \label{eq: nongauss-formula}
    \frac{\diff \nongauss}{\diff \refmu}(x) \explain{def}= 1 + \frac{\lambda}{\lambda_k} \frac{ \thermite{k}(x)\cdot \filter(x)}{ \|\thermite{k} \cdot \filter\|_\infty}.
\end{align}
We first check that the above $\nongauss$ is a valid probability measure. The density defined above is non-negative for any $0 \leq \lambda \leq  \lambda_k$. Furthermore,
\begin{align*}
    \int \frac{\diff \nongauss}{\diff \refmu}(x)  = 1 + \frac{\lambda}{\lambda_k} \cdot \frac{\ip{\thermite{k}}{1}_\filter}{\|\thermite{k} \cdot \filter\|_\infty} = 1 + \frac{\lambda}{\lambda_k} \cdot \frac{\cdot \sqrt{\ip{1}{1}_\filter} \cdot \ip{\thermite{k}}{\thermite{0}}_\filter}{\|\thermite{k} \cdot \filter\|_\infty } \explain{(a)}{=} 1.
\end{align*}
In the step marked (a), we used the orthogonality property $\ip{\thermite{k}}{\thermite{0}}_\filter = 0$ for any $k \in \N$. Hence $\nongauss$ defines a valid probability measure. Next, we verify each of the claims in the statement of the lemma.
\begin{enumerate}
    \item For any $i \leq k-1$, since $x^i$ lies in the span of $\thermite{1:k-1}$, we have $\ip{x^i}{\thermite{k}}_\filter = 0$. Consequently,
    \begin{align*}
         \int x^i  \nongauss(\diff x) =  \int x^i \cdot \frac{\diff \nongauss}{\diff \refmu}(x) \cdot \refmu(\diff x) =\E Z^i + \frac{\lambda}{\lambda_k} \cdot \frac{\ip{x^i}{\thermite{k}}_\filter}{\|\thermite{k} \cdot \filter\|_\infty} = \E Z^i.
    \end{align*}
    \item This claim is immediate from the formula in \eqref{eq: nongauss-formula}.
    \item Following the same steps as in the proof of item (1), we obtain
    \begin{align*}
        \left|\int x^k  \nongauss(\diff x) -\E Z^k \right| = \frac{\lambda}{\lambda_k} \cdot \frac{|\ip{x^k}{\thermite{k}}_\filter|}{\|\thermite{k} \cdot \filter\|_\infty} = \lambda.
    \end{align*}
    \item Observe that
    \begin{align*}
        \int |x|^i \; \nongauss(\diff x) & = \int |x|^i \cdot  \frac{\diff \nongauss}{\diff \refmu}(x) \; \refmu(\diff x) \leq 2 \E|Z|^i, \; Z \sim \gauss{0}{1}.
    \end{align*}
    Hence, $\nongauss$ is sub-Gaussian with variance proxy $\varproxy \leq C$ for some universal constant $C$. 
    \item An inductive argument shows that $\thermite{i}$ is an odd function for odd $i$ and an even function for even $i$. Hence when $k$ is odd, \eqref{eq: nongauss-formula} gives:
    \begin{align*}
          \frac{\diff \nongauss}{\diff \refmu}(x) +   \frac{\diff \nongauss}{\diff \refmu}(-x) = 1,
    \end{align*}
    as claimed.
\end{enumerate}
This concludes the proof of this lemma. 
\end{proof}

\subsection{Information Theoretic Lower Bounds} \label{sec:ngca-it-lb}
In this section, we study information theoretic lower bounds for the Non-Gaussian Component Analysis problem. The main result of this section is stated in proposition below. 

\begin{proposition} \label{prop:ngca-it-lb} Consider the Non-Gaussian Component Analysis problem with a non-Gaussian distribution $\nongauss$ which satisfies the Bounded Signal Strength Assumption (Assumption~\ref{ass: bounded-snr}) with parameters $(\lambda, K)$. If $\dim \geq 32 \ln(2)$ and $K \ssize \lambda^2  \leq \dim/32$, then for any estimator $\hat{\bm V}(\bm x_{1:\ssize}) \in \R^\dim$ we have
\begin{align*}
     \sup_{\bm V \in \{\pm 1\}^\dim } \frac{\E_{\bm V} \|\hat{\bm V}(\bm x_{1:\ssize}) - \bm V \|^2}{\dim} & \geq  \frac{1}{16}.
\end{align*}
\end{proposition}

\begin{proof}

Let $\hat{\bm V}(\bm x_{1:\ssize})$ be any estimator for the Non-Gaussian Component Analysis problem taking values $\R^\dim$. We can construct a $\{\pm 1\}^\dim$-valued estimator from $\hat{\bm V}$ by defining $$\widetilde{\bm V}(\bm x_{1:\ssize}) = \min_{\bm u \in \{\pm 1\}^\dim} \|\bm u - \hat{\bm V}(\bm x_{1:\ssize})\|.$$ 
Observe that by the triangle inequality for any $\bm V \in \{\pm 1\}^\dim$:
\begin{align}\label{eq:rounding-error}
    \|\widetilde{\bm V}(\bm x_{1:\ssize}) - \bm V\| & \leq \|\widetilde{\bm V}(\bm x_{1:\ssize}) - \hat{\bm V}\| + \|\hat{\bm V}(\bm x_{1:\ssize}) - \bm V\| \leq 2\|\widetilde{\bm V}(\bm x_{1:\ssize}) - \hat{\bm V}\|. 
\end{align}
The Varshamov–Gilbert Lemma (see, e.g., \citep[Lemma 2.9]{tsybakov2009introduction}) guarantees the existence of a subset $\mathcal{U} \subset \{\pm 1\}^\dim$ with the properties:
\begin{align*}
    |\mathcal{U}| &\geq 2^{\frac{\dim}{8}}, \quad \|\bm u - \bm u^\prime \|^2  \geq \frac{\dim}{2} \quad \forall \; \bm u \neq \bm u^\prime \in \mathcal{U}.  
\end{align*}
We apply the standard Fano's Inequality for Mutual Information to the random variables $\bm V, \bm x_{1:\ssize}, \widetilde{\bm V}$ distributed as follows:
\begin{align*}
\bm V \sim \unif{\mathcal{U}}, \quad \bm x_{1:\ssize} \explain{i.i.d}{\sim} \dmu{\bm V}, \quad \widetilde{\bm V} = \widetilde{\bm V}(\bm x_{1:\ssize}). 
\end{align*}
Fano's Inequality (see, e.g., \citep[Theorem 2.10.1]{cover1999elements}) yields:
\begin{align*}
    \P(\widetilde{\bm V}(\bm x_{1:\ssize}) \neq \bm V) & \geq 1 - \frac{\mathrm{I}(\bm V; \bm x_{1:\ssize}) + \ln(2)}{\ln |\mathcal{U}|}.
\end{align*}
In the above display, $\mathrm{I}(\bm V; \bm x_{1:\ssize})$ denotes the mutual information between $\bm V$ and $\bm x_{1:\ssize}$. Let $\refmu = \gauss{\bm 0}{\bm I_\dim}$. We can upper bound the mutual information as follows:
\begin{align*}
    \mathrm{I}(\bm V; \bm x_{1:\ssize}) & \explain{(a)}{\leq} \frac{1}{|\mathcal{U}|} \sum_{\bm u \in \mathcal{U}} \mathsf{KL}(\dmu{\bm u}^{\otimes \ssize} ; \refmu^{\otimes \ssize}) \\
    & \explain{(b)}{=} \frac{\ssize}{|\mathcal{U}|} \sum_{\bm u \in \mathcal{U}} \mathsf{KL}(\dmu{\bm u} ; \refmu) \\
    & \explain{(c)}{\leq}  K \cdot \ssize  \cdot \lambda^2.
\end{align*}
In the above display, in the step marked (a), we used the variational formula for Mutual Information. In (b), we used the tensorization property of KL-divergence. In the step marked (c), we upper bounded the KL-divergence by the $\chi^2$-divergence and appealed to the Bounded Signal Strength Assumption (Assumption~\ref{ass: bounded-snr}). In particular when $ \dim \geq 32 \ln(2)$ and $K\ssize \lambda^2 \leq \dim/32$, we have
\begin{align*}
     \P(\widetilde{\bm V}(\bm x_{1:\ssize}) \neq \bm V) & \geq 1/2.
\end{align*}
Hence,
\begin{align*}
    \sup_{\bm V \in \{\pm 1\}^\dim } \frac{\E_{\bm V} \|\hat{\bm V}(\bm x_{1:\ssize}) - \bm V \|^2}{\dim} & \geq  \sup_{\bm V \in \{\pm 1\}^\dim } \frac{\E_{\bm V} \|\widetilde{\bm V}(\bm x_{1:\ssize}) - \bm V \|^2}{4\dim}  \geq \frac{\P(\widetilde{\bm V}(\bm x_{1:\ssize}) \neq \bm V)}{8} \geq \frac{1}{16},
\end{align*}
as claimed. 
\end{proof}

\subsection{A Computationally Inefficient Estimator} \label{sec: ngca-bruteforce}
In this section, we analyze a computationally inefficient, but statistically optimal estimator for the mixture of the Non-Gaussian Component Analysis problem. The main result of this section is the following.

\begin{theorem} \label{thm: bruteforce-ngca-sample-complexity} Assume that the Minimum Signal Strength Assumption (Assumption~\ref{ass: min-snr}) holds with parameters $(k,\lambda)$ and the sub-Gaussian Assumption (Assumption~\ref{ass: subgauss}) holds with variance proxy $\varproxy$.  Then, there is a constant $C_k$ depending only on $k$ such that for any $\epsilon \in (0,1)$, if
\begin{align*}
    \ssize & \geq  \frac{C_k \cdot \varproxy^k \cdot \dim}{\lambda^2 \epsilon^2} \cdot \ln\left( \frac{C_k}{\epsilon} \right) ,
\end{align*}
there is an estimator $\hat{\bm V}(\bm x_{1:\ssize})$ with the guarantee:
\begin{align*}
    \frac{\ip{\bm V}{\hat{\bm V}}^2}{\|\bm V\|^2 \|\hat{\bm V}\|^2} & \geq 1- \epsilon^2,
\end{align*}
with probability $1-2 e^{-\dim}$.
\end{theorem}
In order to motivate the estimator analyzed in Theorem~\ref{thm: bruteforce-ngca-sample-complexity}, we recall that when the non-Gaussian measure $\nongauss$ satisfies the Minimum Signal Strength Assumption (Assumption~\ref{ass: min-snr}) we had shown that if $\bm x \sim \dmu{\bm V}$ and $\bm z \sim \gauss{\bm 0}{\bm I_\dim}$, then
\begin{align*}
    \E \bm x^{\otimes k} - \E \bm z^{\otimes k} & = \pm \lambda \cdot  \frac{\bm V^{\otimes k}}{\sqrt{\dim^k}}.
\end{align*}
This suggests that $\bm V$ can be estimated by computing the best rank-1 approximation to the empirical estimate of the tensor $\E \bm x^{\otimes k} - \E \bm z^{\otimes k}$, i.e.,
\begin{align*}
    \arg\min_{\substack{\bm u \in \sphere{d-1}}}  \left\| \frac{1}{\ssize} \sum_{i=1}^\ssize (\bm x_i^{\otimes k} - \E \bm z^{\otimes k}) \mp \lambda  \cdot  \bm u^{\otimes k}\  \right\|,
\end{align*}
where $\| \cdot \|$ is a suitable measure of the discrepancy between tensors. We will find it convenient to use the following discrepancy measure.
Let $\net{\delta}$ denote the smallest $\delta$-net of $\sphere{\dim-1}$.
For any two functions $f_1,f_2 : \sphere{\dim-1} \mapsto \R$ we define the discrepancy measure:

\begin{align*}
    \left\|f_1(\cdot) - f_2(\cdot) \right\|_{\net{\delta}} \explain{def}{= } \max_{\bm w \in \net{\delta}} | f_1(\bm w) - f_2(\bm w) |. 
\end{align*}

Furthermore since the random variables $\ip{\bm x_i}{\bm w}^k$ are heavy tailed, we will find it convenient to truncate them. Define the truncation function at threshold $\thresh \geq 0$:
\begin{align}
    \trunc{\thresh}(x) \explain{def}{=} \max(\min(x, \thresh),-\thresh).
\end{align}
The final estimator we analyze is given by:
\begin{align}\label{eq:brute-force-ngca}
    \hat{\bm V}_{\delta,\thresh} \explain{def}{=} \arg\min_{\substack{\bm u \in \sphere{d-1}}}  \left\| \frac{1}{\ssize} \sum_{i=1}^\ssize \trunc{\thresh}(\ip{\bm x_i}{\cdot})^k - \E Z^k \mp \lambda  \cdot  \ip{\bm u}{\cdot}^k  \right\|_{\net{\delta}} .
\end{align}
In the above display, $Z \sim \gauss{0}{1}$ and $\thresh, \delta$ are tuning parameters.

To analyze the sample complexity of the estimator in \eqref{eq:brute-force-ngca}, we need several intermediate results which we state and prove.
The proof of Theorem~\ref{thm: bruteforce-ngca-sample-complexity} appears at the end of this section.
The following lemma provides guidance on how to set the threshold $\thresh$ to ensure that
\begin{equation*}
  \E[\trunc{\thresh}(\ip{\bm u}{\bm x})^k] \approx \E[\ip{\bm u}{\bm x}^k].
\end{equation*}
\begin{lemma} \label{lemma: tuning-threshold}Let $\bm x \sim \dmu{\bm V}$. Assume that the Minimum Signal Strength Assumption (Assumption~\ref{ass: min-snr}) holds with parameters $(k,\lambda)$ and the sub-Gaussian Assumption (Assumption~\ref{ass: subgauss}) holds with variance proxy $\varproxy$. Then, there is a universal constant $C$ such that for any $\lambda, \epsilon \geq 0$ if,
\begin{align*}
    \thresh^2  \geq C\cdot k \cdot \varproxy  \cdot \ln\left( \frac{C \cdot k\cdot  \varproxy}{\lambda \epsilon} \right)
\end{align*}
then, for any $\bm u \in \sphere{\dim - 1}$,
\begin{align*}
  |\E[\trunc{\thresh}(\ip{\bm u}{\bm x})^k] - \E[\ip{\bm u}{\bm x}^k]| & \leq \frac{\lambda \epsilon}{6}.
\end{align*}
\end{lemma}
\begin{proof}
We begin by observing that for any $\bm u$ with $\|\bm u \| = 1$, we have
\begin{align*}
    |\E[\trunc{\thresh}(\ip{\bm u}{\bm x})^k] - \E[\ip{\bm u}{\bm x}^k]| & \leq \E[|\ip{\bm u}{\bm x}|^k \Indicator{|\ip{\bm u}{\bm x}| \geq \thresh}] \\
    & \leq \sqrt{\E\ip{\bm u}{\bm x}^{2k} } \cdot \sqrt{\P(|\ip{\bm u}{\bm x}| \geq \thresh)} \\
    & \leq (C \varproxy k)^{\frac{k}{2}} \cdot \exp(-h^2/4\varproxy).
\end{align*}
To obtain the inequality in the last step we observed that $\ip{\bm u}{\bm x}$ is sub-Gaussian with variance proxy $\varproxy$ and used standard moment and tail bounds for sub-Gaussian random variables.
By ensuring:
\begin{align*}
    \thresh^2 \geq 4 \varproxy \cdot \ln\left( \frac{6(C \varproxy k)^{\frac{k}{2}}}{\lambda \epsilon}\right), 
\end{align*}
we obtain the claim of the lemma. 
\end{proof}

We will also need the following concentration result in our analysis. 

\begin{lemma} \label{lemma: tensor-norm-concentration} Assume that the sub-Gaussian Assumption (Assumption~\ref{ass: subgauss}) holds with variance proxy $\varproxy$. Let $\bm x, \bm x_{1: \ssize} \explain{i.i.d}{\sim} \dmu{\bm V}$. There is a universal constant $C$ such that if,
\begin{align*}
    \ssize & \geq  C \cdot \dim \cdot  \left( \frac{(Ck \varproxy)^k}{\lambda^2 \epsilon^2} + \frac{\thresh^k}{\lambda \epsilon} \right) \cdot \ln\left( \frac{C}{\delta} \right),
\end{align*}
then, with probability $1 - 2 e^{-\dim}$,
\begin{align*}
     \left\| \frac{1}{\ssize} \sum_{i=1}^\ssize \trunc{\thresh}(\ip{\bm x_i}{\cdot})^k - \E[\trunc{\thresh}(\ip{\bm x}{\cdot})^k]  \right\|_{\net{\delta}} &\leq \frac{\lambda \epsilon}{6}.
\end{align*}
\end{lemma}

\begin{proof}
For any fixed $\bm w \in \sphere{\dim-1}$, we observe that,
\begin{align*}
    |\trunc{\thresh}(\ip{\bm x}{\bm w})|^k & \leq \thresh^k, \\
    \E |\trunc{\thresh}(\ip{\bm x}{\bm w})|^{2k} & \leq \E[\ip{\bm x}{\bm w}^{2k}]  \leq (Ck\varproxy)^k,
\end{align*}
where $C$ is a universal constant.  Hence, by the Bernstein Inequality, we have
\begin{align*}
    \P \left( \left| \frac{1}{\ssize} \sum_{i=1}^\ssize \trunc{\thresh}(\ip{\bm x_i}{\bm w})^k - \E[\trunc{\thresh}(\ip{\bm x}{\bm w})^k]  \right| > \frac{\lambda \epsilon}{6}\right) & \leq 2 \exp\left( - \frac{\ssize}{C} \cdot \left( \frac{\lambda^2 \epsilon^2}{(Ck\varproxy)^k} \wedge \frac{\lambda \epsilon}{\thresh^k}\right)  \right),
\end{align*}
where $C$ is a universal constant. Using a union bound,
\begin{align*}
    \P \left( \left\| \frac{1}{\ssize} \sum_{i=1}^\ssize \trunc{\thresh}(\ip{\bm x_i}{\cdot})^k - \E[\trunc{\thresh}(\ip{\bm x}{\cdot})^k]  \right\|_{\net{\delta}} > \frac{\lambda \epsilon}{6}\right) & \leq 2 |\net{\delta}| \exp\left( - \frac{\ssize}{C} \cdot \left( \frac{\lambda^2 \epsilon^2}{(Ck\varproxy)^k} \wedge \frac{\lambda \epsilon}{\thresh^k}\right)  \right).
\end{align*}

A well-known bound on $|\net{\delta}|$ is (see for e.g. \citet{vershynin2018high}[Corollary 4.2.13]), $|\net{\delta}| \leq (3/\delta)^\dim$. The hypothesis on the sample size ensures that,
\begin{align*}
    \dim + \dim \ln\left( \frac{3}{\delta} \right) & \leq \frac{\ssize}{C} \cdot \frac{\lambda^2 \epsilon^2}{(Ck\varproxy)^k} \wedge \frac{\ssize}{C} \cdot \frac{\lambda \epsilon}{\thresh^k},
\end{align*}
Hence,
\begin{align*}
    \P \left( \left\| \frac{1}{\ssize} \sum_{i=1}^\ssize \trunc{\thresh}(\ip{\bm x_i}{\cdot})^k - \E[\trunc{\thresh}(\ip{\bm x}{\cdot})^k]  \right\|_{\net{\delta}} > \frac{\lambda \epsilon}{6}\right) & \leq 2 e^{-\dim}.
\end{align*}
\end{proof}

Finally, we will require the following quantitative identifiability result which shows that if, for two unit vectors $\bm u_1, \bm u_2$ if  $\|\ip{\bm u_1}{\cdot}^k - \ip{\bm u_2}{\cdot}^k\|_{\net{\delta}} \approx 0$ then we must have $\bm u_1 \approx \bm u_2$ or $\bm u_1 \approx - \bm u_2$.

\begin{lemma} \label{lemma-tensor-wedin} There is a constant $C_k$ depending only on $k$ such that for any $\bm u_1, \bm u_2 \in \sphere{\dim - 1}$.
\begin{align*}
\|\bm u_1 - \bm u_2\| \wedge \|\bm u_1 + \bm u_2\| & \leq  C_k \left(\|\ip{\bm u_1}{\cdot}^k - \ip{\bm u_2}{\cdot}^k\|_{\net{\delta}} + \delta\right).
\end{align*}
\end{lemma}

\begin{proof}
Define $\alpha = \ip{\bm u_1}{\bm u_2}$ and $\beta = \sqrt{1-|\alpha|^2}$. We decompose $\bm u_2$ as $\bm u_2 = \alpha \bm u_1 + \beta \bm u_0$ where $\bm u_0$ is a unit vector perpendicular $\bm u_1$. We define the  unit vector $\bm w_0 = (3\cdot \mathsf{sign}(\alpha) \cdot \bm u_1 + 4\cdot \bm u_0)/5$. Let $\hat{\bm w}$ be the unit vector in $\net{\delta}$ such that $\|\hat{\bm w} - \bm w_0\| \leq \delta$. We observe that,
\begin{align*}
     \|\ip{\bm u_1}{\cdot}^k - \ip{\bm u_2}{\cdot}^k\|_{\net{\delta}} & \explain{def}{=} \max_{\bm w \in \net{\delta}} |\ip{\bm u_1}{\bm w}^k - \ip{\bm u_2}{\bm w}^k| \\
     & \geq  |\ip{\bm u_1}{\hat{\bm w}}^k - \ip{\bm u_2}{\hat{\bm w}}^k| \\
     & \explain{(a)}{\geq} |\ip{\bm u_1}{{\bm w_0}}^k - \ip{\bm u_2}{{\bm w_0}}^k| - 2k\delta  \\
     & = 5^{-k} \cdot |3^k \mathsf{sign}(\alpha)^k - (3 | \alpha| + 4 \beta)^k| - 2k\delta  \\
     & \geq (3/5)^k \cdot |1- (|\alpha| + 4\beta/3)^k| - 2k\delta  \\
     & = \frac{3^k}{5^k}  \cdot  |1- |\alpha| - 4\beta/3)| \cdot  \left( \sum_{i=0}^{k-1} \left(|\alpha| + \frac{4\beta}{3} \right)^i \right) - 2k\delta \\
     & \explain{(b)}{\geq} \frac{k3^k}{5^k}  \cdot  |1- |\alpha| - 4\beta/3)| - 2k\delta  \\
     & \explain{(c)}{\geq} \frac{1}{3\sqrt{2}} \cdot \frac{k 3^k}{5^k}  \cdot \sqrt{2-2|\alpha|} - 2k\delta \\
     & \explain{(d)}{=} \frac{1}{3\sqrt{2}} \cdot \frac{k 3^k}{5^k} \cdot \|\bm u_1 - \bm u_2\| \wedge \|\bm u_1 + \bm u_2\| - 2k\delta .
\end{align*}
In the above display, in the step marked (a) we used the inequality $|x^k - y^k| \leq k|x-y|$ for any $x,y \in [-1,1]$. Step (b) relies on the estimate $|\alpha| + 4\beta/3 \geq |\alpha| + \beta \geq \sqrt{\alpha^2 + \beta^2} = 1$. Step (c) is obtained by:
\begin{align*}
    \frac{|1- |\alpha| - 4\beta/3)|}{\sqrt{1-|\alpha|}} & = \frac{|1- |\alpha| - 4\sqrt{1-|\alpha|^2}/3)|}{\sqrt{1-|\alpha|}} = \frac{4}{3} \sqrt{1+|\alpha|} - \sqrt{1-|\alpha|} \geq \frac{1}{3}.
\end{align*}
Finally in step (d) we observed that $\|\bm u_1 - \bm u_2\| \wedge \|\bm u_1 + \bm u_2\| = \sqrt{2-2|\alpha|}$. Rearranging the final inequality gives the claim of the lemma. 
\end{proof}

With these supporting lemmas, we can now prove Theorem~\ref{thm: bruteforce-ngca-sample-complexity} which shows the existence of a consistent estimator for the non-Gaussian direction $\ssize \gtrsim \dim/\lambda^2$.  

\begin{proof}[Proof of Theorem~\ref{thm: bruteforce-ngca-sample-complexity}]
The desired estimator is given by $ \hat{\bm V} := \hat{\bm V}_{\delta, \thresh}$ as defined in \eqref{eq:brute-force-ngca} with the choice:
\begin{align*}
    \delta = \frac{\epsilon}{3C_k}, \; \thresh^2 =  C_k \cdot \varproxy  \cdot \ln\left( \frac{C_k \cdot  \varproxy}{\lambda \epsilon} \right),
\end{align*}
for a suitably large constant $C_k$ depending on $k$. For notational simplicity omit the subscripts $\delta,\thresh$ in $\hat{\bm V}_{\delta,\thresh}$. Define $\bm v = \bm V /\|\bm V\|$.  We observe that,
\begin{align*}
    \|\bm v - \hat{\bm V}\| \wedge \|\bm v + \hat{\bm V}\| &\explain{(a)}{\leq} C_k \left(\left\| \ip{\hat{\bm V}}{\cdot}^k - \ip{\bm v}{\cdot}^k \right\|_{\net{\delta}} + \delta \right) \\& \explain{(b)}{\leq} \frac{C_k}{\lambda} \left( \left\| \frac{1}{\ssize} \sum_{i=1}^\ssize \trunc{\thresh}(\ip{\bm x_i}{\cdot})^k  - \E Z^k \mp \lambda \ip{\hat{\bm V}}{\cdot}^k  \right\|_{\net{\delta}} + \right.\\& \hspace{4cm} \left.\left\|\frac{1}{\ssize} \sum_{i=1}^\ssize \trunc{\thresh}(\ip{\bm x_i}{\cdot})^k  - \E Z^k \mp \lambda \ip{\bm v}{\cdot}^k \right\|_{\net{\delta}}  \right) + C_k \delta  \\
    & \explain{(c)}{\leq} \frac{2C_k}{\lambda} \cdot   \left\|\frac{1}{\ssize} \sum_{i=1}^\ssize \trunc{\thresh}(\ip{\bm x_i}{\cdot})^k  - \E Z^k \mp \lambda \ip{\bm v}{\cdot}^k \right\|_{\net{\delta}} + C_k \delta.
\end{align*}
In the above display step (a) relies on Lemma~\ref{lemma-tensor-wedin} and step (b) uses the triangle inequality. The step marked (c) relies on the fact that $\hat{\bm V}$ achieves the minimum discrepancy. By observing that $\E \ip{\bm x}{\bm w}^k = \E Z^k \pm \lambda \ip{\bm w}{\bm v}^k$, we can bound,
\begin{align*}
    &\left\|\frac{1}{\ssize} \sum_{i=1}^\ssize \trunc{\thresh}(\ip{\bm x_i}{\cdot})^k  - \E Z^k \mp \lambda \ip{\bm v}{\cdot}^k \right\|_{\net{\delta}}  \\&\hspace{3cm}\leq \underbrace{\left\|\frac{1}{\ssize} \sum_{i=1}^\ssize \trunc{\thresh}(\ip{\bm x_i}{\cdot})^k  - \E  \trunc{\thresh}(\ip{\bm x}{\cdot})^k\right\|_{\net{\delta}}}_{(\mathsf{I})} +  \underbrace{\left\|\E\trunc{\thresh}(\ip{\bm x}{\cdot})^k  - \E\ip{\bm x}{\cdot}^k \right\|_{\net{\delta}}}_{(\mathsf{II})}.
\end{align*}
As prescribed by Lemma~\ref{lemma: tuning-threshold}, we set $\thresh^2 = C_k \cdot \varproxy  \cdot \ln\left( \frac{C_k \cdot  \varproxy}{\lambda \epsilon} \right)$ to ensure $(\mathsf{II}) \leq \lambda \epsilon/6C_k$. We set $\delta = \epsilon/3C_k$. With this choice of $\delta, \thresh$ if $\ssize$ satisfies:
\begin{align*}
    \ssize & \geq  \frac{C_k \cdot \varproxy^k \cdot \dim}{\lambda^2 \epsilon^2} \cdot  \left( 1 + \frac{\lambda \epsilon}{\varproxy^{\frac{k}{2}}} \ln^{\frac{k}{2}} \left(\frac{C_k \varproxy}{\lambda \epsilon} \right)  \right) \cdot \ln\left( \frac{C_k}{\epsilon} \right),
\end{align*}
then Lemma~\ref{lemma: tensor-norm-concentration} guarantees $(\mathsf{I}) \leq \lambda \epsilon/6 C_k$. Since $\lim_{\alpha \rightarrow 0} \alpha^k \ln(1/\alpha) = 0$, the above sample size requirement can be simplified to:
\begin{align*}
    \ssize & \geq  \frac{C_k \cdot \varproxy^k \cdot \dim}{\lambda^2 \epsilon^2} \cdot \ln\left( \frac{C_k}{\epsilon} \right).
\end{align*}
In summary, we have shown,
\begin{align*}
    \|\bm v - \hat{\bm V}\| \wedge \|\bm v + \hat{\bm V}\| & \leq \frac{\epsilon}{3} + \frac{\epsilon}{3} + \frac{\epsilon}{3} = \epsilon.
\end{align*}
To finish the proof, we observe that,
\begin{align*}
    1- \ip{\bm v}{\hat{\bm V}_{\delta, \thresh}}^2 = \min_{s \in \R} \| \bm v - s \hat{\bm V}\|^2 \leq  \|\bm v - \hat{\bm V}\|^2 \wedge \|\bm v + \hat{\bm V}\|^2 \leq \epsilon^2.
\end{align*}
\end{proof}

\subsection{Low Degree Lower Bound} \label{sec:ngca-LDLR}

In this section, we provide evidence for the computational-statistical gap in the mixture of Gaussians problem.  We provide evidence for the following testing variant of the problem using the Low Degree Framework. We begin by formally defining the testing problem. In this problem, given a dataset $\{\bm x_1,\bm x_2, \dotsc , \bm x_\ssize\} \subset \R^\dim$, the goal is to design a test $\phi: (\R^\dim)^\ssize \mapsto \{0,1\}$ that distinguishes between the null and alternative hypothesis stated below.
\begin{enumerate}
    \item \emph{Null Hypothesis: } In the null hypothesis, the data $\bm x_{1:\ssize}$ is generated as:
    \begin{align*}
        \bm x_i \explain{i.i.d.}{\sim} \refmu \explain{def}{=} \gauss{\bm 0}{\bm I_\dim}.
    \end{align*}
    \item \emph{Alternative Hypothesis: } In the alternative hypothesis, there is a non-Gaussian probability measure $\nongauss$ on $\R$  and an unknown $\bm V \in \R^\dim$ with $\|\bm V \| = \sqrt{\dim}$ such that,
    \begin{align*}
        \bm x_i \explain{i.i.d.}{\sim} \dmu{\bm V}.
    \end{align*}
    Recall that this means that,
\begin{subequations} %
\begin{align}
  \bm x_i  = \eta_i \frac{\bm V}{\|\bm V\|} + \left( \bm I_\dim - \frac{\bm V \bm V^\UT}{\|\bm V\|^2} \right) \cdot \bm z_i, \;
\end{align}
where $\eta_i \in \R, \; \bm z_i \in \R^\dim$ are independent random variables with distributions:
\begin{align}
    \bm z_i \sim \gauss{\bm 0}{\bm I_\dim}, \; \eta_i \sim \nongauss.
\end{align}
\end{subequations}
\end{enumerate}

A test successfully distinguishes between the null and the alternative hypothesis if it is consistent, that is,

\begin{subequations}\label{eq: consistent-test}
\begin{align}
   \lim_{\dim \rightarrow \infty} \refmu(\phi(\bm x_{1:\ssize} = 0)) &= 1, \\
   \lim_{\dim \rightarrow \infty}  \inf_{\substack{\bm V \in \R^\dim\\ \|\bm V\| = \sqrt{\dim}}} \dmu{\bm V}(\phi(\bm x_{1:\ssize} = 1)) & = 1.
\end{align}
\end{subequations}

In order to prove the low-degree lower bound, it will be sufficient to restrict ourselves to simpler Bayesian version of the problem where the parameter $\bm V$ is drawn from the prior $\prior \explain{def}{=} \unif{\{\pm 1\}^\dim}$. Let $\nullmu$ denote the mariginal distribution of the dataset under the alternative hypothesis $\bm x_i \explain{i.i.d.}{\sim} \dmu{\bm V}$ when $\bm V$ is drawn from the prior $\prior = \mathsf{Unif}(\{\pm 1\}^\dim)$:
\begin{align*}
    \nullmu(\cdot) = \int \dmu{\bm V}^{\otimes \ssize}(\cdot) \; \prior( \diff \bm V).
\end{align*}
Recall that Lemma~\ref{lemma: hermite-decomp-ngca} and Definition~\ref{def: integrated-Hermite} lead to the following decomposition of integrated, centered likelihood ratio $\diff\nullmu/\diff \refmu$:
\begin{align*}
 \frac{\diff \nullmu}{\diff\refmu}(\bm x_{1:\ssize}) - 1 &\explain{}{=} \sum_{\substack{\bm t \in \W^\ssize \\ \|\bm t\|_1 \geq 1}} \hat{\nongauss}_{\bm t}  \cdot \intH{\bm t}{\bm x_{1:\ssize}}{1}.    
\end{align*}

We define the low degree approximation to the integrated, centered likelihood ratio:

\begin{align} \label{eq: low-degree-hermite-ngca}
    \lowdegree{\frac{\diff \nullmu}{\diff \refmu} (\bm x) - 1}{t} \explain{def}{=} \sum_{\substack{\bm t \in \W^\ssize\\1 \leq \|\bm t\|_1 \leq t}} \hat{\nongauss}_{\bm t}  \cdot \intH{\bm t}{\bm x_{1:\ssize}}{1}, 
\end{align}
and the corresponding approximation error:
\begin{align*}
    \highdegree{\frac{\diff \nullmu}{\diff \refmu} (\bm x) - 1}{t} \explain{def}{=} \sum_{\substack{\bm t \in \W^\ssize\\1 \|\bm t\|_1 > t}} \hat{\nongauss}_{\bm t}  \cdot \intH{\bm t}{\bm x_{1:\ssize}}{1}.
\end{align*}
The Low Degree Framework for statistical-computational gaps is based on the following conjecture of \citet{hopkins2018statistical}. The statement presented below is from \citet{kunisky2019notes}, and has been instantiated for the Non-Gaussian Component Analysis problem. 
\begin{conjecture}[The Low Degree Likelihood Ratio Conjecture \citealp{hopkins2018statistical,kunisky2019notes}]  \label{conjecture: low-degree} If there exists constants $\epsilon > 0$ and $0\leq C<\infty$ (independent of $\dim$) such that,
\begin{align*}
   \refE \left(  \lowdegree{\frac{\diff \nullmu}{\diff \refmu} (\bm x_{1:\ssize}) - 1}{\ln(\dim)^{1+\epsilon}}^2\right) & \leq C,
\end{align*}
as $d \rightarrow \infty$, then, there is no polynomial-time, consistent test for the Non-Gaussian Component Analysis testing problem with non-Gaussian measure $\nongauss$. 
\end{conjecture}

In light of the above conjecture, the following proposition analyzes the $q$-norm of the integrated, centered likelihood ratio for any $q \geq 2$. 

\begin{align*}
    \left\|\lowdegree{\frac{\diff \nullmu}{\diff \refmu} (\bm x_{1:\ssize}) - 1}{t} \right\|_q^q \explain{def}{=}  \refE  \left|\lowdegree{\frac{\diff \nullmu}{\diff \refmu} (\bm x_{1:\ssize}) - 1}{t}\right|^q.
\end{align*}

The result for $q=2$ is useful in order to appeal to Conjecture~\ref{conjecture: low-degree}. We find the result for large $q$ useful for proving the communication lower bound.  

\begin{proposition} \label{prop: low-degree-norm-ngca} Suppose that $\nongauss$ satisfies the Moment Matching Assumption (Assumption~\ref{ass: moment-matching}) with parameter $k \geq 2$ and the Bounded Signal Strength Assumption (Assumption~\ref{ass: bounded-snr}) with parameters $(\lambda,K)$. There is a constant $C_{k,K}>0$ depending only on $k,K$ such that, for any $q \geq 2$, if, 
\begin{align*}
    t & \leq \frac{1}{C_{k,K}}\cdot \frac{\dim}{(q-1)^2}, \; \ssize \lambda^2 \leq  \frac{1}{C_{k,K}}\cdot \frac{1}{(q-1)^k} \cdot  \frac{\dim^{\frac{k}{2}}}{t^{\frac{k-2}{2}}},
\end{align*}
then,
\begin{align*}
   \left\|\lowdegree{\frac{\diff \nullmu}{\diff \refmu} (\bm x_{1:\ssize}) - 1}{t} \right\|_q^2  & \leq  \frac{ C_{k,K} \cdot (q-1)^k \cdot\ssize \lambda^2 \cdot t^{\frac{k-2}{2}}}{\dim^{\frac{k}{2}}} \leq 1.
\end{align*}
\end{proposition}
In particular, by the Low-Degree Likelihood Ratio Conjecture of \citet{hopkins2018statistical}, this suggests that in the sample size regime $\ssize \lambda^2 \ll \dim^{k/2}$, the $k$-NGCA testing problem is computationally hard.

We also complement this result with the following lemma which shows that the upper bound obtained in Proposition~\ref{prop: low-degree-norm-ngca} is tight for $q=2$.
\begin{lemma} \label{lemma:ldlr-lb-ngca} Suppose $\nongauss$ satisfies Moment Matching Assumption (Assumption~\ref{ass: moment-matching}) with parameter $k$ and the Minimum Signal Strength Assumption (Assumption~\ref{ass: min-snr}) with parameters $(\lambda,k)$.
Suppose that $t \leq \dim$ is even and a multiple of $k$. Then,
we have
\begin{align*}
     \left\|\lowdegree{\frac{\diff \nullmu}{\diff \refmu} (\bm x_{1:\ssize}) - 1}{t} \right\|_2^2  &  \geq \left(\frac{1}{C_k} \cdot  \frac{\ssize \lambda^2 \cdot t^{\frac{k-2}{2}}}{\dim^{\frac{k}{2}}} \right)^{\frac{t}{k}}.
\end{align*}
In the above display $C_k$ denotes a positive constant that depends only on $k$.
\end{lemma}
The remainder of this section is devoted to the proofs of the results mentioned above. 

\subsubsection{Proof of Proposition~\ref{prop: low-degree-norm-ngca}}
\begin{proof}[Proof of Proposition~\ref{prop: low-degree-norm-ngca}]
Using the expression for the degree $t$-approximation to the likelihood ratio given in \eqref{eq: low-degree-hermite-ngca} and the Hypercontractivity estimate of Lemma~\ref{lemma: integrated-hermite-hypercontractivity} we obtain,
\begin{align*}
     \left\|\lowdegree{\frac{\diff \nullmu}{\diff \refmu} (\bm x_{1:\ssize}) - 1}{t} \right\|_q^2  & \leq \sum_{\substack{\bm s \in \W^\ssize\\ 1\leq \|\bm s\|_1 \leq t}}  (q-1)^{\|\bm s\|_1} \cdot \hat{\nongauss}_{\bm s}^2  \cdot \refE \intH{\bm s}{\bm x_{1:\ssize}}{1}^2. 
\end{align*}
Using the estimate on $\refE \intH{\bm s}{\bm x_{1:\ssize}}{1}^2$ obtained in Lemma~\ref{lemma: norm-integrated-hermite} gives us:
\begin{align*}
    \left\|\lowdegree{\frac{\diff \nullmu}{\diff \refmu} (\bm x_{1:\ssize}) - 1}{t} \right\|_q^2  & \leq  \sum_{\substack{\bm s \in \W^\ssize\\ 1\leq \|\bm s\|_1 \leq t}}  \left(  \frac{C \cdot (q-1)^2 \cdot \|\bm s\|_1}{\dim} \right)^{\frac{\|\bm s\|_1}{2}} \cdot \hat{\nongauss}_{\bm s}^2 .
\end{align*}
We recall from Lemma~\ref{lemma: hermite-decomp-ngca} that,
\begin{align*}
    \hat{\nongauss}_{\bm s} \explain{def}{=} \prod_{i=1}^\ssize \hat{\nongauss}_{s_i}, \; \hat{\nongauss}_i \explain{def}{=} \E H_i (\eta), \; \eta \sim \nongauss.
\end{align*}
Recall that $\hat{\nongauss}_0 = 1$ and as a consequence of the Moment Matching Assumption, $\hat{\nongauss}_i = 0$ for any $i \leq k-1$. In particular, this means that, for any $\|\bm s\|_1 < k \|\bm s\|_0$, we have $\hat{\nongauss}_{\bm s}= 0$. Hence,
\begin{align*}
    \left\|\lowdegree{\frac{\diff \nullmu}{\diff \refmu} (\bm x_{1:\ssize}) - 1}{t} \right\|_q^2  & \leq  \sum_{\substack{ 1\leq \|\bm s\|_1 \leq t \\ \|\bm s\|_1 \geq k \|\bm s\|_0}}  \left(  \frac{C \cdot (q-1)^2 \cdot \|\bm s\|_1}{\dim} \right)^{\frac{\|\bm s\|_1}{2}} \cdot \hat{\nongauss}_{\bm s}^2.
\end{align*}
The assumption $C \cdot (q-1)^2 \cdot t/\dim \leq 1/e$ and the observation $k\|\bm s\|_0 \leq \|\bm s\|_1$ guarantee:
\begin{align*}
 \left(  \frac{C \cdot (q-1)^2 \cdot \|\bm s\|_1}{\dim} \right)^{\frac{\|\bm s\|_1}{2}} & \leq  \left(  \frac{C \cdot (q-1)^2 \cdot k  \|\bm s\|_0}{\dim} \right)^{\frac{k\|\bm s\|_0}{2}}.
\end{align*}
Hence,
\begin{align*}
   \left\|\lowdegree{\frac{\diff \nullmu}{\diff \refmu} (\bm x_{1:\ssize}) - 1}{t} \right\|_q^2  &  \explain{}{\leq} C\sum_{\substack{ 1\leq \|\bm s\|_1 \leq t \\ \|\bm s\|_1 \geq k \|\bm s\|_0}}  \left(  \frac{C \cdot (q-1)^2 \cdot k  \|\bm s\|_0}{\dim} \right)^{\frac{k\|\bm s\|_0}{2}} \cdot \hat{\nongauss}_{\bm s}^2 \\
    & \leq \sum_{\substack{ 1\leq \|\bm s\|_0 \leq t/k}}  \left(  \frac{C \cdot (q-1)^2 \cdot k  \|\bm s\|_0}{\dim} \right)^{\frac{k\|\bm s\|_0}{2}} \cdot \hat{\nongauss}_{\bm s}^2 \\
    & =  \sum_{i=1}^{\lfloor \frac{t}{k} \rfloor}  \left(  \frac{C \cdot (q-1)^2 \cdot k  i}{\dim} \right)^{\frac{ki}{2}} \cdot  \sum_{\|\bm s\|_0 = i}  \hat{\nongauss}_{\bm s}^2.
\end{align*}
 Since $\hat{\nongauss}_0 = 1$, we can compute,
\begin{align*}
    \sum_{\|\bm s\|_0 = i}  \hat{\nongauss}_{\bm s}^2 & = \binom{\ssize}{i} \left( \sum_{j=1}^\infty \hat{\nongauss}_j^2 \right)^i
\end{align*}
By the Bounded Signal Strength Assumption, we have
\begin{align*}
        \sum_{\|\bm s\|_0 = i}  \hat{\nongauss}_{\bm s}^2 & \leq \binom{\ssize}{i} \cdot (K^2 \lambda^2)^i  \leq \left( \frac{e K^2 \lambda^2 \ssize}{i}\right)^i
\end{align*}
And hence, we have
\begin{align*}
     \left\|\lowdegree{\frac{\diff \nullmu}{\diff \refmu} (\bm x_{1:\ssize}) - 1}{t} \right\|_q^2 & \leq \sum_{i=1}^{\lfloor \frac{t}{k} \rfloor}   \left( \frac{e \cdot (Ck)^{\frac{k}{2}} \cdot (q-1)^k \cdot i^{\frac{k-2}{2}} \cdot  (K\lambda)^2 \cdot  \ssize}{ \dim^{\frac{k}{2}}}\right)^i \\
     & \leq \sum_{i=1}^{\lfloor \frac{t}{k} \rfloor}   \left( \frac{e \cdot C^{\frac{k}{2}} \cdot k \cdot (q-1)^k \cdot t^{\frac{k-2}{2}} \cdot  (K\lambda)^2 \cdot  \ssize}{ \dim^{\frac{k}{2}}}\right)^i
\end{align*}
The assumption,
\begin{align*}
    \frac{e \cdot C^{\frac{k}{2}} \cdot k \cdot (q-1)^k \cdot t^{\frac{k-2}{2}} \cdot  (K\lambda)^2 \cdot  \ssize}{ \dim^{\frac{k}{2}}} < \frac{1}{2},
\end{align*}
ensures that the above sum is dominated by the geometric series $1+ 1/2 + 1/4 + \dotsb$. Hence,
\begin{align*}
     \left\|\lowdegree{\frac{\diff \nullmu}{\diff \refmu} (\bm x_{1:\ssize}) - 1}{t} \right\|_q^2 & \leq \frac{2e \cdot C^{\frac{k}{2}} \cdot k \cdot (q-1)^k \cdot t^{\frac{k-2}{2}} \cdot  (K\lambda)^2 \cdot  \ssize}{ \dim^{\frac{k}{2}}}.
\end{align*}
This proves the claim of the proposition. 
\end{proof}

\subsubsection{Proof of Lemma~\ref{lemma:ldlr-lb-ngca}}
\begin{proof}[Proof of Lemma~\ref{lemma:ldlr-lb-ngca}]
Using the formula for the degree $t$ approximation to the integrated and centered likelihood ratio along with Lemma~\ref{lemma: integrated-hermite-hypercontractivity}, gives us:
\begin{align*}
     \left\|\lowdegree{\frac{\diff \nullmu}{\diff \refmu} (\bm x_{1:\ssize}) - 1}{t} \right\|_2^2  & = \sum_{\substack{\bm s \in \W^\ssize\\ 1\leq \|\bm s\|_1 \leq t}}  \hat{\nongauss}_{\bm s}^2  \cdot \refE \intH{\bm s}{\bm x_{1:\ssize}}{1}^2 \\
     & \geq \sum_{\substack{\bm s \in \W^\ssize\\ \|\bm s\|_1 = t}}  \hat{\nongauss}_{\bm s}^2  \cdot \refE \intH{\bm s}{\bm x_{1:\ssize}}{1}^2.
\end{align*}
Since $t$ is even and $t \leq \dim$, item (3) of Lemma~\ref{lemma: norm-integrated-hermite} gives us the lower bound:
\begin{align*}
     \left\|\lowdegree{\frac{\diff \nullmu}{\diff \refmu} (\bm x_{1:\ssize}) - 1}{t} \right\|_2^2  &  \geq \left( \frac{t}{C \dim} \right)^{\frac{t}{2}} \cdot  \sum_{\substack{\bm s \in \W^\ssize\\ \|\bm s\|_1 = t}}  \hat{\nongauss}_{\bm s}^2. 
\end{align*}
In order to lower bound the sum involving $\hat{\nongauss}_{\bm s}$, we restrict ourselves to $\bm s \in \W^\ssize$ such that $s_i \in \{0,k\}$ and $\|\bm s\|_0 = t/k$. For any such $\bm s$, we observe that $\hat{\nongauss}_{\bm s} = \hat{\nongauss}_k^{t/k}$. In order to compute $\hat{\nongauss}_k$ we observe that:
\begin{align*}
    x^k & = \sum_{i=0}^k \refE[Z^k H_i(Z)] \cdot H_i(x). 
\end{align*}
Hence,
\begin{align*}
    \lambda &= \left| \refE Z^k - \int x^k \nongauss(\diff x) \right| = \left| \sum_{i=0}^k \refE[Z^k H_i(Z)] \cdot (\refE[H_i(Z)] - \hat{\nongauss}_i) \right| = \refE[Z^k H_k(Z)] \cdot  \hat{\nongauss}_k.
\end{align*}
In the above display, in order to obtain the last equality we observed that $\refE[H_i(Z)] = 0$ for any $i \geq 1$ and $\hat{\nongauss}_i = 0$ for any $1\leq i \leq k-1$ by the Moment Matching Assumption. We define $\alpha_k = \refE[Z^k H_k(Z)]$ and note that $\alpha_k \neq 0$. Hence, $\hat{\nongauss}_{\bm s} = (\lambda/\alpha_k)^{t/k}$ for any  $\bm s \in \W^\ssize$ such that $s_i \in \{0,k\}$ and $\|\bm s\|_0 = t/k$. This gives us the lower bound:
\begin{align*}
    \left\|\lowdegree{\frac{\diff \nullmu}{\diff \refmu} (\bm x_{1:\ssize}) - 1}{t} \right\|_2^2  &  \geq  \left( \frac{t}{C \dim} \right)^{\frac{t}{2}} \cdot \left( \frac{\lambda}{\alpha_k} \right)^{\frac{2t}{k}} \cdot |\{\bm s \in \W^\ssize : s_i \in \{0,k\}, \; \|\bm s\|_0 = t/k \}| \\
    & = \left( \frac{t}{C \dim} \right)^{\frac{t}{2}}\left( \frac{\lambda}{\alpha_k} \right)^{\frac{2t}{k}} \cdot \binom{\ssize}{t/k} \\
    & \geq \left( \frac{t}{C \dim} \right)^{\frac{t}{2}}\left( \frac{\lambda}{\alpha_k} \right)^{\frac{2t}{k}} \cdot \left( \frac{k\ssize}{t} \right)^{\frac{t}{k}}.
\end{align*}
In the above display, in order to obtain the last inequality we used the standard lower bound on the binomial coefficient $\binom{n}{k} \geq (n/k)^k$. To conclude, we have shown,
\begin{align*}
     \left\|\lowdegree{\frac{\diff \nullmu}{\diff \refmu} (\bm x_{1:\ssize}) - 1}{t} \right\|_2^2  &  \geq \left( \frac{k}{C^{\frac{k}{2}} \alpha_k^2} \cdot  \frac{\ssize \lambda^2 \cdot t^{\frac{k-2}{2}}}{\dim^{\frac{k}{2}}} \right)^{\frac{t}{k}}.
\end{align*}
This concludes the proof. 
\end{proof}

\subsection{A Spectral Estimator for Non-Gaussian Component Analysis}  \label{sec:spectral-ngca}
In this section, we analyze a computationally efficient estimator for the Non-Gaussian component analysis problem. We recall that in the Non-Gaussian Component Analysis problem, one seeks to estimate an unknown vector $\bm V \in \R^\dim$ with $\|\bm V\| = \sqrt{\dim}$ from an i.i.d. sample $\bm x_{1:\ssize}$ generated as follows:
\begin{subequations} 
\begin{align}
  \bm x_i  = \eta_i \frac{\bm V}{\|\bm V\|} + \left( \bm I_\dim - \frac{\bm V \bm V^\UT}{\|\bm V\|^2} \right) \cdot \bm z_i, \;
\end{align}
where $\eta_i \in \R, \; \bm z_i \in \R^\dim$ are independent random variables with distributions:
\begin{align}
    \bm z_i \sim \gauss{\bm 0}{\bm I_\dim}, \; \eta_i \sim \nongauss.
\end{align}
\end{subequations}
In the above display, $\nongauss$ is a non-Gaussian probability on $\R$. Let $\dmu{\bm V}$ denote the distribution of $\bm x_i$ described by the above generating process. Throughout this section, we will assume that the non-Gaussian measure $\nongauss$ satisfies the Moment Matching Assumption with parameter $k \geq 2$ (Assumption~\ref{ass: moment-matching}), the Minimum Signal Strength Assumption with parameters $(\lambda,k)$ (Assumption~\ref{ass: min-snr}) along with sub-Gaussian Assumption (Assumption~\ref{ass: subgauss}. Furthermore, we assume that $k$ is even. We will consider the following spectral estimator which estimates the non-Gaussian direction by the leading eigenvector $\hat{\bm V}$ (in the magnitude) of a data-dependent matrix $\hat{\bm M}$:
\begin{subequations} \label{eq: spectral-estimator-ngca}
\begin{align}
    \hat{\bm M} &\explain{def}{=} \frac{1}{\ssize} \sum_{i=1}^\ssize (\|\bm x_i\|^2- \dim)^{\frac{k-2}{2}} \cdot  \bm x_i \bm x_i^\UT - \E[(\|\bm z\|^2 -\dim)^{\frac{k-2}{2}} \cdot \bm z \bm z^\UT], \\
    \hat{\bm V} &\explain{def}{=} \max_{\|\bm u\| = 1} |\bm u^\UT  \hat{\bm M} \bm u|.
\end{align}
\end{subequations}
In the above display $\bm z \sim \gauss{\bm 0}{\bm I_\dim}$. The main result of this section is the following sample complexity bound for the spectral estimator in \eqref{eq: spectral-estimator-ngca}.

\begin{theorem}\label{thm:ngca-spectral} Suppose that $\nongauss$ satisfies the Moment Matching Assumption (Assumption~\ref{ass: moment-matching}) with parameters $(k,\lambda)$ and that the sub-Gaussian Assumption (Assumption~\ref{ass: subgauss}) holds with variance proxy $\varproxy$. Then, there is a constant $C_{k,\varproxy}$ depending only on $k,\varproxy$ such that, for any $\epsilon \in (0,1)$ and any $K \geq 1$, if,
\begin{align} \label{eq: spectral-sample-complexity}
    \ssize & \geq    \frac{C_{k,\varproxy} \cdot K^{k+1} \cdot \dim^{\frac{k}{2}}}{\epsilon^2 \lambda^2} \cdot \ln\left( \frac{C_{k,\varproxy}\cdot K \cdot \dim}{\epsilon \lambda}\right).
\end{align}
then, with probability $1-1/\ssize^K$, we have:
\begin{enumerate}
    \item The estimator $\hat{\bm V}$ defined in \eqref{eq: spectral-estimator-ngca-intro} satisfies the guarantee:
    \begin{align*}
        \frac{\ip{\bm V}{\hat{\bm V}}^2}{\|\bm V\|^2\|\hat{\bm V}\|^2} \geq 1 - \epsilon^2.
    \end{align*}
    \item Furthermore, $\hat{\bm M}$ has a spectral gap in the sense:
    \begin{align*}
         \frac{|\lambda_i(\hat{\bm M})|}{|\lambda_1(\hat{\bm M})|} \leq \epsilon , \; \forall  \; i \geq 2,
    \end{align*}
    where $|\lambda_1(\hat{\bm M})| \geq |\lambda_i(\hat{\bm M})| \dotsb \geq |\lambda_\dim(\hat{\bm M})|$ denote the eigenvalues of $\hat{\bm M}$ sorted in decreasing order of magnitude. 
\end{enumerate}
\end{theorem}
The proof of Theorem~\ref{thm:ngca-spectral} requires two intermediate results, which we introduce next. The following lemma provides intuition regarding why this is a natural estimator for $\bm V$. 
\begin{lemma} \label{lemma: matrix_expectation} Suppose that $\nongauss$ satisfies the Moment Matching Assumption (Assumption~\ref{ass: moment-matching}) with parameter $k$ and the Minimum Signal Strength Assumption (Assumption~\ref{ass: min-snr}) with parameters $(\lambda, k)$. Then, 
\begin{align*}
    \E \hat{\bm M} =  \pm \frac{\lambda}{\dim} \cdot \bm V \bm V^\UT.
\end{align*}
\end{lemma}
\begin{proof}Observe that,
\begin{align*}
    \E \hat{\bm M} = \E[(\|\bm x\|^2- \dim)^{\frac{k-2}{2}} \cdot  \bm x \bm x^\UT] -  \E[(\|\bm z\|^2 -\dim)^{\frac{k-2}{2}} \cdot \bm z \bm z^\UT], 
\end{align*}
where $\bm x \sim \dmu{\bm V}, \; \bm z \sim \gauss{\bm 0}{\bm I_\dim}$. Using the Binomial Theorem, we can expand,
\begin{align} \label{eq: pop-analysis-eq1}
    \E \hat{\bm M}  = \sum_{i=0}^{\frac{k-2}{2}} \binom{\frac{k-2}{2}}{i} \cdot  (- \dim)^{\frac{k-2}{2}-i} \cdot \left( \E[\|\bm x\|^{2i} \cdot \bm x \bm x^\UT] - \E[\|\bm z\|^{2i} \cdot \bm z \bm z^\UT] \right).
\end{align}
For any $\ell \leq k-1$, since $\nongauss$ satisfies the Moment Matching Assumption, we have
\begin{align*}
    \E[\bm x^{\otimes \ell}] - \E[\bm z^{\otimes \ell}] = \bm 0.
\end{align*}
By linearity of expectations, for any $\ell \leq k-1$,
\begin{align} \label{eq: pop-analysis-eq3}
     \E[\|\bm x\|^{\ell-2} \bm x^{\otimes 2}] -  \E[\|\bm z\|^{\ell-2}\bm z^{\otimes 2}] = \bm 0.
\end{align}
Since $\nongauss$ additionally satisfies the  Minimum Signal Strength Assumption (Assumption~\ref{ass: min-snr}) with parameters $(\lambda, k)$, we have
\begin{align}\label{eq: pop-analysis-eq2}
    \E[\bm x^{\otimes k}] -  \E[\bm z^{\otimes k}] = \pm \frac{\lambda}{\sqrt{\dim^k}} \cdot \bm V^{\otimes k}.
\end{align}
Using the linearity of expectations, we obtain,
\begin{align*}
    \E[\|\bm x\|^{k-2} \bm x^{\otimes 2}] -  \E[\|\bm z\|^{k-2}\bm z^{\otimes 2}] = \pm \frac{\lambda}{\dim} \cdot \bm V^{\otimes 2}.
\end{align*}
Substituting \eqref{eq: pop-analysis-eq3} and \eqref{eq: pop-analysis-eq2} in \eqref{eq: pop-analysis-eq1} we obtain,
\begin{align*}
    \E \hat{\bm M} = \pm \frac{\lambda}{\dim} \cdot \bm V \bm V^\UT.
\end{align*}
\end{proof}

The concentration result controls the fluctuation $\|\hat{\bm M} - \E \hat{\bm M}\|_\op$.  

\begin{proposition} \label{prop: matrix-concentration} Suppose that the sub-Gaussian Assumption (Assumption~\ref{ass: subgauss}) holds with variance proxy $\varproxy$. Then, there is a constant $C_{k,\varproxy}$ that depends only on $k,\varproxy$ such that, for any $\epsilon \in (0,1), \; K \geq 1$, if
\begin{align*}
    \ssize \geq \frac{C_{k,\varproxy} \cdot K^{k+1} \cdot \dim^{\frac{k}{2}}}{\epsilon^2} \cdot \ln\left( \frac{C_{k,\varproxy} \cdot K \cdot \dim}{\epsilon}\right),
\end{align*}
then we have 
\begin{align*}
    \P( \|\hat{\bm M} - \E \hat{\bm M}\|_{\op} \geq \epsilon) &\leq \frac{1}{\ssize^K}.
\end{align*}
\end{proposition}
The proof of this result relies on the Matrix Rosenthal Inequality developed by \citet{mackey2014matrix} and can be found in Section~\ref{sec: matrix-concentration}. With this concentration result, we can now prove Theorem~\ref{thm:ngca-spectral}, which provides a performance guarantee for the proposed spectral estimator in \eqref{eq: spectral-estimator-ngca}.
\begin{proof}[Proof of Theorem~\ref{thm:ngca-spectral}]
Consider the event:
\begin{align*}
    \mathcal{E} \explain{def}{=} \left\{\|\hat{\bm M} - \E[\hat{\bm M}] \|_\op \leq \frac{\epsilon \lambda}{2} \right\}.
\end{align*}
The assumption \eqref{eq: spectral-sample-complexity} on the sample size along with Proposition~\ref{prop: matrix-concentration} guarantees that $\P( \mathcal{E}) \geq 1 - 1/\ssize^K$. On the event $\mathcal{E}$, we have, by Weyl's theorem,
\begin{align*}
    |\lambda_i(\hat{\bm M}) - \lambda_i(\E[\hat{\bm M}])| & \leq  \|\hat{\bm M} - \E[\hat{\bm M}]\|_\op \leq  \epsilon \lambda/2.
\end{align*}
By Lemma~\ref{lemma: matrix_expectation} $\E[\hat{\bm M}] = \pm \lambda \cdot \bm V \bm V^\UT /\dim$. Hence, $|\lambda_1(\E[\hat{\bm M}])| = \lambda$ and $|\lambda_i(\E[\hat{\bm M}])| = 0$ for all $i \geq 2$.
Consequently,
\begin{align*}
    \lambda_1(\hat{\bm M}) & \geq\lambda - \epsilon \lambda/2, \\
    \lambda_i(\hat{\bm M}) & \leq \epsilon \lambda/2, \; \forall \; i \; \geq \; 2.
\end{align*}
This gives us the claim about the spectral gap:
\begin{align*}
    \frac{|\lambda_i(\hat{\bm M})|}{|\lambda_1(\hat{\bm M})|} \leq \frac{\epsilon/2}{1-\epsilon/2} \leq \epsilon.
\end{align*}
The Davis-Kahan Theorem then gives,
\begin{align*}
    \frac{\ip{\bm \nu}{\hat{\bm \nu}}^2}{\|\bm \nu\|^2\|\hat{\bm \nu}\|^2} &\geq 1 - \left(\frac{ \|\hat{\bm M} - \E[\hat{\bm M}]\|_\op}{\lambda - \epsilon \lambda/2} \right)^2 \\
    & \geq 1 - \left(\frac{ \epsilon \lambda/2}{\lambda - \epsilon \lambda/2} \right)^2 \\
    & \geq 1 - \epsilon^2.
\end{align*}
This concludes the proof of the theorem. 
\end{proof}

\subsubsection{Concentration Analysis} \label{sec: matrix-concentration}

In this section, we prove the concentration bound on $\hat{\bm M} -\E \hat{\bm M}$ given in Proposition~\ref{prop: matrix-concentration}. We begin by defining:
\begin{align*}
    \bm \Phi_i \explain{def}{=} 
    (\|\bm x_i\|^2-\dim)^{\frac{k-2}{2}} \cdot  \bm x_i \bm x_i^\UT.
\end{align*}
Observe that we can write: 
\begin{align*}
    \hat{\bm M} -\E \hat{\bm M} \explain{def}{=} \frac{1}{\ssize} \sum_{i=1}^\ssize (\bm \Phi_i - \E \bm \Phi_i).
\end{align*}

When $k$ is even, we will use the Matrix Chebychev method along with the Matrix Rosenthal inequality developed by \citet{mackey2014matrix} to analyze the above sum of independent random matrices. We begin by recalling the results of \citeauthor{mackey2014matrix}.

\begin{fact}[Matrix Chebychev Method \citealp{mackey2014matrix}] \label{fact: matrix-chebychev} Let $\overline{\bm \Phi} \in \R^{\dim\times \dim}$ be a random matrix. We have,
\begin{align*}
    \P\left( \|\overline{\bm \Phi}\|_{\op} \geq \epsilon \right) & \leq \inf_{t \geq 1} \frac{\E \|\overline{\bm \Phi}\|_t^t}{\epsilon^t}.
\end{align*}
In the above display, $\|\overline{\bm \Phi}\|_t$ denotes the Schatten-$t$ norm of a matrix:
\begin{align*}
    \|\overline{\bm \Phi}\|_{t}^t & \explain{def}{=} \sum_{i=1}^\dim |\sigma_i(\overline{\bm \Phi})|^t,
\end{align*}
where $\sigma_{i}(\overline{\bm  \Phi}), \; i \; \in \; [\dim]$ denote the singular values of $\overline{\bm \Phi}$.
\end{fact}

\begin{fact}[Matrix Rosenthal Inequality \citealp{mackey2014matrix}] \label{fact: matrix-rosenthal} Let $\bm \Phi,\bm \Phi_1, \bm \Phi_2, \dotsc, \bm \Phi_\ssize \in \R^{\dim \times \dim}$ be $\ssize+1$ i.i.d. random matrices. Define:
\begin{align*}
    \overline{\bm \Phi} \explain{def}{=} \frac{1}{\ssize} \sum_{i=1}^\ssize (\bm \Phi_i - \E \bm \Phi_i).
\end{align*}
Then, for any $t \geq 3/2$, we have:
\begin{align*}
    (\E[\|\overline{\bm \Phi}\|_{4t}^{4t}])^{\frac{1}{4t}} & \leq \frac{\sqrt{4t-1}}{\sqrt{\ssize}} \cdot  \|(\E \bm \Phi^\UT \bm \Phi - \E[\bm \Phi^\UT \bm \Phi])^{\frac{1}{2}}\|_{4t} \vee \|(\E \bm \Phi \bm \Phi^\UT - \E[\bm \Phi \bm \Phi^\UT])^{\frac{1}{2}}\|_{4t}  \\ & \hspace{9cm}+ \frac{(4t-1)}{\ssize^{1-\frac{1}{4t}}} \cdot (\E \|\bm \Phi - \E \bm \Phi\|_{4t}^{4t})^{\frac{1}{4t}}.
\end{align*}
\end{fact}

In order to apply the Matrix Rosenthal inequality, we need to obtain bounds on the matrix moments $\|(\E \bm \Phi^2 - (\E \bm \Phi)^2)^{\frac{1}{2}}\|_{4t}$ and $\E \|\bm \Phi - \E \bm \Phi\|_{4t}$. We will find the following lemma which bounds some scalar moments useful towards this goal. 

\begin{lemma} \label{lemma: moment_bounds_ngca} Let $\bm x \sim \dmu{\bm V}$ satisfy Assumption~\ref{ass: subgauss} with variance proxy $\varproxy$. Furthermore, suppose that $\E \bm x \bm x^\UT = \bm I_\dim$. Then, there is a universal constant $C$ (independent of $\varproxy$) such that,
\begin{enumerate}
    \item $\E[|\ip{\bm x}{\bm u}|^{t}] \leq (Ct \varproxy)^{t/2}$ for any $\bm u \in \R^\dim$ with $\|\bm u\| = 1$, and any $t \in \N$.
     \item $\E |\|\bm x\|^2 - \dim|^{t} \leq  (C \varproxy^2 t)^{t} + (C \varproxy^2 t \dim)^{\frac{t}{2}}$ for any $t \in \N $.
    \item $\E \|\bm x\|^{2t} \leq (C \varproxy^2 t)^{t} +  (C \varproxy^2 t \dim)^{\frac{t}{2}} + (C\dim)^t$  for any $t \in \N$.
\end{enumerate}
\end{lemma}
\begin{proof}
We prove each claim one by one. 
\begin{enumerate}
    \item Since $\bm x$ is a sub-Gaussian vector with variance proxy $\varproxy$, $\ip{\bm u}{\bm x}$ is a sub-Gaussian random variable with variance proxy $\varproxy$. Item (1) now follows from standard estimates on the moments of sub-Gaussian random variables (see for e.g. \citealp[Proposition 2.5.2]{vershynin2018high}).
    \item Observe that $\|\bm x\|^2 - \dim = (\|\bm x\| - \sqrt{\dim})\cdot (\|\bm x \| + \sqrt{\dim}) = (\|\bm x \| - \sqrt{\dim})^2 + 2\sqrt{\dim} \cdot (\|\bm x\| - \sqrt{\dim})$. Using the scalar inequality $|a+b|^t \leq 2^{t-1} (|a|^t + |b|^t)$
    \begin{align*}
        \E |\|\bm x\|^2 - \dim|^{t} & \leq 2^{t-1} \cdot \E |\|\bm x\| - \sqrt{\dim}|^{2t} + 2^{2t-1} \cdot \dim^{\frac{t}{2}} \cdot \E |\|\bm x\| - \sqrt{\dim}|^{t}
        \end{align*}
    \citet[Theorem 3.1.1]{vershynin2018high} shows that the random variable $\|\bm x \| - \sqrt{\dim}$ is sub-Gaussian with variance proxy $C\varproxy^2$ for some universal constant $C$\footnote{Even though \citeauthor{vershynin2018high} shows this for sub-Gaussian vectors with independent coordinates, the result can be applied here since $\bm x$ is obtained by rotating a sub-Gaussian vector with independent coordinates.}. Hence using moment estimates for sub-Gaussian random variables (see for e.g. \citealp[Proposition 2.5.2]{vershynin2018high}), we have
    \begin{align*}
        \E |\|\bm x\| - \sqrt{\dim}|^{t} & \leq (C \varproxy^2 t)^{\frac{t}{2}}. 
    \end{align*}
    Hence,
    \begin{align*}
         \E |\|\bm x\|^2 - \dim|^{t} & \leq (C \varproxy^2 t)^{t} + (C \varproxy^2 t \dim)^{\frac{t}{2}}.
    \end{align*}
    This proves item (2).
\item Using the scalar inequality $|a+b|^t \leq 2^{t-1} \cdot (|a|^t + |b|^t)$ we obtain,
\begin{align*}
    \E \|\bm x\|^{2t} & \leq 2^{t-1} \cdot \left(  \E |\|\bm x\|^2 - \dim|^{t}  + \dim^t \right) \leq (C \varproxy^2 t)^{t} +  (C \varproxy^2 t \dim)^{\frac{t}{2}} + (C\dim)^t.
\end{align*}
\end{enumerate}
This concludes the proof of the lemma. 
\end{proof}

We now use Lemma~\ref{lemma: moment_bounds_ngca} to control the matrix moments required to apply the Matrix Rosenthal Inequality. This is the content of the following lemma. 
\begin{lemma} \label{lemma: moment_bounds_for_rosenthal} Let $\bm x \sim \dmu{\bm V}$ satisfy Assumption~\ref{ass: subgauss} with variance proxy $\varproxy$. Furthermore, suppose that $\E \bm x \bm x^\UT = \bm I_\dim$. Let $\bm \Phi = (\|\bm x\|^2-\dim)^{\ell} \cdot  \bm x \bm x^\UT $ for some $\ell \in \N$.  There is a universal constant $C_\ell$ (depending only on $\ell$) such that the matrix $\bm \Phi$  satisfies the following moment estimates for any $t \geq 1$,
\begin{enumerate}
    \item $\|(\E \bm \Phi^2 - (\E \bm \Phi)^2)^{\frac{1}{2}}\|_{t} \leq C_\ell \cdot \varproxy^{\ell + \frac{1}{2}} \cdot (\varproxy^2 + \dim)^{\frac{\ell+1}{2}} \cdot \dim^{\frac{1}{t}}$.
    \item $(\E \|\bm \Phi - \E \bm \Phi\|_{t}^{t})^{\frac{1}{t}} \leq C_\ell \cdot (\varproxy^2 t + \varproxy \sqrt{t\dim})^{\ell} \cdot (\varproxy^2 t + \varproxy \sqrt{t\dim}+\dim)$.
\end{enumerate}
In the above display, $C$ is a universal constant (independent of $t,\dim,\bm \nu,\bm \Sigma$).
\end{lemma}

\begin{proof}
We prove each part separately.
\begin{enumerate}
    \item Recalling the definition of Schatten $t$-norms, we know that,
    \begin{align*}
        \|(\E \bm \Phi^2 - (\E \bm \Phi)^2)^{\frac{1}{2}}\|_{t} & \leq \dim^{\frac{1}{t}} \cdot \|(\E \bm \Phi^2 - (\E \bm \Phi)^2)^{\frac{1}{2}}\|_{\op} \leq \dim^{\frac{1}{t}} \cdot \sqrt{\|\E \bm \Phi^2\|_{\op} + \|\E \bm \Phi\|_{\op}^2}.
    \end{align*}
    Next, we observe that there is a deterministic unit vector $\bm u \in \R^\dim$ such that $\|\E \bm \Phi^2\|_{\op} = \bm u^\UT  \E \bm \Phi^2 \bm u$. Hence,
    \begin{align}
        \|\E \bm \Phi^2\|_{\op} & = \E \ip{\bm x}{\bm u}^2 \cdot \|\bm x\|^2 \cdot (\|\bm x_i\|^2-\dim)^{2\ell} \nonumber \\
        & \leq \left( \E[\ip{\bm x}{\bm u}^6] \cdot \E[ \|\bm x\|^6] \cdot \E[(\|\bm x_i\|^2-\dim)^{6\ell}] \right)^{\frac{1}{3}} \nonumber \\
        & \explain{(a)}{\leq} \left(C_\ell \cdot \varproxy^3 \cdot (\varproxy^6+\varproxy^3 \dim^{\frac{3}{2}} + \dim^3) \cdot (\varproxy^{12\ell} + \varproxy^{6\ell} \dim^{3 \ell}) \right)^{\frac{1}{3}} \nonumber \\
        & \leq C_\ell \cdot \varproxy \cdot(\varproxy^2 + \varproxy \sqrt{\dim} + \dim) \cdot (\varproxy^{4\ell} + \varproxy^{2l} \dim^\ell) \nonumber \\
        & \leq C_\ell \cdot \varproxy^{2\ell+1} \cdot (\varproxy^2 + \dim)^{\ell+1}. \label{eq: moment-estimate-rosenthal-phi2}
    \end{align}
In the step marked (a), we used the moment bounds obtained in Lemma~\ref{lemma: moment_bounds_ngca}. Likewise, there is a vector $\bm v$ such that $\|\E \bm \Phi\|_{\op} = \bm v^\UT  \E \bm \Phi \bm v$. This gives us,
\begin{align}
     \|\E \bm \Phi\|_{\op} & = \E \ip{\bm x}{\bm u}^2 \cdot (\|\bm x_i\|^2-\dim)^{\ell}  \nonumber \\
        & \leq \left( \E[\ip{\bm x}{\bm u}^4] \cdot \cdot \E[(\|\bm x_i\|^2-\dim)^{2\ell}] \right)^{\frac{1}{2}} \nonumber \\
        & \explain{(b)}{\leq} \left(C_\ell \cdot \varproxy^2  \cdot (\varproxy^{4\ell} + \varproxy^{2\ell} \dim^{\ell}) \right)^{\frac{1}{2}}  \nonumber \\
        & \leq C_\ell \cdot \varproxy  \cdot (\varproxy^{2\ell} + \varproxy^{\ell} \dim^{\frac{\ell}{2}}) \nonumber \\
        & \leq C_\ell \cdot \varproxy^{\ell+1} \cdot (\varproxy^2 + \dim)^{\ell/2}. \label{eq: moment-estimate-rosenthal-phi}
\end{align}
In the step marked (a), we used the moment bounds obtained in Lemma~\ref{lemma: moment_bounds_ngca}. Combining the estimates in \eqref{eq: moment-estimate-rosenthal-phi2} and \eqref{eq: moment-estimate-rosenthal-phi2} gives us:
\begin{align*}
    \|(\E \bm \Phi^2 - (\E \bm \Phi)^2)^{\frac{1}{2}}\|_{t} & \leq C_\ell \cdot \varproxy^{\ell + \frac{1}{2}} \cdot (\varproxy^2 + \dim)^{\frac{\ell+1}{2}} \cdot \dim^{\frac{1}{t}}.
\end{align*}
\item We first observe that,
\begin{align*}
    \|\bm \Phi - \E \bm \Phi\|_{t}^{t} & \leq 2^{t} \cdot \left(  \|\bm \Phi\|_{t}^{t} + \|\E \bm \Phi\|_{t}^{t} \right).
\end{align*}
Hence,
\begin{align*}
    (\E\|\bm \Phi - \E \bm \Phi\|_{t}^{t})^{\frac{1}{t}} & \leq 2 \cdot  (\E\|\bm \Phi\|_{t}^{t})^{\frac{1}{t}}) +  2 \cdot \|\E \bm \Phi\|_{t} \\ & \explain{(b)}{\leq}   2 \cdot  (\E\|\bm \Phi\|_{t}^{t})^{\frac{1}{t}} +  C_\ell \cdot \varproxy^{\ell+1} \cdot (\varproxy^2 + \dim)^{\ell/2} \cdot \dim^{\frac{1}{t}}.
\end{align*}
In step (c) we relied on the estimate in \eqref{eq: moment-estimate-rosenthal-phi}. Since $\bm \Phi$ is a rank-1 matrix,
\begin{align*}
    \|\bm \Phi\|_{t}^{t} = |\|\bm x_i\|^2-\dim|^{t\ell} \cdot \|\bm x\|^{2t}.
\end{align*}
Hence, using the moment bounds in Lemma~\ref{lemma: moment_bounds_ngca}, we obtain,
\begin{align*}
   (\E\|\bm \Phi\|_{t}^{t})^{\frac{1}{t}} &\leq \left(  \E[|\|\bm x_i\|^2-\dim|^{2t\ell}] \cdot \E[\|\bm x\|^{4t}] \right)^{\frac{1}{2t}} \\ & \leq \left( (C_\ell \varproxy^2 t + C_\ell \varproxy \sqrt{t \dim})^{2t \ell} \cdot  (C_\ell \varproxy^2 t + C_\ell \varproxy \sqrt{t \dim} + C_\ell \dim)^{2t} \right)^{\frac{1}{2t}} \\
   & \leq C_\ell \cdot (\varproxy^2 t + \varproxy \sqrt{t\dim})^{\ell} \cdot (\varproxy^2 t + \varproxy \sqrt{t\dim}+\dim).
\end{align*}
Hence,
\begin{align*}
  \left(  \E  \|\bm \Phi - \E \bm \Phi\|_{t}^{t} \right)^{\frac{1}{t}} & \leq C_\ell \cdot (\varproxy^2 t + \varproxy \sqrt{t\dim})^{\ell} \cdot (\varproxy^2 t + \varproxy \sqrt{t\dim}+\dim).
\end{align*}
\end{enumerate}
\end{proof}
With these moment estimates in hand, we are now in the position to prove Proposition~\ref{prop: matrix-concentration}.
\begin{proof}[Proof of Proposition~\ref{prop: matrix-concentration}]Using Matrix Rosenthal inequality (Fact~\ref{fact: matrix-rosenthal}) and the Matrix Chebychev method (Fact~\ref{fact: matrix-chebychev}), we obtain,
\begin{align*}
     \P\left( \|\hat{\bm M} -\E \hat{\bm M}\|_{\op} \geq \epsilon \right) & \leq \inf_{t \geq 1.5} \left( \frac{\sqrt{4t-1}}{\epsilon\sqrt{\ssize} } \cdot  \|(\E \bm \Phi^2 - (\E \bm \Phi)^2)^{\frac{1}{2}}\|_{4t} + \frac{(4t-1)}{\epsilon\ssize^{1-\frac{1}{4t}}} \cdot (\E \|\bm \Phi - \E \bm \Phi\|_{4t}^{4t})^{\frac{1}{4t}} \right)^{4t}.
\end{align*}
Substituting the moment bounds from Lemma~\ref{lemma: moment_bounds_for_rosenthal}, we obtain,
\begin{align*}
    \P\left( \|\hat{\bm M} -\E \hat{\bm M}\|_{\op} \geq \epsilon \right) & \leq \inf_{t \geq 1.5} \left( \frac{C_{k,\varproxy}  \cdot \sqrt{t} \cdot \dim^{\frac{k}{4}+\frac{1}{4t}}}{\epsilon\sqrt{\ssize} }   +  \frac{C_{k,\varproxy} \cdot  t^{\frac{k+2}{2}}}{\epsilon\ssize^{1-\frac{1}{4t}}}+ \frac{C_{k,\varproxy} \cdot (t\dim)^{\frac{k+2}{4}} }{\epsilon\ssize^{1-\frac{1}{4t}}}  \right)^{4t}.
\end{align*}
In the above display, $C_{k,\varproxy}$ denotes a constant that depends only on $\varproxy,k$.
We chose $t = K \ln(\ssize)/4$ for some $K \geq 1$. Note that if $N \geq \dim$, then, $\dim^{\frac{1}{4t}} \leq \ssize^{\frac{1}{4t}} = e^{\frac{\ln(\ssize)}{K\ln(\ssize)}} \leq e^{\frac{1}{K}} \leq e$. Furthermore, if,
\begin{subequations} \label{eq: provision-moment}
\begin{align} 
    \ssize &\geq \frac{9 C^2_{k,\varproxy} \cdot e^4 \cdot K \cdot \ln(\ssize) \cdot \dim^{\frac{k}{2}}}{\epsilon^2}, \\
    \ssize &\geq \frac{3C_{k,\varproxy} \cdot e^2 \cdot  (K\ln(\ssize))^{\frac{k+2}{2}}}{\epsilon}, \\
    \ssize &\geq \frac{3C_{k,\varproxy} \cdot e^2 \cdot  (K\ln(\ssize) \dim)^{\frac{k+2}{4}}}{\epsilon}
\end{align}
\end{subequations}
then, we obtain,
\begin{align*}
     \P\left( \|\hat{\bm M} -\E \hat{\bm M}\|_{\op} \geq \epsilon \right) & \leq e^{-K\ln(\ssize)} = \frac{1}{\ssize^K}.
\end{align*}
Finally we note that the assumption:
\begin{align*}
    \ssize \geq \frac{C_{k,\varproxy} \cdot K^{k+1} \cdot \dim^{\frac{k}{2}}}{\epsilon^2} \cdot \ln\left( \frac{C_{k,\varproxy} \cdot K \cdot \dim}{\epsilon}\right),
\end{align*}
for a suitably large constant $C_{k,\varproxy}$, guarantees \eqref{eq: provision-moment}. This proves the claim. 
\end{proof}

\section{Additional Results for $k$-CCA}
\subsection{Computationally Inefficient Estimators} \label{sec: cca-bruteforce}
In this section, we analyze a computationally inefficient, but statistically optimal estimator for the $k$-CCA problem. The main result of this section is the following.
\begin{theorem}\label{thm:cca-brute-force} Suppose that each view $\mview{\bm x}{\ell}$ is sub-Gaussian with variance proxy $\varproxy$.  Then, there is a constant $C_k$ depending only on $k$ such for any $\epsilon \in (0,1)$ if,
\begin{align*}
    \ssize & \geq  \frac{C_k \cdot \varproxy^k \cdot \dim}{\lambda^2 \epsilon^2} \cdot \ln\left( \frac{C_k}{\epsilon} \right).
\end{align*}
then, the estimator $\hat{\bm V}_{\delta, \thresh}$ with $\delta = \epsilon/(3C_k)$ and $\thresh^2 =  C_k \cdot \varproxy^k  \cdot \ln^k\left( \frac{C_k \cdot  \varproxy}{\lambda \epsilon} \right)$ satisfies
\begin{align*}
    \left\| \hat{\bm V}_{\delta, \thresh} - \bm v_1 \otimes \bm v_2 \otimes \dotsb \otimes \bm v_k \right\|^2 & \leq \epsilon^2,
\end{align*}
with probability $1-2 e^{-\dim}$.
\end{theorem}

In order to motivate the estimator used to obtain Theorem \ref{thm:cca-brute-force}, we recall that in the $k$-CCA problem, a single sample $\bm x = (\mview{\bm x}{1}, \dotsc, \mview{\bm x}{k}) \in \R^{k\dim}$ satisfies:

\begin{align*}
    \E[\mview{\bm x}{1} \otimes \mview{\bm x}{2} \otimes \dotsb \otimes \mview{\bm x}{k}] & = \lambda \cdot  \bm v_1 \otimes \bm v_2 \otimes \dotsb \otimes \bm v_k
\end{align*}

This suggests that $\bm v_1 \otimes \bm v_2 \otimes \dotsb \otimes \bm v_k$ can be estimated by computing the best rank-1 approximation to the empirical estimate of the tensor $ \E[\mview{\bm x}{1} \otimes \mview{\bm x}{2} \otimes \dotsb \otimes \mview{\bm x}{k}]$ i.e.

\begin{align*}
    \arg\min_{\substack{\bm u_{1:k} \in \sphere{d-1}}}  \left\| \frac{1}{\ssize} \sum_{i=1}^\ssize \mview{\bm x}{1}_i \otimes \dotsb \otimes \mview{\bm x}{k}_i - \lambda  \cdot  \bm u_1 \otimes \bm u_2 \otimes \dotsb \otimes \bm u_k  \right\|,
\end{align*}

where $\| \cdot \|$ is a suitable measure of the discrepancy between tensors. We will find it convenient to use the following discrepancy measure: Let $\net{\delta}$ denote the smallest $\delta$-net of $\sphere{\dim-1}$. For any two functions $$f_1,f_2 : \underbrace{\sphere{\dim-1}\times \sphere{\dim-1} \times \dotsb \times \sphere{\dim-1}}_{\text{k times}} \mapsto \R$$ we define the discrepancy measure:

\begin{align*}
    \left\|f_1(\cdot, \cdot, \dotsc, \cdot) - f_2(\cdot, \cdot, \dotsc, \cdot) \right\|_{\net{\delta}} \explain{def}{= } \max_{\bm w_{1:k} \in \net{\delta}} | f_1(\bm w_1, \dotsc, \bm w_k) - f_2(\bm w_1, \dotsc, \bm w_k) |. 
\end{align*}

Furthermore, in order to handle heavy tailed random variables, we will find it convenient to truncate them. Define the truncation function at threshold $\thresh \geq 0$:
\begin{align}
    \trunc{\thresh}(x) \explain{def}{=} \max(\min(x, \thresh),-\thresh).
\end{align}

The final estimator we analyze is given by:

\begin{align*}
    \hat{\bm V}_{\delta,\thresh} &\explain{def}{=} \hat{\bm v}_1 \otimes \hat{\bm v}_2 \otimes \dotsb \otimes \hat{\bm v}_k \\ (\hat{\bm v}_{1:k})& =  \arg\min_{\substack{\bm u_{1:k} \in \sphere{d-1}}}  \left\| \frac{1}{\ssize} \sum_{i=1}^\ssize \trunc{\thresh}\left(\ip{\mview{\bm x}{1}_i}{\cdot} \cdot \ip{\mview{\bm x}{2}_i}{\cdot} \dotsb \ip{\mview{\bm x}{k}_i}{\cdot}\right)  -  \lambda   \ip{\bm u_1}{\cdot} \cdot \ip{\bm u_2}{\cdot}  \dotsb  \ip{\bm u_k}{\cdot}  \right\|_{\net{\delta}}.
\end{align*}
In the above display, $\thresh$ and $\delta$ are tuning parameters. The proof of Theorem \ref{thm:cca-brute-force} requires several intermediate results which we state and prove before presenting the proof of Theorem \ref{thm:cca-brute-force}. The following lemma provides guidance on how to set the threshold $\thresh$.

\begin{lemma}  Suppose that each view $\mview{\bm x}{\ell}$ is sub-Gaussian with variance proxy $\varproxy$. Then, there is a constant $C_k$ depending only on $k$ such that for any $\lambda, \epsilon \geq 0$ if,
\begin{align*}
    \thresh^2  \geq C_k \cdot \varproxy^k  \cdot \ln^{k}\left( \frac{C_k\cdot  \varproxy}{\lambda \epsilon} \right)
\end{align*}
then, for any $\bm u_{1:k} \in \sphere{\dim - 1}$,
\begin{align*}
  \left|\E\left[\trunc{\thresh}\left(\ip{\bm u_1}{\mview{\bm x}{1}} \cdot \ip{\bm u_2}{\mview{\bm x}{2}} \dotsb \ip{\bm u_k}{\mview{\bm x}{k}}\right)\right] - \E\left[\ip{\bm u_1}{\mview{\bm x}{1}} \cdot \ip{\bm u_2}{\mview{\bm x}{2}} \dotsb \ip{\bm u_k}{\mview{\bm x}{k}}\right]\right| & \leq \frac{\lambda \epsilon}{6}.
\end{align*}
\end{lemma}
\begin{proof}
The proof of this lemma is analogous to Lemma \ref{lemma: tuning-threshold} and is omitted. 
\end{proof}

We will also need the following concentration result in our analysis. 

\begin{lemma}Suppose that each view $\mview{\bm x}{\ell}$ is sub-Gaussian with variance proxy $\varproxy$. Then, there is a constant $C_k$ depending only on $k$ such that if,
\begin{align*}
    \ssize & \geq  C_k \cdot \dim \cdot  \left( \frac{ \varproxy^k}{\lambda^2 \epsilon^2} + \frac{\thresh}{\lambda \epsilon} \right) \cdot \ln\left( \frac{C}{\delta} \right),
\end{align*}
then, with probability $1 - 2 e^{-\dim}$,
\begin{align*}
     &\left\| \frac{1}{\ssize} \sum_{i=1}^\ssize \trunc{\thresh}\left(\ip{\mview{\bm x}{1}_i}{\cdot} \cdot \ip{\mview{\bm x}{2}_i}{\cdot} \dotsb \ip{\mview{\bm x}{k}_i}{\cdot}\right) - \E\left[\trunc{\thresh}\left(\ip{\bm u_1}{\mview{\bm x}{1}} \cdot \ip{\bm u_2}{\mview{\bm x}{2}} \dotsb \ip{\bm u_k}{\mview{\bm x}{k}}\right)\right]  \right\|_{\net{\delta}} \\&\hspace{14cm}\leq \lambda \epsilon/6.
\end{align*}
\end{lemma}
\begin{proof}
The proof of this lemma is analogous to the proof of Lemma \ref{lemma: tensor-norm-concentration} and is omitted. 
\end{proof}

Finally, we will require the following quantitative identifiability result which shows that if, for two collections of unit vectors $\bm u_1, \bm u_2, \dotsc, \bm u_k$ and $\bm u_1^\prime, \bm u_2^\prime, \dotsc, \bm u_k^\prime$ if  $$\|\ip{\bm u_1}{\cdot} \cdot \ip{\bm u_2}{\cdot} \cdot \dotsb \cdot \ip{\bm u_k}{\cdot} - \ip{\bm u_1^\prime}{\cdot} \cdot \ip{\bm u_2^\prime}{\cdot} \cdot \dotsb \cdot \ip{\bm u_k^\prime}{\cdot}\|_{\net{\delta}} \approx 0$$ then we must have $\bm u_1 \otimes \bm u_2 \otimes \dotsb \otimes \bm u_k \approx \bm u_1^\prime \otimes \bm u_2^\prime \otimes \dotsb \otimes \bm u_k^\prime$.

\begin{lemma} \label{lemma-tensor-wedin-cca} There is a constant $C_k$ depending only on $k$ such that for any two collections of unit vectors $\bm u_1, \bm u_2, \dotsc, \bm u_k$ and $\bm u_1^\prime, \bm u_2^\prime, \dotsc, \bm u_k^\prime$,
\begin{align*}
&\|\bm u_1 \otimes  \dotsb \otimes \bm u_k - \bm u_1^\prime  \otimes \dotsb \otimes \bm u_k^\prime\| \\ &\hspace{4cm} \leq  C_k \left(\|\ip{\bm u_1}{\cdot} \cdot \ip{\bm u_2}{\cdot} \cdot \dotsb \cdot \ip{\bm u_k}{\cdot} - \ip{\bm u_1^\prime}{\cdot} \cdot \ip{\bm u_2^\prime}{\cdot} \cdot \dotsb \cdot \ip{\bm u_k^\prime}{\cdot}\|_{\net{\delta}} + \delta\right).
\end{align*}
\end{lemma}

\begin{proof}
For any $i \in [k]$ define $\alpha_i = \ip{\bm u_i}{\bm u_i^\prime}$. Without loss of generality we assume that $|\alpha_1| = \min_{i \in [k]} |\alpha_i|$. Recall that,
\begin{align*}
    &\|\ip{\bm u_1}{\cdot} \cdot \dotsb \cdot \ip{\bm u_k}{\cdot} - \ip{\bm u_1^\prime}{\cdot}\cdot \dotsb \cdot \ip{\bm u_k^\prime}{\cdot}\|_{\net{\delta}} \\& \hspace{6cm}\explain{def}{=} \max_{\bm w_{1:k}} |\ip{\bm u_1}{\bm w_1} \cdot \dotsb \cdot \ip{\bm u_k}{\bm w_k} - \ip{\bm u_1^\prime}{\bm w_1} \cdot \dotsb \cdot \ip{\bm u_k^\prime}{\bm w_k} |.
\end{align*}We consider the two cases:
\begin{description}
\item [Case 1 $|\alpha_1| \geq 0.5,\; \alpha_1\alpha_2 \dotsb \alpha_k \geq 0$.] Let $\bm w_1$ be the unit vector in the span of $\bm u_1, \bm u_1^\prime$ that is orthogonal to $\bm u_1$, i.e., 
\begin{align*}
    \bm w_1 = \frac{\bm u_1^\prime - \ip{\bm u_1}{\bm u_1^\prime } \bm u_1}{\sqrt{1- \ip{\bm u_1}{\bm u_1^\prime }^2}},
\end{align*}
let $\bm w_i = \bm u_i$ for $i \geq 2$. Let $\hat{\bm w}_i = \arg\min_{\bm w \in \net{\delta}} \|\bm w_i - \bm w\|$ for $i \geq 2$. Note that $\|\bm w_i - \hat{\bm w}_i \| \leq \delta$. Hence, we can lower bound,
\begin{align*}
     &\|\ip{\bm u_1}{\cdot}  \dotsb  \ip{\bm u_k}{\cdot} - \ip{\bm u_1^\prime}{\cdot} \dotsb \ \ip{\bm u_k^\prime}{\cdot}\|_{\net{\delta}} \\& \hspace{6cm} \geq |\ip{\bm u_1}{\hat{\bm w_1}}  \dotsb  \ip{\bm u_k}{\hat{\bm w_k}} - \ip{\bm u_1^\prime}{\hat{\bm w_1}}  \dotsb  \ip{\bm u_k^\prime}{\hat{\bm w_k}} | \\
     & \hspace{6cm}\geq |\ip{\bm u_1}{\bm w_1}  \dotsb  \ip{\bm u_k}{\bm w_k} - \ip{\bm u_1^\prime}{\bm w_1}  \dotsb  \ip{\bm u_k^\prime}{\bm w_k} | - 2 k \delta \\
     & \hspace{6cm}= \sqrt{1-|\alpha_1|^2} \cdot |\alpha_2| \cdot \dotsb \cdot |\alpha_k| - 2 k \delta\\
     & \hspace{6cm}\geq 2^{1-k} \cdot \sqrt{1-|\alpha_1|^2} - 2 k \delta.
\end{align*}
On the other hand we can upper bound,
\begin{align*}
    \|\bm u_1 \otimes  \dotsb \otimes \bm u_k - \bm u_1^\prime  \otimes \dotsb \otimes \bm u_k^\prime\| & = \sqrt{2 - 2 \cdot \alpha_1 \cdot \alpha_2 \cdot \dotsb \cdot \alpha_k} \\
    & \leq \sqrt{2(1-|\alpha_1|^k)}. 
\end{align*}
We observe that,
\begin{align*}
    \frac{\sqrt{1-|\alpha_1|^2}}{\sqrt{1-|\alpha_1|^k}} & = \frac{\sqrt{1+|\alpha_1|}}{\sqrt{1+|\alpha_1| + |\alpha_1|^2 + \dotsb + |\alpha_1|^{k-1}}} \geq \frac{1}{\sqrt{k}}.
\end{align*}
Hence,
\begin{align*}
     \|\ip{\bm u_1}{\cdot}  \dotsb  \ip{\bm u_k}{\cdot} - \ip{\bm u_1^\prime}{\cdot} \dotsb \ \ip{\bm u_k^\prime}{\cdot}\|_{\net{\delta}} & \geq \frac{ \|\bm u_1 \otimes  \dotsb \otimes \bm u_k - \bm u_1^\prime  \otimes \dotsb \otimes \bm u_k^\prime\| }{C_k} - 2k \delta,
\end{align*}
rearranging which gives the claim of the lemma. 
\item [Case 2 $|\alpha_1| \leq 0.5$ or $\alpha_1\alpha_2 \dotsb \alpha_k < 0$.] In this case, we set $\bm w_i = \bm u_i$ for all $i \in [k]$ and define $\hat{\bm w}_i = \arg\min_{\bm w \in \net{\delta}} \|\bm w_i - \bm w\|$ as before. Arguing as before, we can lower bound,
\begin{align*}
     &\|\ip{\bm u_1}{\cdot}  \dotsb  \ip{\bm u_k}{\cdot} - \ip{\bm u_1^\prime}{\cdot} \dotsb \ \ip{\bm u_k^\prime}{\cdot}\|_{\net{\delta}} \\& \hspace{6cm} \geq |\ip{\bm u_1}{\bm w_1}  \dotsb  \ip{\bm u_k}{\bm w_k} - \ip{\bm u_1^\prime}{\bm w_1}  \dotsb  \ip{\bm u_k^\prime}{\bm w_k} | - 2 k \delta \\
     &\hspace{6cm} = |1 - \alpha_1 \alpha_2 \dotsb \alpha_k| - 2 k \delta \\
     &\hspace{6cm} \geq 1/2 - 2 k \delta.
\end{align*}
On the other hand, we can upper bound,
\begin{align*}
     \|\bm u_1 \otimes  \dotsb \otimes \bm u_k - \bm u_1^\prime  \otimes \dotsb \otimes \bm u_k^\prime\| & = \sqrt{2 - 2 \cdot \alpha_1 \cdot \alpha_2 \cdot \dotsb \cdot \alpha_k} \leq \sqrt{2- 2/2} = 1.
\end{align*}
Hence,
\begin{align*}
     \|\ip{\bm u_1}{\cdot}  \dotsb  \ip{\bm u_k}{\cdot} - \ip{\bm u_1^\prime}{\cdot} \dotsb \ \ip{\bm u_k^\prime}{\cdot}\|_{\net{\delta}} & \geq \frac{\|\bm u_1 \otimes  \dotsb \otimes \bm u_k - \bm u_1^\prime  \otimes \dotsb \otimes \bm u_k^\prime\|}{2} - 2k \delta,
\end{align*}
which gives the claim after rearrangement.
\end{description}
\end{proof}

With these supporting lemmas, we can analyze the sample complexity proposed estimator $\hat{\bm V}_{\delta,\thresh}$ and provide a proof for Theorem \ref{thm:cca-brute-force}. 
\begin{proof}[Proof of Theorem \ref{thm:cca-brute-force}] 
The result follows from the intermediate results proved in this section using arguments identical to the ones employed in the proof of  Theorem \ref{thm: bruteforce-ngca-sample-complexity}. We omit the details.
\end{proof}

\subsection{Low Degree Lower Bound} \label{sec:CCA-LDLR}

In this section, we provide evidence for the computational-statistical gap in the $k$-CCA problem.  We provide evidence for the following testing variant of the problem using the Low Degree framework. We refer the reader to Section \ref{sec:ngca-LDLR} for background on the Low Degree framework for computational-statistical gaps. We begin by formally defining the testing problem. In this problem, given a dataset $\{\bm x_1,\bm x_2, \dotsc , \bm x_\ssize\} \subset \R^{k\dim}$, the goal is to design a test $\phi: (\R^\dim)^\ssize \mapsto \{0,1\}$ that distinguishes between the null and alternative hypothesis stated below.
\begin{enumerate}
    \item \emph{Null Hypothesis: } In the null hypothesis, the data $\bm x_{1:\ssize}$ is generated as:
    \begin{align*}
        \bm x_i \explain{i.i.d.}{\sim} \refmu \explain{def}{=} \gauss{\bm 0}{\bm I_{k\dim}}.
    \end{align*}
    \item \emph{Alternative Hypothesis: } In the alternative hypothesis, there are $k$ unknown vectors $\bm v_{1:k} \in \R^\dim$ with $\|\bm v_1 \| = \|\bm v_2 \| = \dotsb = \|\bm v_k\| = \sqrt{\dim}$ such that,
    \begin{align*}
        \bm x_i  \explain{i.i.d.}{\sim} \dmu{\bm V},
    \end{align*}
    where,
    \begin{align*}
        \bm V = \bm v_1 \otimes \bm v_2 \otimes \dotsb \otimes \bm v_k.
    \end{align*}
    In order to specify the measure $\dmu{\bm V}$, we will specify the likelihood ratio of $\dmu{\bm V}$ with respect to $\refmu$ at $\bm x = (\mview{\bm x}{1}, \mview{\bm x}{2}, \dotsc , \mview{\bm x}{k})$:
    \begin{subequations} \label{eq: llr-cca-ldlr}
 \begin{align}
     \frac{\diff \dmu{\bm V}}{\diff \refmu}(\bm x) \explain{def}{=} 1 + \frac{\lambda}{\lambda_k} \cdot \sgn\left(\frac{\ip{\mview{\bm x}{1}}{\bm v_1}}{\sqrt{\dim}}\right)\cdot \sgn\left(\frac{\ip{\mview{\bm x}{2}}{\bm v_2}}{\sqrt{\dim}}\right) \cdot \dotsb \cdot \sgn\left(\frac{\ip{\mview{\bm x}{k}}{\bm v_k}}{\sqrt{\dim}}\right),
 \end{align}
 where, 
 \begin{align}
     \lambda_k \explain{def}{=} \left(\frac{2}{\pi} \right)^{\frac{k}{2}} = (\E |Z|)^{\frac{k}{2}}, \; Z \sim \gauss{0}{1}.
 \end{align}
 \end{subequations}
 It is straightforward to verify the $\dmu{\bm V}$ is a valid probability measure that satisfies \eqref{eq:k-CCA-correlation}. 
\end{enumerate}

A test successfully distinguishes between the null and the alternative hypothesis if it is consistent, that is,

\begin{subequations}\label{eq: consistent-test-cca}
\begin{align}
   \lim_{\dim \rightarrow \infty} \refmu(\phi(\bm x_{1:\ssize} = 0)) &= 1, \\
   \lim_{\dim \rightarrow \infty}  \inf_{\bm V} \dmu{\bm V}(\phi(\bm x_{1:\ssize} = 1)) & = 1.
\end{align}
\end{subequations}

In order to prove the low-degree lower bound, it will be sufficient to restrict ourselves to simpler Bayesian version of the problem where the parameter $\bm V$ is drawn by sampling $\bm v_{1:k} \explain{i.i.d.}{\sim} \prior = \unif{\{\pm 1\}^\dim}$ and setting $\bm V =  \bm v_1 \otimes \bm v_2 \otimes \dotsb \otimes \bm v_k$. Let $\nullmu$ denote the mariginal distribution of the dataset under the alternative hypothesis $\bm x_i \explain{i.i.d.}{\sim} \dmu{\bm V}$ when $\bm V$ is drawn from this distribution:
\begin{align*}
    \nullmu(\cdot) = \int \dmu{\bm v_1 \otimes \dotsb \otimes \bm v_k}^{\otimes \ssize}(\cdot) \; \prior( \diff \bm v_1) \cdot \prior( \diff \bm v_2) \dotsb  \cdot\prior( \diff \bm v_k).
\end{align*}

In order to analyze the low-degree approximation to the likelihood ratio, we first compute the Hermite decomposition of the likelihood ratio of $\nullmu$ with respect to $\refmu$. This is done in the following lemma. 
\begin{lemma} \label{lemma: hermite-decomposition-cca} For any $\bm X = [\bm x_1, \bm x_2, \dotsc , \bm x_\ssize]$, we have,
\begin{align}
    \frac{\diff\nullmu}{\diff\refmu}(\bm X) - 1 &  =  \sum_{\substack{S \subset [n]\\ |S| \geq 1 }}  \sum_{\mview{\bm t}{1:k} \in \N^{|S|}}  \frac{\lambda^{|S|}}{\lambda_k^{|S|}} \cdot \hat{\nu}_{\mview{\bm t}{1}} \cdot \dotsc \cdot\hat{\nu}_{\mview{\bm t}{k}} \cdot \intH{\mview{\bm t}{1}}{\mview{\bm X}{1}_S}{1} \cdot \dotsb \cdot \intH{\mview{\bm t}{k}}{\mview{\bm X}{k}_S}{1},
\end{align}
where, for any $\mview{\bm t}{\ell} \in \N^{|S|}$
\begin{align*}
    \hat{\nu}_{\mview{\bm t}{\ell}} \explain{def}{=} \prod_{i \in S} \hat{\nu}_{\mview{t}{\ell}_i}, \; \intH{\mview{\bm t}{\ell}}{\mview{\bm X}{\ell}_S}{1} \explain{def}{=} \int \left( \prod_{i\in S} H_{\mview{t}{\ell}_i} \left( \frac{\ip{\mview{\bm x}{\ell}_i}{\bm v_\ell}}{\sqrt{\dim}} \right) \right)  \; \prior(\diff \bm v_\ell).
\end{align*}
In the above display, $\hat{\nu}_{t} = \refE[H_t(Z) \sgn(Z)]$ where $Z \sim \gauss{0}{1}$. 
\end{lemma}
With the above result, we can now define the low degree approximation to the likelihood ratio as:
\begin{align} \label{eq:ldlr-cca}
    &\lowdegree{\frac{\diff\nullmu}{\diff\refmu}(\bm X) - 1}{t} \explain{def}{=} \nonumber \\& \hspace{2cm} \sum_{\substack{S \subset [n]\\ |S| \geq 1 }}  \sum_{\substack{\mview{\bm t}{1:k} \in \N^{|S|}\\ \|\mview{\bm t}{1}\|_1 + \dotsb + \|\mview{\bm t}{k}\|_1 \leq t}}  \frac{\lambda^{|S|}}{\lambda_k^{|S|}} \cdot \hat{\nu}_{\mview{\bm t}{1}} \cdot \dotsc \cdot\hat{\nu}_{\mview{\bm t}{k}} \cdot \intH{\mview{\bm t}{1}}{\mview{\bm X}{1}_S}{1} \cdot \dotsb \cdot \intH{\mview{\bm t}{k}}{\mview{\bm X}{k}_S}{1}.
\end{align}

The following proposition provides a bound on the norm of the low degree approximation to the likelihood ratio. 

\begin{proposition}\label{prop:cca-lldr} There is a finite constant $C_k$ depending only on $k$ such that if,
\begin{align*}
    t & \leq \frac{\dim}{C_k}, \; 
    \ssize \lambda^2 \leq \frac{1}{C_k} \cdot \frac{\dim^{\frac{k}{2}}}{t^{\frac{k-2}{2}}},
\end{align*}
we have,
\begin{align*}
     \refE \left[ \lowdegree{\frac{\diff\nullmu}{\diff\refmu}(\bm x_{1:\ssize}) - 1}{t}^2 \right] & \leq 1.
\end{align*}
\end{proposition}
In particular, by the Low-Degree Likelihood Ratio Conjecture of \citet{hopkins2018statistical}, this suggests that in the sample size regime $\ssize \lambda^2 \ll \dim^{k/2}$, the $k$-CCA testing problem is computationally hard. The remainder of this section is devoted to the proof of Lemma \ref{lemma: hermite-decomposition-cca} and Proposition \ref{prop:cca-lldr}. 
\subsubsection{Proof of Lemma \ref{lemma: hermite-decomposition-cca}}
\begin{proof}[Proof of Lemma \ref{lemma: hermite-decomposition-cca}]
We begin by observing that,
\begin{align*}
    \frac{\diff \dmu{\bm V}}{\diff \refmu}(\bm X) -1  & = \prod_{i=1}^\ssize \left( 1 + \frac{\diff \dmu{\bm V}}{\diff \refmu}(\bm x_{i})  -1 \right) -1  \\
    & = \sum_{\substack{S \subset [n] \\ |S| \geq 1}} \dscore{\bm V}(\bm X_S),
\end{align*}
where,
\begin{align*}
    \dscore{\bm V}(\bm X_S) \explain{def}{=} \prod_{i \in S}  \left( \frac{\diff \dmu{\bm V}}{\diff \refmu}(\bm x_{i})  -1 \right).
\end{align*}
Next, we compute the Hermite expansion of $\dscore{\bm V}(\bm X_S)$. We can expand $\dscore{\bm V}(\bm X_S)$ as follows:
\begin{align*}
    \dscore{\bm V}(\bm X_S) &\explain{def}{=} \prod_{i \in S}  \left( \frac{\diff \dmu{\bm V}}{\diff \refmu}(\bm x_{i})  -1 \right) \\
    & = \frac{\lambda^{|S|}}{\lambda_k^{|S|}} \prod_{\ell=1}^k \prod_{i \in S} \sgn\left(\frac{\ip{\mview{\bm x}{\ell}_i}{\bm v_\ell}}{\sqrt{\dim}}\right) \\
    & =  \frac{\lambda^{|S|}}{\lambda_k^{|S|}} \prod_{\ell=1}^k \prod_{i \in S} \sum_{\mview{t}{\ell}_j=1}^\infty \hat{\nu}_{\mview{t}{\ell}_j} H_{\mview{t}{\ell}_j}\left(\frac{\ip{\mview{\bm x}{\ell}_i}{\bm v_\ell}}{\sqrt{\dim}}\right) \\
    & =  \frac{\lambda^{|S|}}{\lambda_k^{|S|}}  \sum_{\mview{\bm t}{1}, \mview{\bm t}{2}, \dotsc , \mview{\bm t}{k} \in \N^{|S|}} \hat{\nu}_{\mview{\bm t}{1}} \cdot \hat{\nu}_{\mview{\bm t}{2}} \cdot \dotsc \cdot\hat{\nu}_{\mview{\bm t}{k}} \cdot  \prod_{\ell=1}^k \prod_{i \in S}  H_{\mview{t}{\ell}_j}\left(\frac{\ip{\mview{\bm x}{\ell}_i}{\bm v_\ell}}{\sqrt{\dim}}\right).
\end{align*}
Recalling the definition of integrated Hermite polynomials (Definition \ref{def: integrated-Hermite}),
\begin{align*}
     \intH{\mview{\bm t}{\ell}}{\mview{\bm X}{\ell}_S}{1} \explain{def}{=} \int \left( \prod_{i\in S} H_{\mview{t}{\ell}_i} \left( \frac{\ip{\mview{\bm x}{\ell}_i}{\bm v_\ell}}{\sqrt{\dim}} \right) \right)  \; \prior(\diff \bm v_\ell).
\end{align*}
we obtain,
\begin{align*}
   \int \dscore{\bm v_1 \otimes \dotsb \otimes \bm v_k}(\bm X_S) \; \prior( \diff \bm v_{1:k})& = \frac{\lambda^{|S|}}{\lambda_k^{|S|}}  \sum_{\mview{\bm t}{1:k} \in \N^{|S|}} \hat{\nu}_{\mview{\bm t}{1}} \cdot \dotsc \cdot\hat{\nu}_{\mview{\bm t}{k}} \cdot \intH{\mview{\bm t}{1}}{\mview{\bm X}{1}_S}{1} \cdot \dotsb \cdot \intH{\mview{\bm t}{k}}{\mview{\bm X}{k}_S}{1},
\end{align*}
and hence,
\begin{align*}
    \frac{\diff\nullmu}{\diff\refmu}(\bm X) - 1 &  = \sum_{\substack{S \subset [n]\\ |S| \geq 1 }}  \sum_{\mview{\bm t}{1:k} \in \N^{|S|}}  \frac{\lambda^{|S|}}{\lambda_k^{|S|}} \cdot \hat{\nu}_{\mview{\bm t}{1}} \cdot \dotsc \cdot\hat{\nu}_{\mview{\bm t}{k}} \cdot \intH{\mview{\bm t}{1}}{\mview{\bm X}{1}_S}{1} \cdot \dotsb \cdot \intH{\mview{\bm t}{k}}{\mview{\bm X}{k}_S}{1},
\end{align*}
as claimed.
\end{proof}

\subsubsection{Proof of Proposition \ref{prop:cca-lldr}}
\begin{proof}[Proof of Proposition \ref{prop:cca-lldr}]
We begin by recalling \eqref{eq:ldlr-cca} which provides a formula for the low-degree approximation to the likelihood ratio. We observe that since $\mview{\bm t}{\ell} \in \N^{|S|}$, we have,
\begin{align*}
   k|S| & \leq  \|\mview{\bm t}{1}\|_1 + \dotsb + \|\mview{\bm t}{k}\|_1 \leq t.
\end{align*}
Hence, $|S| \leq t/k$. This means that,
\begin{align*}
    &\lowdegree{\frac{\diff\nullmu}{\diff\refmu}(\bm X) - 1}{t} \explain{}{=}\\ & \hspace{2cm} \sum_{\substack{S \subset [n]\\  1 \leq |S| \leq \frac{t}{k} }}  \sum_{\substack{\mview{\bm t}{1:k} \in \N^{|S|}\\ \|\mview{\bm t}{1}\|_1 + \dotsb + \|\mview{\bm t}{k}\|_1 \leq t}}  \frac{\lambda^{|S|}}{\lambda_k^{|S|}} \cdot \hat{\nu}_{\mview{\bm t}{1}} \cdot \dotsc \cdot\hat{\nu}_{\mview{\bm t}{k}} \cdot \intH{\mview{\bm t}{1}}{\mview{\bm X}{1}_S}{1} \cdot \dotsb \cdot \intH{\mview{\bm t}{k}}{\mview{\bm X}{k}_S}{1}.
\end{align*}

The orthogonality property of integrated Hermite polynomials (Lemma \ref{lemma: integrate-hermite-orthogonality}) yields,

\begin{align*}
    &\refE \left[ \lowdegree{\frac{\diff\nullmu}{\diff\refmu}(\bm X) - 1}{t}^2 \right] \\&= \sum_{\substack{S \subset [n]\\  1 \leq |S| \leq \frac{t}{k} }}  \sum_{\substack{\mview{\bm t}{1:k} \in \N^{|S|}\\ \|\mview{\bm t}{1}\|_1 + \dotsb + \|\mview{\bm t}{k}\|_1 \leq t}}  \frac{\lambda^{2|S|}}{\lambda_k^{2|S|}} \cdot \hat{\nu}_{\mview{\bm t}{1}}^2 \cdot \dotsc \cdot\hat{\nu}_{\mview{\bm t}{k}}^2 \cdot \refE[\intH{\mview{\bm t}{1}}{\mview{\bm X}{1}_S}{1}^2] \cdot \dotsb \cdot \refE[\intH{\mview{\bm t}{k}}{\mview{\bm X}{k}_S}{1}^2].
\end{align*}
Appealing to the bound on the second moment of the integrated Hermite polynomials (Lemma \ref{lemma: norm-integrated-hermite}), we obtain,
\begin{align*}
    \refE \left[ \lowdegree{\frac{\diff\nullmu}{\diff\refmu}(\bm X) - 1}{t}^2 \right] & \leq \left(\frac{Ct}{\dim} \right)^{\frac{t}{2}} \cdot \sum_{\substack{S \subset [n]\\  1 \leq |S| \leq \frac{t}{k} }}  \sum_{\substack{\mview{\bm t}{1:k} \in \N^{|S|}\\ \|\mview{\bm t}{1}\|_1 + \dotsb + \|\mview{\bm t}{k}\|_1 \leq t}}  \frac{\lambda^{2|S|}}{\lambda_k^{2|S|}} \cdot \hat{\nu}_{\mview{\bm t}{1}}^2 \cdot \dotsc \cdot\hat{\nu}_{\mview{\bm t}{k}}^2. 
\end{align*}
Observing that,
\begin{align*}
    \sum_{\substack{\mview{\bm t}{1:k} \in \N^{|S|}\\ \|\mview{\bm t}{1}\|_1 + \dotsb + \|\mview{\bm t}{k}\|_1 \leq t}}   \hat{\nu}_{\mview{\bm t}{1}}^2 \cdot \dotsc \cdot\hat{\nu}_{\mview{\bm t}{k}}^2 & \leq \left( \sum_{\bm t \in \N^{|S|}} \hat{\nu}_{\bm t}^2 \right)^k  \leq \left( \sum_{t \in \N} \hat{\nu}_{t}^2 \right)^{k|S|} \explain{(a)}{=} 1.
\end{align*}
In the above display, the step marked (a) follows from the observation that since $\refE[\sgn^2(Z)] = 1$, we have:
\begin{align}\label{eq: sign-parseval}
\sum_{t \geq 1}  \hat{\nu}_t^2 & = 1.
\end{align}
This gives us,
\begin{align*}
    \refE \left[ \lowdegree{\frac{\diff\nullmu}{\diff\refmu}(\bm X) - 1}{t}^2 \right] & \leq \left(\frac{Ct}{\dim} \right)^{\frac{t}{2}} \cdot \sum_{\substack{S \subset [n]\\  1 \leq |S| \leq \frac{t}{k} }} \frac{\lambda^{2|S|}}{\lambda_k^{2|S|}} = \left(\frac{Ct}{\dim} \right)^{\frac{t}{2}} \cdot \sum_{s=1}^{\lfloor \frac{t}{k} \rfloor} \binom{\ssize}{s} \cdot  \frac{\lambda^{2s}}{\lambda_k^{2s}}. 
\end{align*}
The assumption $C t/\dim \leq 1/e$ guarantees that for any $s \leq t/k$
\begin{align*}
 \left(  \frac{C t}{\dim} \right)^{\frac{t}{2}} & \leq  \left(  \frac{C k s}{\dim} \right)^{\frac{ks}{2}}.
\end{align*}
Combining this with the bound $\binom{N}{s} \leq (eN/s)^s$, we obtain,
\begin{align*}
    \refE \left[ \lowdegree{\frac{\diff\nullmu}{\diff\refmu}(\bm X) - 1}{t}^2 \right] & \leq   \sum_{s=1}^{\lfloor \frac{t}{k} \rfloor} \left( \frac{C_k \cdot s^{\frac{k-2}{2}} \cdot \ssize \cdot  \lambda^2 }{\dim^{\frac{k}{2}}} \right)^s. 
\end{align*}
The assumption,
\begin{align*}
    \frac{C_k \cdot t^{\frac{k-2}{2}} \cdot \ssize \cdot  \lambda^2 }{\dim^{\frac{k}{2}}} & \leq  \frac{1}{2},
\end{align*}
ensures that the above sum can be bounded by the geometric sum $1/2 + 1/4 + \dotsb$ which yields the claim of the proposition.
\end{proof}

\subsection{Computationally Efficient Estimators}\label{sec:cca-spectral}

In this section, we design a spectral estimator for the $k$-CCA problem for even $k$. Recall that in the $k$-CCA problem one observes $\ssize$ i.i.d. copies of a random vector $\bm x = (\mview{\bm x}{1}, \mview{\bm x}{2}, \dotsc , \mview{\bm x}{k}) \in \R^{k \dim}$ with the property that,

\begin{align} \label{eq:k-CCA-correlation-recall}
    \E \left[ \mview{\bm x}{1} \otimes \mview{\bm x}{2} \otimes \dotsb \otimes \mview{\bm x}{k} \right] = \lambda \bm v_1 \otimes \bm v_2 \otimes \dotsb \otimes \bm v_k,
\end{align}
and one seeks to estimate the rank-1 tensor $\bm v_1 \otimes \bm v_2 \otimes \dotsb \otimes \bm v_k$.  In order to design our estimator, we will need to assume certain concentration hypothesis on the random vector $\bm x$.

\begin{assumption} [Concentration Assumption] \label{ass:concentration-cca}A random vector $\bm x = (\mview{\bm x}{1}, \mview{\bm x}{2}, \dotsc , \mview{\bm x}{k}) \in \R^{k \dim}$ satisfies the concentration assumption with parameters $(\varproxy, K)$ if,
\begin{enumerate}
    \item Each view $\mview{\bm x}{\ell}$ is marginally sub-Gaussian with variance proxy $\varproxy$ and $\E[\mview{\bm x}{\ell}] = \bm 0$. 
    \item For any tensor $\bm T \in \tensor{\R^\dim}{k/2}$ with $\|\bm T\| = 1$ we have, the moment bounds:
    \begin{align*}
        \E \ip{\bm T}{\mview{\bm x}{1} \otimes \mview{\bm x}{2} \dotsb \otimes \mview{\bm x}{k/2}}^2 & \leq (K \varproxy)^{\frac{k}{2}}, \; \E \ip{\bm T}{\mview{\bm x}{k/2+1} \otimes \mview{\bm x}{k/2+2}  \dotsb \otimes \mview{\bm x}{k}}^2  \leq (K \varproxy)^{\frac{k}{2}}, \\
        \E \ip{\bm T}{\mview{\bm x}{1} \otimes \mview{\bm x}{2} \dotsb \otimes \mview{\bm x}{k/2}}^4 & \leq (K \varproxy)^k, \; \E \ip{\bm T}{\mview{\bm x}{k/2+1} \otimes \mview{\bm x}{k/2+2} \dotsb \otimes \mview{\bm x}{k}}^4  \leq (K\varproxy)^k.
    \end{align*}
\end{enumerate} 
\end{assumption}

Recall in order to obtain our lower bounds we considered $\bm x \sim \dmu{\bm v_1 \otimes \dotsb \otimes \bm v_k}$ where the likelihood ratio of $\dmu{\bm v_1 \otimes \dotsb \otimes \bm v_k}$ with respect to the standard Gaussian measure $\refmu = \gauss{\bm 0}{\bm I_{k\dim}}$ is given by:

 \begin{align} \label{eq:cca-lb-distr}
     \frac{\diff \dmu{\bm v_1 \otimes \dotsb \otimes \bm v_k}}{\diff \refmu}(\bm x) \explain{def}{=} 1 + \frac{\lambda}{\lambda_k} \cdot \sgn\left(\ip{\mview{\bm x}{1}}{\bm v_1}\right)\cdot \sgn\left(\ip{\mview{\bm x}{2}}{\bm v_2}\right) \cdot \dotsb \cdot \sgn\left(\ip{\mview{\bm x}{k}}{\bm v_k}\right),
\end{align}
where, 
 \begin{align*}
     \lambda_k \explain{def}{=} \left(\frac{2}{\pi} \right)^{\frac{k}{2}} = (\E |Z|)^{\frac{k}{2}}, \; Z \sim \gauss{0}{1}.
 \end{align*}
The following lemma shows that the above distribution on $\bm x$ satisfies Assumption \ref{ass:concentration-cca}.

\begin{lemma} The random vector $\bm x \sim \dmu{\bm v_1 \otimes \dotsb \otimes \bm v_k}$ satisfies Assumption \ref{ass:concentration-cca} with $\varproxy = 1, K = 3^{k}$. 
\end{lemma}
\begin{proof}
We observe that,
\begin{align*}
    \int \frac{\diff \dmu{\bm v_1 \otimes \dotsb \otimes \bm v_k}}{\diff \refmu}(\bm x)  \; \refmu(\diff \mview{\bm x}{k}) & = 1.
\end{align*}
Consequently, $\mview{\bm x}{1:k-1} \explain{i.i.d.}{\sim} \gauss{\bm 0}{\bm I_\dim}$. Analogously, $\mview{\bm x}{2:k} \explain{i.i.d.}{\sim} \gauss{\bm 0}{\bm I_\dim}$. This verifies that each view $\mview{\bm x}{\ell}$ is sub-Gaussian with variance proxy $\varproxy = 1$. We can also compute:
\begin{align*}
    \E \ip{\bm T}{\mview{\bm x}{1} \otimes \mview{\bm x}{2} \otimes \dotsb \otimes \mview{\bm x}{k/2}}^2 & = \|\bm T\|^2 = 1.
\end{align*}
Furthermore since $\ip{\bm T}{\mview{\bm x}{1} \otimes \mview{\bm x}{2} \otimes \dotsb \otimes \mview{\bm x}{k/2}}$ is a polynomial of degree $k/2$ in the Gaussian random vectors $\mview{\bm x}{1:k/2}$, by Gaussian hypercontractivity (Fact \ref{fact: hypercontractivity}),
\begin{align*}
    \E \ip{\bm T}{\mview{\bm x}{1} \otimes \mview{\bm x}{2} \otimes \dotsb \otimes \mview{\bm x}{k/2}}^4 & \leq  \left(3^{k/2} \cdot \E\left[\ip{\bm T}{\mview{\bm x}{1} \otimes \mview{\bm x}{2} \otimes \dotsb \otimes \mview{\bm x}{k/2}}^2\right] \right)^2 = 3^k.
\end{align*}
We can analogously obtain the moment bounds for $\mview{\bm x}{k/2+1} \otimes \mview{\bm x}{k/2+2} \otimes \dotsb \otimes \mview{\bm x}{k}$. This proves the claim of the lemma.  
\end{proof}

\eqref{eq:k-CCA-correlation-recall} suggests that we can estimate $\bm v_1 \otimes \dotsb \otimes \bm v_k$ by computing the rank-1 approximation to the empirical cross-mode moment tensor:
\begin{align*}
    \hat{\bm T} & = \frac{1}{\ssize} \sum_{i=1}^\ssize \mview{\bm x}{1}_i \otimes \mview{\bm x}{2}_i \otimes \dotsb \otimes \mview{\bm x}{k}_i. 
\end{align*}
However since computing a rank-1 approximation to a $k$-tensor is non-trivial for $k \geq 3$, we will reshape $\hat{\bm T}$ to a $d^{\frac{k}{2}} \times d^{\frac{k}{2}}$ matrix. Specifically, for a tensor $\bm T \in \tensor{\R^\dim}{k}$, we define the matricization operation $\mat : \tensor{\R^\dim}{k} \rightarrow \R^{\dim^{k/2} \times \dim^{k/2}}$ as follows:
\begin{align*}
    \mat(\bm T)_{i_1 i_2 \dotsc  i_{k/2}, j_1 j_2 \dotsc  j_{k/2} } \explain{def}{=}  T_{i_1, i_2, \dotsc i_{k/2},  j_1, j_2, \dotsc  j_{k/2}}.
\end{align*}
In order to estimate $\bm v_1 \otimes \dotsb \otimes \bm v_k$, we first estimate $\mat(\bm v_1 \otimes \dotsb \otimes \bm v_k)$ by computing the best rank-1 approximation to $\mat(\bm T)$ using SVD:
\begin{subequations} \label{eq: spectral-CCA}
\begin{align}
    (\mview{\hat{\bm U}}{L}, \mview{\hat{\bm U}}{R}) \explain{def}{=} \arg \max_{\substack{\|\mview{\bm U}{L}\| = 1 \\ \|\mview{\bm U}{R}\| = 1}} \ip{\mview{\bm U}{L}}{\mat(\hat{\bm T}) \cdot \mview{\bm U}{R}}. 
\end{align}
We then construct an estimate $\hat{\bm V}$ of $\bm v_1 \otimes \dotsb \otimes \bm v_k$ by reshaping $\mview{\hat{\bm U}}{L} \otimes {\mview{\hat{\bm U}}{R}}$ into a tensor:
\begin{align}
    \hat{\bm V} \explain{def}{=} \mat^{-1}(\mview{\hat{\bm U}}{L} \otimes {\mview{\hat{\bm U}}{R}}).
\end{align}
\end{subequations}

The following theorem, which analyzes the sample complexity of the spectral estimator proposed above, is the main result of this section. 

\begin{theorem} \label{thm:spectral-CCA} Suppose that the Concentration Assumption (Assumption \ref{ass:concentration-cca}) holds with parameters $(\varproxy,K)$. There is a constant $C_k$ depending only on $k$ such that for any $\epsilon \in (0,1)$ if,
\begin{align} \label{eq: spectral-sample-complexity-cca}
    \ssize & \geq \frac{C_k \cdot K^{\frac{k}{2}} \cdot (1+\varproxy)^{k} \cdot \dim^{\frac{k}{2}}}{\lambda^2\epsilon^2} \cdot \ln \left( \frac{C_k \cdot K \cdot (1+\varproxy) \cdot \dim}{\lambda\epsilon}\right),
\end{align}
then, with probability $1-1/\ssize$, we have:
\begin{enumerate}
\item $\mat(\hat{\bm T})$ has a spectral gap in the sense: if $\hat{\sigma}_1 \geq \hat{\sigma}_2$ denote the largest two singular values of $\mat(\hat{\bm T})$ then, $\hat{\sigma}_2/\hat{\sigma}_1 \leq \epsilon/2$. (add power method guarantee, it should converge in $\tilde{O}(1)$ iterations).
    \item The estimator $\hat{\bm V}$ defined in \eqref{eq: spectral-CCA} satisfies the guarantee:
    \begin{align*}
       \| \hat{\bm V} - \bm v_1 \otimes \dotsb \otimes \bm v_k \|  & \leq \epsilon.
    \end{align*} 
\end{enumerate}
\end{theorem}

The proof of Theorem \ref{thm:spectral-CCA} requires some intermediate results, which we state first. The first is a concentration estimate for the norm of a sub-Gaussian vector. 

\begin{fact}[{\citet[Exercises 6.2.5, 6.2.6]{vershynin2018high}}] \label{fact: subgaussian-norm-concentration} There is a universal constant $C$ such that for any  sub-Gaussian vector $\bm x \in \R^\dim$ with $\E \bm x = 0$ and variance proxy $\varproxy$ and any $u \geq 0, t \in \N$ we have,
\begin{align*}
    \P \left( \|\bm x\|_2 \geq C \sqrt{\dim \varproxy} + u \right) & \leq \exp\left(-\frac{u^2}{C \varproxy} \right), \\
    \E \|\bm x\|^{2t} & \leq  (Ct \varproxy \dim)^t.
\end{align*}
\end{fact}
\begin{proof} The concentration inequality appears as an exercise in \citep[Exercise 6.2.5]{vershynin2018high}. In order to obtain the moment bound, we rely on \citep[Exercise 6.2.6]{vershynin2018high} which shows that there is a constant $C$ such that if $C \varproxy \lambda^2 \dim \leq 1$ then $\E \exp(\lambda^2 \|\bm x\|^2) \leq 2$. Observing that $(\lambda^2 \|\bm x\|^2)^t / t! \leq \exp(\lambda^2 \|\bm x\|^2)$ and taking expectations on both sides yields the claimed moment bound. 
\end{proof}

We will also rely on the following concentration estimate on the deviation of $\mat(\hat{\bm T})$ from its expectation. 

\begin{proposition} \label{prop: matrix-concentration-cca} Suppose that the Concentration Assumption (Assumption \ref{ass:concentration-cca}) holds with parameters $(\varproxy,K)$. There is a constant $C_k$ depending only on $k$ such that for any $\epsilon \in (0,1)$ if,
\begin{align*}
    \ssize & \geq \frac{C_k \cdot K^{\frac{k}{2}} \cdot (1+\varproxy)^{k} \cdot \dim^{\frac{k}{2}}}{\epsilon^2} \cdot \ln \left( \frac{C_k \cdot K \cdot (1+\varproxy) \cdot \dim}{\epsilon}\right),
\end{align*}
then,
\begin{align*}
    \P\left( \| \mat(\hat{\bm T}) - \lambda \cdot \mat(\bm v_1 \otimes \dotsb \otimes \bm v_k) \|_\op \geq \epsilon \right) & \leq \ssize^{-1}.
\end{align*}
\end{proposition}
\begin{proof}
The proof of this result is postponed to the end of this section.
\end{proof}
We now present the proof of Theorem \ref{thm:spectral-CCA}. 
\begin{proof}[Proof of Theorem \ref{thm:spectral-CCA}]
Consider the event:
\begin{align*}
    \mathcal{E} \explain{def}{=} \left\{ \| \mat(\hat{\bm T}) - \E[\mat(\hat{\bm T})] \|_\op \leq \frac{\lambda\epsilon}{3\sqrt{2}} \right\}.
\end{align*}
The assumption \eqref{eq: spectral-sample-complexity-cca} on the sample size along with Proposition \ref{prop: matrix-concentration-cca} guarantees that $\P( \mathcal{E}) \geq 1 - 1/\ssize$. Let $\hat{\sigma}_1 \geq \hat{\sigma_2} \geq \dotsb $ and ${\sigma}_1 \geq {\sigma_2} \geq \dotsb$ denote the sorted singular values of $\mat(\hat{\bm T})$ and $\E\mat(\hat{\bm T})$ respectively.On the event $\mathcal{E}$, we have, by Weyl's theorem, 
\begin{align*}
    |\hat{\sigma}_i - \sigma_i| & \leq   \| \mat(\hat{\bm T}) - \E[\mat(\hat{\bm T})] \|_\op  \leq  \frac{\lambda \epsilon}{3\sqrt{2}}.
\end{align*}
Note that since $\E\mat(\hat{\bm T}) = \lambda \mat(\bm v_1 \otimes \dotsb \otimes \bm v_k)$ we know that $\sigma_1 = \lambda$ and $\sigma_i = 0$ for $i \geq 2$. Hence, 
\begin{align*}
    \hat{\sigma}_1 \geq \lambda - \frac{\lambda \epsilon}{3 \sqrt{2}} \geq \frac{2\lambda}{3}, \; \hat{\sigma}_2 \leq \frac{\lambda \epsilon}{3},
\end{align*}
which gives the second claim of the theorem. In order to obtain the first claim we observe that,
\begin{align*}
    &\hspace{3cm}\| \hat{\bm V} - \bm v_1 \otimes \dotsb \otimes \bm v_k \|  = \| \mview{\hat{\bm U}}{L} \otimes \mview{\hat{\bm U}}{R} - \mat(\bm v_1 \otimes \dotsb \otimes \bm v_k)  \| \\
    &\hspace{3cm} \explain{(a)}{\leq} \sqrt{2} \cdot \| \mview{\hat{\bm U}}{L} \otimes \mview{\hat{\bm U}}{R} - \mat(\bm v_1 \otimes \dotsb \otimes \bm v_k)  \|_{\op} \\
    &\hspace{3cm} = \frac{\sqrt{2}}{\lambda}  \cdot \| \lambda \mview{\hat{\bm U}}{L} \otimes \mview{\hat{\bm U}}{R} - \E[\mat(\hat{\bm T})]  \|_{\op} \\
    &\hspace{3cm} \explain{(b)}{\leq} \frac{\sqrt{2}}{\lambda}  \cdot \left( |\hat{\sigma}_1 - \lambda|  +  \| \hat{\sigma}_1 \mview{\hat{\bm U}}{L} \otimes \mview{\hat{\bm U}}{R} - \E[\mat(\hat{\bm T})]  \|_{\op} \right) \\
    &\hspace{3cm} \explain{(b)}{\leq} \frac{\sqrt{2}}{\lambda}  \cdot \left( |\hat{\sigma}_1 - \lambda|  +  \| \hat{\sigma}_1 \mview{\hat{\bm U}}{L} \otimes \mview{\hat{\bm U}}{R} - \mat(\hat{\bm T})  \|_{\op} +  \| \mat(\hat{\bm T}) - \E[\mat(\hat{\bm T})]  \|_{\op} \right) \\
    &\hspace{3cm} \explain{(c)}{\leq} \frac{\sqrt{2}}{\lambda}  \cdot \left(|\hat{\sigma}_1 - \lambda|  +  2  \| \mat(\hat{\bm T}) - \E[\mat(\hat{\bm T})]  \|_{\op}  \right)  \leq \epsilon.
\end{align*}
In the above display, step (a) used the fact that for a rank-2 matrix $\bm A$, $\|\bm A\| \leq \sqrt{2}\|\bm A\|_\op$. The steps marked (b) use the triangle inequality and the step marked (c) relies on the fact that the rank-1 SVD provides the best rank-1 approximation for $\mat(\hat{\bm T})$. This concludes the proof of the result.
\end{proof}

\subsubsection{Proof of Proposition \ref{prop: matrix-concentration-cca}}

\begin{proof}[Proof of Proposition \ref{prop: matrix-concentration-cca}]
We begin by introducing the vectorization operation $\vec : \tensor{\R^\dim}{k/2} \rightarrow \R^{d^{k/2}}$  which maps a tensor $\bm T \in \tensor{\R^\dim}{k/2}$ to a vector in $\R^{d^{k/2}}$ with entries:
\begin{align*}
    \vec(\bm T)_{i_1 i_2 \dotsc i_{k/2}} & \explain{def}{=}  T_{i_1,i_2, \dotsc i_{k/2}}.
\end{align*}
With this notation, we can express $\mat(\hat{\bm T})$ and $\E[\mat(\hat{\bm T})]$ as:
\begin{align*}
    \mat(\hat{\bm T}) & = \frac{1}{\ssize} \sum_{i=1}^\ssize \mview{\bm \chi}{L}_i \otimes \mview{\bm \chi}{R}_i, \;
    \E[\mat(\hat{\bm T})] = \lambda \mview{\bm V}{L} \otimes \mview{\bm V}{R}, 
\end{align*}
where,
\begin{align*}
    \mview{\bm \chi}{L}_i &\explain{def}{=} \vec( \mview{\bm x}{1}_i \otimes \dotsb \mview{\bm x}{k/2}_i), \;  \mview{\bm \chi}{R}_i \explain{def}{=} \vec( \mview{\bm x}{k/2+1}_i \otimes \dotsb \mview{\bm x}{k}_i), \\ \mview{\bm V}{L} &\explain{def}{=} \vec( \mview{\bm v}{1} \otimes \dotsb \mview{\bm v}{k/2}), \; \mview{\bm V}{R} \explain{def}{=} \vec( \mview{\bm v}{k/2+1} \otimes \dotsb \mview{\bm v}{k}).
\end{align*}
We consider the decomposition:
\begin{align*}
    \|\mat(\hat{\bm T}) - \E[\mat(\hat{\bm T})] \|_\op & \leq \underbrace{\|\mat(\hat{\bm T}) - \hat{\bm M} \|_\op}_{\mathsf{(I)}} + \underbrace{\| \E[\mat(\hat{\bm T})] - \E[\hat{\bm M}]\|_\op}_{\mathsf{(II)}} + \underbrace{\|\E[\hat{\bm M}] - \hat{\bm M} \|_\op}_{\mathsf{(III)}},
\end{align*}
where $\hat{\bm M}$ is the following truncated version of $\mat(\hat{\bm T})$:
\begin{align*}
    \hat{\bm M} \explain{def}{=} \frac{1}{\ssize} \sum_{i=1}^\ssize \Indicator{\mathcal{A}_i} \cdot  \mview{\bm \chi}{L}_i \otimes \mview{\bm \chi}{R}_i,
\end{align*}
where the events $\mathcal{A}_i$ are defined as:
\begin{align*}
    \mathcal{A}_i \explain{def}{=} \left\{ \|\mview{\bm x}{\ell}_i \| \leq  C_k \cdot \sqrt{\varproxy} \cdot \left( \sqrt{\dim} + \ln^{\frac{1}{2}} \left( \frac{K\varproxy}{\epsilon} \right) + \ln^{\frac{1}{2}}(N) \right)\right\}.
\end{align*}
In the above display, $C_k$ denotes a suitably large constant depending only on $k$. Next we control the terms $\mathsf{(I)}, \mathsf{(II)}, \mathsf{(III)}$ one by one. We observe that, as a consequence of Fact \ref{fact: subgaussian-norm-concentration}, 
\begin{align*}
    \P\left( \|\mat(\hat{\bm T}) - \hat{\bm M} \|_\op > 0\right) & \leq \sum_{i=1}^\ssize \P(\mathcal{A}_i^c) \leq \frac{1}{2N}.
\end{align*}
Next, we control the term $(\mathsf{II})$. By the variational formula for the operator norm, we know that there is are unit vectors $\bm U, \bm U^\prime \in \R^{\dim^{k/2}}$ such that,
\begin{align*}
    &\left\|\E\left[ \mview{\bm \chi}{L}_1 \otimes \mview{\bm \chi}{R}_1 \right] - \E\left[ \Indicator{\mathcal{A}_1} \cdot  \mview{\bm \chi}{L}_1 \otimes \mview{\bm \chi}{R}_1 \right]  \right\|_{\op}  = \left|\E\left[\Indicator{\mathcal{A}_1^c} \cdot \ip{\mview{\bm \chi}{L}_1}{\bm U} \cdot \ip{\mview{\bm \chi}{R}_1}{\bm U^\prime} \right] \right| \\
    & \hspace{6cm}\explain{(a)}{\leq} \left(\E\left[\ip{\mview{\bm \chi}{L}_1}{\bm U}^4\right]\right)^{1/4}\cdot\left(\E\left[\ip{\mview{\bm \chi}{R}_1}{\bm U}^4\right]\right)^{1/4} \cdot \sqrt{\P(\mathcal{A}_1^c)} \\
    & \hspace{6cm} \explain{(b)}{\leq} (K\varproxy)^{\frac{k}{2}} \cdot \sqrt{\P(\mathcal{A}_1^c)}  \\
    &\hspace{6cm} \explain{(c)}{\leq} \epsilon/2.
\end{align*}
In the above computations, we appealed to the Cauchy-Schwarz inequality in step (a), the Concentration Assumption (Assumption \ref{ass:concentration-cca}) in step (b) and to Fact \ref{fact: subgaussian-norm-concentration} in step (c). In order to control the term $(\mathsf{III})$, we will appeal to the Matrix Bernstein Inequality \citet[Theorem 6.1.1]{tropp2015introduction}. This yields,
\begin{align*}
    \P\left( \|\E[\hat{\bm M}] - \hat{\bm M} \|_\op \geq \frac{\epsilon}{2} \right) & \leq \dim^{\frac{k}{2}} \cdot \exp\left( - \ssize \cdot \left( \frac{\epsilon^2}{32 \sigma^2} \wedge \frac{3 \epsilon}{32 R} \right) \right),
\end{align*}
where, $R$ is an a.s. upper bound on $\max_{i \in [\ssize]} \|  \Indicator{\mathcal{A}_1} \cdot  \mview{\bm \chi}{L}_i \otimes \mview{\bm \chi}{R}_i\|$ and,
\begin{align*}
    \sigma^2 = \left\|\E\left[\Indicator{\mathcal{A}_1} \cdot \|\mview{\bm \chi}{L}_1\|^2 \cdot   \mview{\bm \chi}{R}_i \otimes \mview{\bm \chi}{R}_i\right]\right\|_\op \vee  \left\|\E\left[\Indicator{\mathcal{A}_1} \cdot \|\mview{\bm \chi}{R}_1\|^2 \cdot   \mview{\bm \chi}{L}_i \otimes \mview{\bm \chi}{L}_i\right]\right\|_\op.
\end{align*}
By the definition of $\mathcal{A}_i$, we can set,
\begin{align*}
    R = C_k \cdot \sqrt{\varproxy^k} \cdot \left( \dim^{\frac{k}{2}} + \ln^{\frac{k}{2}} \left( \frac{K\varproxy}{\epsilon} \right) + \ln^{\frac{k}{2}}(N) \right).
\end{align*}
In order to bound $\sigma^2$, we note that by the variational formula for the operator norm, we know that there is are unit vectors $\bm U, \bm U^\prime \in \R^{\dim^{k/2}}$ such that,
\begin{align*}
    \left\|\E\left[\Indicator{\mathcal{A}_1} \cdot \|\mview{\bm \chi}{L}_1\|^2 \cdot \mview{\bm \chi}{R}_1 \otimes \mview{\bm \chi}{R}_1 \right] \right\|_{\op} & = \E\left[\Indicator{\mathcal{A}_1} \cdot \|\mview{\bm \chi}{L}_1\|^2 \cdot \ip{\bm U}{ \mview{\bm \chi}{R}_1}^2 \right] \\&\leq \E \left[ \|\mview{\bm x}{1}_1\|^2 \cdot \dotsb \cdot \|\mview{\bm x}{k/2}_1\|^2 \cdot \ip{\bm U}{ \mview{\bm \chi}{R}_1}^2 \right] \\
    & \explain{(a)}{\leq} \max_{\ell \in [k]} \E\left[ \|\mview{\bm x}{\ell}_1\|^{k} \cdot \ip{\bm U}{ \mview{\bm \chi}{R}_1}^2 \right] \\
    & \explain{(b)}{\leq} (C k \varproxy \dim)^{\frac{k}{2}} \cdot (K \varproxy)^{\frac{k}{2}}.
\end{align*}
In the above computations, we appealed to the AM-GM inequality in step (a). Step (b) relies on Cauchy-Schwarz inequality along with the moment bounds for norms of sub-Gaussian vectors (Fact \ref{fact: subgaussian-norm-concentration}) and the moment bound in Assumption \ref{ass:concentration-cca}. Hence $\sigma^2 \leq (C k \varproxy \dim)^{\frac{k}{2}} \cdot (K \varproxy)^{\frac{k}{2}}$. The sample size assumption in the statement of the proposition guarantees,
\begin{align*}
    \P\left( \|\E[\hat{\bm M}] - \hat{\bm M} \|_\op \geq \frac{\epsilon}{2} \right) & \leq \frac{1}{2\ssize}.
\end{align*}
Combining our estimates on the terms $\mathsf{(I-III)}$, we obtain $ \|\E[\hat{\bm M}] - \hat{\bm M} \|_\op \leq \epsilon$ with probability $1-1/\ssize$, as claimed. 
\end{proof}

\section{Miscellaneous Results}
\label{appendix: misc}

\subsection{Additional Technical Facts and Lemmas}

\begin{fact} [Estimates on Partial Exponential Series~\citep{klar2000bounds}] \label{fact: partial_exp_series} We have, for any $\lambda \geq 0$ and for any $t \in \N_0$ such that $t + 1 \geq \lambda $, we have
\begin{align*}
    \frac{\lambda^t}{t!} \leq \sum_{i = t}^\infty \frac{\lambda^i}{i!} \leq \frac{1}{1-\frac{\lambda}{t+1}} \cdot  \frac{\lambda^t}{t!}
\end{align*}
In particular if $t \geq (e^2 \lambda) \vee \ln(1/\epsilon) \vee 1$, by Stirling's approximation,
\begin{align*}
     \sum_{i = t}^\infty \frac{\lambda^i}{i!} \leq \epsilon.
\end{align*}
\end{fact}

\begin{fact} [A Bound on Hermite Polynomials] \label{fact: hermite-simple-bound}For any $k \in \W$, we have
\begin{align*}
    |H_k(z)| & \leq (1+|z|)^k.
\end{align*}
\end{fact}
\begin{proof}
$H_k$ has the following Taylor series expansion around $z=0$ (see for e.g. \citep[Section 2.4]{hermitewiki}):
\begin{align*}
    H_k(z) & = \sum_{i=0}^k \binom{k}{i} \cdot \frac{\sqrt{i!}}{\sqrt{k!}} \cdot  H_i(0) \cdot  z^i. 
\end{align*}
The values $H_i(0)$ are known explicitly (see for e.g. \citep[Section 2.10]{hermitewiki}):
\begin{align*}
    |H_k(z)| & \leq \sum_{i=0}^k   \binom{k}{i}  \cdot |z|^i = (1+|z|)^k.
\end{align*}

\end{proof}

\begin{fact}[{\citealp[Equation 12]{kunisky2019notes}}] \label{fact: rademacher-moments-lb} Let $\bm V \sim \unif{\{\pm 1\}^\dim}$. Define,
\begin{align*}
    \overline{V} = \frac{1}{\dim} \sum_{i=1}^\dim V_i .
\end{align*}
We have, for any $t \in \W,\; t \leq \dim$,
\begin{align*}
    \E \overline{V}^{2t} & \geq \frac{(2t)!}{2^t \dim^{2t}} \cdot \binom{\dim}{t} \geq \left( \frac{2}{e^2} \cdot \frac{t}{\dim} \right)^t .
\end{align*}
\end{fact}

\begin{lemma} \label{lemma: rademacher-moments}Let $\bm V \sim \unif{\{\pm 1\}^\dim}$. Define,
\begin{align*}
    \overline{V} = \frac{1}{\dim} \sum_{i=1}^\dim V_i
\end{align*}Then for any $t \in \N$,
\begin{align*}
  \sup_{\bm r \in \{0,1\}^\dim} \E\left[ \overline{V}^t \cdot \prod_{i \in [\dim] : r_i = 1} V_i \right] &\leq   4^t \cdot t^{\frac{t}{2}} \cdot d^{-\lceil \frac{t}{2}\rceil}, \\
  \sup_{\substack{\bm r \in \{0,1\}^\dim\\ \|\bm r\|_1 \geq 1}} \E\left[ \overline{V}^t \cdot \prod_{i \in [\dim] : r_i = 1} V_i \right] & \leq 2 \cdot 5^t \cdot t^{\frac{t}{2}} \cdot d^{-\lceil \frac{t+1}{2} \rceil}.
\end{align*}
Furthermore if $t \leq 2(\dim -1)$ and $\dim \geq 3$,
\begin{align*}
    \sup_{\substack{\bm r \in \{0,1\}^\dim\\ \|\bm r\|_1 \geq 1}} \E\left[ \overline{V}^t \cdot \prod_{i \in [\dim] : r_i = 1} V_i \right] & \geq 5^{-t} \cdot t^{\frac{t}{2}} \cdot \dim^{-\lceil\frac{t}{2}\rceil}/2, \\
    \sup_{\substack{\bm r \in \{0,1\}^\dim\\ \|\bm r\|_1 \geq 1}} \E\left[ \overline{V}^t \cdot \prod_{i \in [\dim] : r_i = 1} V_i \right] & \geq 5^{-t} \cdot t^{\frac{t}{2}} \cdot \dim^{-\lceil\frac{t+1}{2}\rceil}/2.
\end{align*}
\end{lemma}

\begin{proof}
Due to coordinate symmetry, degree, and parity considerations, we have
\begin{align*}
    \sup_{\bm r \in \{0,1\}^\dim} \E\left[ \overline{V}^t \cdot \prod_{i \in [\dim] : r_i = 1} V_i \right] & = \sup_{\substack{\ell \in \{0,1,2 \dotsc, t\}\\ t+\ell \text{ is even}}} \E\left[ \overline{V}^t \cdot \prod_{1 \leq i \leq \ell} V_i \right], \\
    \sup_{\substack{\bm r \in \{0,1\}^\dim\\ \|\bm r\|_1 \geq 1}} \E\left[ \overline{V}^t \cdot \prod_{i \in [\dim] : r_i = 1} V_i \right] & = \sup_{\substack{\ell \in \{1,2 \dotsc, t\}\\ t+\ell \text{ is even}}} \E\left[ \overline{V}^t \cdot \prod_{1 \leq i \leq \ell} V_i \right].
\end{align*}
Hence, we focus on proving upper and lower bounds on:
\begin{align*}
     \E\left[ \overline{V}^t \cdot \prod_{1 \leq i \leq \ell} V_i \right].
\end{align*}
We decompose $\overline{V}$ as
\begin{align*}
    \overline{V} = \frac{S_1}{\dim} + \frac{S_2}{\dim},
\end{align*}
where $S_1 = V_1 + V_2 + \dotsb + V_\ell, \; S_2 = V_{\ell+1} + V_{\ell+2} + \dotsb + V_{\dim}$. By the Binomial Theorem,
\begin{align*}
     \E\left[ \overline{V}^t \cdot \prod_{1 \leq i \leq \ell} V_i \right] & = \sum_{i=0}^t \binom{t}{i} \cdot \frac{\E S_2^{i}}{\dim^t} \cdot \E\left[ S_1^{t-i} \prod_{1\leq i \leq \ell} V_i \right]. 
\end{align*}
Observing that when $t-i<\ell$, we have
\begin{align*}
    \E\left[ S_1^{t-i} \prod_{1\leq i \leq \ell} V_i \right] = 0,
\end{align*}
and thus
\begin{align*}
    \E\left[ \overline{V}^t \cdot \prod_{1 \leq i \leq \ell} V_i \right] & = \sum_{i=0}^{t-\ell} \binom{t}{i} \cdot \frac{\E S_2^{i}}{\dim^t} \cdot \E\left[ S_1^{t-i} \prod_{1\leq i \leq \ell} V_i \right]. 
\end{align*}
We now prove an upper bound and lower bound on the above expression.
\begin{description}
\item [Upper Bound:] Since $S_2$ is sub-Gaussian with variance proxy $d-\ell$, we have (see, e.g.,, \citealp[Lemma 1.4]{rigollet2015high})
\begin{align*}
    \E S_2^i & \leq 2^{i} \cdot i^{\frac{i}{2}} \cdot  (d-\ell)^{\frac{i}{2}}.  
\end{align*}
By an analogous argument,
\begin{align*}
    \E\left[ S_1^{t-i} \prod_{1\leq i \leq \ell} V_i \right]& \leq \E[|S_1|^{t-i}]  \leq 2^{t-i}\cdot (t-i)^{\frac{t-i}{2}} \cdot \ell^{\frac{t-i}{2}}.
\end{align*}
Hence, 
\begin{align*}
     \E\left[ \overline{V}^t \cdot \prod_{1 \leq i \leq \ell} V_i \right] & \leq \frac{2^t}{\dim^t} \sum_{i=0}^{t-\ell} \binom{t}{i}  \cdot (t-i)^{\frac{t-i}{2}} \cdot  i^{\frac{i}{2}} \cdot (d-\ell)^{\frac{i}{2}} \cdot \ell^{\frac{t-i}{2}}
\end{align*}
Using the AM-GM Inequality,
\begin{align*}
    (t-i)^{t-i} i^{i} & \leq \left( \frac{(t-i)^2 + i^2}{t} \right)^t \leq t^t.
\end{align*}
Hence,
\begin{align*}
     \E\left[ \overline{V}^t \cdot \prod_{1 \leq i \leq \ell} V_i \right] & \leq \left( \frac{4 t}{\dim}\right)^{\frac{t}{2}} \cdot \sum_{i=0}^{t-\ell} \binom{t}{i}  \cdot \frac{(d-\ell)^{\frac{i}{2}}}{\dim^{\frac{i}{2}}} \cdot \frac{\ell^{\frac{t-i}{2}}}{\dim^{\frac{t-i}{2}}} \\
     & \leq \left( \frac{4 t}{\dim}\right)^{\frac{t}{2}} \cdot \left(\frac{\ell}{\dim}\right)^{\frac{\ell}{2}} \cdot \sum_{i=0}^{t-\ell} \binom{t}{i} \\
     & \leq \left( \frac{16 t}{\dim}\right)^{\frac{t}{2}} \cdot \left(\frac{\ell}{\dim}\right)^{\frac{\ell}{2}} .
\end{align*}
Hence,
\begin{align*}
     \sup_{\bm r \in \{0,1\}^\dim} \E\left[ \overline{V}^t \cdot \prod_{i \in [\dim] : r_i = 1} V_i \right] &= \sup_{\substack{\ell \in \{0,1,2 \dotsc, t\}\\ t+\ell \text{ is even}}} \E\left[ \overline{V}^t \cdot \prod_{1 \leq i \leq \ell} V_i \right]  \leq  \left( \frac{16 t}{\dim}\right)^{\frac{t}{2}} \cdot \left(\sup_{\substack{\ell \in \{0,1,2 \dotsc, t\}\\ t+\ell \text{ is even}}} \left(\frac{\ell}{\dim}\right)^{\frac{\ell}{2}} \right).
\end{align*}
If $\ell \leq t \leq d/e$, the function $(\ell/d)^\ell$ is decreasing, and hence,
\begin{align*}
    \sup_{\bm r \in \{0,1\}^\dim} \E\left[ \overline{V}^t \cdot \prod_{i \in [\dim] : r_i = 1} V_i \right] &\leq  4^t \cdot t^{\frac{t}{2}} \cdot d^{-\lceil \frac{t}{2} \rceil}.
\end{align*}
On the other hand, when $t \geq d/e$, the same upper bound holds since,
\begin{align*}
    4^t \cdot t^{\frac{t}{2}} \cdot d^{-\lceil \frac{t}{2} \rceil} & \geq \frac{4^t e^{-\frac{t}{2}}}{\sqrt{te}} \geq 4^t e^{-t} \geq 1,
\end{align*}
whereas, since $|\overline{V}| \leq 1$, we always have the trivial upper bound,
\begin{align*}
     \sup_{\bm r \in \{0,1\}^\dim} \E\left[ \overline{V}^t \cdot \prod_{i \in [\dim] : r_i = 1} V_i \right] & \leq 1.
\end{align*}
With an analogous argument, we also obtain,
\begin{align*}
    \sup_{\substack{\bm r \in \{0,1\}^\dim\\ \|\bm r\|_1 \geq 1}} \E\left[ \overline{V}^t \cdot \prod_{i \in [\dim] : r_i = 1} V_i \right] &\leq 2 \cdot  5^t \cdot t^{\frac{t}{2}} \cdot d^{-\lceil \frac{t+1}{2} \rceil}.
\end{align*}
\item [Lower Bound: ]Recall that,
\begin{align*}
     \E\left[ \overline{V}^t \cdot \prod_{1 \leq i \leq \ell} V_i \right] & = \sum_{i=0}^{t-\ell} \binom{t}{i} \cdot \frac{\E S_2^{i}}{\dim^t} \cdot \E\left[ S_1^{t-i} \prod_{1\leq i \leq \ell} V_i \right].
\end{align*}
For proving the claim of the lemma, it will be sufficient to lower bound the above expression under the assumption $t+\ell$ is even and $\ell \in \{0,1,2\}$. We observe that each of the terms in the above sum is non-negative. This is because, $\E S_2^i = 0$ when $i$ is odd and, by expanding $S_1^{t-i}$ using the Multinomial Theorem, one sees that:
\begin{align*}
     \E\left[ S_1^{t-i} \prod_{1\leq i \leq \ell} V_i \right] & \geq 0.
\end{align*}
Hence, retaining the term corresponding to $i = (t-\ell)$ we obtain,
\begin{align*}
    \E\left[ \overline{V}^t \cdot \prod_{1 \leq i \leq \ell} V_i \right] & \geq \binom{t}{\ell} \cdot \frac{\E S_2^{t-\ell}}{\dim^t} \cdot \E\left[ S_1^{\ell} \prod_{1\leq i \leq \ell} V_i \right].
\end{align*}
Expanding $S_1^\ell$ using the Multinomial Theorem and comparing the coefficient of $V_1\cdot V_2 \dotsb \cdot V_\ell$, we observe,
\begin{align*}
    \E\left[ S_1^{\ell} \prod_{1\leq i \leq \ell} V_i \right] = 1.
\end{align*}
Hence,
\begin{align*}
     \E\left[ \overline{V}^t \cdot \prod_{1 \leq i \leq \ell} V_i \right] & \geq \binom{t}{\ell} \cdot \frac{\E S_2^{t-\ell}}{\dim^t}.
\end{align*}
Since $t+\ell$ is assumed to be even, so is $t-\ell$. Furthermore since we assume that $\ell \in \{0,1,2\}$ and $t \leq 2(\dim -1)$ we have $t - \ell \leq 2(\dim - \ell)$. Hence using by Fact \ref{fact: rademacher-moments-lb}, we have
\begin{align*}
    \E S_2^{t-\ell} & \geq  \binom{\dim-\ell}{\frac{t-\ell}{2}} \cdot  \frac{(t-\ell)!}{2^{\frac{t-\ell}{2}}}.
\end{align*}
This give us,
\begin{align*}
     \E\left[ \overline{V}^t \cdot \prod_{1 \leq i \leq \ell} V_i \right] & \geq \frac{1}{2^{\frac{t-\ell}{2}}} \cdot  \frac{t!}{\ell!} \cdot \binom{\dim-\ell}{\frac{t-\ell}{2}} \cdot \frac{1}{\dim^t} \\
     & \explain{(a)}{\geq} \frac{1}{2^{\frac{t-\ell}{2}}} \cdot  \frac{t^{t} e^{-t}}{\ell!} \cdot \left( \frac{2(d-\ell)}{t-\ell} \right)^{\frac{t-\ell}{2}} \cdot \frac{1}{\dim^t} \\
     & \geq \frac{t^{\frac{t}{2}} e^{-t}}{\ell!} \cdot \left(1 -\frac{\ell}{\dim} \right)^{\frac{t}{2}} \cdot \frac{1}{\dim^{\frac{t+\ell}{2}}} \\
     & \explain{(b)}{\geq} \frac{t^{\frac{t}{2}} e^{-t}}{2} \cdot 3^{-\frac{t}{2}} \cdot \frac{1}{\dim^{\frac{t+\ell}{2}}} .
\end{align*}
In the step marked (a), we used the standard lower bounds for the Binomial coefficient $\binom{n}{k} \geq (n/k)^k$ and factorial $n! \geq n^n e^{-n}$. In the step marked (b), we used the fact that $\ell \leq 2$ and $\dim \geq 3$.
Hence,
\begin{align*}
     \sup_{\bm r \in \{0,1\}^\dim} \E\left[ \overline{V}^t \cdot \prod_{i \in [\dim] : r_i = 1} V_i \right] &= \sup_{\substack{\ell \in \{0,1,2 \dotsc, t\}\\ t+\ell \text{ is even}}} \E\left[ \overline{V}^t \cdot \prod_{1 \leq i \leq \ell} V_i \right]  \geq \sup_{\substack{\ell \in \{0,1,2\}\\ t+\ell \text{ is even}}} 5^{-t} \cdot t^{\frac{t}{2}} \cdot \dim^{-\frac{t+\ell}{2}}/2.
\end{align*}
Choosing $\ell=0$ if $t$ is even and $\ell = 1$ if $t$ is odd gives us:
\begin{align*}
    \sup_{\bm r \in \{0,1\}^\dim} \E\left[ \overline{V}^t \cdot \prod_{i \in [\dim] : r_i = 1} V_i \right] & \geq 5^{-t} \cdot t^{\frac{t}{2}} \cdot \dim^{-\lceil\frac{t}{2}\rceil}/2.
\end{align*}
Choosing $\ell=2$ if $t$ is even and $\ell = 1$ if $t$ is odd gives us:
\begin{align*}
    \sup_{\substack{\bm r \in \{0,1\}^\dim\\ \|\bm r\|_1 \geq 1}} \E\left[ \overline{V}^t \cdot \prod_{i \in [\dim] : r_i = 1} V_i \right] & \geq 5^{-t} \cdot t^{\frac{t}{2}} \cdot \dim^{-\lceil\frac{t+1}{2}\rceil}/2.
\end{align*}
\end{description}
This concludes the proof of this lemma. 
\end{proof}

\subsection{Analysis on Gaussian Space}
\label{fourier_gauss_appendix}
Consider the functional space $\GaussSpace{d}$ defined as follows:
\begin{align*}
\GaussSpace{d} & \explain{def}{=} \left\{f \colon \R^d \to \R : \E_{\gauss{\bm 0}{\bm I_d}} f^2(\bm Z) < \infty \right\}.
\end{align*}
The multivariate Hermite polynomials for a complete orthonormal basis for $\GaussSpace{d}$.  These are defined as follows: for any $\bm c \in \W^d$, define
\begin{align*}
H_{\bm c} (\bm z) & \explain{def}{=} \prod_{i=1}^d H_{c_i} (z_i) ,
\end{align*}
where, for any $k \in \W$, the $H_k$ are the (probabilist's) orthonormal Hermite polynomials with the property $$\E_{\gauss{0}{1}} H_k(Z) H_l(Z) = \begin{cases} 0 & \text{if $k \neq l$} ; \\ 1 & \text{if $k=l$} . \end{cases} $$
The orthonormality property is inherited by the multivariate Hermite polynomials:
$$\E_{\gauss{\bm 0}{\bm I_d}} H_{\bm c}(\bm Z) H_{\bm d}(\bm Z) = \begin{cases} 0 & \text{if $\bm c \neq \bm d$} ; \\ 1 & \text{if $\bm c=\bm d$} . \end{cases} $$
Since these polynomials form an orthonormal basis of $\GaussSpace{d}$ any $f \in \GaussSpace{d}$ admits an expansion of the form: 
\begin{align*}
f(\bm z) & = \sum_{\bm c \in (\mathbb N \cup \{0\})^d} \hat{f}(\bm c) H_{\bm c}(\bm z).
\end{align*}
In the above display, $\hat{f}(\bm c) \in \R$ are the Hermite (or Fourier) coefficients of $f$. They satisfy the usual Parseval's relation:
\begin{align*}
\sum_{\bm c \in \W^d} \hat f^2(\bm c) & = \E_{\gauss{\bm 0}{\bm I_d}} f^2(\bm Z).
\end{align*}
A particular desirable property of the univariate Hermite polynomials is the following: for any $\mu \in \R, k \in \W$ we have
\begin{align*}
\E_{\gauss{0}{1}} H_k(\mu + Z) & = \frac{\mu^k}{\sqrt{k!}}.
\end{align*}
This implies the following property of multivariate Hermite polynomials which will be particularly useful for us.
\begin{fact} \label{hermite_special_property} For any $\bm \mu \in \R^d$ and any $\bm c \in \W^{d}$, we have,
\begin{align*}
\E_{\gauss{\bm 0}{\bm I_d}} H_k(\bm \mu + \bm Z) & = \frac{\bm \mu ^{\bm c}}{\sqrt{\bm c!}}.
\end{align*}
In the above display, we are using the following notation:
\begin{align}\label{eq:entrywise-power-factorial}
\bm \mu^{\bm c} \explain{def}{=} \prod_{i=1}^m \mu_i^{c_i}, \; \bm c! \explain{def}{=} \prod_{i=1}^m( c_i !).
\end{align}
\end{fact}

\begin{fact}\label{fact: hermite-projection-property} For any  vector $\bm u \in \R^\dim$ with $\|\bm u\| = 1$ we have,
\begin{align*}
    H_{i}(\ip{\bm u}{\bm x}) & = \sum_{\substack{\bm c \in \W^\dim \\ \|\bm c\|_1 = i}} \frac{\bm u^{\bm c}}{\sqrt{\bm c!}} H_{\bm c}(\bm x),
\end{align*}
for any $\bm x \in \R^\dim$. In the above display, the notations $\bm u^{\bm c}$ and $\bm c!$ are as defined in \eqref{eq:entrywise-power-factorial}.
\end{fact}

\begin{fact}\label{fact: correlated-hermite} Let $Z,Z^\prime$ be $\rho$-correlated standard Gaussian random variables:
\begin{align*}
    \begin{bmatrix} Z \\ Z^\prime\end{bmatrix} \sim \gauss{\begin{bmatrix}0 \\ 0\end{bmatrix}}{\begin{bmatrix}1 & \rho \\ \rho & 1 \end{bmatrix}}.
\end{align*}
Then, for any $i,j \in \W$,
\begin{align*}
    \E H_i(Z) H_j(Z^\prime) & = \begin{cases} \rho^i & \text{if $i = j$} ; \\ 0 & \text{if $i \neq j$} . \end{cases}
\end{align*}
\end{fact}
We will also rely on the Gaussian Hypercontractivity theorem which is usually attributed to \citet{nelson1966quartic}.  Our reference for this result was the book of \citet{o2014analysis}.
\begin{fact}[Gaussian Hypercontractivity~\citep{nelson1966quartic}] \label{fact: hypercontractivity} Let $\bm Z \sim \gauss{\bm 0}{\bm I_d}$. Then, for any $q \geq 2$,
\begin{align*}
    \E \left| \sum_{\bm \alpha \in \W^d} c_{\bm \alpha} H_{\bm \alpha}(\bm Z)  \right|^{q} & \leq \left( \sum_{\bm \alpha \in W^d}  (q-1)^{\|\bm \alpha\|_1} \cdot c_{\bm \alpha}^2 \right)^{\frac{q}{2}}
\end{align*}
The inequality is tight for $q=2$.
\end{fact}

\subsection{Fano's Inequality for Hellinger Information} \label{appendix:fano}
In this section, we provide a derivation of the Fano's Inequality for Hellinger Information quoted in Fact~\ref{fact: fano}. This result is a minor modification of a result due to \citet{chen2016bayes}. Although these authors derive a version of Fano's Inequality for Hellinger Information \citep[Corollary 7, item (iii)]{chen2016bayes}, it has a slightly more complicated form than the claim of Fact~\ref{fact: fano}. The simpler form stated in Fact~\ref{fact: fano} (which suffices for our results) is derived by combining Fano's Inequality for the Total Variation (TV) Information proved by \citet[Corollary 7, item (ii)]{chen2016bayes} with standard a comparison between Hellinger and total variation distances. Specifically, \citet[Corollary 7, item (ii)]{chen2016bayes} show that for any estimator
    $\hat{\bm V} : \{0,1\}^{m\budget} \rightarrow \estimatespace$, we have
\begin{align} \label{eq:fano-tv}
    \int_{\paramspace} \E_{\bm V}[\ell(\bm V, \hat{\bm V}(\bm Y))] \; \prior(\diff \bm V) & \geq R_0(\pi) -  \MItv{\bm V}{\bm Y},
\end{align}
where $\MItv{\bm V}{\bm Y}$ denotes the Total Variation Information which is defined as:
\begin{align*}
    \MItv{\bm V}{\bm Y} & \explain{def}{=} \inf_{\Q}  \int \tv{\P_{\bm V}}{\Q} \prior(\diff \bm V).
\end{align*}
In the above display, $\tv{\P_{\bm V}}{\Q}$ denotes the total variation distance between the probability measures $\P_{\bm V}$ and $\Q$. Since $\tv{\P_{\bm V}}{\Q} \leq (2\hell{\P_{\bm V}}{\Q})^{1/2}$ (see for e.g., \citep[Lemma 2.3]{tsybakov2009introduction}) the Total Variation Information can be bounded in terms of the Hellinger Information:
\begin{align*}
     \MItv{\bm V}{\bm Y}  \leq \inf_{\Q}  \int(2\hell{\P_{\bm V}}{\Q})^{1/2} \prior(\diff \bm V) &\explain{(a)}{\leq} \left(2 \inf_{\Q}  \int \hell{\P_{\bm V}}{\Q}  \prior(\diff \bm V) \right)^{1/2} \\
     & \explain{(b)}{=} \sqrt{2 \MIhell{\bm V}{\bm Y}}.
\end{align*}
In the above display step (a) follows from Jensen's Inequality and step (b) follows from the definition of Hellinger Information. Substituting the bound $\MItv{\bm V}{\bm Y} \leq \sqrt{2 \MIhell{\bm V}{\bm Y}}$ in \eqref{eq:fano-tv} immediately yields the claim of Fact~\ref{fact: fano}.

	\bibliographystyle{plainnat}
	\bibliography{ref}
\end{document}